\documentclass[11pt,a4paper]{amsart} 

\usepackage{anysize} 

\usepackage{amsthm}  
\usepackage{amsmath}
\usepackage{amsfonts}
\usepackage{amssymb}
\usepackage[foot]{amsaddr}

\usepackage{multicol} 
\usepackage{graphicx} 
\usepackage[dvipsnames]{xcolor} 
\usepackage{pgf,tikz}
\usepackage{tikz-cd}
\usepackage{mathrsfs}
\usepackage{enumerate}
\usepackage{verbatim} 
\usepackage{array}
\usepackage[pagebackref=true]{hyperref}

\hypersetup{
	colorlinks,
	linkcolor={blue!60!black},
	citecolor={blue!60!black},
	urlcolor={blue!60!black}
}

\usepackage{centernot}
\usepackage{mathtools} 
\usepackage{enumitem} 
\usepackage{mathdots} 
\usetikzlibrary{arrows}

\marginsize{2 cm}{2 cm}{2 cm}{2 cm}

\numberwithin{equation}{section}
\newtheorem{lemma}[equation]{Lemma}
\newtheorem{proposition}[equation]{Proposition}
\newtheorem{theorem}[equation]{Theorem}
\newtheorem{defprop}[equation]{Definition-Proposition}
\newtheorem{corollary}[equation]{Corollary}

\theoremstyle{definition}
\newtheorem{example}[equation]{Example}
\newtheorem{question}[equation]{Question}

\newtheorem{remark}[equation]{Remark}
\newtheorem{definition}[equation]{Definition}
\newtheorem{construction}[equation]{Construction}
\newtheorem{assumption}[equation]{Assumption}
\newtheorem{notation}[equation]{Notation}
\numberwithin{equation}{section}

\newcommand{\C}{{\mathbb C}}
\newcommand{\bH}{{\mathbb H}}
\newcommand{\bP}{\mathbb{P}}
\newcommand{\N}{{\mathbb N}}
\newcommand{\Q}{{\mathbb Q}}
\newcommand{\R}{{\mathbb R}}
\newcommand{\Z}{{\mathbb Z}}
\newcommand{\cL}{{\mathcal L}}

\newcommand{\oL}{\ov{\cL}}
\newcommand{\cS}{{\mathcal S}}
\newcommand{\cM}{{\mathcal M}}
\newcommand{\cN}{{\mathcal N}}
\newcommand{\cF}{{\mathcal F}}
\newcommand{\cH}{{\mathcal H}}

\newcommand{\cE}{{\mathcal E}}
\newcommand{\cK}{{\mathcal K}}
\newcommand{\cC}{{\mathcal C}}

\newcommand{\cA}{{\mathcal A}}

\newcommand{\cO}{{\mathcal O}}
\newcommand{\cT}{{\mathcal T}}
\newcommand{\fm}{{\mathfrak m}}
\newcommand{\eps}{\varepsilon}

\newcommand{\ul}[1]{\underline{#1}}
\newcommand{\wt}[1]{\widetilde{#1}}
\newcommand{\ov}[1]{\overline{#1}}

\newcommand{\cl}{\mathrm{cl}}
\newcommand{\lc}{{\centerdot}}
\newcommand{\h}{\log}

\DeclareMathOperator{\Homm}{\mathcal{H}om}

\DeclareMathOperator{\Spec}{Spec}

\DeclareMathOperator{\Tot}{Tot}
\DeclareMathOperator{\Gdm}{Gdm}
\DeclareMathOperator{\Tors}{Tors}
\DeclareMathOperator{\Sp}{Sp}

\DeclareMathOperator{\Hom}{Hom}
\DeclareMathOperator{\End}{End}
\DeclareMathOperator{\im}{im}
\DeclareMathOperator{\Id}{Id}
\DeclareMathOperator{\Sym}{Sym}
\DeclareMathOperator{\Gr}{Gr}
\newcommand{\inv}{{\operatorname{inv}}}
\DeclareMathOperator{\supp}{supp}
\newcommand{\piG}{\pi_1(G)}

\DeclareMathOperator{\gr}{gr}

\newcommand{\f}[5]{
	\begin{array}{rcl}
		#1\colon #2 & \longrightarrow & #3 \\
		#4 & \longmapsto & #5  \\
	\end{array}
}

\definecolor{olive}{rgb}{0.42, 0.56, 0.14}

\author{Eva Elduque}
\address[E. Elduque]{Departamento de Matem\' aticas, Universidad Aut\' onoma de Madrid and ICMAT, 28049 Madrid, Spain}
\urladdr{https://matematicas.uam.es/~eva.elduque/}
\email{eva.elduque@uam.es}
\author{Mois\'es Herrad\'on Cueto}
\address[M. Herradón Cueto]{Departamento de Matem\' aticas, Universidad Aut\' onoma de Madrid, 28049 Madrid, Spain}
\urladdr{https://matematicas.uam.es/~moises.herradon/}
\email{moises.herradon@uam.es}

\thanks{E. Elduque and M. Herradón Cueto are partially supported by the Grant PID2022-138916NB-I00 funded by MCIN/AEI/ 10.13039/501100011033 and by ERDF A way of making Europe. E. Elduque is also partially supported by the Ramón y Cajal Grant RYC2021‐031526‐I funded by MCIN/AEI /10.13039/501100011033 and by the European Union NextGenerationEU/PRTR}

\title{Hodge theory of abelian covers of algebraic varieties}

\begin{document}
	\begin{abstract}
		Motivated by classical Alexander invariants of affine hypersurface complements, we endow certain finite dimensional quotients of the homology of abelian covers of complex algebraic varieties with a canonical and functorial mixed Hodge structure (MHS). More precisely, we focus on covers which arise algebraically in the following way: if $U$ is a smooth connected complex algebraic variety and $G$ is a complex semiabelian variety, the pullback of the exponential map by an algebraic morphism $f:U\to G$ yields a covering space $\pi:U^f\to U$ whose group of deck transformations is $\pi_1(G)$. The new MHSs are compatible with Deligne's MHS on the homology of $U$ through the covering map $\pi$ and satisfy a direct sum decomposition as MHSs into generalized eigenspaces by the action of deck transformations. This provides a vast generalization of the previous results regarding univariable Alexander modules by Geske, Maxim, Wang and the authors in \cite{mhsalexander,EvaMoises}. Lastly, we reduce the problem of whether the first Betti number of the Milnor fiber of a central hyperplane arrangement complement is combinatorial to a question about the Hodge filtration of certain MHSs defined in this paper, providing evidence that the new structures contain interesting information.
	\end{abstract}

	\maketitle

	\tableofcontents
\section{Introduction}\label{s:intro}

The goal of this note is to develop a Hodge theory for certain (infinite-sheeted) covers of smooth complex algebraic varieties. Let us start by precisely defining our setting: Let $U$ be a smooth connected complex  algebraic variety. Let $G$ be a semiabelian variety. 
 Let $TG$ denote the tangent space of $G$ at the identity, and let $\exp:TG\to G$ be the exponential map of complex Lie groups, which, since $G$ is a commutative algebraic group, is the universal covering map of $G$.

Let $f\colon U\to G$ be an algebraic morphism. The map $f$ determines an abelian cover $\pi:U^f\to U$. Indeed, $\pi:U^f\to U$ is the pullback of $\exp$ by $f$, as shown in the following diagram:
\begin{equation}\label{eq:Uf}
	\begin{tikzcd}
		U^f\subset U\times TG  \arrow[r,"\wt f"] \arrow[d,"\pi"]\arrow[dr,phantom,very near start, "\lrcorner"]&
		TG \arrow[d,"\exp"] \\
		U\arrow[r,"f"] &
		G,
	\end{tikzcd}
\end{equation}
Note that the deck transformation group of $\pi:U^f\to U$ coincides with that of $\exp:TG\to G$, and is thus isomorphic to $\pi_1(G)$, a free abelian group. Hence, the homology groups $H_j(U^f,k)$ have an  $R\coloneqq k[\pi_1(G)]$-action by deck transformations for any field $k$, which in this note will be $\Q, \R$ or $\C$. Also note that, if $g$ is the rank of $\pi_1(G)$, then $R$ is (non-canonically) isomorphic to the Laurent polynomial ring on $g$ variables over $k$.

By Deligne's theory of $1$-motives \cite{DeligneIII}, there are plenty such morphisms $f$: out of the ones that give rise to connected covers, there is one for each mixed Hodge structure quotient of $H_1(U,\Q)$. Hence, the covering spaces considered in this paper are abelian covers which arise from algebraic data that is related to the Hodge theory of $U$, thus providing a natural setting for which to develop a Hodge theory for covering spaces of algebraic varieties. One such morphism $f$ is the generalized Albanese morphism \cite{Iitaka,iitaka_1977}, which yields the  universal torsion-free abelian cover of $U$ (that is, the covering space of $U$ associated to the kernel of the projection of $\pi_1(U)$ into its maximal torsion-free abelian quotient).

A well-studied particular case is that of affine hypersurface complements. Let $$U\coloneqq\C^n\setminus\bigcup_{i=1}^m V(f_i),$$ where $f_i\in\C[x_1,\ldots,x_n]$ are pairwise coprime irreducible polynomials, and let $U^f$ be the cover induced by the map $$f=(f_1,\ldots,f_n):U\to(\C^*)^n.$$
In this case, $f_*:H_1(U,\Z)\to H_1((\C^*)^n,\Z)$ is an isomorphism, and $H_j(U^f,k)$ are classical Alexander invariants of the hypersurface $H=\cup_{i=1}^m V(f_i)$ (see Example~\ref{exam:hypersurfaces}). These kinds of multivariable Alexander invariants are typically studied through their support loci, cf. \cite{DimcaMax}, \cite{libgober92}, \cite{LiuMaxCharacteristic}, \cite{suciuSurvey}. Moreover, each morphism $\pi_1(U)\to\Z^l$ is the morphism induced at the level of fundamental groups of an algebraic morphism $U\to(\C^*)^l$ which factors through $f$. 

Let us note that, unless $G$ is a point, $\pi:U^f\to U$ is an infinite-sheeted cover, and $U^f$ is a complex analytic manifold which in general is not an algebraic variety (nor has the homotopy type of a finite CW complex). Moreover, the homology groups $H_j(U^f,k)$ are finitely generated $R$-modules, so their dimension as $k$-vector spaces is countable, but it will not be finite in general. If $\dim_k H_j(U,k)=\infty$, the dimension of its $k$-dual $H^j(U^f,k)$ will not be countable, and thus $H^j(U^f,k)$ will not be a finitely generated $R$-module. For this reason, and even if Deligne's mixed Hodge theory of algebraic varieties arises in cohomology rather than homology, we will focus on the homology groups of $U^f$ throughout this note.

Even though they are finitely generated over \(R\), the homology groups of $U^f$ can be infinite dimensional, so in order to develop a Hodge theory for them, we need to extract finite dimensional spaces from them. Since the $R$-action is by deck transformations and we want our theory to reflect the fact that $\pi:U^f\to U$ is a covering space, we will also want these finite dimensional spaces to have a natural $R$-module structure. There are two natural ways to do this:
\begin{enumerate}
	\item\label{approach1} Focus on finite dimensional $R$-submodules of $H_j(U^f,k)$:
	\begin{itemize}
		\item If $G=\C^*$, then $R\cong k[t^{\pm 1}]$, so $R$ is a principal ideal domain. Hence, the $R$-module $H_j(U^f,k)$ has a canonical direct sum decomposition into its free part and its torsion part. In particular, $\Tors_R H_j(U^f,k)$ is the maximal $R$-submodule of $H_j(U^f,k)$ which is a finite dimensional $k$-vector space. At this level of generality, a Hodge theory for these torsion submodules was developed in \cite{mhsalexander} (see also \cite{mhsSurvey} for a survey of the main results therein), although there had been prior constructions of MHSs on $\Tors_R H_j(U^f,k)$ in some special situations \cite{DL,HainHomotopy,Liu,KK,Lib96} (see the introduction of \cite{mhsalexander} for a description of the particular cases).
		\item If $G$ is not isomorphic to $\C^*$, then $R$ is not a principal ideal domain, and $H_j(U^f,k)$ no longer decomposes into its free part and its torsion part. However, by analogy with the $G=\C^*$ case, one could still focus on the maximal Artinian submodule of $H_j(U^f,k)$, which is the maximal submodule of $H_j(U^f,k)$ which is a finite dimensional $k$-vector space. If $G\cong(\C^*)^n$ for some $n\geq 1$, this was the approach that was taken in \cite{mellin}, although not for $H_j(U^f,k)$ but for the cohomological Alexander modules defined therein.
	\end{itemize}
	\item\label{approach2} Focus on finite dimensional $R$-module quotients of $H_j(U^f,k)$: This is the approach we take in this paper. More concretely, for every finite index subgroup $H\leq \pi_1(G)$, let\linebreak $\fm_H\coloneqq \left(\gamma-1\mid\gamma\in H\right)\subset k[H]$ be the augmentation ideal of $k[H]\subset R$. Then, the quotient $\frac{H_j(U^f,k)}{(\fm_H)^m H_j(U^f,k)}$ is finite dimensional for all $m\geq 1$ and all $j\geq 0$. The goal of this paper is to endow these quotients with canonical mixed Hodge structures (MHSs). 
\end{enumerate}

These kinds of quotients have interesting applications. For example, they were used in  \cite{Artal-Rybnikov} by Artal Bartolo, Carmona Ruber, Cogolludo Agustín and Marco Buzunáriz to give a proof of the fact that Rybnikov's pair of combinatorially equivalent projective line arrangements from \cite{rybnikov} have non-isomorphic fundamental groups. In their proof, they use objects such as $\frac{H_1(U^f,k)}{\fm^2 H_1(U^f,k)}$ (but with $\Z$ coefficients), where $U^f$ is the universal abelian cover of a line arrangement complement and $\fm$ is the augmentation ideal of $R$.

\begin{remark}
 Let us further justify our choice of quotients to do Hodge theory on. For these quotients to be finite dimensional, they need to be supported in a finite number of points of $\Spec R\cong (\C^*)^g$, where $\Spec$ denotes the maximal spectrum. By \cite[Theorem 2.5]{PSLoci}, we know that, for all $q\geq 0$,
\begin{equation}\label{eq:loci}
	\bigcup_{j\leq q}\supp H_j(U^f,\C)=\bigcup_{j\leq q} (f^*)^{-1}\left(\mathcal V_j(U)\right),
\end{equation}
where $f^*:\Spec R\cong\Hom(\pi_1(G),\C^*)\to \Hom(\pi_1(U),\C^*)$ is the map induced by $f$, and $$\mathcal V_j(U)=\{\text{rank $1$ $\C$-local systems $L$ on $U$}\mid H_j(U,L)\neq 0\}$$
is the $j$-th homology jump loci. Using the structure theorem of (co)homology jump loci \cite[Theorem 1.4.1]{budur2017absolute} one can show that the right hand side of \eqref{eq:loci} is a finite union of torsion translated subtori in $\Spec R$. Hence, when deciding which finite dimensional quotients of $H_j(U^f,\C)$ to study, a natural choice is to force them to be supported at the interesting torsion points, that is, those corresponding to the torsion-translated irreducible components of $\cup_{j\geq 0} \supp H_j(U^f,\C)$. This is precisely what we achieve by looking at $\frac{H_j(U^f,k)}{(\fm_H)^m H_j(U^f,k)}$ for an appropriate choice of $H\leq \pi_1(G)$. Furthermore, by making $m$ grow, we get larger and larger quotients which are supported in the same finite set of points. This corresponds to looking at an infinitesimal neighborhood of these torsion points.
\end{remark}

In principle, approaches~\eqref{approach1} and~\eqref{approach2} might seem unrelated. However, in the case when $G=\C^*$, approach~\eqref{approach2} generalizes approach~\eqref{approach1} as follows.

\begin{remark}[Generalization of \cite{mhsalexander}]\label{rem:generalization}
	Suppose that $G=\C^*$, and identify $R$ with $k[t^{\pm 1}]$. By \cite[Proposition 2.24]{mhsalexander} (based on \cite[Proposition 4.1]{BudurLiuWang}), there exists $N\in \N$ such that $\Tors_R H_j(U^f,k)$ is annihilated by a big enough power of $t^N-1$ for all $j\geq 0$. Hence, for $m\gg1$, there are canonical inclusions
	\begin{equation}\label{eq:relationTorsion}
	\Tors_R H_j(U^f,k)\hookrightarrow \frac{H_j(U^f,k)}{(t^N-1)^m H_j(U^f,k)}
\end{equation}
	for all $j\geq 0$. In \cite{mhsalexander}, $\Tors_R H_j(U^f,k)$ is endowed with a canonical and functorial MHS, but this shows that the MHS on $\frac{H_j(U^f,k)}{(t^N-1)^m H_j(U^f,k)}$ defined in this paper corresponding to the subgroup $H=\langle N\rangle\subset\Z=\pi_1\left((\C^*)^n\right)$  sees more than just the $R$-torsion, it also sees more and more of the free part as we increase the values of $N$ and $m$.
	
	In fact, we show in \cite{compatibility} that the morphism \eqref{eq:relationTorsion} is a morphism of MHSs, so the theory developed in this paper extends the theory developed in \cite{mhsalexander}.
\end{remark}

\begin{remark}[Comparison with \cite{mellin}]
Definition~\ref{def:MHSalexander} endows certain quotients of the cohomology Alexander modules considered in \cite{mellin} with a canonical MHS. However, we will not try to address how the MHS on the maximal Artinian submodules of the cohomology Alexander modules from \cite{mellin} relates to the MHS defined in this note, as the techniques used to define them are very different from one another. The reader may consult \cite[Section 1.4]{mellin} for an explanation of the main differences between~\cite{mhsalexander} and~\cite{mellin} regarding the scope and the methods used.
\end{remark}

Let us note that the MHSs found in earlier work \cite{mhsalexander,mellin} following approach~\eqref{approach1} have applications that go beyond Hodge theory. For example, in \cite{mhsalexander}, the existence and properties of the MHS on $\Tors_R H_j(U^f,k)$  give a bound on the size of the Jordan blocks of $\Tors_R H_j(U^f,k)$ for the $t$-action (see \cite[Corollary 7.20]{mhsalexander}), which in particular implies that $\Tors_R H_1(U^f,k)$ is always a semisimple $R$-module. This was unknown in this sort of generality before, see \cite[Corollary 1.7]{DL} for the case of affine curve complements. A similar bound was obtained for the Jordan blocks of the action of any element of $\pi_1\left((\C^*)^n\right)$ on the maximal Artinian submodules of the cohomological Alexander modules considered in \cite{mellin} (see \cite[Corollary 1.7(c)]{mellin}).

This note is devoted to developing a Hodge theory following approach~\eqref{approach2} which generalizes the theory developed in \cite{mhsalexander} (approach~\eqref{approach1}). The focus is on providing structural results rather than investigating possible applications outside of Hodge theory, which, given the success of the previous approaches, remains a topic for further research.

\subsection{Summary of the main results}\label{sec:summary}

In this paper we prove the following statement, which provides a generalization of \cite[Theorem 1.0.2]{mhsalexander}:

\begin{theorem}\label{thm:mainintro}
	Let $U$ be a smooth connected complex algebraic variety, let $G$ be a semiabelian variety whose tangent space at the identity is denoted by $TG$ and let $f:U\to G$ be an algebraic morphism. Denote by
	$$
	\f{\pi}{U^f\coloneqq\{(u,z)\in U\times TG\mid f(u)=\exp(z)\}\subset U\times TG}{U}{(u,z)\quad\quad\quad\quad\quad\quad\quad\quad}{u}
	$$
	the corresponding cover of $U$, with deck group isomorphic to $\pi_1(G)$, which is a finitely generated free abelian group. Let $R=k[\pi_1(G)]$, for $k=\Q$ or $\R$. Let $H\leq G$ be a finite index subgroup, and let $\fm$ be the augmentation ideal of $R$. Let $j\geq 0$ and $m\geq 1$. The following statements hold:
	\begin{enumerate}
		\item\label{part1} If $k=\R$, $\frac{H_j(U^f,k)}{\fm^m H_j(U^f,k)}$ carries a canonical $k$-MHS (see Definition~\ref{def:MHSalexander}). If $G\cong(\C^*)^n$ for some $n\geq 1$, then this holds for $k=\Q$ too (see Corollary~\ref{cor:QMHS}).
		\item\label{part2} Let $k$ be as in part~\eqref{part1}, and let $m'\geq m$. Then, the projection morphism
		$$
		\frac{H_j(U^f,k)}{\fm^{m'} H_j(U^f,k)}\twoheadrightarrow \frac{H_j(U^f,k)}{\fm^m H_j(U^f,k)}
		$$
		is a MHS morphism (see Remark~\ref{rem:projMHS}).
		\item\label{part3} Let $k$ be as in part~\eqref{part1}. 	For all $\gamma\in\pi_1(G)$, let $\log\gamma\in H_1(G,\Z)$ be the element corresponding to $\gamma$ via the abelianization map. Consider the multiplication map, defined as the only $k$-linear map satisfying
		$$
		\begin{array}{ccc}
			\displaystyle H_1(G,k)\otimes_k \frac{H_j(U^f,k)}{\mathfrak{m}^mH_j(U^f,k)} & \longrightarrow &
      \displaystyle\frac{H_j(U^f,k)}{\mathfrak{m}^mH_j(U^f,k)}\\
      \displaystyle \log \gamma \otimes v &
      \longmapsto &
      \displaystyle\log(\gamma)\cdot v \coloneqq -\sum_{i=1}^{m-1} \frac{(1-\gamma)^i\cdot v }{i}
		\end{array}
		$$
		for all $\gamma\in \pi_1(G)$ and all $v\in \frac{H_j(U^f,k)}{\mathfrak{m}^mH_j(U^f,k)}$. Then, this map  is a MHS morphism.
		\item\label{part4} Let $H$ be a finite index subgroup of $\pi_1(G)$, and let
		$$\fm_H\coloneqq \left(\gamma-1\mid \gamma\in H\right)$$ be the augmentation ideal of $k[H]\subset R$.  Then, the results in parts~\eqref{part1}--\eqref{part3} hold if we substitute $\fm^mH_j(U^f,k)$ by $(\fm_H)^m H_j(U^f,k)$, and $H_1(G,k)$ by $H_1(G_H,k)$, where $G_H\to G$ is the covering space associated to $H\leq \piG$ (see Proposition~\ref{prop:otherideals}).
		\item\label{part5} Let $K_2\leq K_1\leq G$ be a sequence of finite index subgroups. Then, the natural projection
		$$
		\frac{H_j(U^f,k)}{(\mathfrak{m}_{K_2})^mH_j(U^f,k)}\twoheadrightarrow \frac{H_j(U^f,k)}{(\mathfrak{m}_{K_1})^mH_j(U^f,k)}
		$$
		is a MHS morphism  (see Proposition~\ref{prop:proOtherIdeals}).
	\end{enumerate}
\end{theorem}

Moreover, the MHS from Theorem~\ref{thm:mainintro} is functorial in the following sense, both in the domain and the target of $f:U\to G$ (see Theorem~\ref{thm:functorial}  combined with Proposition~\ref{prop:otherideals} and  Corollary~\ref{cor:QMHS}).

\begin{theorem}[Functoriality]\label{thm:functorialintro}
Let $U_1,U_2$ be smooth connected complex algebraic varieties, and let $G_1,G_2$ be semiabelian varieties. Consider a commutative diagram of algebraic morphisms (below, on the left hand side), where $\rho$ is a group homomorphism. 
\begin{equation}\label{eq:baseFunctorialityintro}
	\begin{tikzcd}[row sep=2em]
		U_1
		\arrow[d,"f_1"]
		\arrow[r,"g"]
		& U_2
		\arrow[d,"f_2"]
    &&
	U_1^{f_1}
	\arrow[d,"\widetilde{f_1}"]
	\arrow[r,"\wt g"]
	& U_2^{f_2}
    \arrow[d,"\wt f_2"]
    \\
		G_1
		\arrow[r,"\rho"]
		&
		G_2
    &&
	TG_1
	\arrow[r,"\wt \rho"]
	&
	TG_2
\end{tikzcd}
\end{equation}
On the right hand side, $\wt\rho$ is the unique lift of $\rho$ which is an additive group homomorphism, $\wt{f_1}$ and $\wt{f_2}$ are defined from the pullback diagrams as in \eqref{eq:Uf}, and $\wt g$ is the unique lift of $g$ that makes the diagram commute.

Let $k=\R$ unless both $G_1$ and $G_2$ are affine tori, in which case we may take $k=\Q$. For \(i=1,2\) and for all finite index subgroups $K_i\leq \pi_1(G_i)$, let $\fm_{K_i}$ be the augmentation ideal of $k[K_i]\subset k[\pi_1(G_i)]$.

Under these assumptions, the map $\wt g:U_1^{f_1}\to U_1^{f_2}$ induces MHS morphisms
$$
\wt g_{*,m}: \frac{H_j(U_1^{f_1},k)}{(\fm_{K_1})^mH_j(U_1^{f_1},k)}\to \frac{H_j(U_1^{f_1},k)}{(\fm_{K_2})^mH_j(U_1^{f_1},k)}
$$
 for all $j\geq 0$, $m\geq 1$ and all finite index subgroups $K_1\leq \pi_1(G_1)$ and $K_2\leq \pi_1(G_2)$ such that $\rho_*(K_1)\leq K_2$.
\end{theorem}

Note that this is a more general version of the functoriality found in \cite[Theorem 5.4.9]{mhsalexander}, which corresponds to the diagram~\eqref{eq:baseFunctorialityintro} in the case where $G_1=G_2=\C^*$ and $\rho=\Id$. In other words, while in this paper our MHS behaves functorially in both the domain and the target of $f:U\to G$, the MHS in \cite{mhsalexander} is only functorial in the domain. Because of this more general functoriality, we obtain the following compatibility with Deligne's MHS as a Corollary of Theorem~\ref{thm:functorialintro}, by making $G_2$ be a point (see Corollary~\ref{cor:compatibleDeligne}  combined with Corollary~\ref{cor:QMHS}). This is a generalization of \cite[Theorem 1.0.3]{mhsalexander}, but the proof here is much simpler due to the extra functoriality features in this paper.

\begin{corollary}[Compatibility with Deligne's MHS]\label{cor:compatibleDeligneintro}
	Let $U$ be a smooth connected complex algebraic variety, let $G$ be a semiabelian variety, and let $f:U\to G$ be an algebraic morphism. Let $k=\R$ unless $G$ is isomorphic to an affine torus, in which case we may take $k=\Q$. Let $H$ be a finite index subgroup of $\pi_1(G)$, and let $\fm_H$ be the augmentation ideal of $k[H]\subset k[\pi_1(G)]$.
	
	Then, the covering space map $\pi:U^f\to U$
	induces the MHS morphism
	$$
	\frac{H_j(U^f,k)}{(\fm_H)^m H_j(U^f,k)}\to H_j(U,k)
	$$
	for all $j\geq 0$ and all $m\geq 1$, where $H_j(U,k)$ is endowed with Deligne's MHS.
\end{corollary}

By Theorem~\ref{thm:mainintro} part~\eqref{part3}, the logarithm of deck transformations behaves well with respect to the MHS. Let $\gamma\in\pi_1(G)$, which we interpret as a deck transformation of $\pi:U^f\to U$. In general, $\gamma$ does not preserve the MHS, but its semisimple part does, as exemplified in the following result (see Theorem~\ref{thm:ss} combined with Corollary~\ref{cor:QMHS}), which provides a generalization of~\cite[Theorem 1.3]{EvaMoises}.
\begin{theorem}\label{thm:ssintro}
	Let $U$ be a smooth connected complex algebraic variety, let $G$ be a semiabelian variety, and let $f:U\to G$ be an algebraic morphism. Let $k=\R$ unless $G$ is isomorphic to an affine torus, in which case we may take $k=\Q$. Let $H$ be a finite index subgroup of $\pi_1(G)$, let $\gamma\in\pi_1(G)$, and let $\mathfrak{m}_H$ be the augmentation ideal of $k[H]$. Let $\gamma=\gamma_{ss}  \gamma_{u}$ be the Jordan-Chevalley decomposition of $\gamma$ acting on $\frac{H_j(U^f,k)}{(\mathfrak m_H)^m H_j(U^f,k)}$ as the product of a semisimple and a unipotent operator that commute with each other. Then,
	\[
	\gamma_{ss}  \colon \frac{H_j(U^f,k)}{(\mathfrak m_H)^m H_j(U^f,k)}\to \frac{H_j(U^f,k)}{(\mathfrak m_H)^m H_j(U^f,k)}
	\]
	is a MHS isomorphism for all $j\geq 0$ and all $m\geq 1$.
\end{theorem}

As a consequence of the theorem above, we obtain that the direct sum decomposition of the quotient $\frac{H_j(U^f,\C)}{(\mathfrak m_H)^m H_j(U^f,\C)}$ into its generalized eigenspaces by the action of $\gamma$ is a MHS decomposition, see Corollary~\ref{cor:eigenspaceDecomposition}, which also contains a version for $\R$-coefficients that is extended to $\Q$-coefficients by Corollary~\ref{cor:QMHS}.

\begin{remark}\label{rem:kahler}
	All of the results in this paper for $\R$ and $\C$ coefficients also hold for holomorphic maps $f:U\to G$, where $U$ is a compact Kähler manifold and $G$ a (compact) complex torus which are not necessarily algebraic. Indeed, following the notation of Definition~\ref{def:NavarroAznar},
	$$
	\left(\cA_{U,\R}^\bullet,\quad\left(\cA_{U,\C}^\bullet, F^{\lc}\right),\quad\alpha\right)
	$$
	is a (pure) Hodge complex of weight $0$ which endows the cohomology of $U$ with the usual pure Hodge structure of compact Kähler manifolds (\cite[Théorème 8.8]{navarroAznar}, \cite[Example 2.34]{peters2008mixed}). The constructions of Definition-Proposition~\ref{defPhiPsiY} and Definition~\ref{def:endowedMHS} can be carried out in the exact same way in this setting, and the remaining results in the paper follow from this.
\end{remark}

Lastly, we use the celebrated result of Budur and Saito \cite{budurSaito} on the combinatorial nature of the spectrum of a hyperplane arrangement to reduce the open problem of whether the first Betti number of the Milnor fiber of a central hyperplane arrangement complement in $\C^n$ is combinatorial to a question about the MHSs defined in this paper (see Corollary~\ref{cor:summary}). Specifically, it is reduced to a question about the combinatorial nature of the Hodge filtration of  $\frac{H_2(U^f,\C)}{\fm_HH_2(U^f,\C)}$, where $U$ is an essential line arrangement complement in $\C^2$ of three or more lines, $f:U\to\C^*$ is the defining (reduced) polynomial of the arrangement, and $H$ is a subgroup of $\pi_1(\C^*)$ which is determined by the combinatorial data of the arrangement (see Lemma~\ref{lem:reduction} and Theorem~\ref{thm:combinatorics}). We highlight a couple of aspects of this reduction:
\begin{itemize}
	\item In this case, $H_2(U^f,\C)$ is a free $R=\C[t^{\pm 1}]$-module. Hence, even if $G=\C^*$, the MHS of $\frac{H_2(U^f,\C)}{\fm_HH_2(U^f,\C)}$ is completely new from this paper, as \cite{mhsalexander} only dealt with the torsion part of the homology of $U^f$.
	\item The rank of $H_2(U^f,\C)$ as a free $R=\C[t^{\pm 1}]$-module is determined by the combinatorics of the essential line arrangement, which we denote by $\mathcal H$. Hence, the dimension of $\frac{H_2(U^f,\C)}{\fm_HH_2(U^f,\C)}$ is also determined by combinatorial data of $\mathcal H$. The work in Section~\ref{sec:hyperplanes} shows that, even if $\frac{H_2(U^f,\C)}{\fm_HH_2(U^f,\C)}$ is well understood, its MHS contains interesting information.
	\item The MHS on $\frac{H_2(U^f,\C)}{\fm_HH_2(U^f,\C)}$ only has three non-trivial graded pieces by the Hodge filtration, and we show that the dimension of the middle piece is also determined by the combinatorics of $\mathcal H$. Specifically, Theorem~\ref{thm:combinatorics} reduces the problem of whether the first Betti number of the Milnor fiber of a central hyperplane arrangement complement is combinatorial to the question of whether $\dim_\C F^0 \frac{H_2(U^f,\C)}{\fm_HH_2(U^f,\C)}$ is determined by the combinatorics of $\mathcal H$ for every essential line arrangement $\mathcal H$ in $\C^2$.
\end{itemize}

The last point motivates further work regarding the development of techniques that allow the computation of examples of the MHS defined in this paper (or at the very least of its Hodge filtration). This note is devoted to proving structural results and developing a new theory, not the computation of examples. However, using Remark~\ref{rem:generalization}, note that the examples from \cite[Chapter 10]{mhsalexander} regarding affine hyperplane arrangement complements are also examples of the MHS $\frac{H_j(U^f,\Q)}{\fm_H H_j(U^f,\Q)}$ from this paper for suitable $H$, since in those cases $H_j(U^f,\Q)$ was a semisimple torsion module for the chosen $j$. Similarly, the results from \cite[Sections 5 and 6]{EvaMoises} (the ones which help with the computation of the MHS in \cite{mhsalexander} in cases such as when $U$ is formal or the affine complement of a hypersurface which is transversal at infinity) also apply to the MHS $\frac{H_j(U^f,\Q)}{\fm_H H_j(U^f,\Q)}$ for suitable $H$, since $H_j(U^f,\Q)$ is semisimple and torsion for all but one $j$.

\subsection{Outline of the paper}
This paper provides a vast generalization of the main results in \cite{mhsalexander} and in \cite{EvaMoises}. More precisely, Sections~\ref{sec:thickening} to~\ref{sec:otherideals} and Section~\ref{s:comparison} generalize the results in \cite[Chapters 3--6]{mhsalexander}, which are precisely the chapters which have clear analogies in this general setting (they do not depend on $G$ being $\C^*$), and Section~\ref{sec:eigenspace} generalizes \cite{EvaMoises}.

In Section~\ref{sec:prelim} we recall the relevant background and set notations for the rest of the paper. We review results regarding semiabelian varieties and Albanese morphisms, homology of abelian covers, local systems and how to interpret the homology groups of $U^f$ as the homology groups of a local system $\cL$ on $U$, the compactifications of algebraic varieties that will be used throughout this paper, differential graded algebras (both cdga's and dgla's), and mixed Hodge complexes of sheaves. The latter include the analytic logarithmic Dolbeault mixed Hodge complex of sheaves from Navarro Aznar \cite{navarroAznar}, which endows the cohomology of smooth complex  algebraic varieties with the same MHS as Deligne (which he obtained using holomorphic logarithmic forms).

Sections~\ref{sec:thickening}--\ref{sec:thickeningResolves} provide the theoretical framework needed to develop a Hodge theory for $U^f$. In Section~\ref{sec:thickening} we describe a general procedure to obtain mixed Hodge complexes of sheaves as ``thickenings'', i.e. infinitesimal deformations, of other known mixed Hodge complexes of sheaves. Very roughly speaking, the construction amounts to tensoring the complexes of sheaves by a mixed Hodge structure and twisting the differentials. In Section~\ref{sec:thickDolbeault}, we give an explicit description of how to perform suitable thickenings of the analytic logarithmic Dolbeault mixed Hodge complex of sheaves. In Section~\ref{sec:thickeningResolves} we show that these thickenings realize certain truncated local systems obtained from $\cL$.

In Section~\ref{sec:MHS} we start by endowing the cohomology of the aforementioned truncated local systems with MHSs (see Definition~\ref{def:endowedMHS}), show that those MHS are independent of the choices used in their construction via mixed Hodge complexes of sheaves, and finally arrive at the first three parts of Theorem~\ref{thm:mainintro} for $k=\R$ in Section~\ref{sec:MHSAlexander}. We also endow other related (co)homology groups with canonical MHSs in Section~\ref{sec:MHSAlexander}.

Section~\ref{sec:functoriality} is devoted to proving the version of the functoriality Theorem~\ref{thm:functorialintro} for $k=\R$, $K_1=\pi_1(G_1)$ and $K_2=\pi_1(G_2)$, that is, when $\fm_{K_i}$ is the augmentation ideal of $\R[\pi_1(G_i)]$ (Theorem~\ref{thm:functorial}), as well as the corresponding version of  Corollary~\ref{cor:compatibleDeligneintro}  (Corollary~\ref{cor:compatibleDeligne}).  We also study how the statement of Theorem~\ref{thm:functorialintro} changes if $\rho$ in diagram~\eqref{eq:baseFunctorialityintro} is not a group homomorphism, but just a morphism of algebraic varieties (see Theorem~\ref{thm:dependenceTranslation}). This has the following interesting consequence: if $f$ is the generalized Albanese morphism, then the MHSs obtained in this paper are invariants of the topology and the algebraic structure of $U$, but its isomorphism class does not depend on the choice of  $f$ (which is defined up to translation and isomorphism of semiabelian varieties), see Example~\ref{eg:algebraicInvariant}.

In Section~\ref{sec:otherideals} we prove parts~\eqref{part4} and~\eqref{part5} of Theorem~\ref{thm:mainintro} in the case where $k=\R$. The results in this section can be used to immediately generalize the main results in Section~\ref{sec:functoriality} to the form they have in Section~\ref{sec:summary} (namely Theorem~\ref{thm:functorialintro} and Corollary~\ref{cor:compatibleDeligneintro}) (for  $k=\R$).

Theorem~\ref{thm:ssintro} is proved in Section~\ref{sec:eigenspace} in the case where $k=\R$. The consequent eigenspace decomposition appears in full detail in Corollary~\ref{cor:eigenspaceDecomposition}.

In Section~\ref{s:comparison} we show that the different MHSs defined in Section~\ref{sec:MHS} are in fact defined over $\Q$ if $G\cong(\C^*)^n$ for some $n\geq 1$ (Corollary~\ref{cor:QMHS}). The reason for this distinction is that we perform a thickening of a particular $\Q$-mixed Hodge complex of sheaves and show that it computes the cohomology of the local systems used in this paper, but this specific construction cannot be carried out if $G$ is not an affine torus. We expect the result to be true in general, but were unable so far to find an explicit description of a multiplicative $\Q$-mixed Hodge complex of sheaves that we could use to perform the needed thickenings. The construction of such thickenings over $\Q$ would involve fixing a particular morphism relating the appropriate $\Q$-local systems to the corresponding thickened complexes in a compatible way with the construction over $\R$.

In Section~\ref{sec:hyperplanes} we discuss the applications to the study of Milnor fibers of hyperplane arrangements discussed above.

\subsection{Summary of the techniques and new insights}
Roughly speaking, the strategy  in both \cite{mhsalexander} and this paper is as follows. First, we interpret $H_j(U^f,k)$ as the $j$-th homology group of a rank $1$ local system of free $R$-modules $\cL$ on $U$. We also consider the $R$-dual local system $\oL$ of $\cL$. The local system $\oL$ has infinite dimensional stalks, so we truncate it by quotienting by powers of the augmentation ideal. Then, we create a mixed Hodge complex of sheaves which endows the cohomology groups of these truncated local systems with canonical mixed Hodge structures. These MHSs are used to endow the desired objects ($\Tors_R H_j(U^f,k)$ in \cite{mhsalexander}, where $G=\C^*$, or the aforementioned quotients of $H_j(U^f,k)$ in this paper) with canonical MHSs in different ways. In both cases, the mixed Hodge complexes of sheaves that we use are obtained by thickening known mixed Hodge complexes of sheaves which endow the cohomology of $U$ (or of a finite cover of $U$) with Deligne's MHS. This thickening process consists on tensoring the complexes of sheaves by a finite dimensional vector space $V$ (endowed with a MHS) and  by twisting the differentials.

However, although in principle the techniques seem similar in both papers, there are several new key  technical insights in this note that make the generalization possible:
\begin{itemize}
	\item In \cite{mhsalexander} $V$ was chosen to be $k[t^{\pm 1}]/(t-1)^m$ and the weight and Hodge filtrations were defined by hand. This was possible because we picked coordinates in $G=\C^*$. However, in this paper $G$ can be any semiabelian variety, and we go further than in \cite{mhsalexander} and also explore connections between these different mixed Hodge structures that arise from morphisms between the corresponding semiabelian varieties (see Theorem~\ref{thm:functorialintro}). Hence, we use a coordinate-free description of $V$ (see Remark~\ref{rem:MHSRm} for the definition of the MHS $R_m$ and $R_{-m}$ used in place of $V$ in this paper) and construct the thickened complexes without fixed coordinates (see Sections~\ref{sec:thickDolbeault} and~\ref{sec:thickeningResolves}). As a result of this generality, the new coordinate-free construction is more involved than the construction in \cite{mhsalexander}, see Definition-Proposition~\ref{defprop:PhiPsiG} to illustrate this. However, the choice of the maps used to construct the thickenings can be made explicit by picking coordinates in the case when $G$ is an affine torus (see Section~\ref{s:comparison}).
		
	\item Deligne's mixed Hodge complex of sheaves considered in \cite{mhsalexander}, which consists on logarithmic holomorphic forms on a compactification of $U$, could not be used in the new construction, as it does not contain enough forms to construct a thickening representing the truncations of $\oL$ if $G$ is a semiabelian variety in general (not isomorphic to $(\C^*)^n$).  For this reason, we need to consider a mixed Hodge complex of sheaves due to Navarro Aznar consisting on logarithmic analytic forms on a compactification of $U$ instead.
	
	\item $R$ is a principal ideal domain in the case when $G=\C^*$, so the Universal Coefficient Theorem is available. We use it in \cite[Proposition 2.14]{mhsalexander} to obtain that $\Tors_R H^{j+1}(U,\oL)$ (the torsion part of the $(j+1)$-th cohomology Alexander module) is canonically isomorphic to the $k$-dual of $\Tors_R H_j(U^f,k)$ (the torsion part of the $j$-th homology Alexander module). The MHS on $\Tors_R H_j(U^f,k)$ is endowed through this isomorphism by a MHS on $\Tors_R H^{j+1}(U,\oL)$ obtained using the aforementioned methods. However, there is no analogous duality between the objects we consider in this paper in homology and cohomology, namely, quotients of $H_*(U^f,k)$ and $H^*(U,\oL)$ by powers of the augmentation ideal. To overcome this, we need to define different mixed Hodge complexes of sheaves to study quotients of $H_*(U^f,k)$ and of $H^*(U,\oL)$ respectively: the mixed Hodge complex of sheaves that we use for homology and cohomology are thickenings of the same mixed Hodge complex of sheaves, but the MHSs $R_{-m}$ and $R_m$ that we tensor by are different and dual to one another.
	\end{itemize}

The lack of duality in this paper mentioned in the last point turns out to be a blessing in disguise: In \cite{mhsalexander}, the statements regarding the behavior of the MHS with respect to morphisms which had a natural geometric interpretation for homology Alexander modules but not for cohomology Alexander modules were very difficult to prove. The difficulty stemmed from the fact that the MHS had a natural interpretation in cohomology but not in homology, and the duality map used to define the MHS in homology was explicit but not easy to work with. Examples of these kinds of results are the compatibility with Deligne's MHS on $U$ \cite[Theorem 6.1]{mhsalexander} or the independence of the MHS of the choice of finite cover of $U$ used in the construction \cite[Theorem 5.22]{mhsalexander}. However, the generalization of these results in this paper (namely Corollary~\ref{cor:compatibleDeligne} and Proposition~\ref{prop:proOtherIdeals}) have much simpler proofs due to the fact that the mixed Hodge complex of sheaves constructed in this paper was designed for homology. As a result, and although the construction of the MHS in this note is longer, this paper is shorter in length than \cite{mhsalexander} despite providing a vast generalization of the main results of loc. cit. (Sections~\ref{sec:thickening} to~\ref{sec:otherideals} in this paper) and also of the main result of \cite{EvaMoises} (Section~\ref{sec:eigenspace}).

Lastly, we want to address a possible connection with the works of Hain and Zucker \cite{HainZuckerVMHS,HainHomotopy}, which is also related to the work of Sullivan \cite{sullivan} and Morgan \cite{morgan} (see \cite[Remark 9.25]{peters2008mixed} for the relation). The tautological variations of mixed Hodge structures (VMHS) of \cite{HainZuckerVMHS} (whose stalks are MHSs defined in \cite{HainHomotopy}) for a semiabelian variety $G$ have $R/\fm^m\otimes_R\cL_G$ as underlying local systems, where $m\geq 1$, $\fm$ is the augmentation ideal of $R$ and $\cL_G=\exp_! \ul k_{TG}$. In this paper, the truncated local systems that we consider in order to endow quotients of $H_j(U^f,k)$ by powers of $\fm$ with canonical MHSs are $R_{-m}\otimes_R\oL$, which, by Remarks~\ref{rem:RmVSRmodm} and~\ref{rem:homologyVsCohomology}, are the $k$-dual local systems to $f^{-1}(R/\fm^m\otimes_R\cL_G)$ for each $m\geq 1$. The following questions remain open:
\begin{itemize}
	\item Can the mixed Hodge complexes of sheaves defined in this paper be used to endow $R_{-m}\otimes_R\oL$ with the structure of an admissible VMHS on $U$? If so, since the cohomology of an admissible VMHS is endowed with a MHS, does this MHS coincide with the MHS on $H^*(U,R_{-m}\otimes_R\oL)$ from Definition~\ref{def:endowedMHS}? We note that these MHSs on $H^*(U,R_{-m}\otimes_R\oL)$ are the ingredients needed to endow the quotients of $H_*(U^f,k)$ by powers of the augmentation ideal with canonical MHSs (see Definition~\ref{def:MHSalexander}).
	\item Is the MHS on $H^*(U, f^{-1}(R/\fm^m\otimes_R\cL_G))$ endowed through Hain and Zucker's work related to the MHS from Definition~\ref{def:endowedMHS}? If so, how?
\end{itemize}

 Both techniques have different scopes, so establishing a relationship between the two as in the previous two questions (which pertain to their common intersection) would potentially expand their respective strengths. The tautological VMHS of Hain and Zucker is a very general construction which is a key player in the classification of admissible unipotent VMHS on a smooth quasi-projective variety \cite[Theorem 1.6]{HainZuckerVMHS}. On the other hand, the thickening process described in Section~\ref{sec:thickening} can be applied to any mixed Hodge complex of sheaves, not just those resolving a local system potentially underlying a VMHS. The definition of the tautological VMHS makes heavy use of Chen's iterated integrals and the bar construction. Its description in the case of $(\C^*)^n$ is explicit, but not in general. If a relation between Hain and Zucker's construction and our work could be shown, the thickening construction in this paper could yield an alternative interpretation of Hain and Zucker's construction for semiabelian varieties. We want to note that the explicit description of the thickened mixed Hodge complexes of sheaves defined in this paper is heavily used in many of the proofs. One example of this is the proof of Theorem~\ref{thm:ssintro}, which to the best of our knowledge has no analogue for the tautological VMHS of Hain and Zucker. The explicit nature of the thickened complex in \cite{mhsalexander} also allowed us to stablish a relationship between the MHS on univariable ($G=\C^*$) Alexander modules defined therein and the limit MHS constructed using the nearby cycles functor (see \cite[Theorem 1.8, Theorem 9.8]{mhsalexander}).

\subsection*{Acknowledgements} The authors would like to thank Richard Hain, Lauren\c{t}iu Maxim, Jörg \linebreak Schürmann, Alex Suciu and Botong Wang for helpful discussions.

\section{Preliminaries}\label{sec:prelim}
\subsection{Semiabelian varieties}

We begin this section recalling the Chevalley decomposition of complex connected algebraic groups (cf.  \cite[Theorem 1.1]{conrad}).

\begin{theorem}[Chevalley decomposition]\label{thm:chevalley}
	Let $G$ be a complex connected algebraic group. Then there exists a
	unique normal affine algebraic closed subgroup $H$ of $G$ for which $G/H$ is an abelian variety $A$. That is, there
	is a unique short exact sequence of algebraic groups
	$$
	1\to H\to G\to A\to 1.
	$$
\end{theorem}

\begin{definition}[Semiabelian variety]
	Let $G$ be a complex connected algebraic group, and let $H$ be as in Theorem~\ref{thm:chevalley}. We say that $G$ is a semiabelian variety if $H\cong (\C^*)^n$ for some $n\geq 0$.
\end{definition}

\begin{remark}
	Semiabelian varieties are commutative groups (see \cite[Lemma 4]{iitaka_1977}), so we will denote their Chevalley decompositions with additive notation ($0$ instead of $1$).
\end{remark}

The following is well known.
\begin{proposition}[Functoriality of the Chevalley decomposition]\label{prop:funChevalley}
	Let $f:G_1\to G_2$ be a morphism of algebraic groups between two semiabelian varieties. Let $0\to(G_i)_T\xrightarrow{t_i}  G_i\xrightarrow{(p_A)_i}  (G_i)_A\to 0$ be the Chevalley decomposition of $G_i$ for $i=1,2$. Then, $f((G_1)_T)\subset (G_2)_T$.
\end{proposition}
\begin{proof}
	$(p_A)_2\circ f\circ t_1$ is an algebraic morphism between an affine algebraic group and an abelian variety. Since it is a group homomorphism, it sends the identity in $(G_1)_T$ to the identity in $(G_2)_A$. By \cite[Lemma 2.3]{conrad}, $(p_A)_2\circ f\circ t_1$ is the constant morphism to the identity in $(G_2)_A$. Hence, $f((G_1)_T)\subset\ker (p_A)_2=(G_2)_T$.
\end{proof}

\begin{proposition}\label{prop:torsor}
	Let $G$ be a semiabelian variety. Its Chevalley decomposition
	\[
	0\to G_T\to G\to G_A\to 0
	\]
	gives $G$ the structure of a $G_T$-torsor over $G_A$. This torsor is Zariski-locally trivial, i.e. there is a Zariski open covering of $G_A$ over which $G\cong G_T\times G_A$.
\end{proposition}
\begin{proof}
	This follows from the results in \cite[III.4.]{milne}, concretely, Propositions 4.6 and 4.9.
\end{proof}

\begin{proposition}\label{prop:noHomsToC}
  If \(G\) is a semiabelian variety, the only holomorphic group homomorphism \(G\to \C\) is trivial.
\end{proposition}
\begin{proof}
  Consider a group homomorphism \(\rho\colon G\to \C\), and \(G\)'s Chevalley decomposition as in Proposition~\ref{prop:torsor}. Then, \(\rho|_{G_T}\colon G_T \to \C\) is a holomorphic group homomorphism, so, since the torsion points are mapped to $0$, it must be  trivial. Therefore, \(\rho\) descends to a holomorphic map \(\ov\rho \colon G_A\to \C\), which must be constant, since \(G_A\) is compact.
\end{proof}

\begin{remark}[Universal cover of a semiabelian variety]\label{rem:universalCover}
  As complex manifolds, every semiabelian variety $G$ is isomorphic to $\C^g/\Z^r$ for $g = \dim G$ and some $r\in \Z_{\ge 0}$, where $\Z^r$ is embedded into $\C^g$ as a discrete subgroup. In particular, $r\le 2g$. Also, \(\Z^r\) must generate \(\C^g\) as a \(\C\)-vector space: otherwise, \(G\) would have a nontrivial holomorphic homomorphism to \(\C\), contradicting Proposition~\ref{prop:noHomsToC}. In particular, \(g\le r\). The universal cover of $G$ is given by the exponential map of Lie groups $\exp:TG\to G$, where $TG\coloneqq T_eG$ is the tangent space of $G$ at the identity $e\in G$. Note that, since $G$ is an abelian group, $\exp$ is a group homomorphism, where $TG$ is seen as a group under addition. Note that $TG\cong \C^g$, and $\exp^{-1}(e)$ is identified with the lattice $\Z^r$ through this identification, so $r=\mathrm{rank}\ \piG$.
\end{remark}

\begin{remark}\label{rem:groupStructure}
	Let $f:G_1\to G_2$ be an algebraic morphism between two semiabelian varieties. Up to translation, $f$ is also a group homomorphism. Indeed, up to translation in $G_2$, we can assume that $f$ takes the identity to the identity. Such $f$ induces a linear map (the differential of $f$ at the identity) between the universal covers given by the exponential map, which implies that $f$ is also a group homomorphism.
\end{remark}

\subsection{Alexander modules}\label{ss:alexanderModules}
Let $U$ be a smooth connected complex algebraic variety. Let $G$ be a complex semiabelian variety and let $e\in G$ be its identity. Let $f\colon U\to G$ be an algebraic map. Let $U^f$ be the complex manifold defined as the pullback of the universal cover of $G$ as in the commutative diagram~\eqref{eq:Uf}.

Since $\pi_1(G,e)$ is abelian, $\pi_1(G,x)$ is canonically identified with $\pi_1(G,e)$ for all $x\in G$. Therefore, we will not specify the choice of base points, and will denote the fundamental group of $G$ by $\piG$. The  fundamental group $\piG$ acts on $TG$ by deck transformations. By the universal property of fiber products, $\pi:U^f\to U$ is also a covering map whose deck action comes from the lift of the deck action of $\pi_1(G)$ on $TG$, and its deck transformation group is $\piG$.

Let $k$ be a field (which for us will be $\Q$, $\R$ or $\C$) and let $R = k[\piG]$. Since $\piG\cong \Z^r$, $R$ is non-canonically isomorphic to the ring of Laurent polynomials $k[t_1^{\pm 1},\ldots , t_r^{\pm 1}]$.

\begin{definition}\label{def:alexander}
	Let $k$ be a field. The $i$-th (multivariable) homology Alexander module associated to $(U,f)$ is $H_i(U^f,k)$. It is an $R=k[\piG]$-module via the deck action of $\piG$ on $U^f$.
\end{definition}

\begin{remark}
	Since $U$ has the homotopy type of a finite CW complex, $H_i(U^f,k)$ is a finitely generated $R$-module.
\end{remark}

For our purposes, it will be useful to realize the Alexander modules as homology groups of certain local systems on $U$.

\begin{definition}\label{def:LG}
	Let $k =\Q, \R$ or $\C$. In the notation of \eqref{eq:Uf}, we define $\cL_{G}\coloneqq \exp_! \ul k_{TG}$. 
\end{definition}
The action of $\piG$ on $TG$ by deck transformation induces an automorphism of $\cL_G$, making $\cL_G$ into a local system of rank $1$ free $R$-modules. For any $z\in G$, the stalks are given by
$
(\cL_G)_z=\bigoplus\limits_{z'\in\exp^{-1}(z)} k
$. The monodromy action of a loop $\gamma\in\piG$ on $(\cL_G)_z$ interchanges the summands according to the monodromy action of $\gamma$ on $\exp^{-1}(z)$.

\begin{definition}\label{def:L}
	Let $k = \Q, \R$ or $\C$. In the notation of \eqref{eq:Uf}, we define $\cL\coloneqq f^{-1}\exp_!\ul k_{TG}$, which is a rank 1 local system of free $R$-modules. Similarly, we let $\ov\cL= R\otimes_{\gamma\mapsto \gamma^{-1}} \cL$ denote the same local system, with a new $R$-module structure where $\gamma\in \piG$ acts in the way that $\gamma^{-1}$ acts on $\cL$. 
\end{definition}

\begin{remark}\label{rem:coverVsL}
There is a natural $R$-module isomorphism \(H_i(U^f,k)\cong H_i(U,\cL)\). This follows from the definition of the right hand side, since the chain complex that computes it is the same chain complex that computes the homology of \(U^f\) (see \cite[Section 2.5]{dimca2004sheaves}).
\end{remark}

\begin{remark}
	If $V\subset U$ is a simply connected open set, $\pi^{-1}(V)\cong \piG\times V$. For any  $\gamma\in \pi_1(U)$, the action of $\gamma$ on the stalk $\cL_x$ is given by multiplication by $f_*(\gamma)\in \piG$.
\end{remark}

\begin{remark}\label{remk:sectionsOfL}
	Let $\cS$ be the sheaf (of sets) of lifts of $f$ to $TG$, i.e. $\Gamma(\cS,V)= \{\iota\colon V\to TG \mid \exp\circ \iota =f \}$  for any open set $V\subseteq U$. For every $x\in U$, the stalk $\cS_x$ is canonically isomorphic to $\exp^{-1}(f(x))$, and it carries a $\piG$-action coming from the action on $TG$.
	
	On the other hand, a basis of the stalk of $\cL_G$ at $f(x)$ is given by $\exp^{-1}(f(x))$, where each point $z'$ on the fiber corresponds to the locally constant function that is 1 around $z'$ and $0$ elsewhere on the fiber, and this bijection is also compatible with the $\piG$ action. The same can be said of the stalk of $\cL$ at $x$, since $f^{-1}$ preserves stalks.
	
	This provides us with a map of sheaves $\cS\to \cL$ that sends $\cS$ to a basis of $\cL$ on each open set, and it is compatible with the action of $\piG$. Thus, a (locally defined) function $ \iota\colon U\to TG$ such that $\exp\circ \iota = f$ can be seen as a (local) section of $\cL$, and these locally form a $k$-basis. For $\gamma \in \piG$, we will denote $\gamma\cdot \iota \coloneqq \gamma\circ\iota$.
\end{remark}

\begin{notation}\label{not:iotaBar}
	We will denote the (identity) $R$-antilinear isomorphism $\cL \to \ov\cL$ by $\iota\mapsto \ov\iota$, that is, $\ov\iota$ is the notation that we will use to refer to $\iota$ when seen in $\ov\cL$. This way, for $\iota\in \cL$ and $\gamma\in \piG$, we have $\gamma \cdot\ov\iota = \ov{\gamma^{-1}\circ \iota}$.
\end{notation}

\begin{remark}\label{rem:ConjAndDual}
	There is a canonical isomorphism $\Homm_R(\cL,R)\cong \ov\cL$. On our local $k$-bases of $\cL$ and $\ov\cL$, it is given by the pairing:
	\[
	\ov\cL
	\times 
	\cL
	\longrightarrow
	\ul R_U
	\]
	defined by $\langle \gamma_1 \cdot \ov\iota,\gamma_2\cdot\iota\rangle = \gamma_1\gamma_2\in \piG \subseteq R$ for every $\iota,\gamma_1,\gamma_2$. One readily verifies that this is well-defined. Extending it in a $k$-bilinear way makes it automatically $R$-bilinear and it induces the above isomorphism.
\end{remark}

\begin{remark}\label{rem:ConjAndDualFiniteIndex}
	The observation from Remark~\ref{rem:ConjAndDual} also holds if we replace $R$ by $k[H]$, where $H$ is a finite index subgroup of $\piG$. Indeed, let $H\gamma_1,\ldots,H\gamma_n$ be the distinct elements of $\piG/H$, seen as right cosets. Then, $\{\gamma_1\cdot\iota,\ldots,\gamma_n\cdot\iota\}$ is a local $k[H]$-basis of $\cL$, and $\cL$ is a rank $n$ free $k[H]$-module. Similarly, since $\piG$ is abelian, $\{\gamma_1^{-1}\cdot\ov\iota,\ldots,\gamma_n^{-1}\cdot\ov\iota\}$ is a local $k[H]$-basis of $\ov\cL$. We define the $k[H]$-bilinear pairing $\ov\cL\times\cL\to \underline{k[H]}_U$ by the $k$-bilinear extension of the pairing given by
	$$\langle \delta_1\gamma_i^{-1}\cdot\ov\iota,\delta_2\gamma_j\cdot\iota\rangle=\left\{\begin{array}{lr}
		0 &\text{if }i\neq j,\\
		\delta_1\delta_2 &\text{if }i= j
		\end{array}\right.$$
	for all $\delta_1,\delta_2\in H$, $i,j\in\{1,\ldots,n\}$. One readily verifies that this is well-defined and induces the  isomorphism of sheaves of $k[H]$-modules
	$$
	\Homm_{k[H]}(\cL,k[H])\cong\ov\cL,
	$$
	which is also an isomorphism of sheaves of $R$-modules.
\end{remark}

Remarks~\ref{rem:coverVsL} and \ref{rem:ConjAndDual} motivate the following definition. Indeed, since the stalks of $\cL$ are infinite dimensional vector spaces but rank $1$ free $R$-modules, it seems more reasonable to dualize over $R$ rather than over $k$ to define the cohomological version of Alexander modules.

\begin{definition}\label{def:alexanderCohomology}
	Let $k$ be a field. The $i$-th (multivariable) cohomology Alexander module associated to $(U,f)$ is the $R$-module $H^i(U,\ov\cL)$, where $R=k[\piG]$.
\end{definition}

\subsection{Truncated local systems}\label{sec:truncations}

For the purposes of doing Hodge theory on Alexander modules, we will have to work with truncated versions of the local systems $\cL$ and $\ov\cL$. 

\begin{definition}\label{def:Rminfty}
	Let $m\in\Z_{>0}$, and let $k=\Q,\R,\C$. We define the rings $R_{\infty}$ and $R_{m}$ by $$R_{\infty}\coloneqq \prod\limits_{j=0}^\infty \Sym^j H_1(G,k);\quad R_{m}\coloneqq\frac{R_{\infty}}{\prod\limits_{j=m}^\infty \Sym^j H_1(G,k)}, $$
	and the $R_{m}$-module $R_{-m}$ by
	$$
	R_{-m}\coloneqq \Hom_k(R_{m},k).
	$$
\end{definition}
Note that for all $m>0$, $R_{m}$ and $R_{-m}$ have natural $R_{\infty}$-module structures. Also note that the field $k$ does not appear in the notation for $R_\infty$,  $R_m$ and $R_{-m}$ (like it did not appear in $R$), but whenever we use this notation, the base field will either be explicitly specified or clear from context.

\begin{notation}\label{not:abelianization}
	For all $\gamma\in\piG$ we denote its corresponding element in $H_1(G,\Z)\subset H_1(G,k)$ by $\log\gamma$.
\end{notation}
Even if $\piG$ is abelian and thus isomorphic to $H_1(G,\Z)$, this notation is useful because $\piG$ will be thought of as having multiplication as its group operation, but $H_1(G,\Z)$ has the sum as its group operation.

\begin{definition}[The $R$-module structure of $R_{\infty}$, $R_m$ and $R_{-m}$]\label{def:Rmodstructure}
	 Let $m>0$. The $k$-linear ring monomorphism
	$$
	\begin{array}{rcl}
		R=k[\piG] & \longrightarrow & R_{\infty}\\
    \gamma &\longmapsto &e^{\log\gamma} = \sum_{j=0}^\infty \frac{(\log \gamma)^j}{j!}
	\end{array}
	$$
	endows $R_\infty$, $R_m$ and $R_{-m}$ with $R$-module structures.
\end{definition}

\begin{remark}[$R_m$ and $R/\mathfrak m^m$ are isomorphic $R$-modules]\label{rem:RmVSRmodm}
	Let $\mathfrak m\coloneqq \left( \gamma-1\mid \gamma\in\pi_1(G)\right)$ be the augmentation ideal of $R$, and let $m\geq 1$. The image of $\mathfrak m$ by the $k$-linear ring monomorphism described in Definition~\ref{def:Rmodstructure} lies in $\prod_{j=1}^\infty\Sym^j H_1(G,k)$, so one gets an induced ring homomorphism $R/\mathfrak m^m\to R_m$, which is also an $R$-module homomorphism. In fact, it is an $R$-module isomorphism. To see this, it suffices to see that $R/\mathfrak m^m\to R_m$ is an isomorphism of $k$-vector spaces. Let $\gamma_1,\ldots,\gamma_r$ be a basis of generators of $\pi_1(G)$, and let us consider the bases
	\begin{align*}
		\cup_{j=0}^{m-1}\{(\gamma_1-1)^{i_1}\cdots(\gamma_r-1)^{i_r}&\mid i_l\geq 0 \text{ for all }l=1,\ldots,r,\quad i_1+\cdots+i_r= j\},\\
		\cup_{j=0}^{m-1}\{(\log\gamma_1)^{i_1}\cdots(\log\gamma_r)^{i_r}&\mid i_l\geq 0 \text{ for all }l=1,\ldots,r,\quad i_1+\cdots+i_r= j\}
	\end{align*}
of $R/\mathfrak m^m$ and $R_m$ respectively, where both are ordered in the same way by increasing order of $j$, and amongst the ones with the same $j$, by lexicographical order. The square matrix representing the $R$ module homomorphism between $R/\mathfrak m^m$ and $R_m$ (seen as a $k$-linear homomorphism) in these bases is triangular with ones along the diagonal.
\end{remark}

\begin{remark}\label{rem:RinftyVSInverseLimit}
	Let  $\fm\subset R$ be as in the previous remark. Taking inverse limits in the isomorphism between $\frac{R}{\fm^m}$ and $R_m$ from the previous remark, one obtains an $R$-module isomorphism
	$$
	\varprojlim_m \frac{R}{\fm^m}\cong R_\infty.
	$$
\end{remark}

In light of the Definition~\ref{def:Rmodstructure}, we can think about the local systems $R_\infty\otimes_R\cL$, $R_m\otimes_R\cL$ and $R_{-m}\otimes_R \cL$, and similarly with $\ov\cL$. Note that, by tensoring $\cL$ or $\ov\cL$ with $R_m$ (resp. $R_{-m}$) over $R$, we obtain finite dimensional $k$-local systems whose stalk is isomorphic to $R_m$ (resp. $R_{-m}$). These will be the truncated local systems that we consider. 

Let us understand the relationship between the homology and the cohomology of these truncated local systems. We begin by recalling a well-known duality result for finite dimensional $k$-local systems.

\begin{proposition}[cf. \cite{dimca2004sheaves} section 2.5]\label{prop:dimca-homology}
	Let $L$ be a finite dimensional local system over a field $k$ on a connected algebraic variety $X$. Then, for all $i\geq 0$, there is a natural isomorphism
	\[
	\Hom_k(H_i(X,L),k) \cong H^i(X, \Homm_k(L,k)).
	\]
\end{proposition}

\begin{remark}[Relationship between homology and cohomology]\label{rem:homologyVsCohomology}
	For all $m\neq 0$, we have a chain of natural isomorphisms
	\begin{align*}
		\Homm_k(R_m \otimes_R \cL, k)
		&\cong
		\Homm_R(\cL ,\Hom_k(R_m, k))
		&
		\text{(Tensor-hom adjunction)}\\
		&= \Homm_R(\cL ,R_{-m})\\
		&\cong
		\Homm_R(\cL, R) \otimes_R R_{-m}	&
		\text{(Because \(\cL\) is locally free over $R$)}
		\\
		&\cong\ov \cL\otimes_R R_{-m}.&
		\text{Remark~\ref{rem:ConjAndDual}}
	\end{align*}
Since $R$ is commutative, one can identify $\ov\cL\otimes_R R_{-m}$ with $R_{-m}\otimes_R\ov\cL$. We apply Proposition~\ref{prop:dimca-homology} to $L = R_m \otimes_R \cL$, and the above to get $R$ and $R_\infty$-module isomorphisms for all $i\geq 0$ and $m\neq 0$:
	\begin{align*}
		\Hom_k(H_i(U,R_m \otimes_R \cL),k)
		&\cong
		H^i(U,\Homm_k(R_m \otimes_R \cL,k)). & \text{Proposition~\ref{prop:dimca-homology}}
		\\
		&\cong
		H^i(U,R_{-m}\otimes_R\ov\cL ).
	\end{align*}
\end{remark}

\begin{remark}\label{rem:homologyVsCohomologyFiniteIndex}
	Let $H$ be a finite index subgroup of $\piG$, and let $\pi_H:G_H\to G$ be the corresponding finite cover. $G_H$ is the quotient of $TG$ by $H$, where $H\leq \piG$ acts by deck transformations. In particular, $G_H$ is a commutative algebraic group, which is in fact a  semiabelian variety (see \cite[Section 3]{Conrad2011semistable}, for example), and $\pi_H$ is a morphism of algebraic groups. In that case, we may define $R^H\coloneqq k[H]=k[\pi_1(G_H)]$, and $R^H_\infty$, $R^H_m$ and $R^H_{-m}$ analogously as in Definition~\ref{def:Rminfty} using $H_1(G_H,k)$ instead of $H_1(G,k)$. Note that $\cL$ is locally free of finite rank as a sheaf of $R^H$-modules. Hence,
	using Remark~\ref{rem:ConjAndDual}, the argument in Remark~\ref{rem:homologyVsCohomology} can be replicated to obtain $R^H$ and $R^H_\infty$-module isomorphisms
	$$
	\Hom_k(H_i(U,R^H_m\otimes_{R^H}\cL),k)\cong H^i(U,R_{-m}^H\otimes_{R^H}\ov\cL)
	$$
	for all $m\neq 0$ and all $i\geq 0$.
\end{remark}

The rest of Section~\ref{sec:truncations} will be devoted to establishing the relationship between the (co)homology of these truncated local systems and the homological and cohomological Alexander modules of Definitions~\ref{def:alexander} and~\ref{def:alexanderCohomology}. For this, we will need the following technical result.

	%

\begin{proposition}\label{prop:limvsH}
	Let $(S,\mathfrak a)$ be a complete Noetherian local ring. Let $C^\bullet$ be a complex of finitely generated free $S$-modules. For $m\ge 0$, let $S_m \coloneqq S/\mathfrak a^m$. Then, the natural maps $\Xi_{\infty, m} \colon H^i(C^\bullet)\to H^i(S_m\otimes_{S} C^\bullet)$ induce an isomorphism of $S$-modules:
	\[
	\Xi\colon H^i( C^\bullet) \xrightarrow{\cong} 
	\varprojlim_m H^i\left( S_m\otimes_{S} C^\bullet\right) .
	\]
\end{proposition}\begin{proof}
	First, we show that the map is injective. Let $M=\frac{C^i}{dC^{i-1}}$, and let $N=H^i(C^\bullet)\subseteq M$. Following the definitions, we have that
	\[
	\ker \Xi = \bigcap_m \ker \Xi_{\infty, m} =\bigcap_m (\mathfrak a^{m} M\cap N)\subseteq \bigcap_m \mathfrak a^{m} M = 0,
	\]
	where the last equality follows from Krull's Intersection Theorem.
	
	Let us now prove that $\Xi$ is surjective. It suffices to prove that for every $m$ there exists an $m'\gg m$ such that
	\[
	\im \Xi_{m',m}\subseteq\im \Xi_{\infty, m}  ,
	\]
	where $\Xi_{m',m}$ is the natural map $H^i(S_{m'}\otimes C^\bullet)\to H^i(S_m\otimes C^\bullet)$. Consider the map of short exact sequences:
	\[
	\begin{tikzcd}[row sep = 1.5em]
		0
		\arrow[r]
		&
		\mathfrak a^{m'}C^\bullet
		\arrow[r]
		\arrow[d,hookrightarrow]
		&
		C^\bullet
		\arrow[r]\arrow[d,"="]
		&
		S_{m'}\otimes C^\bullet
		\arrow[r]\arrow[d,twoheadrightarrow]
		&
		0
		\\
		0
		\arrow[r]
		&
		\mathfrak a^{m}C^\bullet
		\arrow[r]
		&
		C^\bullet
		\arrow[r]
		&
		S_m\otimes C^\bullet
		\arrow[r]
		&
		0.
	\end{tikzcd}
	\]
	Taking cohomology, it induces the following map of exact sequences for every $i$:
	\[
	\begin{tikzcd}[row sep = 1.5em]
		H^i(C^\bullet)
		\arrow[r,"\Xi_{\infty, m'}"]\arrow[d,"="]
		&
		H^i(S_{m'}\otimes C^\bullet)
		\arrow[r]\arrow[d,"\Xi_{m',m}"]
		&
		\ker(H^{i+1}(\mathfrak a^{m'}C^\bullet)\to H^{i+1}(C^\bullet))
		\arrow[r]
		\arrow[d,"\star"]
		&
		0
		\\
		H^i(C^\bullet)
		\arrow[r,"\Xi_{\infty, m}"]
		&
		H^i(S_m\otimes C^\bullet)
		\arrow[r]
		&
		\ker(H^{i+1}(\mathfrak a^{m}C^\bullet)\to H^{i+1}(C^\bullet))
		\arrow[r]
		&
		0.
	\end{tikzcd}
	\]
	By the exactness of the rows, it is enough to show that for $m'\gg m$, $\star = 0$. By definition, 
	\[
	\ker(H^{i+1}(\mathfrak a^{m'}C^\bullet)\to H^{i+1}(C^\bullet)) = \frac{\mathfrak a^{m'}C^{i+1}\cap d^{-1}(0)\cap dC^i}{d(\mathfrak a^{m'}C^i)}\subseteq \frac{\mathfrak a^{m'}C^{i+1}\cap dC^i}{d(\mathfrak a^{m'}C^i) }
	\]
	We apply the Artin-Rees Lemma to the module $C^{i+1}$ and its submodule $dC^i$, to conclude that there exists an $m_0\gg 0$ such that for all $m\ge 0$,
	\[
	\mathfrak a^{m+m_0}C^{i+1}\cap dC^i=
	\mathfrak a^{m}(\mathfrak a^{m_0}C^{i+1}\cap dC^i)\subseteq \mathfrak a^m (dC^i)=d(\mathfrak a^m C^i).
	\]
	So, if $m'\ge m_0+m$, the starred map indeed vanishes, as desired.
\end{proof}

\begin{corollary}\label{cor:completionHomology}
	Let $R_\infty$ and $R_m$ be as in Definition~\ref{def:Rminfty}, for $m\geq 1$. 
	The natural maps induce an isomorphism of $R_\infty$-modules
	\[
	R_\infty \otimes_R H^i(U,\ov\cL)=\left( \varprojlim_m R_m\right) \otimes_R H^i(U,\ov\cL) \xrightarrow{\cong} \varprojlim_m H^i\left( U, R_m \otimes_R \ov\cL\right).
	\]
\end{corollary}

\begin{proof}
	Recall that $U$ has the homotopy type of a finite CW-complex, so $H^\bullet(U,\ov\cL)$ is represented by a bounded complex of finitely generated free $R$-modules $C^\bullet$ as in \cite[Section 2.5]{dimca2004sheaves}. Now, by Remark~\ref{rem:RinftyVSInverseLimit}, $R_\infty \cong \left( \varprojlim_m \frac{R}{\fm^m}\right)$, so the ring $R_\infty$ is flat over $R$. In particular, the cohomology of $R_\infty \otimes_R C^\bullet$ is naturally $R_\infty \otimes_R H^i(U,\ov\cL)$.
	
	 Let $(S,\mathfrak a)=\left(R_\infty,\prod_{j=1}^\infty \Sym^j H_1(G,k)\right)$ in Proposition~\ref{prop:limvsH} (so $S_m=R_m$) and apply it to the complex of free $R_\infty$-modules $R_\infty \otimes_R C^\bullet$ to obtain an isomorphism:
	\[
	\left( \varprojlim_m R_m\right) \otimes_R H^i(U,\ov\cL)\xrightarrow{\sim} \varprojlim_m H^i
	\left(
	R_m\otimes_{R_\infty} 
	R_\infty
	\otimes_ R C^\bullet\right)=
	\varprojlim_m H^i\left(R_m\otimes_R C^\bullet\right).
	\]
	Finally, notice that $H^i\left(R_m\otimes_R C^\bullet\right) = H^i\left(U,R_m\otimes_R \ov\cL\right)$.
\end{proof}


\begin{corollary}\label{cor:completionHomology2}
	Let $m\geq 1$, and $R_\infty$, $R_m$ and $R_{-m}$ as in Definition~\ref{def:Rminfty}. There is a natural isomorphism
	\[
	R_\infty \otimes_R H_i(U,\cL)\xrightarrow{\sim}  \varprojlim_m 
	\Hom_k\left(
	H^i(U,R_{-m}\otimes_R \ov\cL)
	,k
	\right)
	\]
\end{corollary}

\begin{proof}
	By the analogous reasoning to the proof of Corollary~\ref{cor:completionHomology}, we have that the natural maps induce an isomorphism:
	\[
	\left( \varprojlim_m R_m\right) \otimes_R
	H_i(U,\cL)\xrightarrow{\sim}
	\varprojlim_m H_i\left(
	U,R_m\otimes_R \cL
	\right).
	\]
	Taking duals in Remark~\ref{rem:homologyVsCohomology}, one obtains
		\[
		H_i\left(
		U,R_m\otimes_R \cL
		\right)\cong \Hom_k\left(H^i(
		U, R_{-m}\otimes_R\ov\cL ),k
		\right).
		\]
\end{proof}

\subsection{Generalized Albanese varieties}

Iitaka (\cite{Iitaka}, \cite{iitaka_1977}) generalized the Albanese morphism of smooth complete complex algebraic varieties to smooth varieties as follows. For a detailed description, see \cite{Fujino-qa}.

\begin{definition}[Iitaka's generalized Albanese maps]\label{def:quasialbanese}
	Let $U$ be a smooth connected complex algebraic variety.
	The \emph{Albanese map} $\alpha_U: U\to G_U$ is a morphism to a semiabelian
	variety $G_U$ satisfying the following universal property:
	for any other morphism $\beta: U \to G'$ to a semiabelian variety $G'$, there exists a unique algebraic morphism $f : G_U\to G'$ such that $\beta=f\circ \alpha_U$. Such $G_U$ is usually called the \emph{Albanese} variety of $U$.
\end{definition}

\begin{remark}[Existence of the Albanese map]\label{remk:existencequasialbanese}
	The Albanese map $\alpha_U$ exists for any smooth connected complex algebraic variety (see \cite{Fujino-qa}), and hence the Albanese variety $G_U$ is well defined up to algebraic isomorphism, which, up to translation, will be a group homomorphism as well by Remark~\ref{rem:groupStructure}. Once $G_U$ is fixed, $\alpha_U$ is uniquely defined defined up to translation in $G_U$ and isomorphism of algebraic groups from $G_U$ to itself.
\end{remark}

\begin{remark}
	If $U$ is a smooth connected complex projective variety, $G_U$ in Definition~\ref{def:quasialbanese} is an abelian variety, and $\alpha_U$ is the usual Albanese map.
\end{remark}

\begin{lemma}[\cite{Fujino-qa}, Lemma 3.11]
	Let $U$ be a smooth connected complex algebraic variety, and let $\alpha_U:U\to G_U$ be its Albanese map. Then,
	$$(\alpha_U)_*: H_1(U, \Z) \to H_1(G_U, \Z)$$
	is surjective. Moreover, the kernel of $(\alpha_U)_*$ coincides with the torsion part of $H_1(U,\Z)$.
\end{lemma}

\begin{remark}\label{remk:universalAbelianCover}
	Consider the pullback of \eqref{eq:Uf} for the map $\alpha_U$. If $H_1(U,\Z)$ is torsion free, $U^{\alpha_U}$ is the universal abelian cover of $U$.
\end{remark}

\begin{example}[Affine hypersurface complements]\label{exam:hypersurfaces}
	Suppose that $U=\C^n\setminus H$ is an affine hypersurface complement, where $H=V(f_1\cdots f_m)$, and $f_i$ are non-constant irreducible polynomials in $\C[x_1,\ldots,x_n]$ such that $f_i$ and $f_j$ do not have any non-constant common factors for all $i\neq j\in\{1,\ldots,m\}$. Then, $H_1(U,\Z)\cong \Z^m$ is generated by a choice of a positively oriented meridians around each of the $m$ irreducible components of $H$. Hence, the map
	$$
	\f{f=(f_1,\ldots,f_m)}{U}{(\C^*)^m}{x}{(f_1(x),\ldots,f_n(x))}
	$$
	induces an isomorphism on first (integral) homology groups, so $U^f$ is the universal abelian cover of $U$. In this case, $H_1(U^f,\Q)$ is generally called the Alexander invariant (with $\Q$-coefficients) of the hypersurface $H$.
	
	Let us see that the map $f=(f_1,\ldots,f_m)$ coincides with the Albanese map of $U$. Since both $\alpha_U$ and $f$ induce isomorphisms in first homology with $\Q$-coefficients, the mixed Hodge structure on $H^1(G_U,\Q)$ is pure of type $(1,1)$. The Chevalley decomposition of $G_U$ induces a short exact sequence between (abelian) fundamental groups, so  if $A$ is the abelian variety in the Chevalley decomposition of $G_U$, $H^1(A,\Q)\xrightarrow{(p_A)^*} H^1(G_U,\Q)$ is an injective morphism between pure Hodge structures of weights 1 and 2 respectively. Thus $H^1(A,\Q)=0$ and $A$ is a point, so $G_U$ is a torus which, looking at the rank of $H_1(G_U,\Q)$, must be isomorphic to $(\C^*)^m$.  By the universal property of the Albanese, there exists a unique algebraic morphism $h:G_U\cong (\C^*)^m\to (\C^*)^m$ such that $f=h\circ\alpha_U$ and which, up to translation in the target, is an algebraic group homomorphism between $(\C^*)^m$ and itself which induces an isomorphism between fundamental groups. This implies that $h$ is an isomorphism of algebraic varieties, so $f$ is the Albanese map of $U$.
\end{example}

\subsection{Compactifications}

Let $U$ be a smooth connected complex algebraic variety, let $G$ be a complex semiabelian variety and let $f:U\to G$ be an algebraic morphism. For the construction of the mixed Hodge structures in this paper, we will need to compactify $f$ in appropriate ways. First of all, the compactifications of $U$ and $G$ that will appear in this paper will always be \emph{good compactifications}, as defined below.

\begin{definition}[Good compactification]
	Let $U$ be a smooth connected complex algebraic variety, and let $X$ be a smooth compactification of $U$. $X$ is a \emph{good compactification} of $U$ if $D\coloneqq X\setminus U$ is a simple normal crossings divisor.
\end{definition}

Let us now explain which compactifications of $G$ will appear in this paper.
\begin{corollary}[Of Proposition~\ref{prop:torsor}]\label{cor:compactification}
	Let $G$ be a semiabelian variety and let $0\to G_T\to G\to G_A\to 0$ be its Chevalley decomposition. Then, $G$ has a good compactification $\ov G$ which has the structure of a fibration as follows:
	\[
	\ov{G_T}\hookrightarrow \ov{G} \twoheadrightarrow G_A,
	\]
	where $\ov{G_T}$ is a compactification of $G_T$ by a product of $\bP^1$'s.
\end{corollary}

\begin{proof}
	Over an open covering of $G_A$ this is the compactification of $(\C^*)^j\times G_A$ by $(\bP^1)^j\times G_A$. These compactifications can be glued: by Proposition~\ref{prop:torsor} the transition functions are multiplication in $G_T$ by locally defined functions $G_A\to G_T$, which fix the divisors at infinity of $\ov{G_T}$. Finally, the divisor at infinity of $\ov G$ has normal crossings, since this can be checked on an open cover.
\end{proof}

\begin{definition}[Allowed compactifications of $G$]\label{def:allowedComp}
  Let $Y$ be a good compactification of $G$. We say that $Y$ is an \emph{allowed compactification} of $G$ if there exists an algebraic map $p:Y\to \ov G$ satisfying that $j_{\ov G}=p\circ j_{Y}$,  where $\ov G$ is a compactification of $G$ such as the one described in Corollary~\ref{cor:compactification}, and $j_{\ov G}$ and $j_Y$ are the inclusions of $G$ into its compactifications $\ov G$ and $Y$. 
\end{definition}

\begin{definition}[Compatible compactifications with respect to  $f$]\label{def:compatibleCompf}
	Let $X$ be a good compactification of $U$ and let $Y$ be an allowed compactification of $G$. We say that $X$ and $Y$ are \emph{compatible compactifications with respect to $f:U\to G$} if $f$ extends to an algebraic morphism $\ov f:X\to Y$.
\end{definition}

More generally, we have the following definition.
\begin{definition}[Compatible compactification with respect to a commutative diagram]\label{def:compatibleCompDiagram}
	Suppose that $U_1$ and $U_2$ are smooth connected complex algebraic varieties, $G_1$ and $G_2$ are complex semiabelian varieties, and that we have the following commutative diagram of algebraic maps
	$$
	\begin{tikzcd}
		U_1\arrow[r, "f_1"]\arrow[d, "h"] & G_1\arrow[d, "g"]\\
		U_2\arrow[r, "f_2"] & G_2.
	\end{tikzcd}
	$$
	Let $X_i$ be a good compactification of $U_i$ and let $Y_i$ be an allowed compactification of $G_i$ for $i=1,2$. We say that $X_1$, $X_2$, $Y_1$ and $Y_2$ are \emph{compatible compactifications with respect to the commutative diagram $g\circ f_1=f_2\circ h$} if the morphisms in the commutative diagram extend to algebraic morphisms which fit into the following commutative diagram:
	$$
	\begin{tikzcd}
		X_1\arrow[r, "\ov{f_1}"]\arrow[d, "\ov h"] & Y_1\arrow[d, "\ov g"]\\
		X_2\arrow[r, "\ov{f_2}"] & Y_2.
	\end{tikzcd}
	$$
\end{definition}

The next result follows from a standard argument.

\begin{lemma}[Existence of compatible compactifications]\label{lem:existenceComp}
	Let $U_1$ and $U_2$ be smooth connected complex algebraic varieties, let $G_1$ and $G_2$ be complex semiabelian varieties, and suppose that we have the following commutative diagram of algebraic maps
	$$
	\begin{tikzcd}
		U_1\arrow[r, "f_1"]\arrow[d, "h"] & G_1\arrow[d, "g"]\\
		U_2\arrow[r, "f_2"] & G_2.
	\end{tikzcd}
	$$
	Then, there exist compatible compactifications of \(U_1,U_2,G_1,G_2\) with respect to the commutative diagram $g\circ f_1=f_2\circ h$.
	
	In particular, if $f:U\to G$ is an algebraic morphism from a smooth connected complex algebraic variety to a complex semiabelian variety, there exist compatible compactifications with respect to $f$.
\end{lemma}

\begin{proof}
	Let $Z_2$ be a good compactification of $U_2$, and let $\ov{G_2}$ be a compactification of $G_2$ as in Corollary~\ref{cor:compactification}. Let $X_2$ be a resolution of singularities of the closure of the graph of $f_2$ inside of $Z_2\times\ov{G_2}$, such that $X_2$ is a good compactification of $U_2$. By construction, $f_2$ extends to an algebraic map $\ov{f_2}:X_2\to\ov{G_2}$.
	
	Now, fix $\ov{G_1}$, a compactification of $G_1$ as in Corollary~\ref{cor:compactification}. By looking at the closure of the graph of $h$ inside of $\ov{G_1}\times\ov{G_2}$ and resolving singularities as in the previous paragraph, we find an allowed compactification $Y_1$ of $G_1$.
	
	Following this argument, we can find good compactifications $X_1'$ and $X_1''$ of $U_1$ such that $f_1$ and $h$ extend to algebraic morphisms $\ov{f_1}':X_1'\to Y_1$ and $\ov{h}'':X_1''\to X_2$. Let $X_1$ be a resolution of singularities of the closure of the graph of the identity map of \(U_1\) inside of $X_1'\times X_1''$, such that $X_1$ is a good compactification of $U$. By construction, there exist algebraic maps $p_1:X_1\to X_1'$ and $p_2:X_1\to X_1''$ which extend the identity from $U_1$ to itself. Let $\ov{f_1}=\ov{f_1}'\circ p_1$, and $\ov{h}=\ov{h}''\circ p_2$.
	
	We claim that $X_1$, $X_2$, $Y_1$ and $\ov{G_2}$ are compatible compactifications with respect to $g\circ f_1=f_2\circ h$. This follows from the fact that $\ov g\circ \ov{f_1}$ and $\ov{f_2}\circ\ov{h}$ both agree on $U_1$, and there exists a unique way of extending them continuously to $X_1$.
\end{proof}

\subsection{Commutative Differential Graded Algebras} 

\begin{definition}[Commutative differential graded algebra (cdga)]
	A commutative differential graded $k$-algebra (cdga) is a triple $$(A,d,\wedge)$$ such that:
	\begin{itemize}
		\item $(A, \wedge)$ is a non-negatively graded unitary associative $k$-algebra.
		\item $a \wedge b = (-1)^{|a||b|} b \wedge a$ for homogeneous $a,b \in A$ of degrees \(|a|\) and \(|b|\).
		\item $(A,d)$ is a cochain complex.
		\item $d(a \wedge b) = da \wedge b + (-1)^{|a|} a \wedge db$ for $a,b \in A$, and \(a\) homogeneous of degree \(|a|\).
	\end{itemize} 
\end{definition}

Notice that when we write a cdga, the field $k$ is implicit.
We often will write $A$ instead of $(A,d,\wedge)$ when the differential and multiplication are understood.

When we discuss Hodge complexes in Section~\ref{ss:MHSsAndComplexes}, we will often work with filtered cdgas whose filtrations are compatible with the differential and the multiplication.

\begin{definition}[cdga filtrations]
	Suppose $(A,d,\wedge)$ is a cdga.
	An \emph{increasing cdga filtration} on $(A, d, \wedge)$ is an increasing filtration $W_{\lc}$ on $A$ such that 
	\[
	W_iA \wedge W_jA \subset W_{i+j}A \quad\text{and}\quad d(W_iA)\subset W_iA
	\]
	for all integers $i$ and $j$.
	By a \emph{decreasing cdga filtration} on $(A, d, \wedge)$ we mean a decreasing filtration $F^{\lc}$ on $A$ such that 
	\[
	F^iA \wedge F^jA \subset F^{i+j}A \quad\text{and}\quad d(F^iA)\subset F^iA
	\]
	for all integers $i$ and $j$.
	
	One defines \emph{cdga filtrations on a sheaf of cdgas} analogously, by looking at the cdgas of sections over arbitrary open subsets. 
\end{definition}

\subsection{Differential graded Lie algebras and deformation theory}

Differential graded Lie algebras (dglas) provide a compact way to package the deformation theory of an object, in our case, a chain complex. We will review the definitions for the purpose of fixing notation. We will work over a (commutative, unital) ring $A$, which we will later assume to be Artinian local.

\begin{definition}
	A differential graded Lie algebra (dgla) over $A$ is a graded $A$-module $M = \bigoplus_{j\in \Z} M^j$ together with two $A$-(bi)linear operations:
	\begin{itemize}
		\item a differential $d\colon M\to M$ which has degree $1$, i.e. $dM^j\subseteq M^{j+1}$, and
		\item a bracket $[\cdot,\cdot]\colon M\otimes_A M\to M$ of degree $0$, i.e. $[M^j,M^{j'}] \subseteq M^{j+j'}$
	\end{itemize}
	subject to the following restrictions: throughout, suppose $a,b,c\in M$ are homogeneous elements of degrees $|a|,|b|,|c|$, respectively.
	\begin{itemize}
		\item $(M,d)$ is a complex, i.e. $d^2=0$.
		\item The bracket is graded-anticommutative:
		\[
		[a,b] = -(-1)^{|a||b|}[b,a].
		\]
		\item The bracket satisfies the graded Jacobi identity:
		\[
		[a,[b,c]] = [[a,b],c] + (-1)^{|a||b|} [b,[a,c]].
		\]
		\item The differential is a graded derivation for the bracket:
		\[
		d[a,b] = [da,b] + (-1)^{|a|}[a,db].
		\]
	\end{itemize}
\end{definition}

\begin{remark}\label{rem:dga-dgla}
	Our main example of a dgla is the following: suppose $(M^\bullet,d,\cdot)$ is a differential graded associative algebra, i.e. $(M^\bullet,d)$ is a complex, and $\cdot$ is an associative product for which $d(a\cdot b)=(da)\cdot b +  (-1)^{|a|} a\cdot (db)$. Then, automatically $(M^\bullet,d,[,])$ is a dgla with the bracket given by
	\[
	[a,b] = a\cdot b - (-1)^{|a||b|} b\cdot a.
	\]
\end{remark}

\begin{definition}
  Let $(A,\fm)$ be an Artinian local $k$-algebra with a fixed map $A\twoheadrightarrow k$, and let $(C^\bullet,d)$ be a bounded complex of $k$-vector spaces. A \emph{deformation} of $(C^\bullet,d)$ over $A$ is a complex $(\wt C^\bullet,\widehat D)$ of free $A$-modules, together with an isomorphism of complexes $k\otimes_A \wt C^\bullet \cong C^\bullet$.
\end{definition}

We will be interested in how endomorphisms of $C^\bullet$ give rise to deformations.

\begin{remark}
	Let $C^\bullet$ be a bounded complex of $k$-vector spaces. Then the vector space of $k$-linear endomorphisms $\End^\bullet_k(C^\bullet)$ is a differential graded associative algebra, where the homogeneous elements of degree $k$ are linear maps $\phi$ such that $\phi(C^j)\subseteq C^{j+k}$. The product is composition, and the differential is the graded commutator with $d$, i.e. if $\phi$ has degree $|\phi|$,
	\[
	d\cdot \phi \coloneqq d\circ \phi - (-1)^{|\phi|} \phi\circ d.
	\]
	Note that with this differential, $H^j(\End^\bullet_k(C^\bullet))$ is the group of homotopy classes of morphisms of complexes $C^\bullet\to C^\bullet[j]$ (recall that by convention the differential on $C^\bullet[j]$ is $(-1)^jd$). By Remark~\ref{rem:dga-dgla}, $\End_k^\bullet(C^\bullet)$ is a dgla with the bracket given by the commutator.
\end{remark}

\begin{remark}
	Let $S_1\to S_2$ be a ring map. If $L^\bullet$ is a $S_1$-dgla, $S_2\otimes_{S_1} L^\bullet$ becomes a dgla with the bracket $[a_1\otimes m_1,a_2\otimes m_2]\coloneqq a_1a_2\otimes[m_1,m_2]$. We are interested in the ring map $k\to A$ and the $k$-dgla $L^\bullet = \End^\bullet_k(C^\bullet)$. In this case, $A\otimes_k L^\bullet = \End^\bullet_A(A\otimes_k C^\bullet)$, and the Lie bracket extended from $k$ coincides with the commutator of endomorphisms. Furthermore, $\fm\End^\bullet(C^\bullet)\coloneqq \fm\otimes_k \End^\bullet_k(C^\bullet)=  \Hom^\bullet_k( C^\bullet,\fm\otimes_k C^\bullet)= \Hom^\bullet_A(A\otimes_k C^\bullet,\fm\otimes_k C^\bullet)$ is a sub-dgla.
\end{remark}


\begin{lemma}\label{lem:dglas}
	Let $k = \Q, \R$ or $\C$, and let $(A,\fm)$ be a local Artinian $k$-algebra with residue field $k$. Let $(C^\bullet,d)$ be a bounded complex of $A$-modules. Suppose $\phi\in \fm\End^1(C^\bullet)$ satisfies the Maurer-Cartan equation, i.e.:
	\[
	d\cdot \phi + \frac{1}{2}[\phi,\phi]=0.
	\]
	Then, $(A\otimes_k C^\bullet,d+\phi)$ is a complex of $A$-modules. Furthermore, for any $\rho\in \fm\End^0(C^\bullet)$, one obtains an isomorphism $e^\rho\coloneqq \sum_{k=0}^\infty \frac{1}{k!} \rho^k$:
	\[
	e^{\rho}\colon (A\otimes_k C^\bullet,d+\phi)\xrightarrow{\sim} (A\otimes_k C^\bullet,d+\phi+[e^\rho ,d+\phi]e^{-\rho})
	\]
	If $[\rho,[\rho,d+\phi]]=0$, then $e^\rho$ is an isomorphism:
	\[
	e^{\rho}\colon (A\otimes_k C^\bullet,d+\phi)\xrightarrow{\sim} (A\otimes_k C^\bullet,d+\phi-d\cdot\rho +[\rho,\phi]).
	\]
	The same result holds for sheaves: Let $(\cK^\bullet,d)$ is a bounded complex of sheaves of $k$-vector spaces. We obtain analogous statements for $A\otimes_k\cK^\bullet$ , $\phi\in \Hom_A^1(A\otimes_k\cK^\bullet, \fm \otimes_k\cK^\bullet)$ and $\rho\in \Hom_A^0(A\otimes_k\cK^\bullet, \fm \otimes_k\cK^\bullet)$.
\end{lemma}

\begin{proof}
	This is all direct computation. Note that $A$ is Artinian local, so $\fm$ is nilpotent, which ensures that $e^\rho$ is well-defined.
	
%
\end{proof}

\begin{remark}\label{rem:dglaTensorM}
	In the notation of Lemma~\ref{lem:dglas}, if we let $M$ be a (left) $A$-module, the analogous statements can be made for the complexes $M\otimes_k C^\bullet$, since these are simply obtained from $A\otimes_k C^\bullet$ by tensoring over $A$ with $M$.
\end{remark}

\begin{remark}\label{rem:dgla-cdga}
	Let $k$ be a field, let $A$ be a local Artinian $k$-algebra with maximal ideal $\fm$. Let $(C^\bullet,d,\cdot)$ be an $A$-cdga and let $(M^\bullet,d)$ be a $C^\bullet$-differential graded (left) module. In other words, multiplication induces an $A$-dga homomorphism $C^\bullet\to \End_A(A\otimes_k M^\bullet)$. Let us abuse notation and use the same letter for elements of $C^\bullet$ and their multiplication endomorphism.
	\begin{enumerate}
		\item For any $\phi\in \fm C^1$, $[\phi,\phi]=0$, so the Maurer-Cartan equation is equivalent to $[d,\phi]=0\in \End_A(A\otimes_k M^\bullet)$, and therefore the condition that $\phi$ is closed in $C^\bullet$ is sufficient for the Maurer-Cartan equation to hold.
		\item For any $\phi\in \fm C^1$ and $\rho\in \fm C^0$, $[\rho,d+\phi]=-d \rho$. Therefore, $[\rho,[\rho,d+\phi]]=0$. Applying Lemma~\ref{lem:dglas}, $e^\rho$ is an isomorphism between $(A\otimes_k M^\bullet,d+\phi)$ and $(A\otimes_k M^\bullet,d+\phi-d \rho)$.
	\end{enumerate}
The same result also holds in the case of sheaves, as in Lemma~\ref{lem:dglas}.
\end{remark}

\subsection{Mixed Hodge structures and complexes}\label{ss:MHSsAndComplexes} 
The purpose of this section is to compile relevant definitions and to set notations related to mixed Hodge structures (MHSs) and mixed Hodge complexes of sheaves. Throughout this section, $k$ will be a subfield of $\R$ and $X$ will be a topological space. We start by recalling how multi-linear algebra constructions behave with respect to MHSs.

\begin{defprop}[MHS on the dual, tensor product and symmetric product, cf. Examples 3.2 in \cite{peters2008mixed}]\label{defprop:multilinear}
	Let $(V,W_{\lc},F^{\lc})$ and $(V',W_{\lc},F^{\lc})$ be $k$-vector spaces endowed with a MHS, where $W_{\lc}$ are the decreasing weight filtrations in $V$ and $V'$ and $F^{\lc}$ are the increasing Hodge filtrations in $V_\C$ and $V'_\C$. Here $V_\C$ (resp. $V'_\C$) denotes $V\otimes_k\C$ (resp. $V'\otimes_k\C$).
	\begin{itemize}
		\item $(\Hom_k(V,k), W_{\lc},F^{\lc})$ is a MHS, where 
		\begin{align*}
			W_{-n}\Hom_k(V,k)&=\{f:V\to k\mid W_{n-1}V\subset\ker f \}\quad\text{for all }n,\text{ and}\\
			F^{-p}\left(\Hom_k(V,k)_\C\right)&=F^{-p}\Hom_\C(V_\C,\C)\\&=\{f:V_\C\to \C\mid F^{p+1}(V_\C)\subset\ker f \}\quad\text{for all }p.\\
		\end{align*}
		\item $(V\otimes_k V', W_{\lc},F^{\lc})$ is a MHS, where
		\begin{align*}
			W_{n}(V\otimes_k V')&=\sum_{m}W_m V\otimes_k W_{n-m}V'\quad\text{for all }n,\text{ and}\\
			F^{p}(V_\C\otimes_\C V'_\C)&=\sum_{m}F^m V_\C\otimes_\C F^{p-m}V'_\C\quad\text{for all }p.\\
		\end{align*}
		\item Let $j\geq 1$. The projection $\underbrace{V\otimes_k V\otimes_k\cdots\otimes_k V}_{j}\to\Sym^j V$ induces a MHS on $\Sym^j V$ given by the image of the filtrations $W_{\lc}$ and $F^{\lc}$ in $\underbrace{V\otimes_k V\otimes_k\cdots\otimes_k V}_{j}$. By convention, $\Sym^0 V=k$ is a pure Hodge structure of type $(0,0)$.
		\item The multiplication map $V\otimes_k \Sym^j V\to \Sym^{j+1} V$ is a MHS morphism.
	\end{itemize}
\end{defprop}

\begin{remark}[MHS in homology]\label{rem:DeligneHomology}
	Let $X$ be a complex algebraic variety. By work of Deligne \cite{DeligneII,DeligneIII}, $H^i(X,k)$ carries a canonical and functorial MHS for all $i\geq 0$. Since $H^i(X,k)$ is finite dimensional, its dual is canonically isomorphic to $H_i(X,k)$. By Definition-Proposition~\ref{defprop:multilinear}, $H_i(X,k)$ also carries a canonical and functorial MHS.
\end{remark}

\begin{remark}[MHS on $R_m$ and $R_{-m}$]\label{rem:MHSRm}
	Let $m>0$, and let $R_m$ and $R_{-m}$ as in Definition~\ref{def:Rminfty}. Since $$R_{m}\cong\prod\limits_{j=0}^{m-1} \Sym^j H_1(G,k)=\bigoplus\limits_{j=0}^{m-1}\Sym^j H_1(G,k)$$ and the direct sum of MHSs is a MHS, the MHS on $H_1(G,k)$ endows both $R_m$ and $R_{-m}$ (its $k$-dual) with a MHS by Definition-Proposition~\ref{defprop:multilinear}.
\end{remark}

In this paper, we will obtain infinite sequences of MHS of the form 
\[
\cdots \twoheadrightarrow
V_{m+1} \twoheadrightarrow
V_{m} \twoheadrightarrow
V_{m-1} \twoheadrightarrow\cdots \twoheadrightarrow V_1.
\]
The inverse limit of such a sequence can be regarded as a pro-MHS. We will not use the definition in this paper, as we will just  construct some pro-MHS and morphisms between them naively.

\begin{remark}[Pro-MHS]
	Let $V=\varprojlim_m V_m$, where each $V_m$ is a $k$-MHS for all $m\geq 1$, and all the morphisms involved are MHS morphisms. This data can be regarded as a \emph{pro-MHS}. There is a category of pro-MHS that can be constructed as the usual abstract nonsense pro-completion: one would simply have to replace the index set $\Z_{> 0}$ by a more general filtered set (or category) to define a pro-MHS in full generality. Morphisms are defined as follows. Suppose we are given two pro-MHS constructed in this way, $V=\varprojlim_m V_m$ and $W=\varprojlim_{m'} W_{m'}$. Then,
	\[
\Hom_{\mathrm{pro-MHS}}\left(
V,W
\right)
=
\Hom_{\mathrm{pro-MHS}}\left(
\varprojlim_{m} V_m, \varprojlim_{m'} W_{m'}
\right)	\coloneqq 
\varprojlim_{m'} \varinjlim_m \Hom_{\mathrm{MHS}}(V_m, W_{m'})
	\]
	Plainly, a morphism consists of: for every $m'$ one must choose an $m$ and a morphism $V_m\to W_{m'}$, and these must be all compatible in the obvious ways. In this paper, the only such morphisms that will appear will be constructed in the most naive way: for every $m'$, we will take $m=m'$. I.e. given morphisms of MHS $V_m\to W_m$ for all $m\geq 1$ commuting with the linear maps $V_{m'}\to V_{m}$ and $W_{m'}\to W_{m}$ for all $m'\geq m\geq 1$, we obtain a \emph{morphism of pro-MHSs} $V\to W$.
\end{remark}

\begin{remark}\label{rem:proMHScategory}
	Inverse limits are left exact, and in the context of inverse limits of finite dimensional vector spaces, they are also right exact, since these inverse limits always satisfy the Mittag-Leffler condition. Hence, the category of pro-MHS has kernels, images and cokernels, and they coincide with the kernels, images and cokernels of the underlying vector spaces \(V\).
\end{remark}

\begin{definition}[{\cite[Definition 3.13]{peters2008mixed}}]
	A \emph{$k$-mixed Hodge complex of sheaves on a topological space $X$} is a triple $$\cK^\bullet= ((\cK^\bullet_k,W_{\lc}), (\cK^\bullet_\C, W_{\lc},F^{\lc}), \alpha)$$ 
	where:
	\begin{itemize}
		\item $\cK^\bullet_k$ is a bounded below complex of sheaves of $k$-vector spaces on $X$ such that $\mathbb{H}^*(X,\cK_k^\bullet)$ are finite-dimensional, and $W_{\lc}$ is an increasing (weight) filtration on $\cK^\bullet_k$.
		\item $\cK^\bullet_\C$ is a bounded below complex of sheaves of $\C$-vector spaces on $X$, $W_{\lc}$ is an increasing (weight) filtration and $F^{\lc}$ a decreasing (Hodge) filtration on $\cK^\bullet_\C$.
		\item $\alpha\colon (\cK^\bullet_k,W_{\lc}) \dashrightarrow (\cK^\bullet_\C,W_{\lc})$ is a \emph{pseudo-morphism of filtered complexes} of sheaves of $k$-vector spaces on $X$ (i.e.\ a chain of morphisms of bounded-below complexes of sheaves as in \cite[Definition 2.31]{peters2008mixed} except that each complex in the chain is filtered, as are all the morphisms) that induces a filtered pseudo-isomorphism
		\[
		\alpha \otimes 1\colon (\cK^\bullet_k \otimes \C, W_{\lc} \otimes \C) \dashrightarrow (\cK^\bullet_\C,W_{\lc})
		\]
		that is, a pseudo-isomorphism on each graded component.
		\item for $m \in \Z$, the $m$-th $W$-graded component 
		\[
		\Gr_m^W\cK^\bullet = \Big(\Gr_m^W\cK^\bullet_k, \big(\Gr_m^W\cK^\bullet_\C, F^{\lc}\big), \Gr_m^W\alpha\Big)
		\]
		is a $k$-Hodge complex of sheaves \cite[Definition 2.32]{peters2008mixed} on $X$ of weight $m$, where $F^{\lc}$ denotes the induced filtration.
	\end{itemize}
\end{definition}

	We have not explicitly defined the concept of a $k$-Hodge complex of sheaves, but we will only use it in the proof of Lemma~\ref{lem:MHStensorMHC}, where we will enumerate the conditions that need to be verified. The strategy that we will follow to prove that the rest of our constructions yield new mixed Hodge complexes of sheaves will consist on proving that its $W$-graded components coincide with complexes that are known to be Hodge complexes of sheaves by Lemma~\ref{lem:MHStensorMHC}.
	
	We will sometimes introduce a $k$-mixed Hodge complex of sheaves on $X$ simply as $\cK^\bullet$ and implicitly assume the components of the triple to be notationally the same as in the above definition. 
\begin{definition}\label{def:multiplicativeMHC}
	A \emph{multiplicative $k$-mixed Hodge complex of sheaves on $X$} is a $k$-mixed Hodge complex of sheaves $\cK^\bullet$ on $X$ such that the pseudo-morphism $\alpha$ has a distinguished representative given by a chain of morphisms of sheaves of cdgas on $X$ (with all but $\cK^\bullet_k$ being a sheaf of $\C$-cdgas), and such that \textit{all} filtrations (including those in the chain) are cdga-filtrations (over $\C$ except for the weight filtration on $\cK^\bullet_k$). 
\end{definition}

From a given mixed Hodge complex of sheaves, one can construct others (translation, Tate twists) as follows:
We can also obtain new MHSs from a given MHS by shifting the filtrations appropriately.

\begin{definition}[Tate twist]\ \\
	
	\begin{itemize}
		%
		\item Suppose $\cK^\bullet$ is a $k$-mixed Hodge complex of sheaves on $X$.
		The \emph{$j$-th Tate twist of $\cK^\bullet$} is the triple
		\[
		\cK^\bullet(j) = \Big(\big(\cK^\bullet_k,W[2j]_{\lc}\big), \big(\cK^\bullet_\C, W[2j]_{\lc},F[j]^{\lc}\big), \alpha\Big)
		\]
		where $W[2j]_i = W_{2j+i}$ and $F[j]^i = F^{j+i}$ are shifted filtrations.
		$\cK^\bullet(j)$ is again a $k$-mixed Hodge complex of sheaves on $X$. For details see \cite[Definition 3.14]{peters2008mixed}.
		
		\item The \emph{$j$-th Tate twist of a $k$-mixed Hodge structure} is defined by shifting the weight and Hodge filtrations with the same formula we used for mixed Hodge complexes above. See \cite[Example 3.2(3)]{peters2008mixed} for an explicit definition.
	\end{itemize}
\end{definition}
Notice that we have changed the convention of \cite[Examples 3.2 (3)]{peters2008mixed} in all of these definitions of Tate twists by selecting not to multiply by $(2 \pi i )^{2j}$.

\begin{remark}\label{rem:TateCohomology}
	The two definitions of Tate twist above are compatible in the following sense. Let $\mathbb{H}^i(X,\cK^\bullet_k(j))$ (resp. $\mathbb{H}^i(X,\cK^\bullet_k)$)  be the $k$-MHS induced in hypercohomology by the $k$-mixed Hodge complex of sheaves $\cK^\bullet(j)$ (resp. $\cK^\bullet$). By \cite[Theorem 3.18]{peters2008mixed},  $\mathbb{H}^i(X,\cK^\bullet_k(j))=\mathbb{H}^i(X,\cK^\bullet_k)(j)$.
\end{remark}

\begin{example}\label{eg:TateRm}
Suppose that $G=\C^*$ in Remark~\ref{rem:MHSRm}, and let $m\geq 1$. We have that $H_1(G,k)$ is a pure Hodge structure of type $(-1,-1)$. Let $s$ be a generator of $H_1(G,\Z)$ (seen inside of $H_1(G,k)$). We have that $\{1=s^0,s,s^2,\ldots, s^{m-1}\}$ is a $k$ basis of $R_m$, and let $\{1^\vee=(s^0)^\vee,s^\vee,(s^2)^\vee,\ldots, (s^{m-1})^\vee\}$ be its dual basis. The following $k$-linear isomorphism defined on a basis as
$$
\begin{array}{cccl}
	A_m:&R_m & \longrightarrow & R_{-m}\\
	&s^j & \longmapsto & (s^{m-1-j})^\vee \quad\text{ for all }j=0,\ldots, m-1
\end{array}
$$
is also an $R_\infty$-linear isomorphism which induces a MHS isomorphism $R_m(1-m)\to R_{-m}$, where $(1-m)$ denotes the $(1-m)$-th Tate twist.
\end{example}

\begin{definition}[Translation of a mixed Hodge complex of sheaves]\label{def:translatedMHC}
	\ 
	\begin{itemize}
		\item If $\mathcal{F}^{\bullet}$ is a complex of sheaves on $X$, then its \emph{translation} by $r\in\Z$ is the complex $\mathcal{F}^{\bullet}[r] = \mathcal{F}^{\bullet+r}$ with differential $d^\bullet[r] = -(1)^rd^{\bullet+r}$.
		\item Suppose that $\cK^\bullet$ is a $k$-mixed Hodge complex of sheaves on $X$.
		The \emph{translation of $\cK^\bullet$} by $r\in \Z$ is the triple
		\[
		\cK^\bullet[r] = ((\cK^\bullet_k[r], W[-r]_{\lc}), (\cK^\bullet_\C[r], W[-r]_{\lc}, F^{\lc}), \alpha[r])
		\]
		where the filtrations are described by:
		\begin{align*}
			&(W[-r])_i\left(\cK^\bullet_k[r]\right) = \left(W_{i-r}\cK^\bullet_k\right)[r],\ (W[-r])_i\left(\cK^\bullet_\C[r]\right) = \left(W_{i-r}\cK^\bullet_\C\right)[r],\ i \in \Z,\\
			&F^p\left(\cK^\bullet_\C[r]\right) = \left(F^{p}\cK^\bullet_\C\right)[r],\ p \in \Z.
		\end{align*}
		This is again a $k$-mixed Hodge complex of sheaves on $X$.
	\end{itemize}
\end{definition}
Note that this does not agree with the translation of a pure Hodge complex as defined in \cite[Lemma-Definition 2.35]{peters2008mixed}. In fact, this notion of translation increases the weight of a pure Hodge complex by 1, whereas the translation in loc. cit. decreases it (contrary to what is stated in loc. cit.). It does agree with the translation of mixed Hodge complexes implicit in \cite[Theorem~3.22]{peters2008mixed}.

\begin{remark}\label{transvstate}
	Suppose $\cK^\bullet$ is a $k$-mixed Hodge complex of sheaves on $X$.
	By \cite[Theorem 3.18.II]{peters2008mixed} the hypercohomology vector spaces $\mathbb{H}^*(X, \cK^\bullet_k)$ inherit $k$-mixed Hodge structures. Furthermore, it can be easily shown that
	\[\mathbb{H}^{*}(X, \cK^\bullet_k[r]) \cong \mathbb{H}^{*+r}(X, \cK^\bullet_k),\]
	where the $k$-mixed Hodge structure on the left-hand side has been induced by the translated $k$-mixed Hodge complex $\cK^\bullet[r]$.
\end{remark}

\begin{definition}[Derived direct image of a mixed Hodge complex of sheaves.]\label{def:directim}
	Let $\cK^\bullet$ be a $k$-mixed Hodge complex of sheaves on $X$ where the filtrations $W_{\lc}$ and $F^{\lc}$ are biregular (i.e. for all $m$, the filtrations induced on $\cK^m$ are finite). Suppose that $g\colon X\rightarrow Y$ is a continuous map between two topological spaces. The derived direct image of $\cK^\bullet$ via $g$ is again a mixed Hodge complex of sheaves, and it is defined as follows (\cite[B.2.5]{peters2008mixed}).
	
	Let $\Tot[\cC_{\Gdm}^\bullet\cF^{\bullet}]$ be the Godement resolution of a complex of sheaves $\cF^{\bullet}$ as defined in \cite[B.2.1]{peters2008mixed}, which is a flabby resolution. Here, $\Tot[\cC_{\Gdm}^\bullet\cF^{\bullet}]$ denotes the simple complex associated to the double complex $\cC_{\Gdm}^\bullet\cF^{\bullet}$. We define $Rg_*\cK^\bullet$ to be the triple
	\[
		\begin{split}
		\Big(\big(g_* \Tot[\cC_{\Gdm}^\bullet \cK^\bullet_k],g_* \Tot[\cC_{\Gdm}^\bullet W_{\lc}]\big),\quad
		\big(g_* \Tot[\cC_{\Gdm}^\bullet \cK^\bullet_\C], g_* \Tot[\cC_{\Gdm}^\bullet W_{\lc}], g_* \Tot[\cC_{\Gdm}^\bullet F^{\lc}]\big),\quad
		g_*\alpha\Big),
		\end{split}
	\]
	where $g_*\alpha$ is the pseudo-morphism of filtered complexes of sheaves of $k$-vector spaces induced by $\alpha$ and the functoriality of both $g_*$ and the Godement resolution.
\end{definition}

\subsection{The analytic logarithmic Dolbeault complex}\label{ss:analyticDolbeault}

Let $U$ be a smooth algebraic variety, let $X$ be a good compactification of $U$ and let $D\coloneqq X\setminus U$. Deligne defined a mixed Hodge complex of sheaves on $X$ whose hypercohomology computes $H^*(U,\R)$, endowing it with a canonical and functorial mixed Hodge structure. If $j\colon U\hookrightarrow X$ is the inclusion, Deligne considered:
\[
(j_*\cE^\bullet_{U,\R},\tau_{\leq \lc}) \to 
(j_*\cE^\bullet_{U,\C},\tau_{\leq \lc}) \xhookleftarrow{\sim} 
(\Omega^\bullet_{X}(\log D),\tau_{\leq \lc})\xhookrightarrow{\sim}
(\Omega_X^\bullet(\log D),W_{\lc},F^{\lc}),
\]
where $\cE^\bullet_{U,\R}$ is the real ($C^{\infty}$) de Rham complex on $U$, $\cE^\bullet_{U,\C}\coloneqq \cE^\bullet_{U,\R}\otimes_{\R}\C$, $\Omega_U^\bullet$ is the holomorphic de Rham complex on $U$, and $\Omega_X^\bullet(\log D)$ is the subcomplex of $j_*\Omega_U^\bullet$ formed by the forms $\omega$ with logarithmic poles along $D$ (both $\omega$ and $d\omega$ have at most a pole of order $1$ along $D$). $\tau_{\leq \lc}$ is the canonical increasing filtration, $F^{\lc}$ is the decreasing trivial filtration, and $W_{\lc}$ is given by the order of the poles (see \cite[Theorem 4.2]{peters2008mixed} for the precise definition).

Let us introduce different de Rham complexes from the ones in Deligne's mixed Hodge complex of sheaves.

\begin{definition}[Real analytic de Rham complex of sheaves]\label{def:analyticDeRham}
	For every smooth complex algebraic variety $Y$,
	\begin{itemize}
		\item $(\cA_{Y,\R}^\bullet,d)$ denotes the real analytic de Rham complex of sheaves on $Y$,
		\item $(\cA_{Y,\C}^\bullet,d)$ denotes the real analytic de Rham complex of sheaves on $Y$ with values over $\C$, i.e. $\cA_{Y,\C}^\bullet\coloneqq \C \otimes_\R \cA_{Y,\R}^\bullet$.
		\item $\cA_{Y,\C}^\bullet$ has a bigrading induced by the complex structure on $Y$, which we denote by $\cA_{Y,\C}^{\bullet\bullet}$, as follows: if $z_1,\ldots,z_n$ are local holomorphic coordinates of $Y$, the forms in $\cA_{Y,\C}^{p,q}$ are locally generated over $\cA_{Y,\C}^0$ by $dz_{i_1}\wedge\ldots\wedge dz_{i_p}\wedge d\ov{z_{j_i}}\wedge\ldots\wedge d\ov{z_{j_q}}$ for $i_1,\ldots,i_p,j_i,\ldots,j_q\in\{1,\ldots,n\}$.
	\end{itemize}
\end{definition}

\begin{remark}\label{rmk:acyclicity}
	$\mathcal A^\bullet_{U,\C}$ (and thus also $\mathcal A^\bullet_{U,\R}$) is a complex of $j_*$-acyclic and $\Gamma$-acyclic sheaves. The $\Gamma$-acyclicity is well known (see \cite[p. 127]{ks}, for example). For the $j_*$-acyclicity, it suffices to show that $\bH^i(V_x\cap U,\mathcal A^k_{U,\C})=0$ for all $i>0$, $k\geq 0$, $x\in X$ and certain $V_x$ neighborhoods of $x$ in $X$ forming a basis. Let $Y=V_x\cap U$, which is a complex (and thus real) analytic manifold. As a real analytic manifold, $Y$ can be embedded into $Y\times \overline Y$ as the diagonal, where $\overline Y$ is the complex conjugate of $Y$, and the restriction of $\mathcal{O}_{Y\times\overline Y}$ (the sheaf of complex analytic functions on $Y\times\overline Y$) to $Y$ is $\mathcal A^0_{U,\C}$. By \cite[Proposition 5.42]{SteinWeinstein}, $Y$ possesses arbitrarily small neighborhoods in $Y\times \overline Y$ which are Stein. Let $W$ be one such Stein neighborhood. By Oka's coherence theorem and Cartan's theorem B, one gets  $\bH^i(W,\mathcal O_{Y\times\overline Y})=0$ for all $i>0$, so $R\Gamma (W,\mathcal O_{Y\times\overline Y})\cong\Gamma (W,\mathcal O_{Y\times\overline Y})$. Since $Y$ is closed in the paracompact space $W$, taking direct limits on $W$ approaching $Y$ yields $R\Gamma (Y,\mathcal A^0_{U,\C})=\varinjlim\limits_{ W}\Gamma (W,\mathcal O_{Y\times\overline Y})=\Gamma(Y,\mathcal A^0_{U,\C})$, 
	obtaining the $j_*$-acyclicity of $\mathcal A^0_{U,\C}$. For $k\geq 1$ we follow the same argument, replacing $\mathcal O_{Y\times\overline Y}$ by $\Omega_{Y\times\overline Y}^k$, which is locally free as an $\mathcal O_{Y\times\overline Y}$-module and thus it is also coherent.
\end{remark}

In \cite{navarroAznar} (see also \cite{burgos} for a similar complex using $C^{\infty}$ functions), Navarro Aznar defined a different mixed Hodge complex of sheaves. The complexes of sheaves involved in its construction are the real and complex-valued logarithmic Dolbeault complexes, defined as follows.

\begin{definition}[Logarithmic Dolbeault Complex]\label{def:logDolbeault}
	Let $U$ be a smooth connected complex algebraic variety, let $X$ be a good compactification of $U$, let $D=X\setminus U$, and let $j:U\hookrightarrow X$ be the inclusion. Let us write local holomorphic coordinates $(z_i)$ on $X$ such that $D$ has equation $z_1\cdots z_r=0$. The real logarithmic Dolbeault complex $\cA^{\bullet}_{X,\R}(\log D)$ is the sub-$\cA^0_{X,\R}$-algebra of $j_*\cA^\bullet_{U,\R}$ generated by the local sections
	\[
	\begin{split}
		\log (z_i\ov {z_i}),\Re\frac{dz_i}{z_i},\Im\frac{dz_i}{z_i}\quad \text{for }1\le i\le r,\quad 
		\Re dz_i,\Im dz_i\quad \text{for }i>r.
	\end{split}
	\]
	Similarly, the complex logarithmic Dolbeault complex is defined by $\cA^{\bullet}_{X,\C}(\log D)\coloneqq \cA^{\bullet}_{X,\R}(\log D)\otimes_\R \C$.
	
	Moreover, we define a bigrading on $\cA^\bullet_{X}(\log D)$ induced by the bigrading on $\cA^{\bullet}_U$ as follows:
	\[
	\cA^{p,q}_{X,\C}(\log D)\coloneqq \cA^{p+q}_{X,\C}(\log D)\cap j_*\cA^{p,q}_{U,\C}.
	\]
\end{definition}

\begin{definition}[Navarro Aznar's mixed Hodge complex of sheaves]\label{def:NavarroAznar}
	Let $U$, $X$, $D$, $j$, $(z_i)$ be as in Definition~\ref{def:logDolbeault}. The following data describes a mixed Hodge complex of sheaves $\wt{\cN^{\bullet}_{X,D}}$:
\[
\wt{\cN^{\bullet}_{X,D}}\coloneqq \left( (\cA^\bullet_{X,\R}(\log D),\wt{W}_{\lc}), (\cA^\bullet_{X,\C}(\log D),\wt{W}_{\lc},F^{\lc}),\alpha\right),
\]	
where
\begin{itemize}
	\item the weight filtration $\wt{W}_{\lc}$ on $\cA^\bullet_{X,\R}(\log D)$ (resp. on $\cA^\bullet_{X,\C}(\log D)$) is the multiplicative increasing filtration generated by assigning weight $0$ to the sections of $\cA_{X,\R}^\bullet$ and weight $1$ to the sections defined locally by $\log (z_i\ov {z_i}),\Re\frac{dz_i}{z_i},\Im\frac{dz_i}{z_i}$ for $1\le i\le r$,
	\item the Hodge filtration $F^{\lc}$ on $\cA^\bullet_{X,\C}(\log D)$ is defined by
	\[
	F^p\cA^\bullet_{X,\C}(\log D)\coloneqq \bigoplus_{p'\geq p }\cA^{p',\bullet}_{X,\C}(\log D), \quad\text{and}
	\]
	\item $\alpha\otimes 1:(\cA^\bullet_{X,\R}(\log D),\wt{W}_{\lc})\otimes_\R\C\to (\cA^\bullet_{X,\C}(\log D),\wt{W}_{\lc},F^{\lc})$ is the identity.
\end{itemize}
\end{definition}

We recall here some important properties of $\wt{\cN^{\bullet}_{X,D}}$.
\begin{theorem}[\cite{navarroAznar}, Theorem 8.8]\label{thm:propertiesNA}
	The inclusion $(\Omega_X^\bullet(\log D),W_{\lc},F^{\lc})\to(\cA^\bullet_{X,\C}(\log D),\wt{W}_{\lc},F^{\lc})$ is a bi-filtered quasi-isomorphism. The (weight and Hodge) filtrations on $\Omega_X^\bullet(\log D)$, which were described at the beginning of the section, coincide with the ones induced by the filtrations on $\cA^\bullet_{X,\C}(\log D)$.
\end{theorem}

\begin{proposition}[Proposition 8.4, \cite{navarroAznar}]\label{prop:adjunctionNA}
	The identity
	$$
	\left(\cA^\bullet_{X,\R}(\log D),\tau_{\leq \lc}\right)\to \left(\cA^\bullet_{X,\R}(\log D),\wt{W}_{\lc}\right)
	$$
	and the inclusion
	$$
	\left(\cA^\bullet_{X,\C}(\log D),\tau_{\leq \lc}\right)\hookrightarrow \left(j_*\cA^\bullet_{U,\C},\tau_{\leq \lc}\right)
	$$
	are both filtered quasi-isomorphisms. Moreover, the second map coincides with the adjunction $\Id\to Rj_*j^{-1}$ applied to $\cA^\bullet_{X,\C}(\log D)$ when seen as a morphism in the derived category.
\end{proposition}

\begin{corollary}\label{cor:comparisonMHC}
	The following is a diagram of filtered quasi-isomorphisms (the last one is bi-filtered), where the maps are either the identity or the natural inclusions.
	\begin{align*}
		&\left(\cA^\bullet_{X,\C}(\log D),\wt{W}_{\lc}\right)\xleftarrow{\Id} \left(\cA^\bullet_{X,\C}(\log D),\tau_{\leq \lc}\right) \hookrightarrow \left(j_*\cE^\bullet_{U,\C},\tau_{\leq \lc}\right)\hookleftarrow \left(\Omega_X^\bullet(\log D),\tau_{\leq \lc}\right)\xrightarrow{\Id}\\ &\left(\Omega_X^\bullet(\log D),W_{\lc}, F^{\lc}\right)\hookrightarrow \left(\cA^\bullet_{X,\C}(\log D),\wt{W}_{\lc}, F^{\lc}\right)
	\end{align*}
	The composition of all of these maps is the identity in the derived category.
\end{corollary}
\begin{proof}
	The (de Rham) resolution $\ul \R_U\to \cE^\bullet_{U,\R}$ factors through $\cA^\bullet_{U,\R}$, which is quasi-isomorphic (through the inclusion map) to the bigger sheaf complex $\cE^\bullet_{U,\R}$, because both resolve the trivial local system $\ul\R_U$ (see \cite[p. 127]{ks}, for example). Since $\cE^\bullet_{U,\R}$ is a complex of soft sheaves, this gives rise to an isomorphism $Rj_* \ul \R_U\to j_*\cE^\bullet_{U,\R}$ in the derived category, and Proposition~\ref{prop:adjunctionNA} shows that the second map in the chain of maps in the statement of this corollary is a (trivially filtered) quasi-isomorphism. The rest of the maps involved are (bi-)filtered quasi-isomorphisms by Theorem~\ref{thm:propertiesNA} and  Proposition~\ref{prop:adjunctionNA}.
	
	The statement about the composition of all of those maps in the derived category follows from the fact that the inclusion  $\Omega_X^\bullet(\log D)\to j_*\cE^\bullet_{U,\C}$ factors through $\cA^\bullet_{X,\C}(\log D)$.
\end{proof}

Note that the filtration $\wt{W}_{\lc}$ on the logarithmic Dolbeault complex is not biregular, a hypothesis which is needed for Definition~\ref{def:directim}, for example. However, we may tweak it a little to get a mixed Hodge complex of sheaves similar to the one from Definition~\ref{def:NavarroAznar} with biregular filtrations as follows.

\begin{defprop}[Modified mixed Hodge complex of sheaves of Navarro Aznar]\label{defprop:modifiedNA}
	Let $U$, $X$, $D$, $j$, $(z_i)$ be as in Definition~\ref{def:logDolbeault}, and let $\wt{W}_{\lc}$, $F^{\lc}$ and $\alpha$ be as in Definition~\ref{def:NavarroAznar}. 
	Let $n\geq \max\{2, \dim_\R U\}$, and let $W_{\lc}^n$ be the increasing cdga filtration on $\cA^\bullet_{X,\R}$ given by
	$$
	W_i^n \cA^j_{X,\R}\coloneqq\left\{\begin{array}{lr}\wt{W_i}\cA^\bullet_{X,\R}&\text{if }i\leq n,\\
		\cA^\bullet_{X,\R}&\text{if }i> n,\end{array}\right.
	$$
	and let $(\cA^j_{X,\C},W_{\lc}^n)\coloneqq (\cA^j_{X,\R}\otimes_\R \C,W_{\lc}^n\otimes_\R\C)$.
	
	Then, the following data describes a mixed Hodge complex of sheaves $\cN^{\bullet}_{X,D,n}$:
	\[
	\cN^{\bullet}_{X,D,n}\coloneqq \left( (\cA^\bullet_{X,\R}(\log D),W_{\lc}^n), (\cA^\bullet_{X,\C}(\log D),W_{\lc}^n,F^{\lc}),\alpha\right),
	\]	
	in which all the filtrations are biregular, and such that the identity morphism
	$$
	\wt{\cN^{\bullet}_{X,D}}\to\cN^{\bullet}_{X,D,n}\
	$$
	induces (bi-)filtered quasi-isomorphisms in its real and complex parts.
\end{defprop}
\begin{proof}
	The identity morphism $
	\wt{\cN^{\bullet}_{X,D}}\to\cN^{\bullet}_{X,D,n}
	$ induces filtered morphisms in its real and complex parts.
	By Proposition~\ref{prop:adjunctionNA}, if $m>n\geq\dim_\R U$, $\Gr_m^{\wt W} \cA^\bullet_{X,\R}$ is quasi-isomorphic to $0$, so it is exact. By induction, one can show that $\wt W_m \cA^\bullet_{X,\R}/\wt W_{n} \cA^\bullet_{X,\R}$ is an exact complex of sheaves for all $m>n$, which implies that $ \cA^\bullet_{X,\R}/\wt W_{n} \cA^\bullet_{X,\R}$ is an exact complex of sheaves, that is, quasi-isomorphic to $0$. This shows that the identity morphism induces quasi-isomorphisms between $\Gr_m^{\wt W}\cA^\bullet_{X,\R}$ and $\Gr_m^{W^n}\cA^\bullet_{X,\R}$ for all $m$. These quasi-isomorphisms are the identity if $m\leq n$. The same holds for $\cA^\bullet_{X,\C}$, which concludes our proof.
\end{proof}

\begin{remark}\label{rem:sameMHS}
	The isomorphism $Rj_* \ul \R_U\to j_*\cE^\bullet_{U,\R}$ described in the proof of Corollary~\ref{cor:comparisonMHC} is the one used to endow $H^*(U,\R)$ with Deligne's canonical mixed Hodge structure, using Deligne's mixed Hodge complex of sheaves described at the beginning of this section.
	
	The mixed Hodge complex of Navarro Aznar (Definition~\ref{def:NavarroAznar}) also induces a mixed Hodge structure on $H^*(U,\R)$ via the composition of $Rj_* \ul \R_U\to j_*\cE^\bullet_{U,\R}$ above with $\cA^\bullet_{X,\C}(\log D) \hookrightarrow j_*\cE^\bullet_{U,\C}$. By Corollary~\ref{cor:comparisonMHC}, both of these mixed Hodge structures on $H^*(U,\R)$ coincide. Consequently, this MHS coincides with the one induced by the modified complex $\cN_{X,D,n}^\bullet$ from Definition-Proposition~\ref{defprop:modifiedNA} for all $n\geq\max\{2,\dim_\R U\}$.
\end{remark}

\section{Thickening of a mixed Hodge complex of sheaves}\label{sec:thickening}

Let $k=\Q$ or $\R$. We will show how to construct a thickened mixed Hodge complex of sheaves for any multiplicative mixed Hodge complex of sheaves $\cK^\bullet$.

\subsection{The definition of the thickening}
The data required for the thickening should be understood as a MHS $V$ and a morphism $V[-1]\to \cK^\bullet$. Precisely, we require the following data.

\begin{assumption}\label{ass:ingredients}
	We consider the following objects.
	\begin{enumerate}
		\item A multiplicative mixed Hodge complex $\cK^\bullet=((\cK^\bullet_k,W_{\lc}),(\cK^\bullet_\C,W_{\lc},F^{\lc}),\alpha)$ on a topological space $X$. $\alpha$ is a filtered pseudo-morphism, which induces the filtered pseudo-isomorphism $\alpha\otimes 1$ after tensoring by $\C$ over $k$:
		\[
		(\cK^\bullet_k\otimes_{k}\C,W_{\lc}) = (\cK^\bullet_0,W_{\lc})\xrightarrow{\alpha_1}(\cK^\bullet_1,W_{\lc})\xleftarrow{\alpha_2} \cdots \xrightarrow{\alpha_{2M-1}}  (\cK^\bullet_{2M-1},W_{\lc})\xleftarrow{\alpha_{2M}} (\cK^\bullet_{2M},W_{\lc}) = (\cK^\bullet_\C,W_{\lc}).
		\]
		In addition, all the weight filtrations $W_{\lc}$ are biregular (i.e. for all $m$ and for all $0\leq i\leq 2M$, the weight
		filtration induced on $\cK^m_i$ is finite), and all the complexes of sheaves are bounded.

		\item A $k$-MHS $(V,W_{\lc},F^{\lc})$.   $V_\C$ will denote the vector space $V\otimes_k\C$.
		\item For every $i=0,\ldots,M$, a morphism
		\[
		\Phi_{2i}\colon (V_{\C},W_{\lc}[1])\to \Gamma\left(X,(\cK^{1,\cl}_{2i},W_{\lc})\right).
		\]
		Where $\cK^{1,\cl}=\ker d\subseteq \cK^1$. Additionally, $\Phi_{2M}$ is required to preserve $F^{\lc}$ and $\Phi_{0}$ must be defined over $k$.
		\item For every $i=1,\ldots, M$, a morphism
		\[
		\Psi_{2i-1}\colon (V_{\C},W_{\lc}[1])\to \Gamma\left(X,(\cK^{0}_{2i-1},W_{\lc})\right)
		\]

		such that
		\[
		d\circ \Psi_{2i-1} = \alpha_{2i-1}\circ\Phi_{2i-2}-\alpha_{2i}\circ \Phi_{2i}.
		\]
		In other words, the maps $\Phi_{2i}$ are only required to be compatible with $\alpha$ up to homotopy, and the homotopies are part of the data.
	\end{enumerate}
\end{assumption}

Our thickening, when seen as a deformation, will be parametrized by the formal neighborhood of the origin in $V$, i.e. if $V^\vee$ is the dual vector space, the base ring will be the completion of $\Sym^\bullet V^\vee$ at its maximal ideal. Concretely, the base ring will be the following: Let $V^\vee$ be the dual MHS. For the rest of this section, and for all $0\le m$, let us generalize Definition~\ref{def:Rminfty} (which assumes $V=H^1(G,k)$):
\[
R_{\infty} \coloneqq \prod_{j=0}^\infty \Sym^j V^\vee
;\quad
R_{m} \coloneqq \frac{R_{\infty}}{\prod_{j=m}^\infty \Sym^j V^\vee}\cong \bigoplus_{j=0}^{m-1} \Sym^j V^\vee
\]
We reuse the notation from Definition~\ref{def:Rminfty} because we will only construct explicit thickened complexes when $V=H^1(G,k)$, but the theory will be carried out with more generality in this section.

$R_{\infty}$ is a ring, whose multiplication is the usual multiplication in the symmetric tensor algebra, and $\prod_{j=m}^\infty \Sym^j V^\vee$ is an ideal for every $m\geq 1$. In fact, if we let
$$
\mathfrak a\coloneqq \prod_{j=1}^\infty \Sym^j V^\vee,
$$
then the ideal
$
\prod_{j=m}^\infty \Sym^j V^\vee
$ equals $\mathfrak a^m$ for all $m\geq 1$. Given a basis $s_1,\ldots, s_r$ of $V^\vee$, we obtain an isomorphism $R_{\infty}\cong k[[s_1,\ldots ,s_{r}]]$, and $R_{m}\cong k[[s_1,\ldots ,s_{r}]]/(s_1,\ldots,s_{r})^m$.

In order to work in both homology and cohomology, we will consider the $k$-dual of a deformed complex. This will require us to work over the $k$-dual modules of $R_{m}$. Let us from now on use $m$ to denote a nonnegative integer, and let:
\[
R_{-m}\coloneqq \Hom_{k}(R_{m},k) .
\]
The $R_{\infty}$-module structure of $R_{m}$ induces a module structure on $R_{-m}$. We will  abuse notation and denote by $R_\infty$, $R_{m}$ and $R_{-m}$ the same constructions but using $V_\C$ instead of $V$ and $\C$ instead of $k$. This abuse of notation will be clarified as follows: The expression $R_\infty\otimes_k -$ will assume that $R_\infty$ is constructed using $V$, whereas $R_\infty\otimes_\C -$ will assume that $R_\infty$ is constructed using $V_\C$.

For $m\ge 0$, $R_{m}$ has a MHS, namely the direct sum of the MHSs on $\Sym^j V^\vee$. Furthermore, \(R_{-m}\) has the dual MHS. In fact, using Definition-Proposition~\ref{defprop:multilinear}, one can see that the multiplication morphisms
\begin{equation}\label{eq:multMHS}
	V^\vee\otimes_k R_m \to R_m,\quad\text{and}\quad V^\vee\otimes_k R_{-m} \to R_{-m}
\end{equation}
are MHS morphisms.

\begin{definition}\label{def:epsi}
	Let $V$ as in Assumption~\ref{ass:ingredients}. We denote by $\eps_k$ the canonical element of $V^\vee\otimes_k V$, namely
	\[
	\eps_k = \sum_{i=1}^r s_i\otimes s_i^\vee,
	\]
	where $\{s_1,\ldots ,s_r\}$ is a basis of $V^\vee$ and $\{s_1^\vee,\ldots ,s_r^\vee\}$ is its dual basis. Similarly, $\eps_\C$ will denote the canonical element of $(V_\C)^\vee\otimes_\C V_\C$.
\end{definition}

\begin{definition}\label{def:imepsi}
	Consider the setup in Assumption~\ref{ass:ingredients}.
	\begin{itemize}
		\item We will denote by $\Phi_{2i}(\eps_\C)\in \Gamma(X,R_\infty\otimes_\C \cK^1_{2i})$ the image of $\eps_\C$ by $\Id_{V^\vee_\C}\otimes_\C \Phi_{2i}$ for all $i=0,\ldots, M$.
	\end{itemize}
Similarly,
\begin{itemize}
\item  If $i=0$, we will denote by $\Phi_{0}(\eps_k)\coloneqq\Id_{V^\vee}\otimes_k \Phi_{0}(\eps_k)\in \Gamma(X,R_\infty\otimes_k \cK^1_{k})$ (recall that $\Phi_0$ is defined over $k$).
\item $\Psi_{2i-1}(\eps_\C)\coloneqq\Id_{V^\vee_\C}\otimes_\C \Psi_{2i-1}(\eps_\C)\in \Gamma(X,R_\infty\otimes_\C \cK^0_{2i-1})$ for all $i=1,\ldots, M$.
\item $\alpha_{2i-1}\Phi_{2i-2}(\eps_\C)\coloneqq \Id_{V^\vee}\otimes_k \alpha_{2i-1}\circ\Phi_{2i-2}(\eps_k)\in \Gamma(X,R_\infty\otimes_\C \cK^1_{2i-1})$ for all $i=1,\ldots, M$.
\item $\alpha_{2i}\Phi_{2i}(\eps_\C)\coloneqq \Id_{V^\vee}\otimes_k \alpha_{2i}\circ\Phi_{2i}(\eps_k)\in \Gamma(X,R_\infty\otimes_\C \cK^1_{2i-1})$ for all $i=1,\ldots, M$.
\end{itemize}
\end{definition}

\begin{remark}
	Let $m\geq 1$. Left multiplication by $\Phi_{2i}(\eps_\C)$ defines an element of $$\Hom_{R_{m}}^1(R_{\pm m}\otimes_\C \cK^\bullet_{2i}, \mathfrak a R_{\pm m}\otimes_\C \cK^\bullet_{2i})$$ for all $i=0,\ldots, M$. Here $R_{-m}$ is seen as an $R_{m}$-module, and $R_m$ is a local Artinian $\C$-algebra with maximal ideal $\mathfrak a$ (where we are abusing notation and denoting by $\mathfrak a$ the image of the ideal $\mathfrak a$ of $R_\infty$ through the ring epimorphism $R_\infty\twoheadrightarrow R_m$). Similarly, left multiplication by the rest of the elements from Definition~\ref{def:imepsi} defines an element of
	$$\Hom_{R_{m}}^1(R_{\pm m}\otimes_\C \cK^\bullet_{j}, \mathfrak a R_{\pm m}\otimes_\C \cK^\bullet_{j})$$ for the appropriate $j$, except for $\Psi_{2i-1}(\eps_\C)$, which defines an element of  $$\Hom_{R_{ m}}^0(R_{\pm m}\otimes_\C \cK^\bullet_{2i-1}, \mathfrak a R_{\pm m}\otimes_\C \cK^\bullet_{2i-1})$$
	for all $i=1,\ldots, M$. In particular, since $\mathfrak a$ is a nilpotent ideal in $R_m$,
	$$
	e^{\Psi_{2i-1}(\eps_\C)}\coloneqq \sum_{j=0}^{\infty}\frac{1}{j!}(\Psi_{2i-1}(\eps_\C))^j=\sum_{j=0}^{m-1}\frac{1}{j!}(\Psi_{2i-1}(\eps_\C))^j\in \Hom_{R_{m}}^0(R_{\pm m}\otimes_\C \cK^\bullet_{2i-1}, R_{\pm m}\otimes_\C \cK^\bullet_{2i-1})
	$$
	is a globally (and well) defined endomorphism for all $i=1,\ldots,M$.
\end{remark}

\begin{lemma}\label{lem:MHStensorMHC}
  Let \(\cK^\bullet = ((\cK^\bullet_k, W_\lc), (\cK^\bullet_\C, W_\lc, F^\lc), \alpha)\) be a mixed Hodge complex of sheaves on a topological space \(X\), and let \(H\) be a $k$-MHS. Then $$
  H\otimes \cK^\bullet\coloneqq ((H\otimes_k\cK^\bullet_k, W_\lc), (H_\C\otimes_\C\cK^\bullet_\C, W_\lc, F^\lc), \Id_H\otimes\alpha),
  $$
  with the natural (tensor) filtrations from Definition-Proposition~\ref{defprop:multilinear}, differential, pseudo-morphism and \(k\)-structure, is also a mixed Hodge complex of sheaves.
\end{lemma}
\begin{proof}
Let us start by noting that, since the complex of sheaves $\cK_k^\bullet$ is a complex of sheaves of $k$-vector spaces, the identity induces an isomorphism for all \(j\):
\begin{equation}\label{eq:grSum}
  \Gr^W_j (H \otimes_k \cK^\bullet_k) \cong \bigoplus_{a + b = j} \Gr^W_{a} H \otimes_k \Gr^W_{b}\cK^\bullet_k,
\end{equation}
and similarly for $H_\C\otimes_\C\cK_i^\bullet$ for every complex of sheaves of $\C$-vector spaces $\cK_i^\bullet$ appearing in the pseudo-isomorphism $\alpha\otimes 1$.

  In order for $H\otimes \cK^\bullet$ to be a mixed Hodge complex of sheaves, it must satisfy the following requirements. We will begin with the more straightforward properties.
  \begin{itemize}
  \item \textbf{The vector spaces \(\bH^*(X, H\otimes_k \cK^\bullet)\) are finite-dimensional:} this follows from the fact that \(k\) is a field and therefore \(H\) is flat, so these are isomorphic to \(H\otimes_k \bH^*(X, \cK^\bullet)\).
  \item \textbf{The differentials preserve the weight and Hodge filtrations:} this is a direct consequence of the definition of the differential on the tensor product as \(\Id_H\otimes d\), where $d$ denotes the differential in $\cK^\bullet$.
  \item \textbf{The maps \(\Id_H\otimes \alpha\) form a pseudo-morphism which becomes a filtered pseudo-isomorphism after tensoring by $\C$ over $k$:} this also follows also from the flatness of \(H_\C\) over \(\C\) and its graded pieces, together with the direct sum decomposition~\eqref{eq:grSum}.
  \end{itemize}
  Finally, we must show that the associated graded for the weight filtration is a pure Hodge complex of sheaves, as defined in \cite[Definition 2.32]{peters2008mixed}. Applying the decomposition~\eqref{eq:grSum}, it suffices to show this for any summand of the form \(\Gr_a^W H\otimes \Gr^W_b \cK^\bullet \). In other words, the problem is reduced to the case where \(H\) is a pure Hodge structure of weight \(a\) and \(\cK^\bullet\) is a pure Hodge complex of weight \(b\), and we need to show that \(H\otimes \cK^\bullet\) is a pure Hodge complex of weight \(a+b\). This amounts to showing the following properties:
  \begin{itemize}
  \item \textbf{The vector spaces \(\bH^*(X, H_k\otimes_k \cK^\bullet)\) are finite-dimensional:} this is the same as above.
  \item \textbf{The maps \(\Id_H\otimes \alpha\) form a pseudo-morphism which becomes a  pseudo-isomorphism after tensoring by $\C$ over $k$:} this is the same as above.
  \item \textbf{The spectral sequence \(\bH^{p+q}(X, \Gr^p_F (H_\C \otimes_\C\cK_\C^\bullet))\Rightarrow \bH^{p+q}(X, H_\C \otimes_\C\cK^\bullet_\C)\) degenerates at \(E_1\):} Let us use the Hodge decomposition of \(H\), namely $F^{p_1}H_\C=\bigoplus_{i\geq p_1}(H_\C)^{i,a-i}$, to decompose the tensor product:
  \begin{equation}\label{eq:hodgeDecomp}
  	\begin{split}
  H_\C \otimes_\C \cK^\bullet_\C &=
  \bigoplus_{p} (H_\C)^{p, a-p} \otimes_\C \cK^\bullet_\C ,
  \\
  F^p(H_\C \otimes_\C \cK^\bullet_\C) &=
  \sum_{p_1+ p_2 = p} F^{p_1}H_\C \otimes_\C F^{p_2}\cK^\bullet_\C =
  \bigoplus_{p_1+ p_2 = p} (H_\C)^{p_1, a-p_1} \otimes_\C F^{p_2}\cK^\bullet_\C .
\end{split}
  \end{equation}
Since \(\cK^\bullet_\C\) is a Hodge complex of sheaves, the spectral sequence for the direct summands converges at \(E_1\), since it only differs from the one for \(\cK^\bullet_\C\) by a tensor with a vector space and a shift in the filtration.
  
  Applying \cite[Lemma A.42]{peters2008mixed}, this in particular implies that the following morphisms, induced by the inclusion, are injective:
  \[
    \bH^{p+q}\left(X, F^p (H_\C\otimes_\C \cK^\bullet_\C)\right)\hookrightarrow \bH^{p+q}(X, H_\C\otimes_\C \cK^\bullet_\C).
  \]
  \item \textbf{The filtration induced by \(F\) endows \(\bH^j(X, H_k\otimes_k \cK^\bullet_k)\) with a Hodge structure of weight \(a+b+j\):} $F^p \bH^j(X, H_\C\otimes_\C \cK^\bullet_\C)$ is, by definition, the image of the morphism induced by the inclusion of sheaves (which is injective, as stated above)
  $$
  \bH^j\left(X, F^p (H_\C\otimes_\C \cK^\bullet_\C)\right)\hookrightarrow \bH^j(X, H_\C\otimes_\C \cK^\bullet_\C).
  $$

Applying \eqref{eq:hodgeDecomp}, $F^p\bH^j(X, H_\C\otimes_\C \cK^\bullet_\C)$ is the direct sum for all $p_1,p_2$ such that $p_1+p_2=p$ of all the images of the morphisms
$$
\bH^j\left(X, (H_\C)^{p_1,a-p_1}\otimes_\C F^{p_2}\cK^\bullet_\C\right)\to \bH^j(X, H_\C\otimes_\C \cK^\bullet_\C),
$$
that is, the direct sum of the images of 
$$
(H_\C)^{p_1,a-p_1}\otimes_\C \bH^j(X,  F^{p_2}\cK^\bullet_\C)\to H_\C\otimes_\C \bH^j(X,  \cK^\bullet_\C),
$$
which coincide with $(H_\C)^{p_1,a-p_1}\otimes_\C F^{p_2}\bH^j(X, \cK^\bullet_\C)$ for all $p_1+p_2=p$. Therefore,
$$
	F^p\bH^j(X, H_\C\otimes_\C \cK^\bullet_\C)  = \bigoplus_{p_1+p_2=p} (H_\C)^{p_1,a-p_1}\otimes_\C F^{p_2}\bH^j(X, \cK^\bullet_\C)=\sum_{p_1+p_2=p} F^{p_1}H_\C\otimes F^{p_2}\bH^j(X, \cK^\bullet_\C).
$$
That is, the filtration $F^\lc$ on $\bH^j(X, H_\C\otimes_\C \cK^\bullet_\C)$ is the tensor filtration on $H_\C\otimes_\C \bH^j(X, \cK^\bullet_\C)$. The rest follows from the fact that $H_\C$ is a Hodge structure of weight $a$ and, since $\cK^\bullet_\C$ is a Hodge complex of sheaves of weight $b$, $\bH^j(X, \cK^\bullet_\C)$ is a Hodge structure of weight $b+j$.

  \end{itemize}
\end{proof}

\begin{defprop}\label{defprop:thickMHC}
	Consider the objects in Assumption~\ref{ass:ingredients}, and let  $\mathbf{\Phi}\coloneqq(\Phi_0,\ldots,\Phi_{2M})$ and $\mathbf{\Psi}\coloneqq(\Psi_1,\ldots,\Psi_{2M-1})$. Let $\eps_\C$ and $\eps_k$ as in Definition~\ref{def:epsi}. Let $m\in \Z$  where $m\neq 0$, and let $(R_{m}, W_{\lc}, F^{\lc})$ be the MHS on $R_{m}$ as in this section.
	
	Then,
	$$
	\cK^\bullet(m,V,\mathbf{\Phi},\mathbf{\Psi})\coloneqq\left(\left(\left(R_{m}\otimes_{k} \cK^\bullet_k, d+\Phi_0(\eps_k)\right),W_{\lc}\right),\left(\left(R_{m}\otimes_{\C} \cK^\bullet_\C, d+\Phi_{2M}(\eps_\C)\right),W_{\lc}, F^{\lc}\right),\wt \alpha\right)
	$$
	is a multiplicative $k$-mixed Hodge complex of sheaves on $X$, where
	\begin{itemize}
		\item The filtrations $W_{\lc}$, $F^{\lc}$ of $R_m\otimes_k \cK_k^\bullet$ and/or $R_m\otimes_\C \cK_j^\bullet$ for all $j=0,\ldots, 2M$ that appear are the tensor filtrations defined as in Definition-Proposition~\ref{defprop:multilinear} from $(R_m,W_{\lc}, F^{\lc})$ and the filtrations in $\cK^\bullet$. 
		\item Everywhere, we write \(d+a\) to denote the sum of the differential $\Id_{R_m}\otimes d$, where $d$ is the differential in $\cK^\bullet$, and left multiplication by \(a\).
		\item $\wt\alpha$ is the filtered  pseudo-isomorphism given by
		\begin{equation}\label{eq:MHCThickening}
			\begin{tikzcd}[column sep = 5em]
				\left((R_{m}\otimes_\C \cK^\bullet_0,d+\Phi_0(\eps_\C)),W_{\lc}\right)
				\arrow[d, "\Id\otimes \alpha_1"]
				&
				\left((R_{m}\otimes_\C \cK^\bullet_2,d+\Phi_2(\eps_\C)),W_{\lc}\right)\arrow[d, "\Id\otimes \alpha_2"]
				\arrow[r,"\Id\otimes \alpha_3"]
				&
				\cdots
				\\
				\left((R_{m}\otimes_\C \cK^\bullet_1,d+\alpha_1\Phi_0(\eps)),W_{\lc}\right)\arrow[r,"e^{\Psi_1(\eps_\C)}","\sim"']
				&
				\left((R_{m}\otimes_\C \cK^\bullet_1,d+\alpha_2\Phi_2(\eps_\C)), ,W_{\lc}\right)
			\end{tikzcd}
		\end{equation}
	\end{itemize}
%
%
\end{defprop}

\begin{proof}
	First note that all of the complexes of sheaves involved in this definition are indeed complexes of sheaves by Lemma~\ref{lem:dglas} and Remark~\ref{rem:dglaTensorM}, and they are bounded  because the complexes in $\cK^\bullet$ are. All the $W_{\lc}$ filtrations that appear in the complexes appearing in $\cK^\bullet(m,V,\mathbf{\Phi},\mathbf{\Psi})$ are increasing, as they are tensor filtrations of increasing filtrations. Similarly, the $F^{\lc}$ filtration of $(R_m\otimes_\C \cK_\C^\bullet, d+\Phi_{2M}(\eps_\C))$ is a decreasing filtration. 
	
	We have to verify the following claims:
\begin{itemize}
	\item \textbf{The vector spaces $\mathbb{H}^*\left(X,(R_m\otimes_k\cK^\bullet_k, d+\Phi_0(\eps_k))\right)$ are finite-dimensional.} If $m=1$, these hypercohomology groups coincide with $\mathbb{H}^*\left(X,\cK^\bullet_k\right)$, which are finite dimensional by the hypothesis that $\cK^\bullet$ is a mixed Hodge complex of sheaves. For $m>1$, we have a short exact sequence
	\begin{equation}\label{eq:sesRm}
		0\to \Sym^{m-1} V^\vee \to R_m\to R_{m-1}\to 0
	\end{equation}
	which induces a short exact sequence of complexes of sheaves
	$$
	0\to (\Sym^{m-1}V^\vee\otimes_k \cK^\bullet_k, d)  \to (R_m\otimes_k\cK^\bullet_k, d+\Phi(\eps_k))\to (R_{m-1}\otimes_k\cK^\bullet_k, d+\Phi(\eps_k))\to 0
	$$
	We have that $\mathbb{H}^*(X,\Sym^{m-1}V^\vee\otimes_k \cK^\bullet_k)\cong \Sym^{m-1}V^\vee\otimes \mathbb{H}^*(X, \cK^\bullet_k)$, so they are finite dimensional vector spaces. The short exact sequence of complexes of sheaves induces a long exact sequence in cohomology groups, so the result for all $m>1$ follows by induction from these long exact sequences.
	
	For $m<0$ the result follows by dualizing \eqref{eq:sesRm} over $k$ and following the same inductive argument.
	
	\item \textbf{The differentials preserve the weight and Hodge filtrations,} that is, the weight and Hodge filtrations are filtrations by subcomplexes of sheaves. We start by showing that $\eps_\C\in (W_0\cap F^0)(V_\C^\vee\otimes_\C V_\C)$. Recall how the filtrations are defined on duals and tensor products in Definition-Proposition~\ref{defprop:multilinear}. Let us assume that a basis $\{s_i^\vee\}$ is chosen in a way compatible with the filtration $W_\lc$ (resp. $F^\lc$), i.e. for every $m$, $W_mV$ (resp. $F^p$) is generated by a subset of this basis. Let $\{s_i\}$ denote the dual basis. If $s_i^\vee\in W_m V$ (resp. $s_i^\vee\in F^pV_\C$), then $s_i$ (seen as a morphism from $V$ to $k$) takes $W_{m-1}V$ (resp. $F^{p+1}V_\C$) to $0$, so $s_i\in W_{-m} V^\vee$ (resp. $s_i\in F^{-p}V_\C^\vee$). In particular, $s_i\otimes s_i^\vee\in W_0 (V^\vee\otimes_k V)$ (resp. $s_i\otimes s_i^\vee\in F^{0}(V_\C^\vee\otimes_\C V_\C)$), so $\eps_\C\in (W_0\cap F^0)(V_\C^\vee\otimes_\C V_\C)$.
	
	Now, recall that $\Phi_{2i}(\eps_\C)\coloneqq(\Id\otimes \Phi_{2i})(\eps_\C)$ for all $i=0,\ldots, M$. By Assumption~\ref{ass:ingredients}, $\Phi_{2i}$ decreases the weight by $1$, so
	\[
	\Phi_{2i}(\eps_\C) \in \Gamma\left(X,W_{-1} \left(V_\C^\vee\otimes_\C \cK^1_{2i} \right)\right).
	\]
	Since $\cK^\bullet$ is a mixed Hodge complex of sheaves, $d$ preserves the weight. Since $\cK^\bullet$ is a multiplicative mixed Hodge complex of sheaves and  the multiplication morphisms in \eqref{eq:multMHS} are MHS morphisms,  multiplication by $\Phi_{2i}(\eps_\C)$ decreases the weight by 1. Hence, applying \(d+\Phi_{2i}(\eps_\C)\) preserves the weight, since \(d\) does. Since $\alpha_{2i}$ also preserves the weight for all $i$, multiplication by $\alpha_{2i-1}\Phi_{2i-2}(\eps_\C)$ and $\alpha_{2i}\Phi_{2i}(\eps_\C)$ decreases the weight by 1, so applying \(d+\alpha_{2i-1}\Phi_{2i-2}(\eps_\C)\) or \(d+\alpha_{2i}\Phi_{2i}(\eps_\C)\) preserves the weight. Similarly, multiplication by $\Phi_{2M}(\eps_\C)$ and the differential \(d+\Phi_{2M}(\eps_\C)\) both preserve the Hodge filtration.
	

	\item \textbf{The associated graded for the weight filtration is a Hodge complex of sheaves.} First, note that, by a similar argument as above, $\Psi_{2i-1}(\eps_\C)$ decreases the weight by $1$, since $\Psi_{2i-1}$ also decreases the weight by $1$. Since $\Phi_{2i}(\eps_\C)$ also decreases the weight by $1$ for all $i=0,\ldots, M$, applying $\Gr^W_{\lc}$ to \eqref{eq:MHCThickening} yields:
	\begin{equation}\label{eq:easyThickening}
		\begin{tikzcd}
			\Gr^W_{\lc}(R_{m}\otimes_\C \cK^\bullet_0,d)
			\arrow[d, "\Gr^W_{\lc}(\Id\otimes \alpha_1)"]
			&
			\Gr^W_{\lc}(R_{m}\otimes_\C \cK^\bullet_2,d)\arrow[d, "\Gr^W_{\lc}(\Id\otimes \alpha_2)"]
			\arrow[rr,"\Gr^W_{\lc}(\Id\otimes \alpha_3)"]
			& &
			\cdots
			\\
			\Gr^W_{\lc}(R_{m}\otimes_\C \cK^\bullet_1,d)\arrow[equals, r]
			&
			\Gr^W_{\lc}(R_{m}\otimes_\C \cK^\bullet_1,d)
		\end{tikzcd}
	\end{equation}
 Hence, $\Gr^W_{j}$ applied to \eqref{eq:MHCThickening} yields (up to some extra identity maps between the sheaf complexes) the same diagram as $\Gr^W_{j}$ applied to $R_m\otimes \cK^\bullet$ (without twisting the differential). The rest follows from Lemma~\ref{lem:MHStensorMHC}
 
 	\item\textbf{The maps $\wt \alpha$ form a filtered pseudo-morphism, which becomes a filtered pseudo-isomorphism after tensoring with $\C$ over $k$:} First, the maps $\Id\otimes \alpha_i$ are clearly morphisms of complexes (they preserve the differential). By Remark~\ref{rem:dgla-cdga}, $e^{\Psi_{2i+1}(\eps)}$ is an isomorphism of complexes. When passing to $\Gr^W_{\lc}$, we obtain \eqref{eq:easyThickening}, which we already showed is a pseudo-isomorphism. Since the filtrations $W_{\lc}$ are biregular, the result now follows by increasing induction and the five lemma. 
 	\end{itemize}
\end{proof}

\subsection{Properties of the thickening}

\begin{proposition}\label{prop:complexMultiplication}
	Suppose we have \((\cK^\bullet, V, \mathbf{\Phi}, \mathbf{\Psi})\) as in Assumption~\ref{ass:ingredients}. Via the embedding \(V^\vee\subset R_{\infty}\), multiplication induces a morphism of mixed Hodge complexes of sheaves for every \(m\in \Z\setminus\{0\}\):
	\[
	V^\vee \otimes \cK^\bullet(m,  V, \mathbf{\Phi}, \mathbf{\Psi})\to \cK^\bullet(m,  V, \mathbf{\Phi}, \mathbf{\Psi}),
	\]
	where 
	\begin{align*}
		&V^\vee \otimes \cK^\bullet(m,  V, \mathbf{\Phi},\mathbf{\Psi})\coloneqq\\
		& \left(\left(\left(V^\vee\otimes_k\otimes R_m\otimes_k \cK_k^\bullet, d+\Phi_0(\eps_k)\right),W_{\lc}\right),\left(\left(V_\C^\vee\otimes_\C\otimes R_m\otimes_\C \cK_\C^\bullet, d+\Phi_{2M}(\eps_\C)\right),W_{\lc}, F^{\lc}\right),\Id_{V^\vee}\otimes\wt\alpha\right),
	\end{align*}
	and the filtrations are tensor filtrations of $V^\vee$ and $\cK^\bullet(m,  V, \mathbf{\Phi},\mathbf{\Psi})$ as in Definition-Proposition~\ref{defprop:multilinear}.
\end{proposition}
\begin{proof}
Note that $V^\vee\otimes\cK^\bullet(m,  V, \mathbf{\Phi},\mathbf{\Psi})$ is a mixed Hodge complex of sheaves by Lemma~\ref{lem:MHStensorMHC}. By \eqref{eq:multMHS}, multiplication induces a mixed Hodge structure morphism \(V^\vee\otimes R_m\to R_m\), so the multiplication morphism $V^\vee \otimes \cK^\bullet(m,  V, \mathbf{\Phi}, \mathbf{\Psi})\to \cK^\bullet(m,  V, \mathbf{\Phi}, \mathbf{\Psi})$ preserves the filtrations. To see that it is a morphism of mixed Hodge complexes of sheaves, we also need to see that it commutes with the pseudo-morphisms of both mixed Hodge complexes of sheaves. That is, it suffices to see that it commutes with \(\Id\otimes \alpha_i\), which is clear since it acts on the first factor, and with \(e^{\Psi_{2i-1}(\eps)}\), which follows from the commutativity of \(R_\infty\).
\end{proof}

\begin{proposition}\label{prop:projectionsMHC}
		Suppose we have \((\cK^\bullet, V, \mathbf{\Phi}, \mathbf{\Psi})\) as in Assumption~\ref{ass:ingredients}. Let $m',m\in \Z$ with $m'\ge m>0$. The projection morphism
		$
		R_{m'}\twoheadrightarrow R_m
		$
		 induces a morphism of mixed Hodge complexes of sheaves:
	\[
\cK^\bullet(m',  V, \mathbf{\Phi}, \mathbf{\Psi})\to \cK^\bullet(m,  V, \mathbf{\Phi}, \mathbf{\Psi}),
	\]
	and the dual $R_{-m}\hookrightarrow R_{-m'}$ of the projection morphism induces a morphism of mixed Hodge complexes of sheaves:
		\[
	\cK^\bullet(-m,  V, \mathbf{\Phi}, \mathbf{\Psi})\to \cK^\bullet(-m',  V, \mathbf{\Phi}, \mathbf{\Psi}).
	\]
\end{proposition}
\begin{proof}
	The proof follows the same steps as the proof of  Proposition~\ref{prop:complexMultiplication}, this time using that the projection $R_{m'}\twoheadrightarrow R_m$ is a MHS morphism, so we omit it.
\end{proof}

\begin{proposition}\label{prop:thickeningFunctorialInV}
	Suppose we have two pieces of data as in Assumption~\ref{ass:ingredients} with a morphism connecting them:
	\[(\cK^\bullet,V,\mathbf{\Phi},\mathbf{\Psi})\longrightarrow(\cK^\bullet,\wt V,\mathbf{\wt \Phi},\mathbf{\wt\Psi}),
	\]
	in other words, there is a MHS morphism $\mu \colon V\to \wt V$, such that the maps $\Phi$'s and $\Psi$'s commute with these. Then, the morphisms between complexes of sheaves induced by $\mu$ and the identity in $\cK^\bullet$
	\[
	\cK^\bullet(m,\wt V,\mathbf{\wt \Phi},\mathbf{\wt\Psi})\to \cK^\bullet(m,V,\mathbf{\Phi},\mathbf{\Psi}),\quad 	\cK^\bullet(-m,V,\mathbf{\Phi},\mathbf{\Psi})\to \cK^\bullet(-m,\wt V,\mathbf{\wt \Phi},\mathbf{\wt\Psi}).
	\]
	are morphisms of mixed Hodge complexes of sheaves for all \(m\in \Z_{\geq 1}\)
\end{proposition}

\begin{proof}
	Let \(R_m\) and \(\wt R_m\) be constructed as in this section for the spaces \(V\) and \(\wt V\), respectively. The morphism \(\mu\colon V\to \wt V\) induces MHS morphisms $\mu^\vee: \wt V^\vee \to  V^\vee$ and  \((\mu^\vee)^{\otimes j}:\Sym^j \wt V^\vee \to \Sym^j V^\vee\). Together they define a ring morphism \(\mu^\vee_\infty:\wt R_\infty \to R_\infty\), \(\wt R_\infty\)-module morphisms \(\mu^\vee_m:\wt R_m\to R_m\) and their duals \(\mu_m:R_{-m}\to \wt R_{-m}\) for all $m\geq 1$. The maps $\mu^\vee_m$ and $\mu_m$ are MHS morphisms for all $m\geq 1$.
	
	Tensoring with the identity morphism of \(\cK^\bullet\), we obtain morphisms between the complexes, which automatically preserve all filtrations since \(\mu^\vee_m\) do as well.

Let us show that these morphisms commute with the differentials (that is, they are morphisms of complexes of sheaves).  We start by showing that \((\mu^\vee_m \otimes_\C \Id_{\cK^\bullet_i}) \circ (d+ \wt\Phi_i(\wt\eps_\C))= (d+\Phi_i(\eps_\C))\circ (\mu^\vee_m \otimes_\C \Id_{\cK^\bullet_i})\), where $\eps_\C$ and $\wt\eps_\C$ are constructed from $V$ and $\wt V$ as in Definition~\ref{def:epsi}, and the subindex $\C$ is changed for $k$ if the degree $i$ is $0$ (in which case $\cK_0^\bullet$ is changed for $\cK_k^\bullet$). Since \(d\) acts as the identity on the factor \(R_m\), it is clear that it commutes with \(\mu^\vee_m\). Therefore, it suffices to show that \(\Phi_i(\eps_\C)\circ \mu^\vee_m  = \mu^\vee_m \circ \wt \Phi_i(\wt\eps_\C)\colon \wt R_m\otimes_\C \cK^\bullet\to  R_m\otimes_\C \cK^\bullet\).
	Suppose we have a simple tensor \(a\otimes c\in \wt R_{m}\otimes_\C \cK^j_i\). Let \(\{\wt s_r^\vee\}\) be a basis of \(\wt V\), and let \(\{\wt s_r\}\) be its dual basis. Then,
	\begin{align*}
		\mu^\vee_m(\wt \Phi_i(\eps_\C)\cdot (a \otimes c))
		&= \mu^\vee_m \left(\sum_r \wt s_r a \otimes \wt \Phi_i(\wt s_r^\vee)\wedge c\right)\\
		&= \sum_r \mu^\vee_m (\wt s_r a) \otimes \wt \Phi_i(\wt s_r^\vee)\wedge c\\
		&= \sum_r \mu^\vee(\wt s_r) \mu^\vee_m(a)\otimes \wt \Phi_i(\wt s_r^\vee)\wedge c\\
		&= \left(\sum_r \mu^\vee(\wt s_r) \otimes \wt \Phi_i(\wt s_r^\vee) \right)\wedge (\mu^\vee_m(a) \otimes c).
	\end{align*}
	Note that if \(\{s_l^\vee\}\) (resp. \(\{\wt s_r^\vee\}\)) is a basis of \(V\) (resp, \(\wt V\)), then
	\[
	\sum_r \mu^\vee(\wt s_r) \otimes \wt \Phi_i(\wt s_r^\vee) =
	\sum_l s_l  \otimes \wt \Phi_i(\mu(s_l^\vee)) 
	\]
	which equals $\sum_l s_l  \otimes \Phi_i(s_l^\vee)$ by hypothesis.
Applying this to the previous string of equalities yields
	\begin{align*}
		\mu^\vee_m(\wt \Phi_i(\eps_\C)\cdot (a \otimes c))
		&= \left(\sum_l s_l  \otimes \Phi_i(s_l^\vee) \right)\wedge (\mu^\vee_m(a) \otimes c)\\
		&= \Phi_i(\eps_\C)\cdot (\mu^\vee_m(a) \otimes c).
	\end{align*}
	This shows that \(\mu^\vee_m\otimes_\C\Id_{\cK^\bullet_i}\) commutes with the differentials of the form $d+\Phi_i(\eps_\C)$ and $d+\wt\Phi_i(\wt\eps_\C)$. The proof follows the same steps for the other differentials appearing in the corresponding thickened mixed Hodge complexes of sheaves, namely those of the form $d+\alpha_j\Phi_i(\eps_\C)$. Hence, \(\mu^\vee_m\otimes_\C\Id_{\cK^\bullet_i}\) commutes with the differentials for all $m\geq 1$.
	
	The proof of the fact that $\wt\Phi_i(\wt\eps_\C)\circ\mu_m=\mu_m\circ \Phi_i(\eps_\C)$ follows the same steps, this time using that for all $a^\vee \in R_{-m}$, and $\wt s\in \wt R_m$, $\wt s\cdot\mu_m(a^\vee)=\mu_m(\mu^\vee_m(\wt s )\cdot a^\vee)$, so we omit it.	
This shows that \(\mu_m\otimes_\C\Id_{\cK^\bullet_i}\) commutes with the differentials of the form $d+\Phi_i(\eps_\C)$ and $d+\wt\Phi_i(\wt\eps_\C)$. As in the previous case for $\mu^\vee_m$, the proof follows the same steps for the other differentials appearing in the corresponding thickened mixed Hodge complexes of sheaves, so  \(\mu_m\otimes_\C\Id_{\cK^\bullet_i}\) commutes with the differentials for all $m\geq 1$.
	
	It remains to show that these morphisms between the complexes of sheaves commute with the pseudo-morphisms at every degree. It is clear that they commute with the morphisms of the form $\Id\otimes\alpha_i$, since these are the identity on the first factor. The commutation with the morphisms of the form $e^{\Psi_i(\eps_\C)}$ and $e^{\wt\Psi_i(\wt\eps_\C)}$ follows similar steps as the ones done for checking that these morphisms induced by $\mu$ between the complexes of sheaves commute with the differentials, so we omit them.
\end{proof}

\begin{proposition}\label{prop:thickenedQuiso}
	Let $m\in \Z\setminus\{0\}$. Let $V$ be a $k$-vector space, where $k=\Q,\R,\C$. Let $M \colon (\cK^\bullet,d)\to(\mathcal G^\bullet,d)$ be a quasi-isomorphism of complexes of sheaves over $k$ on a topological space $X$. Let $\Phi:V\to\Gamma(X,\cK^{1,\cl})$. Then,
	$$
	M_\#\coloneqq\Id_{R_m}\otimes_k M:(R_m\otimes_k \cK^\bullet, d+\Phi(\eps_k))\to (R_m\otimes_k \mathcal G^\bullet, d+(M\circ\Phi)(\eps_k))
	$$
	is a quasi-isomorphism.
\end{proposition}
\begin{proof}
	Note that $M_\#$ is a morphism of complexes (it commutes with the differentials)
	
	Suppose that $m>0$.
	Consider the decreasing filtration $G^{\lc}$ by subcomplexes of $(R_m\otimes_k \cK^\bullet, d+\Phi(\eps_k))$ given by
	$$
	G^p R_m\otimes_k \cK^\bullet=\left(\oplus_{p\leq j\leq m-1} \Sym^j V^\vee\right)\otimes_k \cK^\bullet.
	$$
	Note that $G^0 R_m\otimes_k \cK^\bullet=R_m\otimes_k \cK^\bullet$ and $G^{m} R_m\otimes_k \cK^\bullet=0$. Similarly, we define the decreasing filtration $G^{\lc}$ of $R_m\otimes_k \mathcal G^\bullet, d+(M\circ\Phi(\eps_k))$. Note that $M_\#$ preserves the filtration $G^{\lc}$, and multiplication by $\Phi(\eps_k)$ or $(M\circ\Phi)(\eps_k)$ increases the filtration by $1$. In particular, for $p\geq 0$ we have the following commutative diagrams of short exact sequences:
	$$
	\begin{tikzcd}
		G^{p+1} (R_m\otimes_k \cK^\bullet, d+\Phi(\eps_k))\arrow[r, hook]\arrow[d, "M_\#"]&G^{p} (R_m\otimes_k \cK^\bullet, d+\Phi(\eps_k))\arrow[r,two heads]\arrow[d, "M_\#"]&\Gr_G^{p} (R_m\otimes_k \cK^\bullet, d)\arrow[d, "M_\#"]\\
		G^{p+1} (R_m\otimes_k \mathcal G^\bullet, d+(M\circ\Phi)(\eps_k))\arrow[r,hook]&G^{p} (R_m\otimes_k \mathcal G^\bullet, d+(M\circ\Phi)(\eps_k))\arrow[r,two heads]&\Gr_G^{p} (R_m\otimes_k \mathcal G^\bullet, d)\\
	\end{tikzcd}
	$$
	If $p=m-1$, the vertical arrow on the left is a morphism between two $0$ complexes, and the vertical arrow on the left is a quasi-isomorphism, so the central vertical arrow must also be a quasi-isomorphism. The rest of the proof now follows from decreasing induction and the five lemma, ending at $p=0$, where one can show that the central vertical arrow is a quasi-isomorphism.
	
	The result for $m<0$ follows similarly by defining a decreasing filtration on $R_m\otimes_k \cK^\bullet$ from the dual decreasing filtration on $R_m$ defined as in  Definition-Proposition~\ref{defprop:multilinear} from the one in $R_{-m}$, namely $G^p R_{-m}\coloneqq\{h:R_m\to \C\mid \oplus_{-p+1\leq j\leq m-1}\Sym^j V^\vee\subset \ker h\}$.
\end{proof}

\begin{proposition}\label{prop:qisoThickenings}
	Suppose that we have two pieces of data as in Assumption~\ref{ass:ingredients}:
	\[
	(\cK^\bullet, V, \mathbf{\Phi},\mathbf{\Psi}),
	(\wt\cK^\bullet, V, \mathbf{\wt\Phi},\mathbf{\wt\Psi})
	\]
	Furthermore, suppose that they are connected by a morphism of multiplicative mixed Hodge complexes of sheaves \(M \colon \cK^\bullet\to \wt \cK^\bullet\), that is compatible with the remaining data, in the sense that for every \(i\),
	\begin{align*}
		\wt \Phi_{2i} &= M_{2i} \circ \Phi_{2i};\\
		\wt \Psi_{2i-1} &= M_{2i-1} \circ \Psi_{2i-1}.
	\end{align*}
	Then, \(\Id_{R_m} \otimes M\) is a morphism of mixed Hodge complexes of sheaves between the two thickenings $\cK^\bullet(m,V,\mathbf{\Phi},\mathbf{\Psi})$ and $\wt\cK^\bullet (m,V,\mathbf{\wt\Phi},\mathbf{\wt\Psi})$. Moreover,
	\begin{itemize}
		\item if $M$ is a weak equivalence in the sense of \cite[Lemma-Definition 3.19]{peters2008mixed} (that is, a collection of quasi-isomorphisms), so is \(\Id_{R_m} \otimes M\), and
		\item if $M$ is a filtered quasi-isomorphism between the respective components of the mixed Hodge complexes of sheaves $\cK^\bullet$ and $\wt\cK^\bullet$ (and bi-filtered in the last), so is \(\Id_{R_m} \otimes M\).
	\end{itemize}
\end{proposition}
\begin{proof}
	Let $\alpha$ and $\beta$ denote the pseudo-morphisms in $\cK^\bullet$ and $\wt\cK^\bullet$ respectively.
	
	First of all, \(\Id_{R_m}\otimes_\C M_{2i}\) commutes with the differentials \(d+\Phi_{2i}(\eps_\C)\) and \(d+\wt\Phi_{2i}(\eps_\C)\) because  \(M\) commutes with all \(\Phi\)'s, and it commutes with differentials of the form \(d+\alpha_j\circ \Phi_{2i}\) and \(d+\beta_j\wt\Phi_{2i}(\eps_\C)\) because \(M\) commutes with the \(\alpha\)'s as well. Next, \(\Id_{R_m}\otimes M\) commutes with all the maps in the pseudo-morphisms \eqref{eq:MHCThickening} (the $\alpha$'s and the $\beta$'s): It commutes with maps of the form \(\Id_{R_m}\otimes \alpha_i\) and \(\Id_{R_m}\otimes \beta_i\) because \(M\), being a morphism of mixed Hodge complexes of sheaves, must commute with the \(\alpha_i\)'s and \(\beta_i\)'s, and it commutes with maps of the form \(e^{\Psi_i(\eps)}\) because \(M\) is required to preserve the multiplicative structure and commute with the \(\Psi\)'s. Lastly, \(\Id_{R_m}\otimes M\) preserves all the filtrations because both \(\Id_{R_m}\) and \(M\) do. This concludes the proof of the fact that $\Id_{R_m}\otimes M$ is a morphism of mixed Hodge complexes of sheaves.
	
	The first point in the ``moreover'' part of the statement follows from Proposition~\ref{prop:thickenedQuiso}.
	
	Lastly, the proof of the second point for the weight filtration in the ``moreover'' part of the statement follows from the fact that the differentials become untwisted after passing to the graded pieces, and from the direct sum decomposition of the graded pieces in terms of graded pieces of the tensor appearing in the proof of Lemma~\ref{lem:MHStensorMHC}. For the Hodge filtration in the last component of both mixed Hodge complexes of sheaves, one can again use  the direct sum decomposition of the graded pieces in terms of graded pieces of the tensor, and use an inductive argument similar to the one in the proof of Proposition~\ref{prop:thickenedQuiso}, defining the filtration that $G^\lc$ induces on these graded pieces.
%
\end{proof}

\section{The thickening of the logarithmic Dolbeault complex}\label{sec:thickDolbeault}
In the previous section, we showed how to construct a thickened mixed Hodge complex of sheaves from a multiplicative mixed Hodge complex of sheaves together with the extra data $(V,\mathbf{\Phi},\mathbf{\Psi})$ of Assumption~\ref{ass:ingredients}. In this section, we apply this construction in the case where the multiplicative mixed Hodge complex of sheaves is the modified logarithmic Dolbeault mixed Hodge complex of sheaves from Definition-Proposition~\ref{defprop:modifiedNA} (based on Navarro Aznar's mixed Hodge complex of sheaves from Definition~\ref{def:NavarroAznar}) and $V$ is the first cohomology of a semiabelian variety $G$. The following result clarifies which ingredients we will use in order to construct a thickening of the logarithmic Dolbeault complex. In it, note that the mixed Hodge complex of Definition-Proposition~\ref{defprop:modifiedNA} is extended by an extra term (using the identity morphism) to fit Assumption~\ref{ass:ingredients}.

\begin{lemma}\label{lem:df/f}
	Let $U\xrightarrow{f} G$ be an algebraic morphism from a smooth variety to a semiabelian variety. Let
	$$
	0	\to G_T \xrightarrow{t}  G \xrightarrow{p_A} G_A\to 0
	$$ be the Chevalley decomposition of $G$. 
	
	Suppose that we have the ingredients $(X,Y,\Phi_{\R}^Y,\Phi_\C^Y,\Psi^Y)$, satisfying the following properties:	
	\begin{enumerate}
		\item $X$ is a good compactification of $U$ and $Y$ is an allowed compactification of $G$  (Definition~\ref{def:allowedComp}) such that $X$ and $Y$ are compatible with respect to $f$ (Definition~\ref{def:compatibleCompf}), i.e. $f$ extends to $\ov f:X\to Y$. 
		Let $E\coloneqq Y\setminus G$ and let $D\coloneqq X\setminus U$.
		\item Let $n\geq \max\{2,\dim_\R U\}$, and let $\cN^{\bullet}_{X,D,n}$ be the (multiplicative) mixed Hodge complex from \cite{navarroAznar} (see Definition-Proposition~\ref{defprop:modifiedNA}):
		\[
		\cN^{\bullet}_{X,D,n}\coloneqq \left((\cA^\bullet_{X,\R}(\log D),W_{\lc}^n) ,
		(\cA^\bullet_{X,\C}(\log D),W_{\lc}^n,F^{\lc}),\alpha\right),
		\]
		where $\alpha$ is the filtered pseudo-morphism such that $\alpha\otimes 1$ is the filtered pseudo-isomorphism
		\[
		(\cA^\bullet_{X,\C}(\log D),W_{\lc}^n)=(\cA^\bullet_{X,\R}(\log D)\otimes_{\R}\C,W_{\lc}^n)\xrightarrow{\Id}(\cA^\bullet_{X,\C}(\log D),W_{\lc}^n)\xleftarrow{\Id}(\cA^\bullet_{X,\C}(\log D),W_{\lc}^n).
		\]
		\item Let $H \coloneqq H^1(G;\R)$, together with its mixed Hodge structure.
		\item\label{itemPhi} $\Phi_{\R}^Y,\Phi_\C^Y$ are linear maps which are a section of the cohomology map, with the following domain and target:
		\begin{align*}
			\Phi_\R^Y\colon H &\to\Gamma(Y,\cA_{Y,\R}^{1,\cl}(\log E)),\\
			\Phi_\C^Y \colon H_\C\coloneqq H\otimes_{\R}\C &\to\Gamma(Y,\cA_{Y,\C}^{1,\cl}(\log E)).
		\end{align*}
		Here, $\cA_{Y,k}^{1,\cl}(\log E)$ denotes the closed $k$-valued forms in $\cA_{Y,k}^{1}(\log E)$. These maps satisfy the following three conditions:
		\begin{itemize}
			\item For \(k=\R,\C\), the image of \(\Phi_k^Y\) is contained in \(\Gamma(Y, \wt W_1\cA^{1,\cl}_{Y,k}(\log E))\), where $\wt W_{\lc}$ is as in Definition~\ref{def:NavarroAznar}.
			\item Both $\Phi_\R^Y$ and $\Phi_\C^Y$ send classes that are pulled back from $G_A$ to forms whose restriction to $G$ is in the image of $p_A^*:\Gamma(G_A,\cA_{G_A,k}^1)\to \Gamma(G,\cA_{G,k}^1)$ for $k=\R,\C$ respectively.
			\item $\Phi_\C^Y$ sends classes that are represented by holomorphic forms on $G$ to $(1,0)$-forms. 
		\end{itemize}
		\item\label{itemPsi} $\Psi^Y$ is a linear map (a homotopy)
		\[
		\Psi^Y \colon H_{\C}\to \Gamma(Y,\cA_{Y,\C}^{0}(\log E))
		\]
		such that
		\[
		d\circ \Psi^Y = \C\otimes \Phi_{\R}^Y-\Phi_{\C}^Y.
		\]
	\end{enumerate}
	Then, $(\cK^\bullet,V,\Phi_0,\Phi_2,\Psi_1)\coloneqq(\cN^{\bullet}_{X,D,n},H,\ov f^*\circ\Phi_{\R}^Y,\ov f^*\circ\Phi_{\C}^Y,\ov f^*\circ\Psi^Y)$ satisfy Assumption~\ref{ass:ingredients}.
\end{lemma}
\begin{proof}
	We need to show that $\Phi_0,\Phi_2,\Psi_1$ satisfy the conditions of Assumption~\ref{ass:ingredients}. We do this in several steps. Note that, since $n\geq 2$, the filtrations of the logarithmic Dolbeault complex $W_j^n$ and $\wt{W}_j$ coincide for $j=0,1,2$, so we can use $\wt{W}_{\lc}$ in our arguments.

		\textbf{$\Phi_R^Y$ preserves the weight filtration,} as follows: $$\Phi_\R^Y:(H,W_{\lc}[1])\to \Gamma\left(Y,\left(\cA_{Y,\R}^{1,\cl}(\log E),W_{\lc}^n\right)\right).$$
		Since $G$ is smooth, $\Gr_i^W H=0$ when $i$ is not contained in $\{1,2\}$. Therefore, after the shift, the non trivial graded pieces correspond to indices contained in $\{0,1\}$. The weight $1$ is preserved by hypothesis.
		
		For weight $0$, note that
    \begin{align*}
      W_0(H^1(G,\R)[1])&=W_1H^1(G,\R) = H^1(G_A,\R),  &&\text{ and }
                       &
      \wt W_0\cA^1_{Y,\R}(\log E)&=\cA^1_{Y,\R}.
    \end{align*}
Since $Y$ is an allowed compactification of $G$, there exists a compactification $\ov G$ as in Corollary~\ref{cor:compactification} and an algebraic map $p:Y\to \ov G$ such that $p\circ j_Y=j_{\ov G}$, where $j_Z:G\hookrightarrow Z$ is the inclusion for $Z=Y,\ov G$. In particular, $p_A$ extends to a fibration $\ov{p_A}:\ov G\to G_A$. Let $a\in H^1(G_A,\R)$. By hypothesis, $\Phi^Y_\R(a)|_{G}= p_A^*\omega$ for some $\omega\in\Gamma(G_A,\cA^1_{G_A,\R})$. Note that $\Phi^Y_\R(a)$ and $p^*( \ov{p_A}^*\omega)$ have the same restriction to $G$, and since $G$ is dense in $Y$, they must coincide. In particular,  $\Phi^Y_\R(a)\in \Gamma(Y,\cA^1_{Y,\R})=\Gamma(Y,\wt W_0\cA^1_{Y,\R}(\log E))$.
%
%

		\textbf{$\Phi_\C^Y$ preserves the weight filtration:} $\Phi_\C^Y:(H_\C,W_{\lc}[1])\to \Gamma\left(Y,\left(\cA_{Y,\C}^{1,\cl}(\log E),\wt W_{\lc}\right)\right)$ also respects the weight filtration by the analogous argument over \(\C\).

    \textbf{$\Psi^Y$ preserves the weight filtration:} it maps $(H_\C,W_{\lc}[1])$ to $\Gamma\left(Y,\left(\cA_{Y,\C}^{0}(\log E),\wt W_{\lc}\right)\right)$. Since $\Phi^Y_\R$ and $\Phi^Y_\C$ respect the filtrations (up to a shift), the relationship between $\Psi^Y, \Phi^Y_\R$ and $\Phi^Y_\C$ implies that it suffices to show that, for $d:\cA_{Y,\C}^0(\log E)\to\cA_{Y,\C}^1(\log E)$, $d^{-1}\left(\wt W_j\cA_{Y,\C}^1(\log E)\right)=\wt W_j\cA_{Y,\C}^0(\log E)$ for all $j\ge 0$ (we only need to apply this fact for \(j \in \{0,1\}\)).
		
   To do this, we will show that for all \(j\ge 1\), \(d^{-1}\left(\wt W_{j-1}\cA^1_{Y,\C}(\log E)\right) \cap \wt W_{j}\cA^0_{Y,\C}(\log E) = \wt W_{j-1}\cA^0_{Y,\C}(\log E)\). We apply Proposition~\ref{prop:adjunctionNA}, which ensures that
    \[
    \cH^0\left(
      \gr^{\wt W}_j \cA^\bullet_{Y,\C}(\log E) 
    \right)
\cong 
\cH^0\left(
  \gr^\tau_j\cA^\bullet_{Y,\C}(\log E) 
\right) \overset{j > 0}= 0.
    \]
    Spelling out the definition of \(\cH^0\), this means that if \(\alpha\in \wt W_j\cA^0_{Y, \C}(\log E)\) is such that \(d\alpha\in \wt W_{j-1} \cA^1_{Y,\C}(\log E)\), then \(\alpha\in \wt W_{j-1}\cA^0_{Y,\C}(\log E)\), as desired. By induction on \(j\), we have that for any \(j>j'\ge 0\), if \(\alpha\in \wt W_j\cA^0_{Y,\C}(\log E)\) is such that \(d\alpha\in \wt W_{j'} \cA^1_{Y,\C}(\log E)\), then \(\alpha \in \wt W_{j'}\cA^0_{Y,\C}(\log E)\).

		\textbf{$\Phi_{\C}^Y$ respects the Hodge filtration} (without any shifts): The relevant pieces are $F^0=H^1(G,\C)$ and $F^1$. For $F^0$, we have that $H^1(G,\C)=F^0H^1(G,\C)$. Automatically, its image lands in $F^0\cA^1_{Y,\R}(\log E) = \cA^1_{Y,\R}(\log E)$. Next, by Deligne's theory of MHS, $F^1H^1(G,\C)$ is composed of the classes of holomorphic forms, and  $\Phi_{\C}$ maps these to $(1,0)$ forms by hypothesis, that is, to $F^1\cA^1_{Y,\R}(\log E)$.
		
		\textbf{$\ov f^*$ takes logarithmic forms to logarithmic forms and respects the filtrations} $\wt W_{\lc}$ and $F^{\lc}$, so in particular it also respects $W_0^n$ and $W_1^n$. Also, $\ov f^*$ commutes with the differential $d$. Hence, $\Phi_0,\Phi_2,\Psi_1$ satisfy the conditions of Assumption~\ref{ass:ingredients}.
\end{proof}

In order to apply Lemma~\ref{lem:df/f}, we need to make sure that such $\Phi_{\R}^Y,\Phi_{\C}^Y$ and $\Psi^Y$ satisfying the assumptions therein exist, which we achieve in Definition-Proposition~\ref{defPhiPsiY} and in Corollary~\ref{cor:Phidf/f}. Before that, we start by recalling some general facts about abelian Lie groups in order to fix notation (Lemma~\ref{lem:generatecoh}). Then, we will state the definitions of the maps $\Phi_{\R}^G,\Phi_{\C}^G$ and $\Psi^G$ in Definition-Proposition~\ref{defprop:PhiPsiG}, which are a first approximation to the definitions of $\Phi_{\R}^Y,\Phi_{\C}^Y$ and $\Psi^Y$. The images of the maps $\Phi_{\R}^G,\Phi_{\C}^G$ and $\Psi^G$ consist of analytic forms on $G$. We later extend these to $\Phi_{\R}^Y,\Phi_{\C}^Y$ and $\Psi^Y$ in Definition-Proposition~\ref{defPhiPsiY}.

\begin{lemma}\label{lem:generatecoh}
	Let $G$ be a complex semiabelian variety. Let \(\Lambda\) be the kernel of the exponential map \(TG\to G\). Let the Chevalley decomposition of $G$ be given by
	\begin{equation}\label{eq:chevalley}
		0	\to G_T \xrightarrow{t}  G \xrightarrow{p_A} G_A\to 0.
	\end{equation}
	Let $\Omega_X^\bullet$ denote the holomorphic de Rham complex of sheaves on $X$ for every smooth complex algebraic variety $X$, and if $X$ is a complex Lie group, let $\Omega_X^{1,\inv}$ (resp. $\cA_{X,k}^{1,\inv}$ for $k=\R,\C$) denote the sheaf of holomorphic (resp. analytic) invariant $1$-forms on $X$. Then,
	\begin{enumerate}
		\item\label{itemLambda} For \(k=\R,\C\), there are natural isomorphisms \(\Gamma(G,\cA_{G,k}^{1,\inv})\cong \Hom_\R(TG,k)\) and \(H^1(G,k)\cong \Hom_\Z(\Lambda, k)\). The map that sends a form to its cohomology class corresponds to the restriction to \(\Lambda\).
		\item\label{itemHolomorphic} There is a natural isomorphism \(\Gamma(G,\Omega^{1,\inv}_G)\cong \Hom_\C(TG,\C)\).
		\item\label{item0} The restriction $\Gamma(G,\Omega^{1, \inv}_{G})\to\Gamma(G_T,\Omega^{1, \inv}_{G_T})$ is surjective.
		\item\label{item1} The projection of invariant forms onto their cohomology classes $\Gamma(G,\cA^{1, \inv}_{G,\R})\to H^1(G,\R)$ is a surjection, and the same holds for $G_A$ and $G_T$. The statement is also true for $\C$-coefficients. Furthermore, in the case of $G_A$ this projection is an isomorphism (both with $\R$ and $\C$ coefficients).
		\item\label{item-1} $\Gamma(G,\Omega^{1,\inv}_{G})$ can be seen as a subspace of $H^1(G,\C)$ through the projection of forms onto their cohomology classes, which is an injective map. The same holds for $G_A$ and $G_T$. Furthermore, in the case of $G_T$ this injection is an isomorphism.
		\item\label{item2} $H^1({G}_A,\C)$ can be seen as a subspace of $H^1(G,\C)$ via $(p_A)^*:H^1({G}_A,\C)\to H^1(G,\C)$, which is injective.
		\item\label{itemInvariantHolo}  The cohomology class of every closed holomorphic 1-form is represented by an invariant holomorphic form.
		\item\label{item3} $\Gamma(G,\Omega^{1, \inv}_{G})$ and  $H^1({G}_A,\C)$  generate $H^1(G,\C)$ as a complex vector space.
	\end{enumerate}
\end{lemma}
\begin{proof}
	Note that all invariant forms appearing in the statement of this lemma are closed, since they pull back to constant forms on the corresponding universal cover (a complex vector space), and the differential commutes with the pullback. Hence, invariant forms do represent cohomology classes, and the statements in parts~\eqref{itemLambda}, \eqref{item1}, \eqref{item-1}  and \eqref{item3} make sense. 
	
	Since \(G_A\) and \(G_T\) are semiabelian varieties, every statement that is proved for \(G\) applies to them as well.
	
	\begin{enumerate}
		\item The isomorphism \(\Gamma(G,\cA_{G,k}^{1,\inv})\cong \Hom_\R(TG,k)\) comes from pulling back an invariant form through the exponential map $TG\to G$, which yields a constant form. Constant forms on a vector space are identified with its dual. For the second isomorphism, note that \(TG\) is the universal cover of \(G\), and therefore \(\Lambda\) is canonically \(\pi_1(G)\) and also \(H_1(G,\Z)\). Furthermore, since \(TG\) is a vector space, the pairing between a constant form  seen as an element of \(\Hom_\R(TG,k)\) and \(x\in TG\) is the same as the integral of the form on a path from \(0\) to \(x\). If \(x\in \Lambda\), this path is the pullback of a loop in \(\pi_1(G)\), and the statement follows from de Rham's theorem.
		\item The isomorphism $\Gamma(G,\Omega^{1, \inv}_{G})\cong \Hom_{\C}(T G,\C)$ is analogous to the real analytic setting.
		\item The morphism $p_A$ in \eqref{eq:chevalley} is a fibration with fiber $G_T$, so it induces a short exact sequence between tangent spaces at the identity, and hence the restriction $\Hom_{\C}(T G,\C)\to \Hom_{\C}(T G_T,\C)$ is a surjection. Now, use part \eqref{itemHolomorphic}.
		\item Since \(G = TG/\Lambda\), \(\Lambda\) is discrete, and in particular any $\Z$-basis is \(\R\)-linearly independent. The statement for \(G\) follows from part \eqref{itemLambda}. In the case of $G_A$, the fact that the projection is an isomorphism now follows from the fact that both spaces have the same real dimension, namely $2\dim_{\C} G_A$.
		\item From above, we have a natural isomorphism $\Gamma(G,\Omega^{1, \inv}_{G})\cong \Hom_{\C}(T G,\C)$, and $H^1(G,\C)$ is naturally identified with $\Hom_{\Z}(\Lambda,\C)$. Since $\Lambda$ generates $TG$ as a $\C$-vector space, this restriction is injective. In the case of $G_T$, $\Lambda$ is a $\C$-basis of $TG_T$, and the injection is an isomorphism.

		\item 
		Consider the long exact sequence of the fibration $p_A$ on homotopy groups. Since the universal covers of all the spaces involved are contractible and their fundamental groups are all abelian, we get a short exact sequence in homology
		$$
		0\to H_1(G_T,\Z)\xrightarrow{t_*}H_1(G,\Z)\xrightarrow{(p_A)_*}H_1(G_A,\Z)\to 0,
		$$
		and in particular, a short exact sequence in cohomology
		\begin{equation}\label{eq:chevalleyCohomology}
			0\to H_1(A,\C)\xrightarrow{(p_A)^*}H^1(G,\C)\xrightarrow{t^*} H^1(G_T,\C)\to 0.
		\end{equation}
	
		\item Let us consider first the cases where \(G\) is an algebraic torus and an abelian variety. If \(G\) is a torus, then every cohomology class is represented by an invariant holomorphic form, by part \eqref{item-1}. If \(G\) is an abelian variety, the Hodge decomposition tells us that the space of classes of holomorphic forms has complex dimension equal to \(\dim G\), so it suffices to compare dimensions.
		
		For a general \(G\), consider a closed form \(\alpha\in \Gamma(G,\Omega^1_G)\). By the torus case, its restriction to \(G_T\) is represented by an invariant holomorphic form \(\wt\alpha_T\) on \(G_T\) which, by \eqref{item0}, is the restriction of some \(\alpha_T\in \Gamma(G,\Omega_G^{1,\inv})\). Then, \(\alpha - \alpha_T\) is a holomorphic form that vanishes on \(G_T\). By the short exact sequence \eqref{eq:chevalleyCohomology}, its cohomology class comes from \(G_A\), and by the abelian variety case, it is represented by an invariant holomorphic form \(\alpha_A\). Then, in cohomology, \(\alpha = \alpha_T + (p_A)^*\alpha_A\), which is the class of an invariant holomorphic form, as desired.
		\item Using the short exact sequence \eqref{eq:chevalleyCohomology}, $H^1(G_A,\C)$ together with the image of any section of $t^*$ generate $H^1(G,\C)$. Combining parts~\eqref{item-1} and \eqref{item0}, such a section can be constructed from a section of \(\Gamma(G,\Omega^{1,\inv}_{G})\to \Gamma(G_T,\Omega^{1,\inv}_{G_T})\). 
	\end{enumerate}

	
\end{proof}

\begin{remark}\label{rem:Lattices}
	Consider the Chevalley decomposition \eqref{eq:chevalley} and \(\Lambda\) as in Lemma~\ref{lem:generatecoh}. Since \(G_T\) is an algebraic torus, \(\Lambda\cap TG_T\cong \pi_1(G_T)\) is freely generated by a \(\C\)-basis of \(TG_T\), and since \(G_A\) is an abelian variety, the image of \(\Lambda\) in \(TG_A\) is a full rank lattice in \(G_A\).
	
	We can choose a way of extending \(\Lambda\) to a full rank lattice in \(G\). Let \(\Lambda' \coloneqq i\cdot (\Lambda \cap TG_T)\). Then, \(\Lambda'\oplus (\Lambda \cap TG_T)\) is a full rank lattice in \(TG_T\), and \(\Lambda\oplus \Lambda'\) is a full rank lattice in \(TG\).
\end{remark}

\begin{defprop}[Definition of $\Phi_{\R}^G$, $\Phi_{\C}^G$ and $\Psi^G$]\label{defprop:PhiPsiG}
	Let $Y$ be an allowed compactification of a complex semiabelian variety $G$, let $j_Y:G\to Y$ be the inclusion and let $E\coloneqq Y\setminus G$. Let
	$$
	0	\to G_T \xrightarrow{t}  G \xrightarrow{p_A} G_A\to 0
	$$ be the Chevalley decomposition of $G$. 
	
	\begin{itemize}
		\item We define $\Phi_{\C}^G$ as the unique $\C$-linear map whose restrictions to $H^1(G_A,\C)$ and the cohomology classes of $\Gamma(G,\Omega^{1, \inv}_{G})$ are as follows:
		\begin{enumerate}
			\item  $$
			{(\Phi^G_{\C})|}_{H^1({G}_A,\C)}:H^1({G}_A,\C)\to \Gamma(G,\cA^{1, \inv}_{G,\C})
			$$
			is given by the composition of the isomorphism found in Lemma~\ref{lem:generatecoh}\eqref{item1} $H^1({G}_A;\C)\cong\Gamma({G}_A,\cA^{1,\inv}_{G_A,\C})$  and the pullback by $p_A$.
			\item $${(\Phi^G_{\C})|}_{\Gamma(G,\Omega^{1,\inv}_G)}:\Gamma(G,\Omega^{1,\inv}_G)\to \Gamma(G,\cA^{1, \inv}_{G,\C})$$ is the map given by the inclusion of sheaves.
		\end{enumerate}
		
		\item We define $\Phi_{\R}^G$ as the composition
		$$
		H^1(G,\R)\hookrightarrow H^1(G,\C)\xrightarrow{\Phi_{\C}^G}\Gamma(G,\cA^{1}_{G,\C})\xrightarrow{\Re} \Gamma(G,\cA^{1}_{G,\R}).
		$$
		where $\Re$ is the real part.
		\item We define $\Psi^G:H^1(G,\C)\to\Gamma(G,\cA^0_{G,\C})$ as the unique linear map satisfying that $d\Psi^G=\C\otimes\Phi_\R^G-\Phi_\C^G$ whose image lies in $\Hom_{\R\text{-Lie groups}}(G,\C)$.
	\end{itemize}
\end{defprop}
\begin{proof}
	To see that $\Phi^G_{\C}$ is well defined, we need to see that ${(\Phi^G_{\C})}|_{\Gamma(G,\Omega^{1,\inv}_G)}$ and ${(\Phi^G_{\C})}|_{H^1({G}_A,\C)}$ agree on $\Gamma(G,\Omega^{1,\inv}_G)\cap H^1({G}_A,\C)$. We are going to use the notation for $\Lambda$ and $\Lambda'$ from Remark~\ref{rem:Lattices}. We will give a global definition of $\Phi_{\C}^G$ and we will check that it agrees with the definition that we gave on each of the subspaces. The uniqueness follows from Lemma~\ref{lem:generatecoh}, part~\eqref{item3}.
	
	Consider the natural isomorphism
	$$
	\Hom_\Z(\Lambda,\C) \oplus \Hom_\Z(\Lambda',\C)\cong \Hom_{\R}(TG,\C)\cong \Gamma(G,\cA^{1,\inv}_{G,\C}),
	$$
	and let $(\alpha,\alpha')\in \Hom_\Z(\Lambda,\C) \oplus \Hom_\Z(\Lambda',\C)$. Seeing this inside of $\Gamma(G,\cA^{1,\inv}_{G,\C})$, we have that $t^*(\alpha,\alpha')$ is the restriction to $\Hom_\Z(\Lambda\cap T{G_T},\C) \oplus \Hom_\Z(\Lambda',\C)\cong \Gamma(G_T,\cA^{1,\inv}_{G_T,\C})$. The elements of $\Hom_\Z(\Lambda\cap T{G_T},\C) \oplus \Hom_\Z(\Lambda',\C)$ which correspond to elements of  $\Gamma(G_T,\Omega^{1\inv}_{G_T})\cong \Hom_\Z(\Lambda\cap T{G_T},\C)\cong \Hom_\C(TG_T,\C)$ are the ones satisfying that $\alpha' = -i\circ \alpha|_{\Lambda \cap TG_T}\circ i$. Consider the following chain of isomorphisms:
	\begin{equation}\label{eq:defPhi}
		\begin{split}
      H^1(G,\C)
      &\cong \Hom_{\Z}(\Lambda,\C)\cong \left\{
			\substack{	\displaystyle	(\alpha,\alpha') \in  \Hom_\Z(\Lambda,\C) \oplus \Hom_\Z(\Lambda',\C)\cong \Hom_{\R}(TG,\C)\cong \Gamma(G,\cA^{1,\inv}_{G,\C})
				\\ \displaystyle	\text{where }\alpha' = -i\circ \alpha|_{\Lambda \cap TG_T}\circ i}
			\right\}\\[2em]
			&\cong \{
			(\alpha,\alpha') \in \Gamma(G,\cA^{1,\inv}_{G,\C})\mid
			t^*(\alpha,\alpha') \in \Gamma(G_T,\Omega^{1,\inv}_{G_T})
			\}\subset \Gamma(G,\cA^{1,\inv}_{G,\C}).
		\end{split}
	\end{equation}
	We claim that the composition above coincides with the definition that we have given of $\Phi_{\C}^G$: We start by showing that both definitions agree on $H^1(G_A,\C)$. Let $\beta\in H^1(G_A,\C)\cong\Gamma(G_A,\cA^{1,\inv}_{G_A,\C})\cong \Hom_\R(TG_A,\C)\cong \Hom_\Z(\Lambda/(\Lambda\cap TG_T),\C)$, where the first of these isomorphisms is the one in part~\eqref{item1} of Lemma~\ref{lem:generatecoh}. Let $q$ be the quotient $q:TG\to TG_A$. By our definition, ${(\Phi^G_{\C})|}_{H^1({G}_A,\C)}(\beta)=\beta\circ q\in \Hom_{\R}(TG,\C)\cong \Gamma(G,\cA^{1,\inv}_{G,\C})$ for all $\beta\in \Hom_\R(TG_A,\C)$. Note that $\beta \circ q\circ \iota_{\Lambda}$ is just $\beta\in H^1(G_A,\C)$ seen inside of $H^1(G,\C)\cong \Hom_\Z(\Lambda,\C)$, where $\iota_{\Lambda}:\Lambda \to TG$ is the inclusion. The chain of isomorphisms in \eqref{eq:defPhi} sends $\beta\circ q\circ \iota_{\Lambda}$ to $(\beta\circ q\circ \iota_{\Lambda}, 0)$, which corresponds to $\beta\circ q$ under the isomorphism $\Hom_\R(TG,\C)\cong \Hom_\Z(\Lambda,\C) \oplus \Hom_\Z(\Lambda',\C)$.
	
	Let us now see that both definitions agree on $\Gamma(G,\Omega^{1,\inv}_G)\cong \Hom_\C(TG,\C)$. Let $\beta\in \Hom_\C(TG,\C)$.  By our definition, ${(\Phi^G_{\C})|}_{\Gamma(G,\Omega^{1,\inv}_G)}(\beta)$ equals $\beta$ itself, but seen inside of $\Hom_{\R}(TG,\C)$. In Lemma~\ref{lem:generatecoh}~\eqref{item1}, we see $\beta$ in $H^1(G,\C)\cong \Hom_\Z(\Lambda,\C)$ as $\beta\circ\iota_{\Lambda}$. The chain of isomorphisms \eqref{eq:defPhi} sends $\beta\circ\iota_{\Lambda}$ to $(\beta\circ\iota_{\Lambda},\beta\circ\iota_{\Lambda'})$, where $\iota_{\Lambda'}:\Lambda'\hookrightarrow TG$ is the inclusion. This equals $\beta$ itself. Hence, we have seen that $\Phi_\C^Y$ is well defined.
	
	Let us now construct $\Psi^G$. Suppose $\alpha\in H^1(G,\C)\cong \Hom_\Z(\Lambda, \C)$. We will see $\alpha$ as an element of $\Hom_\Z(\Lambda, \C)$. Then, by our construction, $(\C\otimes \Phi_\R) (\alpha)$ vanishes on $\Lambda'$, while $\Phi_\C(\alpha)|_{\Lambda'}=-i\circ \alpha |_{\Lambda \cap TG_T} \circ i$. They both agree on $\Lambda$, so their difference is the element $\beta\in \Hom_\Z(\Lambda\oplus \Lambda', \C)$ that vanishes on $\Lambda$ and agrees with $i\circ \alpha |_{\Lambda \cap TG_T} \circ i$ on $\Lambda'$. Going back through the isomorphism $\Hom_\Z(\Lambda\oplus \Lambda', \C)\cong \Hom_\R(TG, \C)$, $\beta$ corresponds to a linear map vanishing on the $\R$-span of $\Lambda$. The pullback of an invariant form to the universal cover $TG$ yields a constant $1$-form. Let us pull back the form $\beta=(\C\otimes \Phi_\R-\Phi_\C)(\alpha)$ to a form in $TG$.  Note that this pulled back 1-form on $TG$ is exact: a linear function on a vector space seen as an invariant 1-form is the differential of itself, seen as a function (in coordinates, $\sum a_idz_i$ is the differential of $\sum a_iz_i$). In other words, it is the differential of the linear function $h$ vanishing on the span of $\Lambda$ and agreeing with $i\circ \alpha |_{\Lambda \cap TG_T} \circ i$ on $\Lambda'$ (i.e. $\beta$ seen as a function). Lastly, note that $h:TG\to\C$ descends to $G$, since it is $\Lambda$-invariant (it vanishes on $\Lambda$ and it is $\R$-linear). This function can be defined to be $\Psi^G(\alpha)$ (it is uniquely defined up to constants amongst the functions $\Psi^G(\alpha)$ that satisfy that $d\Psi^G(\alpha)=\C\otimes\Phi_\R^G(\alpha)-\Phi_\C^G(\alpha)$). Note that we have defined $\Psi^G$ as a linear map, and furthermore, it is a homomorphism $H^1(G,\C)\to \Hom_{\R-\text{Lie groups}}(G,\C)$. In fact, since $\Psi^G(\alpha)$ is uniquely defined up to adding a constant function, our choice of $\Psi^G$ such that its image is in $\Hom_{\R-\text{Lie groups}}(G,\C)$ is unique.
\end{proof}

\begin{defprop}[Definition of $\Phi_{\C}^Y$, $\Phi_{\R}^Y$ and $\Psi^Y$]\label{defPhiPsiY}
	Let $Y$ be an allowed compactification of a complex semiabelian variety $G$, let $j_Y:G\to Y$ be the inclusion and let $E\coloneqq Y\setminus G$.
	
	The images of the maps $\Phi_{\C}^G$, $\Phi_{\R}^G$ and $\Psi^G$ consist of logarithmic forms in $\Gamma(Y,\cA^{1,\cl}_{Y,\C}(\log E))$,\linebreak $\Gamma(Y,\cA^{1,\cl}_{Y,\R}(\log E))$ and $\Gamma(Y,\cA^{0}_{Y,\R}(\log E))$ respectively, where
	$\Gamma(Y,\cA^l_{Y,k}(\log E))$ is seen as a subspace of $\Gamma(G,\cA^{l}_{Y,k})$ through
	$$
	\Gamma(Y,\cA^{l}_{Y,k}(\log E))\subset\Gamma(Y,(j_Y)_*\cA^{l}_{G,k})\cong \Gamma(G,\cA^{l}_{G,k}),
	$$
	for $k=\R,\C$ and $l=0,1$. Hence we can define
	\begin{align*}
		\Phi_{\C}^Y&:H^1(G,\C)\to\Gamma(Y,\wt W_1\cA^{1,\cl}_{Y,\C}(\log E))\subseteq \Gamma(Y,\cA^{1,\cl}_{Y,\C}(\log E)),\\
		\Phi_{\R}^Y&:H^1(G,\R)\to\Gamma(Y,\wt W_1\cA^{1,\cl}_{Y,\R}(\log E))\subseteq \Gamma(Y,\cA^{1,\cl}_{Y,\R}(\log E)),\text{ and}\\
		\Psi^Y&:H^1(G,\C)\to \Gamma(Y,\cA^{0}_{Y,\R}(\log E))
	\end{align*}
	as the maps $\Phi_{\C}^G$, $\Phi_{\R}^G$ and $\Psi^G$ of Definition-Proposition~\ref{defprop:PhiPsiG} respectively.
\end{defprop}

\begin{proof}
	Since the images of $\Phi_{\C}^G$, $\Phi_{\R}^G$ and $\Psi^G$ consist of invariant forms and those are closed, any form on $Y$ that extends them must also be closed.
	
	Let us check that the image of $\Phi^G_{\C}:H^1(G,\C)\to \Gamma(G,\cA^{1, \inv}_{G,\C})\subset \Gamma(G,\cA^{1}_{G,\C})$ lies in the space $\Gamma(Y,\wt W_1\cA^{1,\cl}_{Y,\C}(\log E))$. Since $Y$ is an allowed compactification of $G$, there exists a compactification $\ov G$ of $G$ as in Corollary~\ref{cor:compactification} and an algebraic map $p:Y\to\ov G$ such that $p\circ j_Y=j_{\ov G}$, where $j_{\ov G}:G \to \ov G$ is the inclusion. Let $E'=\ov G\setminus G$. First of all, the image of $\Phi_\C^G$ is contained in $\Gamma(\ov G,\wt W_1\cA_{\ov G,\C}(\log E')^{1,\cl})$, in fact, all invariant forms are logarithmic of weight $\wt W_{\lc}$ equal to $1$ (recall that we already know that invariant forms are closed). This can be verified over an open cover of $G_A$: over a small enough open set $U_A$ of $G_A$, $\ov G$ is isomorphic to $(\bP^1)^m\times U_A$, and an explicit basis of the space of invariant forms can be written down using local coordinates. By pulling back through $p$, we see that the elements in the image of $\Phi^G_{\C}$ all extend (necessarily uniquely) to elements of $\Gamma(Y,\wt W_1\cA^{1,\cl}_{Y,\C}(\log E))$. 
	
	The fact that the image of $\Phi^G_{\R}:H^1(G,\R)\to \Gamma(G,\cA^{1, \inv}_{G,\R})\subset \Gamma(G,\wt W_1\cA^{1}_{G,\R})$ lies in $\Gamma(Y,\cA^{1,\cl}_{Y,\R}(\log E))$ follows from the definition of $\Phi^G_{\R}$ as the real part of $\Phi^G_{\C}$ and from the previous paragraph.
	
	Let us show that the elements in the image of $\Psi^G$ extend to globally defined elements in the space $\Gamma(Y,\cA^0_{Y,\C}(\log E))$. It suffices to see that they extend to globally defined elements in $\Gamma(\ov G,\cA^0_{\ov G,\C}(\log E'))$, and then pull those back through $p:Y\to \ov G$. This can be verified over an open cover of $G_A$ as before: over a small enough open set $U_A$ of $G_A$, $G$ is isomorphic to $(\C^*)^n\times U_A$, and $\ov G$ is isomorphic to $(\bP^1)^m\times U_A$. Let $(z_1,\ldots,z_n)$ be (complex) coordinates of the $(\C^*)^n$ factor. One can check that the elements in the image of $\Psi^G$ (as defined explicitly in the proof of Definition-Proposition~\ref{defprop:PhiPsiG}) are the functions $(\C^*)^n\times U_A\to \C$ of the form $\sum_{i=1}^n a_i \log(|z_i|)$ for $a_1,\ldots,a_n\in\C$. Hence, these all lie in $\Gamma(\ov G,\cA^0_{\ov G,\C}(\log E'))$.
\end{proof}

\begin{corollary}\label{cor:Phidf/f}
	Let $Y$ be an allowed compactification of a complex semiabelian variety $G$, let $j_Y:G\to Y$ be the inclusion and let $E\coloneqq Y\setminus G$. Then, 
	the maps $\Phi_{\C}^Y$ and $\Phi_{\R}^Y$ satisfy the assumptions of part~\eqref{itemPhi} in Lemma~\ref{lem:df/f}.
\end{corollary}
\begin{proof}
	Let
	$
	0	\to G_T \xrightarrow{t}  G \xrightarrow{p_A} G_A\to 0
	$ be the Chevalley decomposition of $G$. Let $k=\R,\C$. The first condition (the image of \(\Phi^Y_{k}\) is contained in weight 1) is part of Definition-Proposition~\ref{defPhiPsiY}. The fact that $\Phi_{\C}^Y$ and $\Phi_{\R}^Y$ are sections of the cohomology map follows immediately from the definition of $\Phi_{k}^G$ (Definition-Proposition~\ref{defprop:PhiPsiG}).
	
	The fact that $\Phi_{k}^Y$  maps forms which are pulled back from $H^1(G_A,\C)$ to forms whose restriction to $G$ is in the image of $p_A^*:\Gamma(G_A,\cA^1_{G_A,k})\to \Gamma(G,\cA^1_{G,k})$ also follows by  Definition-Proposition~\ref{defprop:PhiPsiG}.
	
	Lastly, by Lemma~\ref{lem:generatecoh} \eqref{itemInvariantHolo}, classes of holomorphic forms are represented by invariant holomorphic forms. By definition, $\Phi_{\C}^Y$ maps these to holomorphic forms, which in particular are $(1,0)$-forms.
\end{proof}

Applying Lemma~\ref{defprop:thickMHC} to the objects $(\cN^{\bullet}_{X,D,n},H^1(G,\R),\ov f^*\circ\Phi_{\R}^Y,\ov f^*\circ\Phi_{\C}^Y,\ov f^*\circ\Psi^Y)$, which satisfy the assumptions of Lemma~\ref{lem:df/f} by Corollary~\ref{cor:Phidf/f}, we get a thickened mixed Hodge complex of sheaves. We describe this mixed Hodge complex of sheaves explicitly in the following definition.

\begin{definition}[The thickened logarithmic Dolbeault mixed Hodge complex of sheaves]\label{def:thickening}
	Let $U$ be a smooth connected complex algebraic variety, let $G$ be a complex semiabelian variety, and let $f:U\to G$ be an algebraic morphism, which extends to $\ov f:X\to Y$, where $X,Y$ are compatible compactifications of $U,G$ with respect to $f$ as in Definition~\ref{def:compatibleCompf}. Let $\displaystyle R_{m}\coloneqq \frac{\prod_{j=0}^{\infty}\Sym^j H_1(G,k)}{\prod_{j=m}^{\infty}\Sym^j H_1(G,k)}$ and let $R_{-m}\coloneqq \Hom_k(R_{m},k)$ for all $m\geq 1$, and $k=\R,\C$.
	
	Let $m\in \Z\setminus \{0\}$, and let $n\geq\max\{2,\dim_\R U\}$. We denote by $(R_m\otimes\cN_{X,D,n}^\bullet,d+\ov f^*\circ\Phi^Y(\eps))$ the thickened mixed Hodge complex  with real part
	\[
	\left(\left(R_{m}\otimes_\R\cA^\bullet_{X,\R}(\log D),d+\ov f^*\circ\Phi_\R^Y(\eps_{\R})\right), W_{\lc}^n\right),
	\]
	complex part $\left(\left(R_{m}\otimes_\C\cA^\bullet_{X,\C}(\log D),d+\ov f^*\circ\Phi_\C^Y(\eps_{\C})\right), W_{\lc}^n, F^{\lc}\right)$, and a filtered isomorphism
	\begin{multline*}
		\alpha=e^{\ov f^*\circ\Psi^Y(\eps_{\C})}:
		\left(\left(R_{m}\otimes_\R\cA^\bullet_{X,\R}(\log D),d+\ov f^*\circ\Phi_\R^Y(\eps_{\R})\right)\otimes_{\R}\C,W_{\lc}^n\right)\xrightarrow{\sim}
		\\
		\left(\left(R_{m}\otimes_\C\cA^\bullet_{X,\C}(\log D),d+\ov f^*\circ\Phi_\C^Y(\eps_{\C})\right), W_{\lc}^n\right).
	\end{multline*}
Here,  $W_{\lc}^n$ denotes the tensor filtration of the weight filtration in $R_m$ and the filtration $W_{\lc}^n$ of $\cA_{X,k}^\bullet(\log D)$ from Definition-Proposition~\ref{defprop:modifiedNA}, and $F^{\lc}$ denotes the tensor filtration of the Hodge filtration in $R_m$ and the filtration $F^{\lc}$ from Definition-Proposition~\ref{defprop:modifiedNA}.
\end{definition}

\begin{remark}
	Note that technically, applying Definition-Proposition~\ref{defprop:thickMHC} to the objects $$(\cN^{\bullet}_{X,D,n},H^1(G,\R),\ov f^*\circ\Phi_{\R}^Y,\ov f^*\circ\Phi_{\C}^Y,\ov f^*\circ\Psi^Y)$$ yields a thickened mixed Hodge complex with four terms $(\cK_0^{\bullet},\cK_1^{\bullet}, \cK_2^{\bullet},\cK_3^\bullet)$, but since two of the maps between them are the identity ($\cK_0^{\bullet}=\cK_1^{\bullet}$, $\cK_2^{\bullet}=\cK_3^{\bullet}$), we have simplified the notation in the definition above.
\end{remark}

The maps defined in Definition-Proposition~\ref{defPhiPsiY} satisfy the following functoriality property.

\begin{corollary}\label{cor:PhiFunctorial}
	Suppose that we have a map of semiabelian varieties $g\colon G_1\to G_2$. Let $Y_2$ be an allowed compactification of $G_2$. Then, there exists an allowed compactification $Y_1$ of $G_1$ such that $g$ extends to $\ov{g}:Y_1\to Y_2$. Moreover, for every such allowed compactification $Y_1$,  the maps $(\Phi_\R^{Y_1},\Phi_\C^{Y_1},\Psi^{Y_1})$ and $(\Phi_\R^{Y_2},\Phi_\C^{Y_2},\Psi^{Y_2})$ are compatible in the sense that $\ov g^*\circ \Phi_\R^{Y_1} = \Phi_\R^{Y_2}\circ g^*$, $\ov g^*\circ \Phi_\C^{Y_1} = \Phi_\C^{Y_2}\circ g^*$ and $\ov g^*\circ \Psi^{Y_1} = \Psi^{Y_2}\circ g^*$.
\end{corollary}
\begin{proof}
	Let $\ov{G_1}$ be a compactification of $G_1$ as in Corollary~\ref{cor:compactification}. We can obtain $Y_1$ as a resolution of singularities of the closure of the graph of $g$ inside of $\ov{G_1}\times Y_2$.
	
	Recall that by Proposition~\ref{prop:funChevalley}, $g$ must preserve the Chevalley decomposition. By Lemma~\ref{lem:generatecoh}, $\Phi^{Y_i}_{\C}$ is completely determined by its restriction to $\Gamma(G_i,\Omega^{1, \inv}_{G_i})$ and $H^1({(G_i)}_A,\C)$, both seen as subspaces of $H^1(G_i,\C)$. With the definition of $\Phi^{G_i}_\C$ from Definition-Proposition~\ref{defprop:PhiPsiG}, it is straightforward to see that $g^*\circ \Phi_\C^{G_1} = \Phi_\C^{G_2}\circ g^*$. Hence, $\ov g^*\circ \Phi_\C^{Y_1} = \Phi_\C^{Y_2}\circ g^*$.
	In Definition-Proposition~\ref{defPhiPsiY}, $\Phi^{Y_i}_\R$ is defined as the real part of $\Phi^{Y_i}_\C$, so these are compatible as well. Finally, $\Psi^{Y_i}$ is determined up to constants by the condition that it is a homotopy between $\Phi^{Y_i}_\R$ and $\Phi^{Y_i}_\C$, so it is uniquely determined if one requires that its image is composed of homomorphisms of $\R$-Lie groups $G_i\to \C$, and compatible with $g$.
\end{proof}

\begin{example}[The case $G=\C^*$]\label{eg:TateRmMHC}
	Let $m\geq 1$, and let $n\geq\max\{2,\dim_\R U\}$. If $G=\C^*$, the $R_\infty$-linear isomorphism of MHS  $A_m:R_m(1-m)\to R_{-m}$ from Example~\ref{eg:TateRm} lifts to an isomorphism of mixed Hodge complexes of sheaves:
	$$
	A_m\otimes \Id :(R_m\otimes\cN^\bullet_{X,D,n}, d+\ov f^*\circ\Phi^Y(\eps))(1-m)\longrightarrow (R_{-m}\otimes\cN^\bullet_{X,D,n}, d+\ov f^*\circ\Phi^Y(\eps)),
	$$
	where $(1-m)$ denotes a Tate twist. Indeed, the commutativity with the differentials is immediate: $d$ leaves the first factor of the tensor product unchanged, and $\ov f^*\circ\Phi^Y(\eps)$ acts on the first factor of the tensor product by multiplication by elements of $R_\infty$, which commutes with $A_m$. The commutativity with the pseudo-morphism is also immediate because it leaves the first factor of the tensor product unchanged.
\end{example}

\section{Thickened logarithmic Dolbeault complexes and local systems}\label{sec:thickeningResolves}

Let $f:U\to G$ be an algebraic morphism from a smooth variety to a semiabelian variety. Let $\Phi_\R^G$ be as in Definition-Proposition~\ref{defprop:PhiPsiG}, let $m\in\Z\setminus\{0\}$ and let $R_m$ as in Definition~\ref{def:Rminfty} (with $\R$-coefficients throughout this section). Recall the definition of the twisted differential from Definition-Proposition~\ref{defprop:thickMHC}. This section is devoted to showing that $(R_m\otimes_\R \cA_{U,\R}^\bullet, d+f^*\circ\Phi_\R^G(\eps_{\R}))$ is a resolution of $R_m\otimes_R \oL$ (see Lemmas~\ref{lem:thickeningIsLocalSystem} and~\ref{lem:nuResolves}), explicitly defining the morphism that makes the former a resolution of the latter (namely the one defined in Construction~\ref{con:nu} below). Recall that the definition of $\cL$ and $\oL$ can be found in Definition~\ref{def:L}.

\begin{construction}[Definition of $(\Phi_\R^G)^\vee$]\label{con:Phidual}
	Recall that $\Phi_\R^G$ is a map from $H^1(G,\R)$ to $\Gamma(G,\cA^{1,\inv}_{G,\R})$ (extended in a unique way by $\Phi_{\R}^Y$ to logarithmic forms on $Y$), and recall that $\Gamma(G,\cA^{1,\inv}_{G,\R})\cong \Hom_\R(TG,\R)$ by Lemma~\ref{lem:generatecoh}\eqref{itemLambda}. Under this identification, we can consider its dual $(\Phi_\R^G)^\vee$ as a morphism
	\begin{equation}\label{eq:Phidual}
		(\Phi_\R^G)^\vee\colon TG\to H_1(G,\R).
	\end{equation}
Note that since $\Phi_\R^G$ is a section of the cohomology map, $(\Phi_\R^G)^\vee$ fixes $H_1(G,\R)\subset TG$. Furthermore, we will also use the notation $(\Phi_\R^G)^\vee$ to denote the map
	\begin{equation}\label{eq:Phidualinduced}
		(\Phi_\R^G)^\vee\colon TG\otimes_\R \cA^0_{U,\R}\to H_1(G,\R)\otimes_\R \cA^0_{U,\R}
	\end{equation}
	induced by  $(\Phi_\R^G)^\vee$ in \eqref{eq:Phidual}. 
\end{construction}

Recall from Remark~\ref{remk:sectionsOfL} that a local $\R$-basis of $\cL$ at any point of $U$ is given by lifts $\iota$ of $f$ to $TG$, i.e. maps $\iota\colon U\to TG$ such that $\exp\circ\iota=f$. The sheaf $\oL$ is a local system of rank 1 free $\R[\pi_1(G)]$-modules. Recall from Notation~\ref{not:iotaBar} that $\cL$ and $\oL$ are identified through the identity map $\cL\to\oL$ that maps $\iota$ to $\ov\iota$, which is an $R$-antilinear isomorphism.

Recall that $R_\infty = \prod_{j=0}^\infty \Sym^j H_1(G,\R)$ and $R\coloneqq \R[\pi_1(G)]$. Moreover, recall Notation~\ref{not:abelianization} and Definition~\ref{def:Rmodstructure}.

\begin{construction}[Definition of $e^{-(\Phi_\R^G)^\vee}$]\label{con:nu}
	Let $m\in\Z\setminus\{0\}$. 
	Let $\{\gamma_i\}$ be a basis of $\pi_1(G)$, so that $\{\h\gamma_i\}$ is a $\Z$-basis of $\Lambda=H_1(G,\Z)\subset TG$, and let $\{e_j\}$ be chosen so that $\{\h \gamma_i,e_j\}$ form an $\R$-basis of $TG\supset H_1(G,\Z)$.
	Since they form a basis, any lift $\iota\colon U\to TG$ may be written as $\iota=\sum \h\gamma_i\otimes g_i + e_j\otimes h_j$ for some $g_i,h_j\in \cA^0_{U,\R}$, so $$\ov \iota=\sum \h\gamma_i\otimes g_i + e_j\otimes h_j.$$ Hence, we can see $\ov\cL$ as a subsheaf of $TG\otimes_{\R}\cA_{U,\R}^0$, and restrict $(\Phi_\R^G)^\vee$ as in \eqref{eq:Phidualinduced} to $\ov\cL$.
	
	 Up to a sign, we postcompose $(\Phi_\R^G)^\vee$ as in \eqref{eq:Phidualinduced} with the exponential map, to obtain the following:
	 
    \begin{align*}
		{\displaystyle e^{-(\Phi_\R^G)^\vee}\colon}
		{\displaystyle R_m\otimes_R \oL\subset R_m\otimes_R( TG\otimes_\R \cA^0_{U,\R})}
		&
		\longrightarrow
		{\displaystyle R_m \otimes_\R \cA^0_{U,\R}=R_m\otimes_{R_{\infty}}(R_\infty \otimes_\R \cA^0_{U,\R})}
		\\
		\\
		{\displaystyle
			\alpha\otimes \ov\iota = 
			\alpha \left(
			\sum \h\gamma_i\otimes g_i + e_j\otimes h_j
			\right)
		}
		&
		\longmapsto
		\alpha e^{-(\Phi_\R^G)^\vee(\bar\iota)}\\
     &=  \alpha\sum_{k=0}^\infty \frac{1}{k!} \left(-\sum_{i} \log \gamma_i\otimes g_i - \sum_j (\Phi_\R^G)^\vee ( e_j)\otimes h_j\right)^k.
	\end{align*}
	Note that the product of $k$ many elements in $H_1(G,\R)\otimes \cA^0_{U,\R}$ is an element of $\Sym^k H_1(G,\R)\otimes_\R \cA^0_{U,\R}\subset R_{\infty}\otimes_\R \cA^0_{U,\R}$, so, since $R_m$ is an $R_\infty$-module, it makes sense to multiply $\alpha\in R_m$ by the elements in the first factor of the tensor product of $\frac{1}{k!} \left(-\sum_{i} \log \gamma_i\otimes g_i - \sum_j (\Phi_\R^G)^\vee ( e_j)\otimes h_j\right)^k$ for all $k$.
\end{construction}

\begin{proposition}\label{prop:nuR-linear}
	The map $e^{-(\Phi_\R^G)^\vee}$ defined in Construction~\ref{con:nu} is well-defined on the tensor product (over $R$), and is $R$-linear.
\end{proposition}

\begin{proof}
	Let us show that the above formula is well-defined on the tensor product (over $R$).  The same reasoning will show us that $e^{-(\Phi^G_\R)^\vee}$ is $R$-linear.
	
	Recall that $R$ acts on $\ov\cL$ by letting $\gamma_0\in \pi_1(G)$ act by translation by $\gamma_0^{-1}$, that is, $\gamma_0\cdot \ov\iota=\ov{\gamma_0^{-1}\cdot \iota}$, which  corresponds with postcomposing $\ov\iota$ with translation by $-\h \gamma_0=\h\gamma_0^{-1}$ (namely the element in $H_1(G,\Z)\subset TG$ corresponding to $\gamma_0\in\pi_1(G)$). Furthermore, $R$ is embedded in $R_\infty$ by the ring $\gamma\mapsto e^{\log\gamma}$ of Definition~\ref{def:Rmodstructure}. Hence, we need to show that for any $\gamma\in R$, the image of $e^{\log \gamma}\alpha\otimes \ov\iota$ equals the image of $\alpha\otimes \gamma \cdot\ov\iota$.  It is enough to check this for $\gamma_0\in \pi_1(G)$, since $R$ is generated by $\pi_1(G)$. With the above notations:
	\begin{align}\label{eq:translation}
		\begin{split}
		\alpha	e^{-(\Phi_\R^G)^\vee (\gamma_0\cdot\ov\iota)} &=
		\alpha e^{\displaystyle -(\Phi_\R^G)^\vee \left(
			-\log\gamma_0\otimes 1 + \sum \h\gamma_i\otimes g_i + e_j\otimes h_j
			\right)}
		\\
		&=\alpha
		\exp\left(\log \gamma_0 \otimes 1- \sum_i \log \gamma_i\otimes g_i - \sum_j (\Phi_\R^G)^\vee( e_j ) \otimes h_j\right)
		\\
		&=\alpha
		e^{\log\gamma_0}e^{-(\Phi_\R^G)^\vee  \ov\iota}.
	\end{split}
	\end{align}
\end{proof}

\begin{proposition}\label{prop:sectionDualBasis}
	Let $\{\log \gamma_i,  e_j\}$ be an $\R$-basis of $TG$, where $\{\log \gamma_i\}$ is the $\Z$-basis of $H_1(G,\Z)$ corresponding to a basis $\{\gamma_i\}$ of $\pi_1(G)$. Let $\{\log\gamma_i^\vee,e_j^\vee\}$ be its dual basis. Suppose a locally defined $\ov\iota\in \ov\cL$ is given by $\ov\iota =\sum_i \log \gamma_i\otimes g_i + \sum_j e_j\otimes h_j\in TG\otimes_{\R}\cA^0_{U,\R}$. Then,   
	\begin{itemize}
		\item All the 1-forms \(dg_i\) and \(dh_j\) are the pullback of invariant 1-forms on \(G\), namely:
		\[
		dg_i = f^*\log \gamma_i^\vee,\quad dh_j = f^* e_j^\vee,
		\]
		where $\log\gamma_i^\vee,e_j^\vee:TG\to\R$ are seen in $\Gamma(G,\cA^{1,\inv}_{G,\R})$ through the isomorphism described in Lemma~\ref{lem:generatecoh}, part (\ref{itemLambda}).
		\item For every \(j\), the function $h_j$ is the composition of \(f\) with the (globally defined) unique differentiable homomorphism $G\to \R$ mapping $\exp(e_j)$ to $1$ and the rest of the elements of $\{\exp(e_l)\}$ to $0$.
	\end{itemize}
\end{proposition}
\begin{proof}
	
	Let us start with the first statement. Let $V$ be the open set in $U$ such that $\ov\iota$ is a map from $V$ to $TG$. By definition of the dual basis, for all $x\in V$ we have that
	\[
	g_i(x) = \langle (\log \gamma_i)^\vee
	, \ov\iota (x)\rangle;\quad
	h_j(x) = \langle (e_j)^\vee , \ov\iota (x)\rangle.
	\]
	In any small neighborhood of $\exp(\ov\iota(x))$ in $G$, we can define $a_i, b_j:G\to \R$ such that $a_i\circ\exp =(\log \gamma_i)^\vee$ and $b_j\circ\exp=e_j^\vee$. Hence, locally we have that $g_i=f^*(a_i)$, and $h_j=f^*(b_j)$ in a neighborhood of $x$. Thus, $dg_i=f^*(da_i)$, and $dh_j=f^*(db_j)$. Note that $\exp^* da_i=d((\log \gamma_i)^\vee)$ is a constant 1-form on $TG$. The identification $\Hom_\R(TG,\R)\cong \cA^{1,\inv}_{G,\R}$ from Lemma~\ref{lem:generatecoh}\eqref{itemLambda} implies that $da_i=(\log \gamma_i)^\vee$ (seen as an element of $\cA^{1,\inv}_{G,\R}$) and similarly, $db_j=e_j^\vee$. This concludes the proof of the first statement.
	
	For the second statement, we just need to see that $b_j$ is defined globally, and that it coincides with the homomorphism $G\to\R$ described. Note that $b_j$ is defined globally because $e_j^\vee$ is invariant by the action of $\log\gamma_i$. Since $b_j\circ \exp=e_j^\vee$ and $\exp$ is a surjective homomorphism, $b_j$ is a group homomorphism which takes $\exp(e_j)$ to $1$ and the image by $\exp$ of the rest of the elements of the basis $\{\log\gamma_i,e_l\}$ to $0$.
\end{proof}
\begin{lemma}\label{lem:thickeningIsLocalSystem}
	For any $m\in \Z\setminus\{0\}$, the complex $(R_m \otimes_\R \cA_{U,\R}^\bullet,d+f^*\circ\Phi_\R^G(\eps_{\R}))$ has non-zero cohomology only in degree $0$. The kernel of the differential in degree $0$ is a local system of $R_\infty$-modules whose stalks are isomorphic to $R_m$.
\end{lemma}

\begin{proof}
	We will show that locally there is an isomorphism between $(R_m \otimes_\R \cA_{U,\R}^\bullet, d+ f^*\circ\Phi_\R^G(\eps_\R))$ and $(R_m\otimes_\R \cA_{U,\R}^\bullet, d)$. Let us consider a simply connected open set $V$ of $U$. Over such an open set, all closed 1-forms are exact, and in particular the restriction to $V$ of the image of $f^*\circ \Phi_{\R}^G\colon H^1(G,\R)\to \Gamma(U,\cA^1_{U,\R})$ consists of exact forms. Let $h\colon H^1(G,\R)\to \Gamma(V,\cA^0_{V,\R})$ be a linear map such that $d\circ h = (f^*\circ \Phi_\R^G)|_{V}$. 
	Applying Lemma~\ref{lem:dglas}, multiplication by $e^{h(\eps_\R)}$ is an isomorphism:
	\[
	(R_m \otimes_\R \cA^\bullet_{V,\R} ,d+f^*\circ\Phi_\R^G(\eps_{\R}))\cong (R_m \otimes_\R \cA^\bullet_{V,\R},d + f^*\circ\Phi_\R^G(\eps_{\R})-(d\circ h)(\eps_{\R}) + [h(\eps_\R),f^*\circ\Phi_\R^G(\eps_\R)]).
	\]
	Note that $h(\eps_\R)$ and $f^*\circ\Phi_\R^G(\eps_{\R})$ commute because they are elements of a cdga, so the differential on the right hand side above is simply $d$. Note that $(\cA^\bullet_{V,\R} ,d)$ is a complex of acyclic sheaves (with respect to the global sections functor) which resolves the trivial local system $\ul\R_V$ (see \cite[p. 127]{ks}, for example). This shows that $(R_m \otimes_\R \cA^\bullet_{V,\R} ,d+f^*\circ\Phi_\R^G(\eps_{\R}))$ is isomorphic to the resolution of a trivial local system with stalk $R_m$, which is exact in all places except for degree $0$, as desired. 
\end{proof}

\begin{lemma}\label{lem:nuResolves}
	The morphism $e^{-(\Phi_\R^G)^\vee}$ defined as in Construction~\ref{con:nu} is an isomorphism onto the kernel of
	$$
	d + f^*\circ\Phi_\R^G(\eps_{\R}):R_m\otimes_{\R}\cA^0_{U,\R}\to R_m\otimes_{\R}\cA^1_{U,\R}
	$$ 
\end{lemma}
\begin{proof}
	
	Let us start by proving that $d\circ e^{-(\Phi_\R^G)^\vee}=-(f^*\circ\Phi_\R^G(\eps_\R))\circ e^{-(\Phi_\R^G)^\vee}$, which will show that the image of $e^{-(\Phi_\R^G)^\vee}$ is contained in the kernel of $d+f^*\circ\Phi_\R^G(\eps_\R)$. Let $\{ e_i\}$ be an $\R$-basis of $TG$ and let $\{e_i^\vee\}$ be its dual basis. Let $\ov\iota =  \sum  e_i\otimes h_i$ be a local generator of $\oL$. We must compute $d(e^{-(\Phi_\R^G)^\vee}( \ov\iota))$, i.e.
	\begin{align*}
		d(e^{-(\Phi_\R^G)^\vee}( \ov\iota)) &=d
		\left(
		e^{-(\Phi_\R^G)^\vee  (\ov\iota)}
		\right)
		\\
		&=
		d\left(
		\sum_{k=0}^\infty \frac{1}{k!} \left(- \sum_i (\Phi_\R^G)^\vee ( e_i)\otimes h_i\right)^k
		\right)
		\\
		&=
		\sum_{k=1}^\infty \frac{1}{(k-1)!} \left(- \sum_i (\Phi_\R^G)^\vee ( e_i)\otimes h_i\right)^{k -1} \left(- \sum_i (\Phi_\R^G)^\vee ( e_i)\otimes dh_i\right)
		\\
		&=
		\left(- \sum_i (\Phi_\R^G)^\vee (e_i)\otimes dh_i\right)
		\cdot e^{-(\Phi_\R^G)^\vee}(\ov\iota )
	\end{align*}
	Using Proposition~\ref{prop:sectionDualBasis}, $dh_i = f^*( e_i^\vee)$. Finally, note that if $\{s_j\}$ is an $\R$-basis of $H_1(G,\R)$, and $\{s_j^\vee\}$ is its dual basis, then
\begin{align*}
	- \sum_i (\Phi_\R^G)^\vee (e_i)\otimes f^*( e_i^\vee) &=  - \Id\otimes f^*\left( \sum_i (\Phi_\R^G)^\vee (e_i)\otimes e_i^\vee\right)\\
	&=- \Id\otimes f^*\left(\sum_j s_j \otimes \Phi_\R^G(s_j^\vee)\right) = -f^*\circ\Phi_\R^G(\eps_{\R}).	
\end{align*}
	So indeed the image of $e^{-(\Phi_\R^G)^\vee}$ is contained in the desired kernel.
	
	Using Lemma~\ref{lem:thickeningIsLocalSystem}, we know that the kernel of $d$ is a local system, of the same real dimension as $R_m$. To show $e^{-(\Phi_\R^G)^\vee}$ is an isomorphism onto the kernel, we only need to prove that $e^{-(\Phi_\R^G)^\vee}$ is either injective or surjective on stalks, since we know the dimensions agree.
	
	The stalks of $\ker d+f^*\circ\Phi^G_\R(\eps_\R)\subset R_m\otimes_\R\cA^0_{U,\R}$ are finitely generated $R_{|m|}$-modules. If $m>0$,  Nakayama's Lemma implies that one can show that $e^{-(\Phi_\R^G)^\vee}$ is surjective on stalks by taking the quotient by the maximal ideal of $R_{m}$, reducing to the case $m=1$ (which is clear, the complex is just $(\cA^\bullet_{U,\R} ,d)$). If $m< 0$, we can show that $e^{-(\Phi_\R^G)^\vee}$ is injective by noting that $R_{-1}\subset R_m$ is contained in every non-zero sub $R_{-m}$-module of $R_m$. Therefore, to show that $e^{-(\Phi_\R^G)^\vee}$ is injective, it is enough to show that, identifying the stalk with $R_m$, the kernel of $e^{-(\Phi_\R^G)^\vee}$ intersects $R_{-1}$ trivially, which is again clear.
\end{proof}

\begin{remark}\label{rem:commutativityAm}
	Let $m\geq 1$ and suppose that $G=\C^*$. The isomorphism $A_m:R_m\to R_{-m}$ from Example~\ref{eg:TateRm} extends to an $R_\infty$-linear isomorphism
	$$
	A_m\otimes_\R\Id_{\oL}:R_m\otimes_R\oL\to R_{-m}\otimes_R\oL.
	$$
	Moreover, it is immediate from the definition of $e^{-(\Phi_\R^{\C^*})^\vee}$ that $$\left(A_m\otimes_\R \Id_{\cA^0_{U,\R}}\right)\circ e^{-(\Phi_\R^{\C^*})^\vee}:R_m\otimes_R\oL\to R_{-m }\otimes_\R \cA^0_{U,\R}$$
	coincides with
	$$
	e^{-(\Phi_\R^{\C^*})^\vee}\circ \left(A_m\otimes_\R\Id_{\oL}\right).
	$$
\end{remark}

\section{Mixed Hodge structures}\label{sec:MHS}

\subsection{The MHS on \texorpdfstring{$H^j(U,R_m\otimes_R\ov\cL)$}{Hj(U, Rm ⊗ L)}}

	Let $f:U\to G$ be an algebraic morphism between a smooth complex connected algebraic variety $U$ and a semiabelian variety $G$. Let $X$ and $Y$ be compatible compactifications of $U$ and $G$ with respect to $f$ as in Definition~\ref{def:compatibleCompf}, let $j:U\to X$ be the inclusion, and let $D\coloneqq X\setminus U$. Note that
	$$
	(R_m\otimes_\R \cA_{U,\R}^\bullet, d+f^*\circ\Phi_\R^G(\eps_\R))=j^{-1}\left(R_m\otimes_\R\cA_{X,\R}^\bullet(\log D), d+\ov f^*\circ\Phi_\R^Y(\eps_\R)\right),$$
	where $\ov f:X\to Y$ extends $f$ and $\Phi_\R^Y$ is as in Definition-Proposition~\ref{defPhiPsiY}. Hence, $$Rj_*(R_m\otimes_R \oL)\cong Rj_*j^{-1}\left(R_m\otimes_\R\cA_{X,\R}^\bullet(\log D),d+\ov{f}^*\circ\Phi_\R^Y(\eps_\R)\right).
	$$
	In this case, the adjunction $\Id\to Rj_*j^{-1}$ applied to the complex of sheaves
	\[
	\left(R_m\otimes_\R\cA_{X,\R}^\bullet(\log D), d+\ov f^*\circ\Phi_{\R}(\eps_{\R})\right)
\]
	is the real part of the thickened logarithmic Dolbeault mixed Hodge complex of sheaves $(R_m\otimes \cN_{X,D,n}^\bullet, d+\ov f^*\circ\Phi^Y(\eps))$ from Definition~\ref{def:thickening}. It is an isomorphism in the derived category by Proposition~\ref{prop:adjunctionNA} and Proposition~\ref{prop:thickenedQuiso} (see Remark~\ref{rmk:acyclicity}).
	
	In Section~\ref{sec:thickeningResolves} we saw that the morphism $$e^{-(\Phi_\R^G)^\vee}\colon R_m\otimes_R\ov\cL\to \left(R_m\otimes_\R\cA^\bullet_{U,\R}, d+f^*\circ\Phi_\R^G(\eps_\R)\right)$$ is a quasi-isomorphism. The first goal of this section is to show that the mixed Hodge complex of sheaves $(R_m\otimes \cN_{X,D,n}^\bullet, d+\ov f^*\circ\Phi^Y(\eps))$ endows $H^i(U,R_m\otimes_R \oL)$ with an $\R$-MHS for all $i$, and to describe the map via which these MHS are induced.

\begin{definition}[MHS on $H^*(U,R_m\otimes_R \ov\cL)$]\label{def:endowedMHS}
	Let $f:U\to G$ be an algebraic morphism between a smooth complex connected algebraic variety $U$ and a semiabelian variety $G$. Let $Y$ be an allowed compactification of $G$, and let $X$ be a good compactification of $U$ such that $f$ extends to $\ov f:X\to Y$. Let $D=X\setminus U$, let $m>0$ and let $n\geq\max\{2,\dim_\R U\}$.
	\begin{itemize}
		\item Suppose that $m<0$. The thickened logarithmic Dolbeault mixed Hodge complex of sheaves $(R_m\otimes\cN^\bullet_{X,D,n}, d+\ov f^*\circ\Phi^Y(\eps))$ from Definition~\ref{def:thickening} endows
		$H^*(U,R_m\otimes_R \ov\cL)$ with a mixed Hodge structure via this sequence of isomorphisms in the derived category.
		\begin{equation}\label{eq:endow}
			\begin{tikzcd}[column sep = 2em]
				Rj_*(R_m\otimes_R \ov\cL)\arrow[r,"Rj_*e^{-(\Phi_{\R}^G)^\vee}"] & Rj_*(R_m\otimes_\R\cA^\bullet_{U,\R}, d+ f^*\circ\Phi_\R^G(\eps_{\R}))\arrow[d,equals]\\
				(R_m\otimes_\R\cA^\bullet_{X,\R}(\log D), d+\ov f^*\circ\Phi_\R^Y(\eps_{\R}))
				\arrow[r, shift left = 2.2ex, phantom, "\scriptsize\text{ adjunction }"]		
				\arrow[r]
				 & Rj_*j^{-1}(R_m\otimes_\R\cA^\bullet_{X,\R}(\log D), d+\ov f^*\circ\Phi_\R^Y(\eps_{\R}))
			\end{tikzcd}
		\end{equation}
		\item Suppose that $m>0$. Let $e=\dim_\C H_1(G,\C)$. The \textbf{Tate twisted} thickened logarithmic Dolbeault mixed Hodge complex of sheaves $(R_m\otimes\cN^\bullet_{X,D,n}, d+\ov f^*\circ\Phi^Y(\eps))(e)$ endows
		$H^*(U,R_m\otimes_R \ov\cL)$ with a mixed Hodge structure via the same sequence of isomorphisms as in \eqref{eq:easyThickening}, namely
			$$
			\begin{tikzcd}[column sep=5em]
				Rj_*(R_m\otimes_R \ov\cL)\arrow[r,"Rj_*e^{-(\Phi_{\R}^G)^\vee}"] & Rj_*(R_m\otimes_\R\cA^\bullet_{U,\R}, d+ f^*\circ\Phi_\R^G(\eps_{\R}))\arrow[d,equals]\\
				(R_m\otimes_\R\cA^\bullet_{X,\R}(\log D), d+\ov f^*\circ\Phi_\R^Y(\eps_{\R}))(e)\arrow[r, "\text{ adjunction }"] & Rj_*j^{-1}(R_m\otimes_\R\cA^\bullet_{X,\R}(\log D), d+\ov f^*\circ\Phi_\R^Y(\eps_{\R}))
			\end{tikzcd}
			$$
	\end{itemize}
\end{definition} 

\begin{remark}
	The Tate twist when $m$ is positive but not when $m$ is negative might seem arbitrary in the previous definition, but it is not. Indeed, the case $m$ negative will be used to endow quotients of the homology Alexander modules with MHSs, and those MHSs will be functorial (see Section~\ref{sec:functoriality}) without the need for any twists. However, the case when $m$ is positive is related to the MHS on the torsion part of the cohomology Alexander modules defined in \cite{mhsalexander} in the case when $G=\C^*$ (see Remark~\ref{rem:comparisonC^*}), where the twist was needed to enjoy good functoriality properties (see \cite[Theorem 6.1]{mhsalexander}). In any case, the focus of this paper is the case where $m$ is negative.
\end{remark}

\begin{example}[The case $G=\C^*$]
	If $G=\C^*$, and $m\geq 1$, the MHS on $H^*(U,R_m\otimes_R\oL)$ and $H^*(U,R_{-m}\otimes_R\oL)$ from Definition~\ref{def:endowedMHS} are related as follows:  Let $A_m:R_m\to R_{-m}$ be the $R_\infty$-linear identification from Example~\ref{eg:TateRm}. The isomorphism $A_m\otimes\Id_{\oL}:R_m\otimes_R\oL\to R_{-m}\otimes_R\oL$ from Remark~\ref{rem:commutativityAm} lifts (by the commutativity with $e^{-(\Phi_\R^{\C^*})^\vee}$ explained therein) to the isomorphism of mixed Hodge complexes of sheaves $$A_m\otimes\Id:(R_m\otimes\cN^\bullet_{X,D,n}, d+\ov f^*\circ\Phi^Y(\eps))(1-m)\longrightarrow (R_{-m}\otimes\cN^\bullet_{X,D,n}, d+\ov f^*\circ\Phi^Y(\eps))$$ from Example~\ref{eg:TateRmMHC}. By Remarks~\ref{rem:TateCohomology} and~\ref{transvstate}, $A_m$ induces the following isomorphism of MHS
	$$
	H^j(A_m\otimes\Id_{\oL}):H^j(U,R_m\otimes_R\oL)(2-m)\to H^j(U,R_{-m}\otimes_R\oL),
	$$
	where $(2-m)$ denotes the $(2-m)$-th Tate twist.
\end{example}

In Section~\ref{ss:independence} we will see that the previous definition is independent of the choice of $n$ and the choice of compatible compactifications. Before  that, let us show some properties of the MHS from Definition~\ref{def:endowedMHS} while assuming the independence of those choices.

\begin{remark}[The pro-MHS on $H^*(U,R_\infty\otimes_R \ov\cL)$]\label{rem:endowedproMHS}
	Let $m'>m>0$. In that case, the projection morphism $p_{m',m}\colon R_{m'}\to R_m$ is an $R_\infty$-linear mixed Hodge structure morphism, and it induces a projection $p_{m',m}\otimes \Id\colon R_{m'}\otimes_R \oL\to R_{m}\otimes_R \oL$. This morphism extends via $e^{-(\Phi_\R^G)^\vee}$ and the morphisms in Definition~\ref{def:endowedMHS} to a morphism of mixed Hodge complexes of sheaves (that is, a morphism between the corresponding complexes of sheaves respecting the filtrations and the pseudo-morphism):
	$$
	p_{m',m}\otimes\Id\colon (R_{m'}\otimes\cN^\bullet_{X,D,n}, d+\ov f^*\circ\Phi^Y(\eps))\to (R_m\otimes\cN^\bullet_{X,D,n}, d+\ov f^*\circ\Phi^Y(\eps)).
	$$
	In particular, $p_{m',m}$ induces a MHS morphism
	$$
	H^*(U,R_{m'}\otimes_R \oL)\to H^*(U,R_{m}\otimes_R \oL).
	$$
 By Proposition~\ref{prop:limvsH}, taking the inverse limit for $m>0$, one obtains a pro-MHS on $H^*(U,R_\infty\otimes_R \ov\cL)$.
\end{remark}

\begin{remark}[The pro-MHS on $R_\infty\otimes_R H_*(U,\cL)$]\label{rem:endowedproMHShomology}
		Let $m'>m>0$. In that case, the dual  $p_{m',m}^\vee\colon R_{-m}\hookrightarrow R_{-m'}$ of the projection morphism $p_{m',m}$ from Remark~\ref{rem:endowedproMHS} is also a mixed Hodge structure morphism which is $R_\infty$-linear, and it induces an inclusion $p_{m',m}^\vee\otimes \Id\colon R_{-m}\otimes_R \oL\to R_{-m'}\otimes_R \oL$. Note that, by Remark~\ref{rem:homologyVsCohomology}, dualizing this inclusion (over $\R$) yields the projection
		$$
		p_{m',m}\otimes\Id\colon R_{m'}\otimes_R \cL\to R_{m}\otimes_R \cL.
		$$
		The morphism $p_{m',m}^\vee\otimes \Id$ extends via $e^{-(\Phi_\R^G)^\vee}$ and the morphisms in Definition~\ref{def:endowedMHS} to a morphism of mixed Hodge complexes of sheaves:
	$$
	p_{m',m}^\vee\otimes\Id\colon (R_{-m}\otimes\cN^\bullet_{X,D,n}, d+\ov f^*\circ\Phi^Y(\eps))\to (R_{-m'}\otimes\cN^\bullet_{X,D,n}, d+\ov f^*\circ\Phi^Y(\eps)).
	$$
	In particular, $p_{m',m}^\vee$ induces a MHS morphism
	$$
	H^*(U,R_{-m}\otimes_R \oL)\to H^*(U,R_{-m'}\otimes_R \oL).
	$$
	If $\Hom_\R(H^*(U,R_{-m}\otimes_R \ov\cL,\R)$ is endowed with the dual MHS, these morphisms endow their limit \linebreak $\varprojlim_{m}\Hom_\R(H^*(U,R_{-m}\otimes_R \ov\cL,\R)$ with a pro-MHS. By Corollary~\ref{cor:completionHomology2}, this endows $R_\infty\otimes_R H_*(U,\cL)$ with a pro-MHS. In fact, using the isomorphism $\Hom_\R(H^*(U,R_{-m}\otimes_R \oL),\R)\cong H_*(U,R_{m}\otimes_R \cL)$ from  Remark~\ref{rem:homologyVsCohomology}, we see that the dual of the MHS morphism induced by $p_{m',m}^\vee$ in cohomology is the morphism induced in homology by $p_{m',m}\otimes\Id:R_{m'}\otimes_R\cL\to R_{m}\otimes_R\cL$. With this interpretation, the pro-MHS on $R_\infty\otimes_R H_*(U,\cL)$ is given by the isomorphism $R_\infty\otimes_R H_*(U,\cL)\cong \varprojlim_{m} H_*(U,R_{m}\otimes_R\cL)$ and the morphisms induced in homology by the projections $p_{m',m}\otimes\Id:R_{m'}\otimes_R\cL\to R_{m}\otimes_R\cL$.
\end{remark}

\begin{remark}\label{rem:multiplicationCohomology} Let $m\in\Z\setminus\{0\}$. 
	The action of \(H_1(G,\R)\subset  R_\infty\) on $R_m$ induces a multiplication morphism
	\begin{equation}\label{eq:multiLocal} H_1(G,\R)\otimes_\R(R_m\otimes_R\ov\cL)\to  R_m\otimes_R\ov\cL.
	\end{equation}
	Since the morphism $e^{-(\Phi_\R^G)^\vee}:R_m\otimes_R\ov\cL\to \ker(d+f^*\circ\Phi_\R^G(\eps_\R))\subset R_m\otimes_\R\cA^0_{U,\R}$ from Construction~\ref{con:nu} is $R_\infty$-linear, the multiplication morphism from \eqref{eq:multiLocal} extends to a morphism of mixed Hodge complexes of sheaves
	$$
	H_1(G,\R)\otimes (R_m\otimes\cN_{X,D,n}^\bullet,d+\ov f^*\circ\Phi^Y(\eps))\to (R_m\otimes\cN_{X,D,n}^\bullet,d+\ov f^*\circ\Phi^Y(\eps))
	$$
	for any $n\geq\max\{2,\dim_\R U\}$ by Proposition~\ref{prop:complexMultiplication}. Therefore, the multiplication morphism \eqref{eq:multiLocal} induces a MHS morphism for every \(m\in \Z\setminus\{0\}\):
	\[
	H_1(G,\R)\otimes H^*(U,R_m\otimes_R \ov\cL)\to H^*(U,R_m\otimes_R \ov\cL).
	\]
	By Proposition~\ref{prop:limvsH} and Remark~\ref{rem:endowedproMHS}, taking the inverse limit for $m>0$, one obtains a pro-MHS morphism replacing \(m\) by \(\infty\).
\end{remark}

\begin{remark}\label{rem:multiplicationHomology}
	Let $m>0$. Since the $\R$-dual of multiplication by elements of $H_1(G,\R)$ in an $R_\infty$-module is multiplication  by elements of $H_1(G,\R)$, the fact that the multiplication map
	$$
	H_1(G,\R)\otimes_\R H^*(U,R_{-m}\otimes_R \ov\cL)\to H^*(U,R_{-m}\otimes_R \ov\cL)
	$$
	from Remark~\ref{rem:multiplicationCohomology} is a MHS morphism implies that the multiplication map
	$$
	H_1(G,\R)\otimes_\R\Hom_\R(H^*(U,R_{-m}\otimes_R \ov\cL),\R)\to \Hom_\R(H^*(U,R_{-m}\otimes_R \ov\cL),\R)
	$$
	is also a MHS morphism, where $\Hom_\R(H^*(U,R_{-m}\otimes_R \ov\cL),\R)$ is endowed with the dual MHS. This can be easily checked using the definition of the tensor and dual MHSs from Definition-Proposition~\ref{defprop:multilinear}. By Remark~\ref{rem:endowedproMHShomology}, taking the inverse limit for $m>0$ one obtains a pro-MHS morphism
	$$
		H_1(G,\R)\otimes_\R \left(R_\infty\otimes_R H_*(U,\cL)\right)\to R_\infty\otimes_R H_*(U,\cL).
	$$
\end{remark}

\begin{remark}[Multiplication by elements of $H_1(G,\R)$ if $G\cong(\C^*)^n$]\label{rem:multiplicationTorus}
	Suppose that $G\cong(\C^*)^n$, in which case $H_1(G,\R)$ is pure of type $(-1,-1)$. Let $a\in H_1(G,\R)$ be a non-zero element. Since the span of $a$ is a sub-MHS of $H_1(G,\R)$, Remark~\ref{rem:multiplicationCohomology} implies that multiplication by $a$ is a MHS morphism
	$$
	a:H^*(U,R_m\otimes_R\ov\cL)\to H^*(U,R_m\otimes_R\ov\cL)(-1)
	$$
	for every $m\in\Z\setminus\{0\}$, where $(-1)$ denotes the Tate twist. By Proposition~\ref{prop:limvsH} and Remark~\ref{rem:endowedproMHS}, taking the inverse limit for $m>0$, one obtains a pro-MHS morphism replacing \(m\) by \(\infty\). Similarly, Remarks~\ref{rem:endowedproMHShomology} and~\ref{rem:multiplicationHomology} imply that
	$$
	a:R_\infty\otimes_RH^*(U,\cL)\to R_\infty\otimes_RH^*(U,\cL)(-1)
	$$
	is a pro-MHS morphism.
\end{remark}

\subsection{Independence of the choices}\label{ss:independence}

Note that there are some choices involved in Definition~\ref{def:endowedMHS}, namely the choice of compactifications, the number $n\geq\max\{2,\dim_\R U\}$ and the choice of the maps $\Phi_{\R}^G$, $\Phi_{\C}^G$ and $\Psi^G$. Note that the choice of the maps $\Phi_{\R}^G$, $\Phi_{\C}^G$ and $\Psi^G$ was canonical, so we will not attempt to modify those. However, it is important that the MHS in Definition~\ref{def:endowedMHS} does not depend on compactifications or on $n$. This section shows this.

\begin{lemma}[Independence of $n$]\label{lem:indepn}
	Under the same notation as in Definition~\ref{def:endowedMHS}, the MHS on $H^*(U,R_m\otimes_R\ov\cL)$ endowed by the mixed Hodge complex of sheaves $(R_m\otimes\cN^\bullet_{X,D,n}, d+\ov f^*\circ\Phi^Y(\eps))$ does not depend on the choice of $n\geq\max\{2,\dim_\R U\}$.
\end{lemma}
\begin{proof}
	Let $n'\geq n \geq\max\{2,\dim_\R U\}$. The identity map induces a morphism of complexes of sheaves (as in \cite[Definition 3.16]{peters2008mixed})
	$$
	(R_m\otimes\cN^\bullet_{X,D, n'}, d+\ov f^*\circ\Phi^Y(\eps))\to (R_m\otimes\cN^\bullet_{X,D,n}, d+\ov f^*\circ\Phi^Y(\eps)).
	$$
	Since the identity is a quasi-isomorphism, this is what is called a weak equivalence, which induces isomorphisms of MHS in hypercohomology (see \cite[Lemma-Definition 3.19]{peters2008mixed}).
\end{proof}

\begin{lemma}[Independence of the compactification of $U$, fixing the compactification of $G$]\label{lem:indepCompactU}
	Let $Y$ be an allowed compactification of $G$. Let $X_1$ and $X_2$ be two good compactifications of $U$ such that $f:U\to G$ extends to algebraic maps $\ov f_1:X_1\to Y$ and $\ov f_2:X_2\to Y$. Let $D_i$ be the simple normal crossings divisor $X_i\setminus U$ for $i=1,2$. Then, the MHS on $H^*(U,R_m\otimes_R\ov\cL)$ induced by $(R_m\otimes\cN_{X_1,D_1,n}^\bullet, d+(\ov f_1)^*\circ \Phi^Y(\eps))$ coincides with the MHS induced by $(R_m\otimes\cN_{X_2,D_2,n}^\bullet, d+(\ov f_2)^*\circ \Phi^Y(\eps))$ for all $m\in \Z\setminus\{0\}$. 
\end{lemma}

\begin{proof}
	Let $Z$ be a good compactification of $U$ obtained as a resolution of singularities of the closure of the diagonal $U\to U\times U$ in $X_1\times X_2$. Then, there exist algebraic maps $\pi^i:Z\to X_i$ for $i=1,2$. Let $D=Z\setminus U$. It is enough to show that  the MHS on $H^*(U,R_m\otimes_R\ov\cL)$ induced by $(R_m\otimes\cN_{X_1,D_1,n}^\bullet, d+(\ov f_1)^*\circ \Phi^Y(\eps))$ coincides with the MHS induced by $(R_m\otimes\cN_{Z,D,n}^\bullet, d+(\ov f_1\circ\pi^1)^*\circ \Phi^Y(\eps))$.
	
	The proof follows the same steps as \cite[Theorem 5.21]{mhsalexander}, so we omit some details. 
	The pullback of forms through $\pi^1$ induces a morphism
	$$
	\cN_{X_1,D_1,n}^\bullet
	\to (\pi^1)_*\cN_{Z,D,n}^\bullet
	$$
	 which respects the filtrations. Although $(\pi^1)_*\cN_{Z,D,n}^\bullet$ is not a mixed Hodge complex of sheaves, the proof of Proposition~\ref{prop:qisoThickenings} implies that the morphism $\cN^\bullet_{X_1,D_1,n}\to (\pi^1)_*\cN_{Z,D,n}^\bullet$ extends to the thickenings by $(\ov f_1)^*\circ \Phi^Y(\eps)$ and $(\ov f_1\circ\pi^1)^*\circ\Phi^Y(\varepsilon)$, respecting the filtrations. Composing with $(\pi^1)_*$ of the canonical map from the mixed Hodge complex of sheaves $(R_m\otimes\cN_{Z,D,n}^\bullet,d+(\ov f_1\circ\pi^1)^*\circ \Phi(\eps))$ into its Godement resolution, we obtain a morphism of mixed Hodge complexes of sheaves
	$$
	(R_m\otimes\cN_{X_1,D_1,n}^\bullet, d+(\ov f_1)^*\circ \Phi(\eps))\to R(\pi^1)_*(R_m\otimes\cN_{Z,D,n}^\bullet, d+(\ov f_1\circ\pi^1)^*\circ \Phi(\eps)),
	$$
	where the latter is a mixed Hodge complex of sheaves (see Definition~\ref{def:directim}). If we restrict to $U$, the map between these mixed Hodge complexes of sheaves is just the map from the analytic forms on $U$ (real or complex) to its Godement resolution. Hence, this map induces the identity between the cohomology of $R_m\otimes_R\ov\cL$ itself, which concludes the proof. 
\end{proof}

\begin{lemma}[Independence of the compactification of $G$, fixing the compactification of $U$]\label{lem:indepCompactG}
	Let $Y_i$ be an allowed compactification of $G$ for $i=1,2$. Suppose that $X$ is a good compactification of $U$ such that $f:U\to G$ extends to algebraic maps $\ov f_1:X\to Y_1$ and $\ov f_2:X\to Y_2$. Let $D$ be the simple normal crossings divisor $X\setminus U$. Then, the MHS on $H^*(U,R_m\otimes_R\ov\cL)$ induced by $(R_m\otimes\cN_{X,D,n}^\bullet, d+(\ov f_1)^*\circ \Phi^{Y_1}(\eps))$ coincides with the MHS induced by $(R_m\otimes\cN_{X,D,n}^\bullet, d+(\ov f_2)^*\circ \Phi^{Y_2}(\eps))$. 
\end{lemma}

\begin{proof}
	First, we find an allowed compactification $Y$ of $G$ such that there exist algebraic maps $\pi^i:Y\to Y_1$, as in the first sentence of the proof of Lemma~\ref{lem:indepCompactU}. Now, take $Z$ to be a good compactification of $U$ obtained by doing a resolution of singularities of the closure of the graph of $f:U\to G$ inside of $X\times Y$. Let $\widehat D\coloneqq Z\setminus U$. We get an algebraic map $p:Z\to X$, and $f$ extends to $\ov f: Z\to Y$. Hence, it suffices to show that the MHS on $H^*(U,R_m\otimes_R\ov\cL)$ induced by $(R_m\otimes\cN_{X,D,n}^\bullet, d+(\ov f_1)^*\circ \Phi^{Y_1}(\eps))$ coincides with the MHS induced by $(R_m\otimes\cN_{Z,\widehat D,n}^\bullet, d+\ov f^*\circ \Phi^Y(\eps))$.
	
	Note that $(\pi^1\circ\ov f)^*\circ \Phi^{Y_1}_k=(\ov f)^*\circ \Phi^{Y}_k$ for $k=\R,\C$ (both are extensions of $f^*\circ\Phi^G_k$). By Lemma~\ref{lem:indepCompactU}, the mixed Hodge complex of sheaves $(R_m\otimes\cN_{Z,\widehat D,n}^\bullet, d+(\pi^1\circ\ov f)^*\circ \Phi^{Y_1}(\eps))$ induces the same MHS on $H^*(U,R_m\otimes_R\ov\cL)$ as $(R_m\otimes\cN_{X, D,n}^\bullet, d+(\ov f_1)^*\circ \Phi^{Y_1}(\eps))$.
\end{proof}

\begin{theorem}[Independence of the compactifications of $G$ and $U$]
	Let $Y_i$ be an allowed compactification of $G$, and let $X_i$ be a good compactification of $U$ such that $f$ extends to $\ov {f_i}:X_i\to Y_i$, for $i=1,2$. Let $D_i=X_i\setminus U$. Then, the MHS on $H^*(U,R_m\otimes_R\ov\cL)$ endowed by $(R_m\otimes\cN_{X, D,n}^\bullet, d+(\ov {f_i})^*\circ \Phi^{Y_i}(\eps))$ is the same for $i=1,2$.
\end{theorem}

\begin{proof}
	This follows from Lemmas~\ref{lem:indepCompactU} and \ref{lem:indepCompactG} by finding suitable compactifications lying above the ones given, using the same methods for doing so as in the proof of these two lemmas (resolution of singularities of the closure of the diagonal of $U$ or of $G$).
\end{proof}

\subsection{The MHS on quotients of Alexander modules}\label{sec:MHSAlexander}
In this section, we obtain other MHSs from the MHS on $H^*(U,R_m\otimes_R\ov\cL)$ given in  Definition~\ref{def:endowedMHS}.

\begin{corollary}
	Let \(\mathfrak a\) be the maximal ideal of \(R_\infty\). For every \(m\in \Z\setminus\{0\}\) and every \(m'\in \Z\ge 0\), \(\mathfrak a^{m'}H^*(U,R_m\otimes_R \ov\cL)\) is a sub-MHS of \(H^*(U,R_m\otimes_R \ov\cL)\), and similarly replacing \(m\) by \(\infty\) and ``MHS'' by ``pro-MHS''. Therefore, the quotients
	\[
	\displaystyle\frac{ H^*(U,R_m\otimes_R \ov\cL)}{ \mathfrak a^{m'}H^*(U,R_m\otimes_R \ov\cL)}
	\]
	are quotient MHSs as well.
\end{corollary}
\begin{proof}
	Note that \(\mathfrak a^{m'}H^*(U,R_m\otimes_R \ov\cL)\) is the image of the map in Remark~\ref{rem:multiplicationCohomology} composed with itself \(m'\) times:
	\[
	(H_1(G,\R))^{\otimes m'}\otimes H^*(U,R_m\otimes_R \ov\cL)\to
	(H_1(G,\R))^{\otimes (m'-1)}\otimes H^*(U,R_m\otimes_R \ov\cL)\to \cdots
	\]
	The proof for $m=\infty$ follows from  Remark~\ref{rem:multiplicationCohomology} and Remark~\ref{rem:proMHScategory}.
\end{proof}

\begin{corollary}\label{cor:quotientI}
	The MHS in Definition~\ref{def:endowedMHS} induces the following two sequences of MHS for \(m\in\Z_{\ge 1}\):
	\[
	R_m\otimes_R H^i(U,\oL);\quad R_m\otimes_R H_i(U,\cL).
	\]
	The latter MHS induces through the $R$-module isomorphism $H_i(U^f,\R)\cong H_i(U,\cL)$ from Remark~\ref{rem:coverVsL} a MHS on  $R_m\otimes_R H_i(U^f,\R)$.
	
	Moreover, the quotient maps induced by $R_{m'}\twoheadrightarrow R_m$ for all $m'\geq m$ between these are MHS morphisms. 
\end{corollary}
\begin{proof}
	By Remarks~\ref{rem:RmVSRmodm}, \ref{rem:RinftyVSInverseLimit} and~\ref{rem:homologyVsCohomology}, and  Corollaries~\ref{cor:completionHomology} and~\ref{cor:completionHomology2}, we have the following $R_\infty$-module isomorphisms:
	\begin{align*}
		R_\infty \otimes_R H^i(U,\oL) &\cong \varprojlim_m H^i(U,R_m\otimes_R \oL);\\
		R_\infty \otimes_R H_i(U,\cL) &\cong \varprojlim_m H_i(U,R_m\otimes_R \cL) \cong \varprojlim_m \Hom_\R(H^i(U,R_{-m}\otimes_R\oL),\R).
	\end{align*}
	These isomorphisms endow the right hand side spaces with pro-MHS by Remarks~\ref{rem:endowedproMHS} and~\ref{rem:endowedproMHShomology}.  Furthermore, by Remarks~\ref{rem:multiplicationCohomology} and~\ref{rem:multiplicationHomology}, the multiplication maps
	\begin{align*}
		H_1(G,\R)\otimes_\R (R_\infty\otimes_R H^i(U,\ov\cL))&\to R_\infty\otimes_R H^i(U,\ov\cL)\\
		H_1(G,\R)\otimes_\R(R_\infty\otimes_R H_i(U,\cL))&\to R_\infty\otimes_R H_i(U,\cL)
	\end{align*}
	are pro-MHS morphisms. By Remark~\ref{rem:proMHScategory}, the images and cokernels of the composition of pro-MHS morphisms are pro-MHSs as well. In particular, $R_m\otimes_R H^i(U,\oL)$ and $R_m\otimes_R H_i(U,\cL)$ are pro-MHSs, but they are also finite dimensional vector spaces, so they must be MHSs.
	
	For the ``moreover'' part of the statement, note that the quotient maps
	\begin{align*}
		R_{m'}\otimes_R H^i(U,\ov\cL)&\to R_{m}\otimes_R H^i(U,\ov\cL)\\
			R_{m'}\otimes_R H_i(U,\cL)&\to R_{m}\otimes_R H_i(U,\cL)
	\end{align*}
induced by $R_{m'}\twoheadrightarrow R_m$ for all $m'\geq m$ are induced in the quotients by the identity morphisms in $R_{\infty}\otimes_R H^i(U,\ov\cL)$ and $R_{\infty}\otimes_R H_i(U,\cL)$ respectively, so they are MHS morphisms.
\end{proof}

\begin{definition}[MHS on quotients of the Alexander modules by powers of the augmentation ideal]\label{def:MHSalexander}
	Let $m\geq 1$, and let $\mathfrak{m}$ be the augmentation ideal of $R=k[\pi_1(G)]$.
	\begin{itemize}
		\item The $R$-module isomorphism $R/\fm^m\cong R_m$ from Remark~\ref{rem:RmVSRmodm} induces  isomorphisms
		\begin{align*}
			\frac{H^i(U,\oL)}{\mathfrak{m}^mH^i(U,\oL)}&\cong R/\fm^m\otimes_R H^i(U,\oL)\cong R_m\otimes_R H^i(U,\oL),\\\frac{H_i(U,\cL)}{\mathfrak{m}^mH_i(U,\cL)}&\cong R/\fm^m\otimes_R H_i(U,\cL)\cong R_m\otimes_R H_i(U,\cL),\\
		\end{align*}
		The right hand sides of these isomorphisms are MHS by Corollary~\ref{cor:quotientI}, which we use to define MHS on $\frac{H^i(U,\oL)}{\mathfrak{m}^mH^i(U,\oL)}$ and $\frac{H_i(U,\cL)}{\mathfrak{m}^mH_i(U,\cL)}$ for all $i\geq 0$ and for all $m\geq 1$.
		\item The $R$-module isomorphism $H_i(U^f,\R)\cong H_i(U,\cL)$ from Remark~\ref{rem:coverVsL} (where $R$ acts on $H_i(U^f,\R)$ by deck transformations) endows $\frac{H_i(U^f,\R)}{\fm^m H_i(U^f,\R)}$ with a MHS. 
	\end{itemize}
\end{definition}

\begin{remark}\label{rem:projMHS}
	Since the isomorphisms $R/\fm^{m'}\cong R_{m'}$ and $R/\fm^m\cong R_m$ from Remark~\ref{rem:RmVSRmodm} form a commutative diagram with the projections $R/\fm^{m'}\twoheadrightarrow R/\fm^{m}$ and $R_{m'}\twoheadrightarrow R_m$ for all $m'\geq m\geq 1$, the projection morphisms
	$$
		\frac{H^i(U,\oL)}{\mathfrak{m}^{m'}H^i(U,\oL)}\to \frac{H^i(U,\oL)}{\mathfrak{m}^{m}H^i(U,\oL)},\quad \frac{H_i(U,\cL)}{\mathfrak{m}^{m'}H_i(U,\cL)}\to \frac{H_i(U,\cL)}{\mathfrak{m}^{m}H_i(U,\cL)}
	$$
are MHS morphisms.
\end{remark}

Recall that $\pi_1(G)$ acts on $H_i(U,\cL)\cong H_i(U^f,\R)$ by deck transformations. The following result states that the nilpotent logarithm of deck transformations respects the MHS on quotients of Alexander modules.

\begin{corollary}\label{cor:multiplication}
	For all $\gamma\in\pi_1(G)$, let $\log\gamma\in H_1(G,\Z)$ be the element corresponding to $\gamma$ via the abelianization map. Let $\mathfrak{m}$ be the augmentation ideal of $R\coloneqq \R[\pi_1(G)]$, and let $m\geq 1$.
	
	Then, the multiplication map defined as the only $\R$-linear map satisfying that
	$$
	\begin{array}{ccc}
		H_1(G,\R)\otimes_\R \frac{H^i(U,\oL)}{\mathfrak{m}^mH^i(U,\oL)} & \longrightarrow & \frac{H^i(U,\oL)}{\mathfrak{m}^mH^i(U,\oL)}\\
		\log \gamma \otimes v & \longmapsto & \log(\gamma)\cdot v
	\end{array}
	$$
	for all $\gamma\in \pi_1(G)$ and all $v\in \frac{H^i(U,\oL)}{\mathfrak{m}^mH^i(U,\oL)}$  is a MHS morphism, where $\log(\gamma)\cdot v$ denotes the multiplication by $\log(\gamma)\coloneqq\log(1+(\gamma-1))$, seen as a power series in $\gamma-1\in\mathfrak{m}$.
	
	Moreover, if $G\cong(\C^*)^n$ for some $n\geq 1$, then for all $\gamma\in\pi_1(G)$, multiplication by $\log(\gamma)$ is a MHS morphism from $\frac{H^i(U,\oL)}{\mathfrak{m}^mH^i(U,\oL)}$ to its $(-1)$-st Tate twist.
	
	Furthermore, the same results hold if we replace $H^i(U,\oL)$ by $H_i(U,\cL)$ or $H_i(U^f,\R)$ everywhere.
\end{corollary}

\begin{proof}
	Note that the multiplication morphisms
	$$
	H_1(G,\R)\otimes_\R (R_m\otimes_R H^i(U,\oL))\to H^i(U,\oL),\quad H_1(G,\R)\otimes_\R (R_m\otimes_R H_i(U,\cL))\to H_i(U,\cL)
	$$
	are MHS morphisms because they are induced by the multiplication morphisms on $R_\infty\otimes_R H^i(U,\oL)$ and $R_\infty\otimes_R H_i(U,\oL)$ respectively, which are pro-MHS morphism by the proof of Corollary~\ref{cor:quotientI}. Also note that the isomorphism $R/\fm^m\cong R_m$ from Remark~\ref{rem:RmVSRmodm} takes $\log(\gamma)\in R/\fm^m$ to $\log \gamma\in R_m$. The result now follows from Remarks~\ref{rem:multiplicationCohomology} and~\ref{rem:multiplicationTorus}, and from the way the MHS of Definition~\ref{def:MHSalexander} are constructed. Note that the dual MHS of a $j$-th Tate twist corresponds to the $(-j)$-th Tate twist of the dual MHS.
\end{proof}

\section{Functoriality}\label{sec:functoriality}

In this section we prove the following theorem, which is stated in terms of the homology of covers instead of the homology of local systems (recall Remark~\ref{rem:coverVsL}) due to the geometric meaning of the morphisms to which it applies.

\begin{theorem}[Functoriality]\label{thm:functorial}
	Let $U_1,U_2$ be smooth connected complex algebraic varieties, and let $G_1,G_2$ be semiabelian varieties. Consider a commutative diagram of algebraic morphisms
	\begin{equation}\label{eq:baseFunctoriality}
	\begin{tikzcd}[row sep=2em]
		U_1
		\arrow[d,"f_1"]
		\arrow[r,"g"]
		& U_2
		\arrow[d,"f_2"]\\
		G_1
		\arrow[r,"\rho"]
		&
		G_2
	\end{tikzcd}
	\end{equation}
	where $\rho$ is a group homomorphism. Let
	\begin{equation}\label{eq:topFunctoriality}
	\begin{tikzcd}[row sep=2em]
		U_1^{f_1}
		\arrow[d,"\widetilde{f_1}"]
		\arrow[r,"\wt g"]
		& U_2^{f_2}
		\arrow[d,"\wt f_2"]\\
		TG_1
		\arrow[r,"\wt \rho"]
		&
		TG_2
	\end{tikzcd}
	\end{equation}
	be a commutative diagram which is the unique lift of \eqref{eq:baseFunctoriality}  satisfying that $\wt\rho$ is an additive group homomorphism, and such that $\wt f_1$ and $\wt f_2$ are defined from the pullback diagrams as in~\eqref{eq:Uf}.
	
	For \(i=1,2\), let \(R^i= \R[\pi_1(G_i)]\) and let $\fm_i$ be the augmentation ideal of $R^i$.  For \(m\in \Z_{\ge 1}\), let
	$
	R^i_m = \frac{\prod_{j=0}^\infty\Sym^j H_1(G_i,\R)}{\prod_{j=m}^\infty\Sym^j H_1(G_i,\R)}
	$. Then, the following statements hold for the morphisms induced in homology by $\wt g:U_1^{f_1}\to U_1^{f_2}$:
	
	\begin{enumerate}
		\item 
		$$
		\wt g_{*,m}\colon R_m^1\otimes_{R^1} H_j(U_1^{f_1},\R)\to R_m^2\otimes_{R^2} H_j(U_2^{f_2},\R)
		$$
		is a MHS morphism for all $j\geq 0$ and for all $m\geq 1$, where the domain and the target have the MHS from Corollary~\ref{cor:quotientI}.
		\item Equivalently,
		$$
		\wt g_{*,m}\colon\frac{H_j(U_1^{f_1},\R)}{\fm_1^m H_j(U_1^{f_1},\R)}\to \frac{H_j(U_2^{f_2},\R)}{\fm_2^m H_j(U_2^{f_2},\R)}
		$$
		is a MHS morphism for all $j\geq 0$ and for all $m\geq 1$, where the domain and the target have the MHS from Definition~\ref{def:MHSalexander}.
	\end{enumerate} 

	
\end{theorem}

Before we prove Theorem~\ref{thm:functorial}, let us interpret its statement in more detail: the commutative diagram \eqref{eq:baseFunctoriality} induces a commutative cube
\begin{equation}\label{eq:cubefunctoriality}
	\begin{tikzcd}[row sep=2.5em]
		U_1^{f_1} \arrow[rr,"\wt g"] \arrow[dr,swap,"\widetilde{f_1}"] \arrow[dd,swap,"\pi^1"] &&
		U_2^{f_2} \arrow[dd,swap,"\pi^2" near start] \arrow[dr,"\widetilde{f_2}"] \\
		& T G_1 \arrow[rr,crossing over,"\widetilde{\rho}" near start] &&
		T G_2 \arrow[dd,"\exp"] \\
		U_1 \arrow[rr,"g" near end] \arrow[dr,swap,"f_1"] && U \arrow[dr,swap,"f_2"] \\
		& G_1 \arrow[rr,"\rho"] \arrow[uu,<-,crossing over,"\exp" near end]&& G_2
	\end{tikzcd}
\end{equation}
as follows: the left and right sides of the cube are pullback diagrams, $\wt\rho$ is the unique lift of $\rho$ to the universal covers that is a group homomorphism, and $\wt g$ is determined uniquely by $g$ and $\wt\rho$. The top of this cube is the commutative diagram~\eqref{eq:topFunctoriality}.

Also note that the morphism
$$
\wt g_*: H_j(U_1^{f_1},\R)\to H_j(U_2^{f_2},\R)
$$
induced in homology by $\wt g$ for all $j\geq 0$ satisfies that $\wt g_*(\gamma\cdot -)=\rho_*(\gamma)\cdot\wt g_*(-)$ for all $\gamma\in\pi_1(G_1)$, which justifies that the maps $\wt{g}_{*,m}$ are well defined for all $j\geq 0$ and all $m\geq 1$.

\subsection{Proof of Theorem~\ref{thm:functorial}}

\begin{remark}\label{rem:outlineFunctoriality}
	In the setting of Theorem~\ref{thm:functorial}, the commutative diagram \eqref{eq:baseFunctoriality} factors as
		\begin{equation}\label{eq:baseFunctoriality2}
		\begin{tikzcd}[row sep=2em]
			U_1
			\arrow[d,"f_1"]
			\arrow[r,"\Id"]
			&
				U_1
			\arrow[d,"\rho\circ f_1"]\arrow[r, "g"]
			& 	U_2
			\arrow[d,"f_2"]\\
			G_1
			\arrow[r,"\rho"]
			&
			G_2\arrow[r, "\Id"]
			&
			G_2,
		\end{tikzcd}
	\end{equation}
 so the map \(U_1^{f_1}\to U_2^{f_2}\) factors through \(U_1^{\rho\circ f_1}\). Therefore, it is enough to consider the cases where \(\rho=\Id\) and \(g = \Id\), which we will do in Theorems~\ref{thm:funcEasy} and~\ref{thm:functGroupHom} respectively.
\end{remark}

Let $\cL_1\coloneqq f_1^{-1}\exp_!\ul{\R}_{TG_1}$, and let $\cL_2\coloneqq f_2^{-1}\exp_!\ul{\R}_{TG_2}$. Let $\ov{\cL_1}$ (resp. $\ov{\cL_2}$) be $\cL_1$ (resp. $\cL_2$) with the conjugate $R^1$ (resp. $R_2$)-module structure, as in Definition~\ref{def:L}. Before we prove Theorem~\ref{thm:functorial}, we need to recall how the MHS on $R_m^i\otimes_{R^i} H_j(U_i^{f_i},\R)$ was defined (Definition~\ref{def:MHSalexander}) using mixed Hodge complexes of sheaves, since the proof will need to realize the morphism $\wt g_*$ as a morphism at the level of the corresponding complexes of sheaves.  The MHS on $R_m^i\otimes_{R^i} H_j(U_i^{f_i},\R)$ is induced from the MHS on $H^j(U_i,R_{-m'}^i\otimes_{R^i}\ov{\cL_i})$ for all $m'\geq 1$ (from Definition~\ref{def:endowedMHS}) as follows:
\begin{enumerate}
	\item\label{step1} The isomorphism $H_j(U_i, R_{m'}^i\otimes_{R^i} \cL_i)\cong \Hom_\R(H^j(U_i, R_{-m'}^i\otimes_{R^i} \ov{\cL_i}),\R) $ from Remark~\ref{rem:homologyVsCohomology} endows $H_j(U_i, R_{m'}^i\otimes_{R^i} \cL_i)$ with the dual MHS of $H^j(U_i, R_{-m'}^i\otimes_{R^i} \ov{\cL_i})$ for all $m'\geq 1$.
	\item\label{step2} The isomorphism $R_\infty^i\otimes_{R^i}H_j(U_i,\cL_i)\cong \varprojlim_{m'} H_j(U_i,R_{m'}^i\otimes_{R^i}\cL_i)$ from Corollary~\ref{cor:completionHomology2} endows $R_\infty^i\otimes_{R^i}H_j(U_i,\cL_i)$ with a pro-MHS, where the morphisms in the inverse limits are the ones induced by the projections $R_{m''}^i\twoheadrightarrow R_{m'}^i$ for all $m''>m'$  (Remark~\ref{rem:endowedproMHShomology}).
	\item\label{step3} $R_m^i\otimes_{R^i}H_j(U_i,\cL_i)$ is endowed with a MHS in Corollary~\ref{cor:quotientI} as the cokernels of the multiplication map
	$$
	\underbrace{H_1(G_1,\R)\otimes_\R\ldots\otimes_\R H_1(G_1,\R)}_m\otimes_\R \left(R_\infty^{i}\otimes_{R^i}H_j(U_i,\cL_i)\right)\to \left(R_\infty^{i}\otimes_{R^i}H_j(U_i,\cL_i)\right),
	$$
	which is a morphism of pro-MHS (Remark~\ref{rem:multiplicationHomology}).
	\item\label{step4} $R_m^i\otimes_{R^i}H_j(U_i^{f_i},\R)$ is endowed with the MHS from $R_m^i\otimes_{R^i}H_j(U_i,\cL_i)$ through the natural isomorphism $H_j(U_i^{f_i},\R)\cong H_j(U_i,\cL_i)$ from Remark~\ref{rem:coverVsL}, which comes from an isomorphism at the level of chain complexes (Corollary~\ref{cor:quotientI}).
\end{enumerate}

The following two lemmas address the question of how to realize the morphism $\wt g_*$ as a morphism between complexes of sheaves.

\begin{lemma}\label{lem:homologySheafMapEasy}
	Suppose that $\rho=\Id$ in the setting from Theorem~\ref{thm:functorial}. Let us denote $G\coloneqq G_1=G_2$, $R=\R[\pi_1(G)]$ and $R_m\coloneqq R_m^1=R_m^2$ for all $m\geq 1$. Then, $\ov{\cL_1}=g^{-1}\ov{\cL_2}$, and the map $\wt g_*: H_j(U_1^{f_1},\R)\to H_j(U_2^{f_2},\R)$ is induced through steps \eqref{step1}--\eqref{step4} above for all $j\geq 0$ by the adjunction $\Id\to Rg_*g^{-1}$ applied to the sheaves $R_{-m'}\otimes\ov{\cL_2}$ for all $m'>1$.
\end{lemma}

\begin{proof}
Following \cite[p.60]{maxbook}, let $$S_j(U_i,\cL_i)\coloneqq\left\{\sum_{\sigma}l_\sigma \sigma \text{ (finite linear combination)}\left| \begin{array}{c}\sigma:\Delta^j\to U_i \text{ is a singular }j\text{-simplex, and}\\ l_\sigma\in\Gamma(\Delta^{j},\sigma^{-1}\cL_i)\end{array}\right.\right\}.$$
Through the usual differential of singular homology and restrictions of $l_\sigma$ to the faces of $\sigma$, we obtain $S_\bullet(U_i,\cL_i)$, the singular chain complex that computes $H_j(U_i,\cL_i)$ for all $j$. The morphism in homology induced by $\wt g_*$ through the isomorphism in Step~\eqref{step4} above comes from the following map of chain complexes.
$$
\f{\widehat{g}}{S_j(U_1,\cL_1)}{S_j(U_2,\cL_2)}{\sum_\sigma l_\sigma \sigma}{\sum_\sigma l_\sigma g\circ\sigma} 
$$
Since $\cL_1=g^{-1}\cL_2$, we have that $\Gamma(\Delta^{j},\sigma^{-1}\cL_1)=\Gamma(\Delta^{j},(g\circ\sigma)^{-1}\cL_2)$, so this definition makes sense. A similar definition can be given for a map between the chain complexes corresponding to truncated local systems, namely
$$
\begin{array}{ccc}
	\f{\widehat{g}_{m'}}{	S_j(U_1,R_{m'}\otimes_R\cL_1)}{S_j(U_2,R_{m'}\otimes_R\cL_2)}{\sum_\sigma (a\otimes l_\sigma) \sigma}{\sum_\sigma (a\otimes l_\sigma) g\circ\sigma} 
\end{array}
$$
where $a\in R_{m'}$ and $l_\sigma\in\Gamma(\Delta^j,\sigma^{-1}\cL_1)$. Note that $\Gamma(\Delta^j,\sigma^{-1}R_{m'}\otimes_R\cL_1)=R_{m'}\otimes_R \Gamma(\Delta^j,\sigma^{-1}\cL_1)$, so this is well defined. Analogously, we may define $\widehat{g}_\infty$. If we pass to the inverse limit, these morphisms $S_\bullet (U_1,R_{m'}\otimes_R\cL_1)\to S_\bullet (U_2,R_{m'}\otimes_R\cL_2)$ for all $m'$ induce through the isomorphism in Step~\eqref{step2} the same morphism in homology as $\widehat{g}_\infty$ (note that $R_\infty$ is a flat $R$-module). The multiplication map in Step~\eqref{step3} above can be lifted to a morphism of chain complexes
$$
\underbrace{H_1(G_1,\R)\otimes_\R\ldots\otimes_\R H_1(G_1,\R)}_{m}\otimes_\R \left(S_\bullet(U_i,R_\infty\otimes_R\cL_i)\right)\to S_\bullet(U_i,R_\infty\otimes_R\cL_i),
$$
inducing a multiplication map
$$
\underbrace{H_1(G_1,\R)\otimes_\R\ldots\otimes_\R H_1(G_1,\R)}_{m}\otimes_\R H_j(U_i,R_\infty\otimes_R\cL_i)\to H_j(U_i,R_\infty\otimes_R\cL_i)
$$
for $i=1,2$, and for all $j\geq 0$. For a fixed $j$, the morphism that $\widehat g_\infty$ induces in the cokernel of these multiplication maps coincides with $R_m\otimes H_j(\wt g)$. All that is left to see is that the morphisms induced by $\widehat g_{m'}$ in homology agree with the dual of the morphisms induced in cohomology by adjunction $\Id\to Rg_*g^{-1}$ applied to the sheaf $R_{-m}\otimes_R\ov\cL_2$ through the isomorphism in Step~\eqref{step1} above.

The isomorphism in Step~\eqref{step1} can be realized at the level of chains as follows: Let $S^\bullet (U_i,R_{-m'}\otimes_R\ov{\cL_i})$ the complex obtained by taking the $\R$-dual of $S_\bullet (U_i,R_{m'}\otimes_R \cL_i)$. The dual of $\widehat g_{m'}$ is
$$
\f{(\widehat g_{m'})^\vee}{S^j (U_2,R_{-m'}\otimes_R\ov{\cL_2})}{S^j (U_1,R_{-m'}\otimes_R\ov{\cL_1})}{H}{H\circ \widehat g_{m'}.}
$$

Note that $S^j(-,R_{-m'}\otimes_R\ov{\cL_i})$ is a presheaf. Let $\wt S^j(-,R_{-m'}\otimes_R\ov{\cL_i})$ be its sheafification. We will use facts stated in \cite[p.360, section F]{spanier}. $ S^\bullet(-,R_{-m'}\otimes_R\ov{\cL_i})$ is a complex of fine presheaves which is a resolution (in the category of presheaves on $U_i$) of $R_{-m'}\otimes_R\ov{\cL_i}$, where the resolution map $R_{-m}\otimes_R\ov{\cL_i}\to S^0(-,R_{-m'}\otimes_R\ov{\cL_i})$ is given locally on $V_i\subset U_i$ by
$$b\mapsto \left(\begin{array}{ccc}
	S_0(V_i,R_m\otimes_R\cL_i)&\longrightarrow& \R\\
	\sum_{x\in V_i}l_{\sigma_x} \sigma_x&\longmapsto& \sum_\sigma b_x(l_{\sigma_x}),
	\end{array}\right)$$
where $\sigma_x$ is the map from $\Delta^0$ to $x\in V_i$,  $l_{\sigma_x}$ is in the stalk of $R_m\otimes_R\cL_i$ at the point $x$, and we are using that $R_m\otimes_R\cL_i$ and $R_{-m}\otimes_R\ov{\cL_i}$ are $\R$-dual local systems. Hence,
 $\wt S^\bullet(-,R_{-m'}\otimes_R\ov{\cL_i})$ is a resolution of $R_{-m'}\otimes_R\ov{\cL_i}$ (in the category of sheaves on $U_i$) of fine sheaves. In particular $\wt S^\bullet(-,R_{-m'}\otimes_R\ov{\cL_i})$ is a complex of acyclic sheaves with respect to pushforwards, so it can be used to compute $H^j(U_i,R_{-m'}\otimes_R\ov{\cL_i})$. Moreover, the sheafification morphism induces an isomorphism
$$
H^j(S^j(U_i,R_{-m'}\otimes_R\ov{\cL_i}))\xrightarrow{\cong} H^j(U_i,R_{-m'}\otimes_R\ov{\cL_i})
$$
for all $j$. It suffices to show that the morphism induced by $(\widehat g_{m'})^\vee$ in cohomology coincides through this isomorphism with the map $R_{-m'}\otimes_R\ov{\cL_2}\to Rg_*(R_{-m'}\otimes_R\ov{\cL_1})$ induced by the adjunction $\Id\to Rg_*g^{-1}$ in sheaf cohomology. Consider the morphism
$$
\f{a_{m'}}{R_{-m}\otimes_R\ov{\cL_2}}{g_*(R_{-m}\otimes_R\ov{\cL_1})}{\alpha\otimes\iota}{\alpha\otimes\iota\circ g},
$$
where $\iota:V\to TG$ satisfies that $\exp\circ\iota=f_2$. Note that $(\widehat g_{m'})^\vee$ can be easily extended to a morphism of complexes of pre-sheaves $(\widehat g_{m'})^\vee:S^\bullet(-,R_{-m}\otimes_R\ov{\cL_2})\to g_*S^\bullet(-,R_{-m}\otimes_R\ov{\cL_2})$. The result follows from the commutativity of this diagram and the fact that the complexes of (pre)sheaves that appear are fine.
$$
\begin{tikzcd}
	  \ & R_{-m'}\otimes_R\ov\cL_2\arrow[d,"a_{m'}"]\arrow[ldd,swap, "\text{resolution}"]\arrow[ddd,bend left=100,"\text{adjunction} \Id\to Rg_*g^{-1}"]\\
	  & g_*\left(R_{-m'}\otimes_R\ov\cL_1\right)\arrow[d, "g_* \text{ of resolution}"]\\
	S^\bullet(-,R_{-m'}\otimes_R\ov{\cL_2})\arrow[r, "(\widehat g_{m'})^\vee"]\arrow[d, "\text{sheafification}"]&g_*S^\bullet(-,R_{-m'}\otimes_R\ov{\cL_1})\arrow[d, "g_*\text{ of sheafification}"]\\
	\wt S^\bullet(-,R_{-m'}\otimes_R\ov{\cL_2})&g_*\wt S^\bullet(-,R_{-m'}\otimes_R\ov{\cL_1})& \  \\
\end{tikzcd}
$$
\end{proof}

	\begin{lemma}\label{lem:homologySheafMap}
	Suppose that $g=\Id$ in the setting of Theorem~\ref{thm:functorial}, and let $U\coloneqq U_1=U_2$. Let \(\rho_*\) denote the induced map in homology, which generates a map \(R^1_{m'}\to R^2_{m'}\) for all \(m'\ge 1\). After applying the isomorphisms in Steps~\eqref{step1}--\eqref{step4} above, the map $\wt g_{*,m}$ in Theorem~\ref{thm:functorial} is induced by the following maps of sheaves for all $m'\geq 1$:
	\[
	\begin{array}{rcl}
		R^2_{-m'}\otimes_{R^2} \ov{\cL_2} &\longrightarrow & R^1_{-m'}\otimes_{R^1} \ov{\cL_1} \\
		\phi\otimes \ov{\wt\rho\circ \iota_0} & \mapsto & \phi \circ \rho_* \otimes \ov{\iota_0}.
	\end{array}
	\]
	Here \(\phi\) is any element of \(R^2_{-m'} =\Hom_{\R}(R_{m'}^2,\R) \) and \(\iota_0\) is a local generator of \(\cL_1\), i.e. a local lift \(U\to T G_1\) satisfying that $\exp\circ\iota_0=f_1$.	
\end{lemma}

\begin{proof}
	Seeing $U^{f_i}$ as subsets of $U\times TG_i$, we have that $\wt g(u,z)=(u,\wt\rho(z))$, so the map $\wt g_*:H_j(U^{f_1},\R)\to H_j(U^{f_2},\R)$ coincides through the identifications in Step~\eqref{step4} with the map that the morphism
	\begin{equation}\label{eq:morphism}
	\begin{array}{ccc}
		 \cL_1& \longrightarrow &  \cL_2\\
		\iota &\longmapsto&\rho\circ\iota.
	\end{array}
	\end{equation}
	induces in homology, where the morphism above is described in terms of local sections (local lifts of $f_i$ to $U\to TG_i$). Notice that the morphism \eqref{eq:morphism} descends to the truncated local systems as
	\begin{equation}\label{eq:morphismTruncated}
		\begin{array}{ccc}
			R_{m'}^1\otimes_{R^1} \cL_1& \longrightarrow & R_{m'}^2\otimes_{R^2} \cL_2\\
			a\otimes\iota_0 &\longmapsto&\rho_*(a)\otimes \wt\rho\circ\iota_0
		\end{array}
	\end{equation}
  for all $m'\geq 0$. By Step \eqref{step2}, the inverse limit of the maps in homology induced by these morphisms coincides with the map in homology induced by
	\begin{equation}\label{eq:morphismInfinity}
		\begin{array}{ccc}
			R_\infty^1\otimes_{R^1} \cL_1& \longrightarrow & R_\infty^2\otimes_{R^2} \cL_2\\
			a\otimes\iota_0 &\longmapsto&\rho_*(a)\otimes \wt\rho\circ\iota_0,
		\end{array}
	\end{equation}
which makes the following diagram commute, where the horizontal arrows are multiplication.
$$
\begin{tikzcd}
	H_1(G_1,\R)\otimes_\R \left(R_\infty\otimes H_j(U,\cL_1)\right)\arrow[r]\arrow[d, "\rho_*\otimes\left(\rho_*\otimes H_j(\eqref{eq:morphism})\right)"] & R_\infty\otimes H_j(U,\cL_1)\arrow[d, "\rho_*\otimes H_j(\eqref{eq:morphism})=H_j(\eqref{eq:morphismInfinity})=\varprojlim_{m'} H_j(\eqref{eq:morphismTruncated})"]\\
	H_1(G_2,\R)\otimes_\R \left(R_\infty\otimes H_j(U,\cL_2)\right)\arrow[r] & R_\infty\otimes H_j(U,\cL_2).\\
\end{tikzcd}
$$
The equality $\rho_*\otimes H_j(\eqref{eq:morphism})=H_j(\eqref{eq:morphismInfinity})$ follows because $R_\infty^i$ is a flat $R^i$-module. In light of Step~\eqref{step3}, the commutativity of the diagram above tells us that we just need to identify what the morphism $R_{-m'}^2\otimes_{R^2}\ov{\cL_2}\to R_{-m'}^1\otimes_{R^1}\ov{\cL_1}$ corresponding to \eqref{eq:morphismTruncated} is through the chain of isomorphisms at the level of sheaves from Remark~\ref{rem:groupStructure} (see Step~\eqref{step1}), and check that it agrees with the one described in the statement of this Lemma.

Let $m'\geq 1$, let \(\phi\in R_{-m'}^2 = \Hom_\R(R_{m'}^2,\R)\), and let \(\iota_0\) be a local section of \(\cL_1\). Note that \(\wt\rho\circ \iota_0\) generates \(\cL_2\) locally over $R^2$, so any element of \(R^2_{-m'} \otimes_{R^2} \oL_2\) can be written as \(\phi\otimes \ov{\wt\rho\circ \iota_0}\) for some \(\phi\in R^2_{-m'}\). For two lifts $\iota,\iota'\colon U\to G_i$, let us use $\langle \ov{\iota'},\iota\rangle\in \pi_1(G_i)\subset R^i$ the pairing between $\ov{\cL_i}$ and $\cL_i$ from Remark~\ref{rem:ConjAndDual}, which is defined by $\langle \ov\iota,\iota\rangle = 1$. We have the following chain of isomorphisms from Remark~\ref{rem:homologyVsCohomology}: 
	\begin{multline*}
  \begin{array}{ccccccc}
		R_{-m'}^2\otimes_{R^2} \ov\cL_2 
		& \cong &
		R_{-m'}^2\otimes \Homm_{R^2}(\cL_2,R^2)
		&\cong &
		\Homm_{R^2}(\cL_2,\Homm_\R(R_{m'}^2,\R))
		\\[0.5em]
		\phi\otimes \ov{\wt\rho\circ\iota_0}
		& \leftrightarrow &
		\phi \otimes (\iota \mapsto \langle \ov{\wt\rho\circ\iota_0},\iota\rangle)
		&\leftrightarrow &
		\iota\mapsto (b\mapsto \phi( \langle \ov{\wt\rho\circ\iota_0},\iota\rangle\cdot b))
	\end{array}
  \\
\begin{array}{ccccccc}
		\cong &
		\Homm_\R(R_{m'}^2\otimes_{R^2} \cL_2,\R)
		\\[0.5em]
\leftrightarrow &
		(b\otimes \iota) \mapsto \phi(  \langle \ov{\wt\rho\circ\iota_0},\iota\rangle\cdot b).
	\end{array}
	\end{multline*}
	Now, we have the following morphism
	\begin{equation}\label{eq:morphismDuals}
	\begin{array}{ccc}
		\Homm_\R(R_{m'}^2\otimes_{R^2} \cL_2, \R)& \longrightarrow & \Homm_\R(R_{m'}^1\otimes_{R^1} \cL_1, \R)\\
		H &\longmapsto&\left(a\otimes\iota_0\to H(\rho_*(a)\otimes \wt \rho\circ\iota_0)\right),
	\end{array}
	\end{equation}
which is the $\R$-dual of \eqref{eq:morphismTruncated}. Hence,
	 the composition of the chain of isomorphisms above with \eqref{eq:morphismDuals} takes $\phi\otimes\ov{\wt\rho\circ\iota_0}$ to $\left(a\otimes\iota_0\mapsto \phi(\rho_*(a))\right)$.
 Going backwards through the chain of isomorphisms of Remark~\ref{rem:homologyVsCohomology}, we have:
	\begin{multline*}
	\begin{array}{ccccccc}
		\Homm_\R(R_{m'}^1\otimes_{R^1} \cL_1,\R)
		& 
    \cong &
		\Homm_{R^1}(\cL_1,\Homm_\R(R_{m'}^1,\R))
		&
    \cong & R_{-m'}^1\otimes \Homm_{R^1}(\cL_1,R^1) 
		\\[0.5em]
		(a\otimes \iota_0)\mapsto \phi( \rho_*(a))
		& \leftrightarrow &
		\iota_0 \mapsto \phi\circ\rho_*& \leftrightarrow & \phi\circ\rho_*\otimes\left(\iota_0\mapsto\langle \ov \iota_0,\iota_0\rangle\right)
	\end{array}
  \\
  \begin{array}{ccccccc}
     \cong & R_{-m'}^1\otimes \ov\cL_1 
		\\[0.5em]
    \leftrightarrow & \phi\circ\rho_*\otimes \ov{\iota_0}
	\end{array}
	\end{multline*}
	In conclusion, the map of sheaves is the one we claimed.
\end{proof}

\begin{remark}\label{rem:translation}
	Suppose that $\rho$ is an algebraic morphism, but not a group homomorphism. In that case, one cannot pick a canonical lift $\wt\rho$ to the universal covers, so there are many choices of $\wt\rho$ (each of which determines a choice of $\wt g$) that make the cube~\eqref{eq:cubefunctoriality} commutative. However, notice that the hypothesis that $\rho$ is a group homomorphism was not used in the proof of Lemma~\ref{lem:homologySheafMap}. Hence, if $g=\Id$, the map 
	$$
	\wt g_{*,m}:R^1_m\otimes_{R^1}H_j(U^{f_1},\R)\to R^2_m\otimes_{R^2}H_j(U^{f_2},\R)
	$$
	induced by the choice of $\wt g$ determined by the choice of $\wt\rho$ is induced by the morphisms $R^2_{-m'}\otimes_{R^2} \ov{\cL_2}\to R^1_{-m'}\otimes_{R^1} \ov{\cL_1}$ through Steps~\eqref{step1}--\eqref{step4}. Note that if $\rho$ is a homeomorphism (for example, a translation in a semiabelian variety $G$) and $g=\Id$, then $\wt g_{*,m}$ is an isomorphism.
\end{remark}

%

\begin{theorem}[Functoriality, $\rho=\Id$]\label{thm:funcEasy}
	Theorem~\ref{thm:functorial} holds if $\rho=\Id$.
\end{theorem}

\begin{proof}
	Let us denote $G\coloneqq G_1=G_2$, $R=\R[\pi_1(G)]$ and $R_m\coloneqq R_m^1=R_m^2$ for all $m\geq 1$. By Lemma~\ref{lem:homologySheafMapEasy} it suffices to show that the adjunction morphism $\Id\to Rg_*g^{-1}$ applied to the sheaves $R_{-m'}\otimes\ov{\cL_2}$ induces MHS isomorphisms in cohomology for all $m'\geq 1$.
	
	 Let $Y_1, Y_2$ be compactifications of $G$ and $X_1, X_2$ be compactifications of $U_1$ and $U_2$ such that
	$$
	\begin{tikzcd}
		X_1\arrow[r, "\ov{g}"]\arrow[d,"\ov{f_1}"]& X_2\arrow[d,"\ov{f_2}"]\\
		Y_1\arrow[r, "\ov{\Id}"]& Y_2
	\end{tikzcd}
	$$ forms a compatible compactification with respect to the commutative diagram \eqref{eq:baseFunctoriality}, where $\rho=\Id$. Let $j_i:U_i\to X_i$ be the inclusion for $i=1,2$. We have that $R (j_2)_*$ of the adjunction morphism yields $R (j_2)_*(R_{-m'}\otimes_R\ov{\cL_2})\to R(\ov g)^*R (j_1)_*(R_{-m'}\otimes_R\ov{\cL_1})$. At the level of the thickened logarithmic Dolbeault complexes which are quasi-isomorphic to $R (j_i)_*(R_{-m'}\otimes_R\ov{\cL_i})$ for $i=1,2$, this is the composition of
	$$\left(R_{-m'}\otimes_\R\cA^\bullet_{X_2,\R}(\log D_2), d+(\ov{f_2})^*\circ\Phi_\R^{Y_2}(\varepsilon_\R)\right)\to (\ov{g})_*\left(R_{-m'}\otimes_\R\cA^\bullet_{X_1,\R}(\log D_1), d+(\ov{f_1})^*\circ\Phi_\R^{Y_1}(\varepsilon_\R)\right),
	$$
	given by the pullback of forms through $\ov g$ (which is a morphism of complexes of sheaves by the proof of Proposition~\ref{prop:qisoThickenings} and the fact that $\Phi_\R^{Y_1}=\ov\Id^*\circ \Phi_\R^{Y_2}$) with $(\ov g)_*$ of the inclusion of $$\left(R_{-m'}\otimes_\R\cA^\bullet_{X_1,\R}(\log D_1), d+(\ov{f_1})^*\circ\Phi_\R^{Y_1}(\varepsilon_\R)\right)$$ into its Godement resolution. Since $\ov g$ is algebraic, the first of these morphisms respects the weight filtrations $\wt W_{\lc}$. Hence, picking $n\geq\max\{2,\dim_\R U_1, \dim_\R U_2\}$, both of these morphisms will respect the weight filtrations $W_{\lc}^n$, which are biregular. Recall the definition of the derived direct image of a mixed Hodge complex of sheaves (Definition~\ref{def:directim}). Using Proposition~\ref{prop:qisoThickenings}, we see that composition above extends to a morphism  of mixed Hodge complexes of sheaves
	$$
	\left(R_{-m'}\otimes\cN^\bullet_{X_2,D_2,n}, d+\ov{f_2}^*\circ\Phi^{Y_2}(\varepsilon)\right)\to R(\ov{g})_* \left(R_{-m'}\otimes\cN^\bullet_{X_1,D_1,n}, d+\ov{f_1}^*\circ\Phi^{Y_1}(\varepsilon)\right),
	$$
	where the morphism between the complex part is also given by the pullback by $\ov{g}$. Indeed, pullback by $\ov{g}$ respects both the weight and Hodge filtrations there, and it is straightforward to check that $R(\ov g)_*\left(e^{\ov{f_1}^*\circ\Psi^{Y_1}(\varepsilon_\C)}\right)$ composed with the real part of the morphism (tensored by $\otimes_\R\C$) coincides with the composition of the complex part of this morphism and $e^{\ov{f_2}^*\circ\Psi^{Y_2}(\varepsilon_\C)}$. This proves that the morphism induced by adjunction $\Id\to Rg_*g^{-1}$ applied to $R_{-m}\otimes_R\ov{\cL^2}$ yields MHS morphisms in cohomology
	$$
	H^j(U_2,R_{-m'}\otimes_R\ov{\cL^2})\to H^j(U_1, R_{-m'}\otimes_R\ov{\cL^1})
	$$
	for all $m'\geq 1$ and for all $j\geq 1$.

\end{proof}

\begin{theorem}[Functoriality, $g=\Id$]\label{thm:functGroupHom}
	Theorem~\ref{thm:functorial} holds if $g=\Id$.
\end{theorem}
\begin{proof}
	Let us denote $U\coloneqq U_1=U_2$. By Lemma~\ref{lem:homologySheafMap} it suffices to show that the maps of sheaves $R^2_{-m'}\otimes_{R^2}\ov{\cL_2}\to R^1_{-m'}\otimes_{R^1}\ov{\cL_1}$ given by $\phi\otimes\ov{\wt\rho\circ\iota_0}\mapsto \phi\circ\rho_*\otimes\ov{\iota_0}$ induces MHS isomorphisms in cohomology for all $m'\geq 1$, where $\phi\in R^2_{-m'}$ and $\iota_0$ is a local generator of $\cL_1$.
	
	Let $Y_1, Y_2$ be compactifications of $G_1$ and $G_2$, and let $X$ be a compactification of $U$ such that
	$$
	\begin{tikzcd}
		X\arrow[r, "\Id"]\arrow[d,"\ov{f_1}"]& X\arrow[d,"\ov{f_2}"]\\
		Y_1\arrow[r, "\ov{\rho}"]& Y_2
	\end{tikzcd}
	$$ forms a compatible compactification with respect to the commutative diagram \eqref{eq:baseFunctoriality}, where $g=\Id$. Let $j:U\to X$ be the inclusion, and let $D=X\setminus U$. 
	Now, we consider the thickened complexes from Definition~\ref{def:thickening}.
	\[
	\left(R_{-m}^i\otimes\cN^\bullet_{X,D,n}, d+\ov f_i^*\circ\Phi^{Y_i}(\eps)\right),
	\]
	where $n\geq \max\{2,\dim_\R U\}$.
	
Let $\rho_m:R^2_{-m}\to R^1_{-m}$ be the dual of the morphism $R^1_m\to R^2_m$ induced by $\rho_*:H_1(G_1,\R)\to H_1(G_2,\R)$. By Proposition~\ref{prop:thickeningFunctorialInV}, since $\rho^*\colon H^1(G_2,\R)\to H^1(G_1,\R)$ is a MHS morphism, \(\rho_m\otimes \Id\) induces a morphism of mixed Hodge complexes (taking into account Corollary~\ref{cor:PhiFunctorial}, which ensures that the \(\Phi\)'s and \(\Psi\)'s we have defined commute with \(\rho^*\)).
	
	The resulting morphism of mixed Hodge complexes of sheaves induces a morphism of MHSs, as desired. It remains to show that it agrees with the one in Lemma~\ref{lem:homologySheafMap}. First, note that it results, up to natural quasi-isomorphism (see Definition~\ref{def:endowedMHS}) from applying $Rj_*$ to
	\[
	R^2_{-m} \otimes_\R \cA^{\bullet}_{U,\R}
	\xrightarrow{\rho_m\otimes \Id}
	R^1_{-m} \otimes_\R \cA^{\bullet}_{U,\R}
	\]
so we just need to show that the following diagram commutes:
	\[
	\begin{tikzcd}
		R^2_{-m}\otimes_{R^2} \ov{\cL_2}
		\arrow[r,"e^{-(\Phi_\R^{G_2})^\vee}"]
		\arrow[d,"\phi\otimes \ov{\wt\rho\circ \iota_0}\mapsto \phi\circ \rho_*\otimes \ov{\iota_0}"']
		&
		R^2_{-m} \otimes_\R\cA^\bullet_{U, \R}
		\arrow[d,"\rho_m\otimes \Id"]
		\\
		R^1_{-m}\otimes_{R^1} \ov{\cL_1} 
		\arrow[r,"e^{-(\Phi_\R^{G_1})^\vee}"]
		&
		R^1_{-m} \otimes_\R\cA^\bullet_{U, \R}
	\end{tikzcd}
	\]
	
	The proof is done via direct computation. A generator of \(\oL_2\) can be given as \(\ov{\wt\rho\circ \iota}\), where \(\iota\) is a generator of \(\cL_1\). Let us write \(\iota = \sum_j e_j\otimes h_j\), where \(h_j\) are analytic functions on \(U\) and \(\{ e_j\}\) is an $\R$-basis of \(TG_1\). Applying \(\wt\rho\) results in the element \(\wt\rho\circ\iota = \sum_j \wt\rho(e_j)\otimes h_j\). If \(\alpha\in R^2_{-m}\), then
	\begin{align*}
		&(\rho_m\otimes \Id)(e^{-(\Phi_\R^{G_2})^\vee})\left(
		\alpha\otimes\ov{\wt\rho\circ\iota}
		\right)
		\\
		&=
		(\rho_m\otimes \Id)\left( \alpha \cdot \exp\left(
		-\sum_j (\Phi_\R^{G_2})^\vee(\wt \rho( e_j))\otimes h_j
		\right)
		\right)      \\
		&=
		(\rho_m\otimes \Id)\left( \alpha \cdot \exp\left(
		-\sum_j \rho_*(\Phi_\R^{G_1})^\vee( e_j)\otimes h_j
		\right)
		\right)
		&\left(\begin{array}{c} \text{Corollary~\ref{cor:PhiFunctorial}, taking duals, using}\\\text{that $\rho$ is a homomorphism}\end{array}
			\right)
	\end{align*}
	Going through the other path, the generator $\alpha \otimes \sum_j  \wt\rho(e_j)\otimes h_j$ is mapped to $\alpha\circ\rho_* \otimes \sum_j e_j\otimes h_j$ through the vertical arrow. Therefore, we have:
	\begin{align*}
		e^{-(\Phi_\R^{G_1})^\vee}\left(\alpha\circ\rho_* \otimes \sum_j e_j\otimes h_j\right)
		&=
		\rho_m(\alpha) \cdot \exp \left(
		- \sum_j (\Phi_\R^{G_1})^\vee(e_j)\otimes h_j
		\right).
	\end{align*}
	To show that the above two expressions coincide, we just need to show that for any $\beta\in H_1(G_1,\R)$ and any $\alpha\in \Hom(\Sym^j H_1(G_2,\R), \R)$,
	\[
	(\rho_*\beta\cdot \alpha)\circ\rho_* =  \beta \cdot(\alpha\circ\rho_*).
	\]
	The proof is a computation:  we use the fact that the product
	\[
	H_1(G_i,\R)\otimes (\Sym^j H_1(G_i,\R))^\vee\to (\Sym^{j-1} H_1(G_i,\R))^\vee
	\]
	is the dual of \(H^1(G_i,\R)\otimes \Sym^{j-1} H_1(G_i,\R) \to \Sym^j H_1(G_i,\R)\). For any \(\gamma \in \Sym^{j-1} H_1(G_1,\R)\),
	\begin{align*}
		\langle (\rho_*\beta\cdot \alpha)\circ\rho_*,\gamma\rangle
		=
		\langle \rho_*\beta\cdot \alpha ,\rho_*\gamma\rangle 
		=
		\langle \alpha ,\rho_*\beta \cdot\rho_*\gamma\rangle
		=
		\langle \alpha ,\rho_*(\beta\cdot \gamma)\rangle =\langle \alpha\circ\rho_*,\beta\cdot \gamma\rangle = \langle \beta\cdot \alpha\circ\rho_*,\gamma\rangle.
	\end{align*}
\end{proof}

\subsection{A more general statement}

Suppose that $\rho:G_1\to G_2$ in the statement of Theorem~\ref{thm:functorial} is an algebraic morphism but not a group homomorphism. In that case, Remark~\ref{rem:groupStructure} says that there exists a group homomorphism $\rho_1:G_1\to G_2$ and a translation $\rho_2:G_2\to G_2$ such that $\rho=\rho_2\circ\rho_1$, and the commutative diagram \eqref{eq:baseFunctoriality} can be decomposed as
$$
\begin{tikzcd}[row sep=2em]
	U_1
	\arrow[d,"f_1"]
	\arrow[r,"g"]
	&
	U_2
	\arrow[d,"\rho_2^{-1}\circ f_2"]\arrow[r, "\Id"]
	& 	U_2
	\arrow[d,"f_2"]\\
	G_1
	\arrow[r,"\rho_1"]
	&
	G_2\arrow[r, "\rho_2"]
	&
	G_2.
\end{tikzcd}
$$
Theorem~\ref{thm:functorial} says that the commutative square on the left induces a MHS homomorphism between the quotients of Alexander modules of $(U_1,f_1)$ and $(U_2,\rho_2^{-1}\circ f_2)$ by powers of the respective augmentation ideals. Hence, in this more general setting, it suffices to understand what happens for the commutative square on the right. This is done in the following result. The main issue is that, while any group homomorphism between semiabelian varieties lifts to a unique group homomorphism between its universal covers, algebraic morphisms between semiabelian varieties don't have a canonical lift to their universal covers in general. To avoid this dependence on the base points, we will compose it with another map in (co)homology in a way that the composition does not depend on the choice of base points used to construct the lift.

\begin{theorem}\label{thm:dependenceTranslation}
	Consider the commutative diagram
	$$
	\begin{tikzcd}
		U\arrow[r,"\Id"]\arrow[d, "f"]& U\arrow[d, "\rho\circ f"]\\
		G\arrow[r, "\rho"] & G
	\end{tikzcd}
$$
	where $\rho:G\to G$ is a translation, that is, multiplication by an element $x\in G$. Let $y\in TG$ such that $\exp(y)=x$. Let $\cL_1\coloneqq f^{-1}\exp_!\ul{\R}_{TG}$ and $\cL_2\coloneqq (\rho\circ f)^{-1}\exp_!\ul{\R}_{TG}$. Then, the following hold:
	\begin{itemize}
		\item If $\wt\rho:TG\to TG$ is addition by $y$, then $\exp\circ\wt\rho=\rho\circ\exp$.
		\item 	Let $\zeta_m^y:R_{-m}\otimes_R\ov{\cL_2}\to R_{-m}\otimes_R\ov{\cL_1}$ be the morphism from Remark~\ref{rem:translation} given by $\phi\otimes\ov{\wt\rho\circ\iota}\mapsto \phi\circ \rho_*\otimes\ov{\iota}$. The composition $e^{-(\Phi_\R^G)^\vee(y)}\circ \zeta_m^y$ induces through Steps \eqref{step1}--\eqref{step4} the morphism
		$$
		e^{-(\Phi_\R^G)^\vee(y)}\circ \wt{\Id}_{*,m}: R_m\otimes_R H_j(U^f,\R)\to R_m\otimes_R H_j(U^{\rho\circ f},\R),
		$$
		where $e^{-(\Phi_\R^G)^\vee(y)}$ denotes the multiplication by $e^{-(\Phi_\R^G)^\vee(y)}\in R_\infty$, and $\wt{\Id}_{*,m}$ is the map induced by the lift $\wt{\Id}:U^f\to U^{\rho\circ f}$ of $\Id$ which is determined by $\wt{\rho}$ as in Theorem~\ref{thm:functorial}.
		\item The morphism $e^{-(\Phi_\R^G)^\vee(y)}\circ \zeta_m^y$ is independent of the choice of $y\in TG$ such that $\exp(y)=x$.
		\item $e^{-(\Phi_\R^G)^\vee(y)}\circ \wt{\Id}_{*,m}$ is an isomorphism of MHS.
	\end{itemize}

 In particular, $\frac{H_j(U^f,\R)}{\fm^mH_j(U^f,\R)}$ and  $\frac{H_j(U^{\rho\circ f},\R)}{\fm^mH_j(U^{\rho\circ f},\R)}$ are canonically isomorphic, although not through the map $\wt\Id$ in general.
\end{theorem}

\begin{proof}
	The proof of the first point is immediate. The second point is a consequence of Remark~\ref{rem:translation} and the fact that the $\R$-dual of multiplication by an element of $R_\infty$ is also multiplication by an element of $R_\infty$.
	
	For the third point, notice that, since $G$ is path connected, $\rho$ is homotopic to the identity in $G$. In particular, $\zeta_m^y(-\phi\otimes\ov{\wt\rho\circ\iota})=\phi\otimes\ov\iota$. Notice also that for all $y,y'\in TG$ such that $\exp(y)=\exp(y')$ one has that $y-y'\in H_1(G,\Z)\subset TG$. In particular, $y-y'$ is fixed by $(\Phi_\R^G)^\vee$, so
	$$
	(\Phi_\R^G)^\vee(y)-(\Phi_\R^G)^\vee(y')=(\Phi_\R^G)^\vee(y-y')=y-y'\in H_1(G,\Z).
	$$
	Hence,
	$$
		e^{-(\Phi_\R^G)^\vee(y)}\circ\zeta_m^y(\phi\otimes\ov{\iota+y})=e^{-(\Phi_\R^G)^\vee(y)}\circ(\phi\otimes\ov{\iota})=\left(e^{-(\Phi_\R^G)^\vee(y)}\cdot\phi\right)\otimes\ov{\iota},
	$$
	and, if we let $\gamma\in \pi_1(G)$ be the element corresponding to $y-y'\in H_1(G,\Z)$ (i.e. $\log\gamma=y-y'$ using the notation of Remark~\ref{rem:RmVSRmodm}), then 
	\begin{align*}
		e^{-(\Phi_\R^G)^\vee}\circ\zeta_m^{y'}(\phi\otimes\ov{\iota+y})&=e^{-(\Phi_\R^G)^\vee(y')}\circ(\phi\otimes\ov{\iota+y-y'})=(e^{-(\Phi_\R^G)^\vee(y')}\cdot\phi)\otimes\gamma^{-1}\ov{\iota}\\
		&=e^{-(y-y')}(e^{-(\Phi_\R^G)^\vee(y')}\cdot\phi)\otimes\ov{\iota}=(e^{-(\Phi_\R^G)^\vee(y)}\cdot\phi)\otimes\ov{\iota},
	\end{align*}
which concludes the proof of the third point.

For the fourth point a similar computation to that of \eqref{eq:translation} yields that the morphism $\zeta_m^y:R_{-m}\otimes_R\ov{\cL_2}\to R_{-m}\otimes_R\ov{\cL_1}$ lifts through the morphisms $e^{-(\Phi_\R^G)^\vee}$ from Construction~\ref{con:nu} to a morphism $R_{-m}\otimes_\R \cA_{U,\R}^0\to R_{-m}\otimes_\R \cA_{U,\R}^0$ given by multiplication by $e^{(\Phi_\R^G)^\vee(y)}\in R_\infty$ in the first factor. This is done using that $\rho_*:R_m\to R_m$ is the identity. Hence, 	$e^{-(\Phi_\R^G)^\vee(y)}\circ\zeta_m^y$  lifts through the morphisms $e^{-(\Phi_\R^G)^\vee}$ from Construction~\ref{con:nu} to the identity morphism in $R_{-m}\otimes_\R \cA_{U,\R}^0$, which in turn lifts to the identity morphism in $R_{-m}\otimes_\R \cA_{U,\R}^\bullet$. Finding compatible compactifications of $U$ and $G$ with respect to the commutative diagram in the statement of this lemma, it is clear that the identity morphism can be realized at the level of mixed Hodge complexes of sheaves from Definition-Proposition~\ref{defprop:modifiedNA}, and hence it induces a morphism of MHS in hypercohomology. In particular, $e^{-(\Phi_\R^G)^\vee(y)}\circ \wt{\Id}_{*,m}$ is a morphism of MHS which is an isomorphism of vector spaces, so it is an isomorphism of MHS.
\end{proof}

The following example conveys that the MHS defined in this paper have the potential of distinguishing (up to algebraic isomorphism) algebraic varieties whose cohomology groups have isomorphic MHS. This could be interesting in the case of affine hypersurface complements (Example~\ref{exam:hypersurfaces}).

\begin{example}\label{eg:algebraicInvariant}
	Let $U$ be a smooth connected complex algebraic variety. By Remark~\ref{remk:existencequasialbanese}, its Albanese map $\alpha_U:U\to G$ is completely determined up to translation in the target.  By Remark~\ref{remk:universalAbelianCover}, $U^{\alpha_U}$ is the  universal (torsion-free) abelian cover of $U$, which is a topological invariant (i.e. it does not depend on $\alpha_U$, just on $U$). Let $\fm$ be the augmentation ideal of $\R[\pi_1(G)]$. By Theorem~\ref{thm:dependenceTranslation}, the isomorphism class of the mixed Hodge structure on $\frac{H_j(U^{\alpha_U},\R)}{\fm^m H_j(U^{\alpha_U},\R)}$ is an algebraic invariant of $(U,m)$, that is, it depends on the algebraic structure of $U$ and on the value of $m\geq 1$, but not on the choice of Albanese map $\alpha_U$. 
\end{example}

\subsection{Compatibility with Deligne's MHS}

We end this section by showing the compatibility of the MHS defined in this paper with Deligne's MHS, as a consequence of functoriality.

\begin{corollary}\label{cor:compatibleDeligne}
	Let $U$ be a smooth connected complex algebraic variety, let $G$ be a semiabelian variety, and let $f:U\to G$ be an algebraic morphism. Let $\pi:U^f\to U$ be the pullback of $\exp:TG\to G$ by $f$. Let $R=\R[\pi_1(G)]$, and let $\fm$ be its augmentation ideal. Then, the map that $\pi$ induces in homology factors through the MHS morphism
	$$
	\frac{H_j(U^f,\R)}{\fm^m H_j(U^f,\R)}\to H_j(U,\R)
	$$
	for all $j\geq 0$ and all $m\geq 1$, where $H_j(U,\R)$ is endowed with Deligne's MHS.
\end{corollary}
\begin{proof}
	The statement follows from applying Theorem~\ref{thm:functorial} in the case where $U_1=U_2=U$, $G_1=G$, $G_2$ is a point, $f_1=f$ and $g$ is the identity. Indeed, in this case $R^2=\R=R^2_m$, $U^{f_2}=U$ and $\wt g=\pi$, so in particular the augmentation ideal $\fm_2$ of $R^2$ is $(0)$. The thickened logarithmic Dolbeault mixed Hodge complex of sheaves $(R^2_m\otimes\cN^\bullet_{X,D,n}, d+\ov{f_2}^*\circ\Phi^Y)$ from Definition~\ref{def:thickening} constructed form $(U,f_2,m)$ coincides with the mixed Hodge complex of sheaves $(\cN^\bullet_{X,D,n}, d)$ from Definition-Proposition~\ref{defprop:modifiedNA}, so by Remark~\ref{rem:sameMHS}, the pro-MHS on $R^2_\infty\otimes_{R^2} H_j(U^{f_2},\R)=H_j(U^{f_2},\R)$ from Corollary~\ref{cor:quotientI} coincides with Deligne's MHS on $H_j(U,\R)$. Hence, the MHS on $H_j(U^{f_2},\R)\cong\frac{H_j(U^{f_2},\R)}{\fm_2^mH_j(U^{f_2},\R)}$ from Definition~\ref{def:MHSalexander} coincides with Deligne's MHS on $H_j(U,\R)$.
\end{proof}

\section{Completion with respect to other ideals}\label{sec:otherideals}
Let $U$ be a smooth connected complex algebraic variety, let $G$ be a complex semiabelian variety and let $f:U\to G$ be an algebraic map. Let $H$ be a finite index subgroup of $\pi_1(G)$, and let $\mathfrak{m}_H$ be the augmentation ideal of $R^H\coloneqq \R[H]\subset \R[\pi_1(G)]=R$. Note that $\mathfrak{m}_HR$ is the ideal of $R$ given by
$$
\mathfrak{m}_HR=\left(\gamma-1\mid \gamma\in H\right).
$$
The goal of this section is to endow  
$\frac{H_i(U^f,k)}{(\mathfrak{m}_H)^m H_i(U^f,k)}$ with a canonical MHS for all $i\geq 0$ and for all $m\geq 1$.
Here, $(\mathfrak{m}_H)^m H_i(U^f,k)$, which in principle is an  $R^H$-submodule of $H_i(U^f,k)$, can also be seen as an $R$-submodule by identifying it with $(\mathfrak{m}_HR)^m H_i(U^f,k)$.

For this, we start by passing to the finite cover induced by $H$ as follows: let $p_H:G_H\to G$ be the covering space corresponding to $H$, where $G_H$ is a semiabelian variety and $p_H$ is morphism of algebraic groups (see Remark~\ref{rem:homologyVsCohomologyFiniteIndex}). Note that this determines $p_H:G_H\to G$ up to unique isomorphism of semiabelian varieties on the domain.

The pair $(f,p_H)$ determines the following pullback diagram:

\begin{equation}\label{eq:finitecover}
	\begin{tikzcd}
		U_H\subset U\times G_H  \arrow[r,"f_H"] \arrow[d,"\pi_H"]\arrow[dr,phantom,very near start, "\lrcorner"]&
		G_H \arrow[d,"p_H"] \\
		U\arrow[r,"f"] &
		G.
	\end{tikzcd}
\end{equation}
Note that $\pi_H:U_H\to U$ is a finite cover of $U$, with deck transformation group $H$ (the same  as $p_H:G_H\to G$). Note also that $p_H$ induces a unique isomorphism of vector spaces $\wt{p_H}:TG_H\to TG$, such that $\exp\circ \wt{p_H}=p_H\circ\exp$.

We define the following map, which is easily seen to be an isomorphism of complex analytic varieties, where $U_H^{f_H}$ is constructed from the pullback diagram of $(f_H,\exp)$ as $U^f$ is constructed from the pullback diagram of $(f,\exp)$:
$$
\f{\theta_H}{U^f}{U_H^{f_H}}{(u,z)}{\left(\left(u,\exp({\wt{p_H}}^{-1}(z))\right), \wt{p_H}^{-1}(z)\right)}
$$
It fits into the following commutative cube:

\begin{equation}\label{eq:commcubeotherideals}  
	\begin{tikzcd}[row sep=2.5em]
		U_H^{f_H} \arrow[rr,"\theta_H^{-1}"] \arrow[dr,swap,"\widetilde{f_H}"] \arrow[dd,swap,"\pi'"] &&
		U^{f} \arrow[dd,swap,"\pi" near start] \arrow[dr,"\widetilde{f}"] \\
		& T G_H \arrow[rr,crossing over,"\widetilde{p_H}" near start] &&
		T G \arrow[dd,"\exp"] \\
		U_H \arrow[rr,"\pi_H" near end] \arrow[dr,swap,"f_H"] && U \arrow[dr,swap,"f"] \\
		& G_H \arrow[rr,"p_H"] \arrow[uu,<-,crossing over,"\exp" near end]&& G
	\end{tikzcd}
\end{equation}
Here, the bottom face of the cube is the pullback diagram above, and the left and right faces are also pullback diagrams.

A straightforward computation shows that $\pi_H\circ\pi'\circ \theta_H=\pi$, so $\pi'\circ \theta_H:U^f\to U_H$ is a covering space and $\theta_H$ is an isomorphism of covering spaces over $U_H$. Hence, $\theta_H$ induces an isomorphism $(\theta_H)*:H_j(U^{f},k)\to H_j(U_H^{f_H},k)$ of $\R[H]$-modules for all $j\geq 0$. Note that $\gamma\in H\leq \pi_1(G)$ acts on $H_j(U^{f},k)$ by $(p_H)_*(\gamma)$, where the latter is seen as a deck transformation of $\pi:U^f\to U$. In particular, $\theta_H$ induces isomorphisms of $\R[H]=\R[\pi_1(G_H)]$-modules
$$
\frac{H_j(U^{f},k)}{(\mathfrak{m}_H)^m H_j(U^{f},k)}\cong \frac{H_j(U_H^{f_H},k)}{(\mathfrak{m}_H)^m H_j(U_H^{f_H},k)}.
$$

\begin{proposition}\label{prop:otherideals}
	Let $H$ be a finite index subgroup of $\pi_1(G)$. Let $\mathfrak{m}$ be the augmentation ideal of $R=\R[\pi_1(G)]$, and let $\mathfrak{m}_H$ be the  augmentation ideal of $R^H=\R[H]\subset R$. Let $f:U\to G$ be an algebraic map, where $U$ is a smooth connected complex algebraic variety, and let $\pi:U^f\to G$ be the corresponding abelian cover corresponding to $f$ as in \eqref{eq:Uf}. Then,
	$$
	\frac{H_j(U^f,\R)}{(\mathfrak{m}_H)^m H_j(U^f,\R)}
	$$
	has a canonical MHS for all $j\geq 0$ and for all $m\geq 1$ such that the natural projection morphism
	$$
	\frac{H_j(U^f,\R)}{(\mathfrak{m}_H)^m H_j(U^f,\R)}\twoheadrightarrow \frac{H_j(U^f, \R)}{\mathfrak{m}^m H_j(U^f,\R)}
	$$
	is a morphism of mixed Hodge structures.
\end{proposition}
\begin{proof}
	Using the notation in the discussion above, the isomorphism induced by $\theta_H$ can be used to endow $\frac{H_j(U^f,\R)}{(\mathfrak{m}_H)^m H_j(U^f,\R)}$ with a MHS from the canonical MHS of $\frac{H_j(U_H^{f_H},\R)}{(\mathfrak{m}_H)^m H_j(U_H^{f_H},\R)}$ from Definition~\ref{def:MHSalexander}. For this MHS on $\frac{H_j(U^f,\R)}{(\mathfrak{m}_H)^m H_j(U^f,\R)}$ to be canonical it must not depend on the choice of $G_H$, which is determined up to unique isomorphism of semiabelian varieties, but this follows from functoriality (Theorem~\ref{thm:functorial}). The statement about the projection map follows from functoriality (Theorem~\ref{thm:functorial}) applied to the commutative diagram~\eqref{eq:finitecover}. 
\end{proof}

\begin{proposition}\label{prop:proOtherIdeals}
	Let $\pi_1(G)=K_0\geq K_1\geq K_2\ldots$ a sequence such that $K_i$ is a finite index subgroup of $K_{i-1}$ for all $i\geq 1$. Then, the following is a diagram of MHS morphisms, where all the maps involved are the natural projections and the MHS are the ones from Proposition~\ref{prop:otherideals}.
	\begin{equation}\label{eq:cubeotherideals}
		\begin{tikzcd}
			\ddots &	 \vdots\arrow[d,two heads] & \vdots  \arrow[d,two heads]&  \vdots \arrow[d,two heads]\\		
			\ldots\arrow[r,two heads]&	\frac{H_j(U^f,\R)}{(\mathfrak{m}_{K_2})^3H_j(U^f,\R)}\arrow[r,two heads]\arrow[d,two heads] & \frac{H_j(U^f,\R)}{(\mathfrak{m}_{K_2})^2H_j(U^f,\R)}\arrow[r,two heads]\arrow[d,two heads]& \frac{H_j(U^f,\R)}{\mathfrak{m}_{K_2}H_j(U^f,\R)}\arrow[d,two heads]\\
			\ldots\arrow[r,two heads] &	\frac{H_j(U^f,\R)}{(\mathfrak{m}_{K_1})^3H_j(U^f,\R)}\arrow[r,two heads]\arrow[d,two heads] & \frac{H_j(U^f,\R)}{(\mathfrak{m}_{K_1})^2H_j(U^f,\R)}\arrow[r,two heads]\arrow[d,two heads]& \frac{H_j(U^f,\R)}{\mathfrak{m}_{K_1}H_j(U^f,\R)}\arrow[d,two heads]\\
			\ldots\arrow[r,two heads] &	\frac{H_j(U^f,\R)}{\mathfrak{m}^3H_j(U^f,\R)}\arrow[r,two heads] & \frac{H_j(U^f,\R)}{\mathfrak{m}^2H_j(U^f,\R)}\arrow[r,two heads]& \frac{H_j(U^f,\R)}{\mathfrak{m}H_j(U^f,\R)}
		\end{tikzcd}
	\end{equation}
\end{proposition}

\begin{proof}
	Construct a chain of covering spaces $$\ldots\to G_{K_2}\to G_{K_1} \to G$$
	and use those semiabelian varieties to endow the different modules in the diagram with canonical MHS as in the proof of Proposition~\ref{prop:otherideals}. The vertical projections are morphisms of MHS by Proposition~\ref{prop:otherideals}. The horizontal projections are morphism of MHS by Remark~\ref{rem:projMHS}.
\end{proof}

\begin{example}
	Suppose that $G=\C^*$. In this case, $R=\R[\pi_1(G)]$ is isomorphic to $\R[t^{\pm 1}]$, where $t$ is a generator of $\pi_1(G)$. The map $p:\C^*\to\C^*$ that sends $z$ to $z^N$ is a finite cover.  In this setting, Proposition~\ref{prop:proOtherIdeals} applied to the chain of subgroups $\pi_1(G)=\langle t\rangle\geq \langle t^N\rangle\geq \langle t^{2N}\rangle\geq \langle t^{4N}\rangle$ yields the commutative diagram \eqref{eq:cubeotherideals}, where $(\fm_{K_l})^m=(t^{2^{l-1}N}-1)^m$ for all $l\geq 1$ and all $m\geq 1$, and $(\fm)^m=(t-1)^m$.
%
%
\end{example}

\section{Eigenspace decomposition}\label{sec:eigenspace}
In this section, $U$, $G$, $f$, $G_H$, $p_H$, $U_H$, $f_H$, $\pi$, $\pi_H$, $\pi'$, $\theta_H$, and the augmentation ideals $\mathfrak m\subset R=\R[\pi_1(G)]$, and $\mathfrak m_H\subset R^H=\R[H]$ will be as in Section~\ref{sec:otherideals}. Let $\mathcal L$ (resp. $\mathcal L_H$) be the local system of $\R[\pi_1(G)]$-modules given by $(f)^{-1}\exp_!\underline k_{TG}$ (resp. $(f_H)^{-1}\exp_!\underline k_{TG_H}$), and let $\overline{\mathcal L}$ (resp. $\overline{\mathcal L_H}$) be the corresponding local system endowed with its conjugate $\R[\pi_1(G)]$-module (resp. $\R[H]$-module) structure.

The rings $R^H_{\infty}$ and $R^H_{m}$ for $m>0$ will be defined from $G_H$, that is, $R^H_{\infty}\coloneqq \prod_{j=0}^\infty\Sym^j H_1(G_H,\R)$, $R^H_{m}\coloneqq \frac{R^H_{\infty}}{\prod_{j=m}^\infty\Sym^j H_1(G_H,\R)}$. Similarly, for all $m>0$, $R^H_{-m}\coloneqq \Hom_\R(R^H_{m},\R)$. The same construction can be carried out for $\C$ coefficients, and, abusing notation, will be denoted equally, as in Section~\ref{sec:thickening}.

The goal of this section is to prove the following theorem, which provides a generalization of Theorem~\cite[Theorem 1.3]{EvaMoises}.
\begin{theorem}\label{thm:ss}
	Let $\gamma\in\pi_1(G)$, which acts on $H_j(U^f,\R)$ by a deck transformation of $\pi:U^f\to U$. Let $H$ be a finite index subgroup of $\pi_1(G)$, and let $\mathfrak{m}_H$ be the augmentation ideal of $\R[H]$. Let $\gamma=\gamma_{ss}  \gamma_{u}$ be the Jordan-Chevalley decomposition of $\gamma$ acting on $\frac{H_j(U^f,\R)}{(\mathfrak m_H)^m H_j(U^f,\R)}$ as the product of a semisimple (i.e. diagonalizable) operator and a unipotent (i.e. $\gamma_u-\Id$ is nilpotent) operator that commute with each other. Then,
	\[
	\gamma_{ss}  \colon \frac{H_j(U^f,\R)}{(\mathfrak m_H)^m H_j(U^f,\R)}\to \frac{H_j(U^f,\R)}{(\mathfrak m_H)^m H_j(U^f,\R)}
	\]
	is a MHS isomorphism for all $j\geq 0$ and all $m\geq 1$.
\end{theorem}

Since $\gamma$ acts quasi-unipotently  on $\frac{H_j(U^f,\R)}{(\mathfrak m_H)^m H_j(U^f,\R)}$ (i.e. $\gamma^N-\Id$ is a nilpotent operator, where $N$  is the order of the class of $\gamma$ in  $\pi_1(G)/H$) and $\pi_1(G)$ is an abelian group, Theorem~\ref{thm:ss} immediately implies the following result. It uses the terminology of the lesser known $\C$-MHSs, see cf. \cite[Definition 2.1]{EvaMoises} for a definition.

\begin{corollary}[Eigenspace decomposition]\label{cor:eigenspaceDecomposition}
	Let $\gamma\in\pi_1(G)$, let $N$ be the order of the class of $\gamma$ in the  quotient $\pi_1(G)/H$ and let $k=\R,\C$. Let $g(x)\in k[x]$ be a monic irreducible factor of $x^N-1$, and let  $E_{g}^\gamma$ be the kernel of $g(\gamma_{ss}):\frac{H_j(U^f,k)}{(\mathfrak m_H)^m H_j(U^f,k)}\to \frac{H_j(U^f,k)}{(\mathfrak m_H)^m H_j(U^f,k)}$ for $m\gg0$, so in particular, if $g(x)=x-\lambda$ for some $N$-th root of unity $\lambda\in k$, $E_{g}^\gamma$ is the generalized eigenspace of eigenvalue $\lambda$.
	
	 Then, the inclusion $E_{g}^\gamma\hookrightarrow \frac{H_j(U^f,k)}{(\mathfrak m_H)^m H_j(U^f,k)}$ endows $E_{g}^\gamma$ with a $k$-MHS, and the direct sum decomposition $$\frac{H_j(U^f,k)}{(\mathfrak m_H)^m H_j(U^f,k)}=\bigoplus_g E_{g}^\gamma$$
	is a MHS decomposition, that is, the MHS on the right hand side is a direct sum of MHS as in the left hand side.
	
	Moreover, let $n$ be the rank of $\pi_1(G)$, and let $\{\gamma_1,\ldots,\gamma_n\}$ be a basis of $\pi_1(G)$ as a $\Z$-module. Consider all the $n$-tuples $\ov{g}=(g_1,\ldots,g_n)$ such that $g_i(x)\in k[x]$ is a monic irreducible factor of $x^{N_i}-1$, where $N_i$ is the order of the class of $\gamma_i$ in the quotient $\pi_1(G)/H$. Denote $E_{\ov g}\coloneqq \bigcap\limits_{i=1}^n E_{g_i}^{\gamma_i}$. Then,
	$$
	\frac{H_j(U^f,k)}{(\mathfrak m_H)^m H_j(U^f,k)}=\bigoplus_{\ov g} E_{\ov g}
	$$
	is a finer $k$-MHS decomposition.
\end{corollary}

Theorem~\ref{thm:ss} will be proved by passing to the finite cover $U_H$ of $U$, so first, we need to specify certain identifications between local systems on $U$ and $U_H$. Let $\mathcal L_H\coloneqq f_H^{-1}\exp_!\ul{\R}_{TG}$ and $\mathcal L\coloneqq f^{-1}\exp_!\ul{\R}_{TG}$.  From the commutative cube \eqref{eq:commcubeotherideals}, we deduce the following chain of canonical identifications:
\begin{align*}
	&(\pi_H)_*\mathcal L_H=(\pi_H)_*(f_H)^{-1}\exp_!\underline \R_{TG_H}=(\pi_H)_!(f_H)^{-1}\exp_!\underline \R_{TG_H}=(\pi_H)_!(\pi')_!(\widetilde{f_H})^{-1}\underline k_{TG_H}=\\
	&\pi_!(\theta_H^{-1})_!(\widetilde{f_H})^{-1}\underline \R_{TG_H}=
	\pi_!(\theta_H)^{-1}(\widetilde{f_H})^{-1}\underline \R_{TG_H}=\pi_!(\widetilde f)^{-1} (\wt{p_H}^{-1})^{-1}\underline \R_{TG}= \pi_!(\widetilde f)^{-1} \underline \R_{TG}=\mathcal L.
\end{align*}
Here, the superscript $-1$ is used to describe the inverse image functor when the function is in parentheses, and the inverse of a bijective map when there are no parentheses, we hope that the use is clear from the context.  Recall that the $\R[H]$-module structure on $\mathcal L_H$ side is by deck transformations of $\pi':U_H^{f_H}\to U_H$, and the $\R[H]$-module structure on $\mathcal L$ is by seeing $H$ inside of $\pi_1(G)$ and thus considering the elements of $H$ as deck transformations of $\pi:U^{f}\to U$.  Hence, the chain of identifications above induces an isomorphism of $\R[H]$-modules between the (conjugate) local systems
$$
\theta_{\overline{\mathcal L_H}}:(\pi_H)_*\overline{\mathcal L_H}\to\overline{\mathcal L}.
$$

Let $\gamma\in\pi_1(G)$, and let $\cT_\gamma:U^f\to U^f$ be the corresponding deck transformation of $\pi$.
By definition, $\gamma\in \R[\pi_1(G)]$ has an action on $\cL$ induced by $\cT_\gamma$, i.e.  if $\iota$ is a local section of $\cL$ seen as a map $\iota:U\to TG$ such that $\exp\circ\iota=f$, $\gamma\cdot\iota = \iota+\log\gamma$. 
The action of $\cT_{\gamma}$ descends to the deck transformation of $\pi_H$, which we will also denote $\cT_{\gamma}\colon U_H\to U_H$. 

Now, note that for all $(u,z)\in U^f\subset U\times TG$, $\gamma\cdot (u,z)=(u,z+\log\gamma)$, where $\log\gamma$ is the element of $H_1(G,\Z)$ corresponding to $\gamma$ through the abelianization. In that sense, we can think of $\gamma\in\pi_1(G)$ as acting on $U_H^{f_H}$ through the isomorphism $\theta_H$ as follows: for every $((u,s),w)\in U_H^{f_H}\subset U_H\times TG_H\subset U\times G_H\times TG_H$, $$\gamma\cdot ((u,s),w)\coloneqq\theta_H\circ(\gamma\cdot)\circ\theta_H^{-1}((u,s),w)=\left(\left(u,s\cdot\exp(\wt{p_H}^{-1}(\log\gamma))\right),w+\wt{p_H}^{-1}(\log\gamma)\right).$$

The following result below is a generalization of \cite[Lemma 3.1]{EvaMoises} (where $G=\C^*$). The proof is a straightforward verification on local sections which follows the same steps as in \cite[Lemma 3.1]{EvaMoises}, so we omit it. In it, we describe the $\gamma$-action on $\ov{\cL_H}$, similarly as to how we have described the $\gamma$-action on $U_H^{f_H}$.

\begin{lemma}\label{lem:t-liftstoEquiv}
	Let $\gamma\in\pi_1(G)$. There is a morphism of sheaves $\cM_\gamma\colon  \overline{\cL_H}  \to (\cT_\gamma)_* \overline{\cL_H} $ such that after taking $(\pi_H)_*$ it becomes multiplication by $\gamma$, i.e. the following composition is multiplication by $\gamma$:
	\[
	\ov\cL \xrightarrow[\sim]{\theta_{ \overline{\cL_H} }^{-1}} (\pi_H)_* \overline{\cL_H}  \xrightarrow{(\pi_H)_* \cM_\gamma} (\pi_H)_* (\cT_\gamma)_* \overline{\cL_H}    = (\pi_H\circ \cT_\gamma)_*  \overline{\cL_H}  = (\pi_H)_* \overline{\cL_H}  \xrightarrow[\sim]{\theta_{ \overline{\cL_H} }} \ov\cL.
	\]
	Furthermore, for every local section $\ov\iota$ of $\ov{\cL_H}$ (seen as a map $\iota:U_H\to TG_H$ such that $\exp\circ\iota=f_H$), $\cM_\gamma$ is given by $$\cM_\gamma(\ov\iota)= (\wt{p_H}^{-1}\circ(-\log \gamma)\circ\wt{p_H}) \circ \ov\iota\circ \cT_\gamma=(\ov\iota-\wt{p_H}^{-1}\log\gamma)\circ \cT_\gamma,$$ where $(-\log \gamma):TG\to TG$ is defined by $z\to z-\log\gamma$. 
\end{lemma}

\begin{remark}
	From now on, we are going to be working on complexes of sheaves defined over $U_H$. We will look at the map that $\cM_\gamma$ induces between $R_m^H\otimes_{R^H} \overline{\cL_H}$ and itself, namely $\Id\otimes\cM_\gamma$. For simplicity in the notation, we will denote $\Id\otimes\cM_\gamma$ also by $\cM_\gamma$ from now on. Note 
\end{remark}

The following is a higher dimensional generalization of \cite[Lemma 3.3]{EvaMoises}.
\begin{lemma}\label{lem:t-liftstodR}
	Let $m\in \Z\setminus\{0\}$. Under the quasi-isomorphism $$e^{-(\Phi_{\R}^{G_H})^\vee} \colon  R_{m}^H\otimes_R\overline{\cL_H}  \to \left( R_{m}^H\otimes_{\R}\cA^\bullet_{U_H,\R}, d+(f_H)^*\circ \Phi_{\R}^{G_H}(\varepsilon_\R)\right)$$ from Construction~\ref{con:nu}, the map $\cM_\gamma$ (where $\cM_\gamma$ was defined in Lemma~\ref{lem:t-liftstoEquiv}) becomes the following morphism of complexes of sheaves, defined $R_\infty^H$-linearly as
	\[
	\f{\wt\cM_\gamma}{\left( R_{m}^H\otimes_{\R}\cA^\bullet_{U_H,\R}, d+(f_H)^*\circ \Phi_{\R}^{G_H}(\varepsilon_\R)\right)}{(\cT_\gamma)_*\left( R_{m}^H\otimes_{\R}\cA^\bullet_{U_H,\R}, d+(f_H)^*\circ \Phi_{\R}^{G_H}(\varepsilon_\R)\right)}{1\otimes\omega}{e^{\wt{p_H}^{-1}(\log\gamma)}\otimes(\cT_\gamma)^*\omega.}
	\]
	In other words, the following diagram commutes:
	\[
	\begin{tikzcd}
		R_{m}^H\otimes_{R^H}\overline{\cL_H} \arrow[r,"\cM_\gamma"]\arrow[d,"e^{-(\Phi_{\R}^{G_H})^\vee}"] &
		(\cT_\gamma)_* R_{m}^H\otimes_{R^H}\overline{\cL_H} \arrow[d,"(\cT_\gamma)_*e^{-(\Phi_{\R}^{G_H})^\vee}"] \\
		\left( R_{m}^H\otimes_{\R}\cA^\bullet_{U_H,\R}, d+(f_H)^*\circ \Phi_{\R}^{G_H}(\varepsilon_\R)\right) \arrow[r,"\wt\cM_\gamma"] &
		(\cT_\gamma)_*\left( R_{m}^H\otimes_{\R}\cA^\bullet_{U_H,\R}, d+(f_H)^*\circ \Phi_{\R}^{G_H}(\varepsilon_\R)\right).
	\end{tikzcd}
	\]
\end{lemma}

\begin{proof}
	Since $p_H^*:H^1(G,\R)\to H^1(G_H,\R)$ is surjective, $\wt\cM_\gamma$ is a morphism of complexes of sheaves. Moreover, since $\cT_\gamma$ is a deck transformation of $\pi_H:U_H\to U$,
	$$\cT_\gamma^*\circ f_H^*\circ\Phi_\R^{G_H}\circ p_H^*=\cT_\gamma^*\circ f_H^*\circ p_H^*\circ\Phi_\R^G=\cT_\gamma^*\circ \pi_H^*\circ f^*\circ\Phi_\R^G=\pi_H^*\circ f^*\Phi_\R^G=f_H^*\circ\Phi_\R^{G_H}.$$
	
	The rest of the proof is a direct application of the definitions of the morphisms involved which can be checked on elements of the form $1\otimes\ov\iota$, where $\ov\iota$ is a local section of $\ov{\cL_H}$. It uses that $(\Phi_{\R}^{G_H})^\vee$ fixes $H_1(G_H,\R)$, so in particular it fixes $\wt{p_H}^{-1}\log\gamma$.
\end{proof}


Let $\wt\cM_\gamma^{ss}$ be the following morphism of complexes of sheaves, defined $R_{\infty}^H$-linearly as
\[
\f{\wt\cM_\gamma^{ss}}{\left( R_{m}^H\otimes_{\R}\cA^\bullet_{U_H,\R}, d+(f_H)^*\circ \Phi_{\R}^{G_H}(\varepsilon_\R)\right)}{(\cT_\gamma)_*\left( R_{m}^H\otimes_{\R}\cA^\bullet_{U_H,\R}, d+(f_H)^*\circ \Phi_{\R}^{G_H}(\varepsilon_\R)\right)}{1\otimes\omega}{1\otimes(\cT_\gamma)^*\omega}
\]

\begin{lemma}\label{lem:gammasscohomology}
	The morphism $\cM_\gamma^{ss}$ induces an isomorphism of MHS in cohomology
	$$
	h_{ss}:H^j(U_H,R_{m}^H\otimes_{R^H} \ov{\cL_H})\to  H^j(U_H,R_{m}^H\otimes_{R^H} \ov{\cL_H})
	$$
	for all $j\geq 0$ and all $m\in \Z\setminus\{0\}$ through the quasi-isomorphism $e^{-(\Phi_{\R}^{G_H})^\vee}$.
\end{lemma}
\begin{proof}
	We denote also by $\cT_\gamma:G_H\to G_H$ be the deck transformation of $p_H:G_H\to G$ induced by $\gamma\in\pi_1(G)$. Let $Y_1, Y_2$ be compactifications of $G_H$ and $X_1, X_2$ be compactifications of $U_H$ such that the first commutative diagram in 
	$$
	\begin{tikzcd}
		X_1\arrow[r, "\ov{\cT_\gamma}"]\arrow[d,"\ov{f_H}"]& X_2\arrow[d,"\ov{f_H}"]\\
		Y_1\arrow[r, "\ov{\cT_\gamma}"]& Y_2
	\end{tikzcd} \quad \quad\quad	\begin{tikzcd}
	U_H\arrow[r, "\cT_\gamma"]\arrow[d,"f_H"]& U_H\arrow[d,"f_H"]\\
	G_H\arrow[r, "\cT_\gamma"]& G_H
\end{tikzcd}
	$$ forms a compatible compactification with respect to the second commutative diagram.
	
	Let $D_i\coloneqq X_i\setminus U_H$ for $i=1,2$. Note that $j_*\wt\cM_\gamma^{ss}$ restricts to a morphism of sheaf complexes
	$$
	\left(R_{m}^H\otimes_\R\cA^\bullet_{X_2,\R}(\log D_2), d+(\ov{f_H})^*\circ\Phi_\R^{Y_2}(\varepsilon_\R)\right)\to (\ov{\cT_\gamma})_*\left(R_{m}^H\otimes_\R\cA^\bullet_{X_1,\R}(\log D_1), d+(\ov{f_H})^*\circ\Phi_\R^{Y_1}(\varepsilon_\R)\right),
	$$
	where $X, Y$ are compatible compactifications of $U_H$, $G_H$ with respect to $f_H$, and $D=X\setminus U_H$, and this restriction is the pullback by $\ov{\cT_\gamma}$.
	
	Recall the definition of the derived direct image of a mixed Hodge complex of sheaves (Definition~\ref{def:directim}). Composing this restriction of $j_*\wt\cM_\gamma^{ss}$  with $(\ov{\cT_\gamma})_*$ of the inclusion of $$\left(R_{m}^H\otimes_\R\cA^\bullet_{X_1,\R}(\log D_1), d+(\ov{f_H})^*\circ\Phi_\R^{Y_1}(\varepsilon_\R)\right)$$ into its Godement resolution, and using that the pullback by algebraic functions respects the weight filtration $\wt W_{\lc}$, we obtain a morphisms of filtered complexes
	\begin{small}
	$$
	\left(R_{m}^H\otimes\cA^\bullet_{X_2,\R}(\log D_2),d+(\ov{f_H})^*\circ\Phi_\R^{Y_2}(\varepsilon_\R),W_{\lc}^n\right)\to R(\ov{\cT_\gamma})_* \left(R_{m}^H\otimes\cA^\bullet_{X_1,\R}(\log D_1),d+(\ov{f_H})^*\circ\Phi_\R^{Y_1}(\varepsilon_\R),W_{\lc}^n\right)
	$$
\end{small}
	given by the pullback by $\ov{\cT_\gamma}$, for any $n\geq\max\{2,\dim_\R U_H\}$. The result follows from the fact that this extends to a morphism of mixed Hodge complexes of sheaves
	$$
	\left(R_m^H\otimes\cN^\bullet_{X_2,D_2,n}, d+\ov{f_H}^*\circ\Phi^{Y_2}(\varepsilon)\right)\to R(\ov{\cT_\gamma})_* \left(R_m^H\otimes\cN^\bullet_{X_1,D_1,n}, d+\ov{f_H}^*\circ\Phi^{Y_1}(\varepsilon)\right),
	$$
	where the morphism between the complex part is also given by the pullback by $\ov{\cT_\gamma}$. Indeed, pullback by $\ov{\cT_\gamma}$ respects both the weight and Hodge filtrations there, and it is straightforward to check that $R(\ov{\cT_\gamma})_*\left(e^{\ov{f_H}^*\circ\Psi^{Y_1}(\varepsilon_\C)}\right)$ composed with the real part of the morphism (tensored by $\otimes_\R\C$) coincides with the composition of the complex part of this morphism and $e^{\ov{f_H}^*\circ\Psi^{Y_2}(\varepsilon_\C)}$.
\end{proof}

\begin{lemma}\label{lem:hss}
	Let $N$ be the order of the class of $\gamma$ in the quotient $\pi_1(G)/H$, and let $j\geq 0$, $m\in\Z\setminus\{0\}$. Let
	$$h_{ss}:H^j(U_H,R_{m}^H\otimes_{R^H} \ov{\cL_H})\to  H^j(U_H,R_{m}^H\otimes_{R^H} \ov{\cL_H})$$
	as in Lemma~\ref{lem:gammasscohomology}, and let $$h:H^j(U_H,R_{m}^H\otimes_{R^H} \ov{\cL_H})\to  H^j(U_H,R_{m}^H\otimes_{R^H} \ov{\cL_H})$$ be the map induced by $\cM_\gamma$ in cohomology. Then,
	\begin{itemize}
		\item $(h_{ss})^N=\Id$.
		\item $(h\circ (h_{ss})^{-1}-\Id)^{|m|}=0$
		\item $h_{ss}$ and $h$ commute.
	\end{itemize}
	In particular, $h_{ss}$ is the semisimple part in the Jordan-Chevalley decomposition of $h$.
\end{lemma}
\begin{proof}
	Note that $h_{ss}$ and $h$ are the maps induced in cohomology by $\wt\cM_\gamma$ and $\wt\cM_\gamma^{ss}$  through the quasi-isomorphism $e^{-(\Phi_{\R}^{GH})^\vee}$. These statements can be easily checked by looking at $\wt\cM_\gamma$ and $\wt\cM_\gamma^{ss}$.
	The first statement is a consequence of the fact that $(\cT_\gamma)^N$ is the identity on $U_H$. The second is a consequence of the fact that $(e^{\wt{p_H}^{-1}(\log\gamma)}-1)^{|m|}\in R_{\infty}^H$ acts as multiplication by $0$ in $R_{m}^H$.
\end{proof}

We are now ready to prove the main result in this section.
\begin{proof}[Proof of Theorem~\ref{thm:ss}]
	Let $m>0$. By Lemma~\ref{lem:t-liftstoEquiv}, multiplication by $\gamma\in \pi_1(G)\subset\R[\pi_1(G)]$ in $\ov\cL$ determines an isomorphism of sheaves from $\gamma:R_{-m}^H\otimes_{R^H} \ov\cL\to R_{-m}^H\otimes_{R^H} \ov\cL$ to itself, which, through the isomorphism $\theta_{\ov{\cL_H}}$ determines an isomorphism
	$$
	(\pi_H)_*\cM_\gamma:(\pi_H)_*\left(R_{-m}^H\otimes_{R^H} \ov{\cL_H}\right)\to (\pi_H)_*\left(R_{-m}^H\otimes_{R^H} \ov{\cL_H}\right)
	$$
	Let $\gamma_{ss}:H^j(U,R_{-m}^H\otimes_{R^H} \ov\cL)\to H^j(U,R_{-m}^H\otimes_{R^H} \ov\cL)$ be the semisimple part of the isomorphism induced by $\gamma$ in cohomology. Now, taking duals and applying Lemma~\ref{lem:hss}, we obtain the commutative diagram
	$$
	\begin{tikzcd}
		\Hom_\R\left(H^j(U,R_{-m}^H\otimes_{R^H} \ov\cL),\R\right)\arrow[d,"(\gamma_{ss})^\vee"]\arrow[r, "\cong"] & \Hom_\R\left(H^j(U_H,R_{-m}^H\otimes_{R^H} \ov{\cL_H}),\R\right)\arrow[d, "(h_{ss})^\vee", "\text{MHS}"']\\
		\Hom_\R\left(H^j(U,R_{-m}^H\otimes_{R^H} \ov\cL),\R\right)\arrow[r, "\cong"] & \Hom_\R\left(H^j(U_H,R_{-m}^H\otimes_{R^H} \ov{\cL_H}),\R\right),
	\end{tikzcd}
	$$
	where the arrow on the right is a MHS morphism by Lemma~\ref{lem:gammasscohomology}.
	
	We can apply Remark~\ref{rem:homologyVsCohomology} to the right column of this diagram and Remark~\ref{rem:homologyVsCohomologyFiniteIndex} to the left column to obtain
	\begin{equation}\label{eq:beforeLimit}
		\begin{tikzcd}
			H_j(U,R_{m}^H\otimes_{R^H} \cL)\arrow[d,"\gamma_{ss}"]\arrow[r, "\cong"] & H_j(U_H,R_{m}^H\otimes_{R^H} \cL_H)\arrow[d, "(h_{ss})^\vee", "\text{MHS}"']\\
			H_j(U,R_{m}^H\otimes_{R^H} \cL)\arrow[r, "\cong"] & H_j(U_H,R_{m}^H\otimes_{R^H} \cL_H).
		\end{tikzcd}
	\end{equation}
	Note that the arrow at the left has been labeled $\gamma_{ss}$ because the dual of multiplication by $\gamma$ is multiplication by $\gamma$, and taking duals respects the Jordan-Chevalley decomposition. Now, these maps are defined for all $m>0$ and commute with taking inverse limits, so by Corollary~\ref{cor:completionHomology2},
	\begin{equation}\label{eq:afterLimit}
		\begin{tikzcd}
			R_{\infty}^H\otimes_{R^H} H_j(U, \cL)\arrow[d,"\varprojlim\limits_m\gamma_{ss}"]\arrow[r, "\cong"] & R_{\infty}^H\otimes_{R^H} H_j(U_H, \cL_H)\arrow[d, "\varprojlim\limits_m(h_{ss})^\vee", "\text{pro-MHS}"']\\
			R_{\infty}^H\otimes_{R^H} H_j(U, \cL)\arrow[r, "\cong"] & R_{\infty}^H\otimes_{R^H} H_j(U_H, \cL_H).
		\end{tikzcd}
	\end{equation}
	In the previous commutative diagram, the horizontal arrows were induced by $\theta_{\ov{\cL_H}}$ after tensoring by $R_{-m}^H$ over $R^H$, taking the $j$-th cohomology, taking $\R$-duals and performing an inverse limit. However, notice that this is  the map induced in homology by the identification $\theta_H:U^f\to U_H^{f_H}$ from \eqref{eq:commcubeotherideals}, under the identification from Remark~\ref{rem:coverVsL}. Indeed, both of these maps come from the natural identifications arising from the commutative cube \eqref{eq:commcubeotherideals}.
	
	Let us see that the maps in \eqref{eq:beforeLimit} all commute with multiplication by any element of $H_1(G_H,\R)\subset R_{\infty}^H$. It is a well known fact of the Jordan-Chevalley decomposition that the semisimple part of a matrix $A$ with real entries can be written as a polynomial on $A$ as follows: if $p_A(x)=\prod_{k=1}^l(x-\lambda_k)^{n_k}$ is the characteristic polynomial of $A$ for $\lambda_1,\ldots,\lambda_l$ distinct elements in $\C$, Bézout's identity implies that we can pick polynomials $C_k(x),D_k(x)\in\C[x]$ such that $C_k(X)\cdot (x-\lambda_k)^{n_k}+D_k(x)\cdot \prod_{j\neq k}(x-\lambda_j)^{n_j} =1$, and $D_k$ can be chosen so that its constant term is $0$. Let $P(x)\coloneqq \sum_{k=1}^l \lambda_k D_k(x)\cdot \prod_{j\neq k}(x-\lambda_j)^{n_j}$, which is a polynomial with $0$ constant term. Since every vector in $\ker (A-\lambda_k I)^{n_k}$ is an eigenvector of $P(A)$ of eigenvalue $\lambda_k$ for all $k=1,\ldots,l$, $P(A)$ is the semisimple part of $A$. Now, since the Jordan-Chevalley decomposition commutes with taking duals, we have that $(h_{ss})^\vee$ is a polynomial in $h^\vee$ with no constant term, so in particular $(h_{ss})^\vee$ commutes with every linear operator that commutes with $h^\vee$. Since the action of $\gamma$ on the left hand side of \eqref{eq:beforeLimit} commutes with multiplication by any element of $H_1(G_H,\R)$, we obtain that it commutes with all the maps in \eqref{eq:beforeLimit}, as desired.
	
	In Corollary~\ref{cor:quotientI} $R_m^H\otimes_{R^H} H_j(U_H, \cL_H)$ is endowed with a MHS as the cokernel of a multiplication map
	$$
	\underbrace{H_1(G_H,\R)\otimes\ldots\otimes H_1(G_H,\R)}_m\otimes \left(R_\infty^H\otimes_{R^H} H_j(U_H,\cL_H)\right)\to R_\infty^H\otimes_{R^H} H_j(U_H,\cL_H),
	$$
	so, by \eqref{eq:afterLimit},
	$$
	(h_{ss})^\vee:R_m^H\otimes_{R^H} H_j(U_H,\cL_H)\to R_m^H \otimes_{R^H} H_j(U_H,\cL_H)
	$$
	is a mixed Hodge structure isomorphism. Note that there exists $m'\gg1$ such that the natural map
	$$
	 H_j(U_H,R_{m'}^H\otimes_{R^H}\cL_H)\to R_m^H\otimes_{R^H} H_j(U_H,\cL_H)
	$$
	is surjective and a MHS morphism. Since this surjection commutes with $(h_{ss})^\vee$, Lemma~\ref{lem:gammasscohomology} implies that $\left((h_{ss})^\vee\right)^N:R_m^H\otimes_{R^H} H_j(U_H,\cL_H)\to R_m^H\otimes_{R^H} H_j(U_H,\cL_H)$ is the identity, it commutes with $h$, and $h^\vee\circ\left((h_{ss})^\vee\right)^{-1}$ is unipotent. In other words,  $(h_{ss})^\vee:R_m^H\otimes_{R^H} H_j(U_H,\cL_H)\to R_m^H\otimes_{R^H} H_j(U_H,\cL_H)$ is the semisimple part in the Jordan-Chevalley decomposition of $h^\vee$, and it is a MHS isomorphism.
	
	Recall the definition of the MHS on $\frac{H_j(U^f,\R)}{\fm_H H_j(U^f,\R)}$ from Proposition~\ref{prop:otherideals}, which uses the MHS from Definition~\ref{def:MHSalexander}. Under the isomorphism $H_j(U_H,\cL_H)\cong H_j(U,\cL)$ coming from $\theta_H$ and the identification from Remark~\ref{rem:coverVsL}, $(h_{ss})^\vee$ corresponds to the semisimple part of the map induced by $\gamma$, so
	$$
	\gamma_{ss}: \frac{H_j(U^f,\R)}{\fm_H H_j(U^f,\R)}\cong R_m^H\otimes_{R^H} H_j(U,\cL)\to R_m^H\otimes_{R^H} H_j(U,\cL)\cong \frac{H_j(U^f,\R)}{\fm_H H_j(U^f,\R)}
	$$
	is a mixed Hodge structure isomorphism, concluding the proof.
%
%
\end{proof}

\section{The \texorpdfstring{$\Q$}{Q}-MHS in the case \texorpdfstring{$G=(\C^*)^n$}{G = C*n}}\label{s:comparison}

In \ref{def:endowedMHS}, a canonical MHS was defined on $H^*(U,R_m\otimes_R\ov\cL)$ for all $m\in\Z\setminus\{0\}$. All of the other MHS defined in this paper are induced from these ones through morphisms defined over $\Q$. The goal of this section is to prove that the MHSs of this paper are actually defined over $\Q$, in the specific case where $G=(\C^*)^n$, although we expect the result to be true in general. Note that the construction in Section~\ref{sec:MHS} only uses morphisms defined over $\Q$,  and the results in Sections~\ref{sec:MHS}, \ref{sec:functoriality}, \ref{sec:otherideals} and~\ref{sec:eigenspace} only involve morphisms defined over $\Q$. Therefore, the results therein also hold for the MHS with $\Q$-coefficients.

Let $U$ be a smooth connected complex algebraic variety, and let $f:U\to (\C^*)^n$ be an algebraic morphism, where $n\geq 1$. Let $X$ and $(\mathbb P^1)^n$ be compactifications of $U$ and $(\C^*)^n$ which are compatible with $f$, and let $\ov f:X\to (\mathbb P^1)^n$ be the extension of $f$ to those compactifications. Let $D\coloneqq X\setminus U$ and let $n'\geq\max\{2,\dim_\R U\}$. Pick coordinates $(z_1,z_2,\ldots,z_n)$ of $(\C^*)^n$, and note that $\left\{\left[\frac{1}{2\pi i}\frac{dz_j}{z_j}\right]\mid j=1,\ldots, n\right\}$ form a basis of $H^1((\C^*)^n,\Z)$. With this choice of coordinates, we have that $f=(f_1,\ldots,f_n)$, where $f_j:U\to \C^*$ for all $j=1,\ldots, n$.

Note that $G_A$ in the Chevalley decomposition of $G=(\C^*)^n$ is a point, and $G=G_T$. Hence, by Definition-Proposition~\ref{defprop:PhiPsiG}, $\Phi^{(\mathbb P^1)^n}_\C$ factors through $\Gamma\left((\mathbb P^1)^n, \Omega_{(\mathbb P^1)^n}^1(\log E)\right)$, where $E=(\mathbb P^1)^n\setminus(\C^*)^n$. It is straightforward to see that maps $\ov f^*\circ\Phi_{\C}^{(\mathbb P^1)^n}, 	\ov f^*\circ\Phi_{\R}^{(\mathbb P^1)^n}$, and $\ov f^*\circ\Psi^{(\mathbb P^1)^n}$ appearing in Definition~\ref{def:thickening} have the following form:
\begin{small}
\begin{align}\label{eq:explicitPhiPsi}
	\begin{split}
		&\f{	\ov f^*\circ\Phi_{\C}^{(\mathbb P^1)^n}}{\left(H^1((\C^*)^n,\C),W_{\lc}[1],F^{\lc}\right)}{\left(\Gamma(X,\Omega^{1}_{X}(\log D)),W_{\lc},F^{\lc}\right)\subset\left(\Gamma(X,\cA^{1}_{X,\C}(\log D)),W_{\lc}^{n'},F^{\lc}\right)}{\left[\frac{1}{2\pi i}\frac{dz_j}{z_j}\right]}{\frac{1}{2\pi i}\frac{df_j}{f_j},}\\
		&\f{\ov f^*\circ\Phi_{\R}^{(\mathbb P^1)^n}}{\left(H^1((\C^*)^n,\R),W_{\lc}[1]\right)}{\left(\Gamma(X,\cA^{1}_{X,\R}(\log D)),W_{\lc}^{n'}\right)}{\left[\frac{1}{2\pi i}\frac{dz_j}{z_j}\right]}{\Re\frac{1}{2\pi i}\frac{df_j}{f_j}=\frac{1}{2\pi}\Im \frac{df_j}{f_j},}\\
		&
		\f{\ov f^*\circ\Psi^{(\mathbb P^1)^n}}
    {\left(H^1((\C^*)^n,\C),W_{\lc}[1]\right)}
    {\left(\Gamma(X,\C\otimes_{\R}\cA_{X,\R}^0(\log D)),W_{\lc}^{n'}\right)}
    {\left[\frac{1}{2\pi i}\frac{df_j}{f_j}\right]}
    {-\frac{1}{2\pi i}\log(|f_j|).}
	\end{split}
\end{align}
	\end{small}
Recall the filtrations in the target of the first of these maps, which were defined in Section~\ref{ss:analyticDolbeault}, and note that it respects the filtrations (recall that $H^1((\C^*)^n,\R)$ is pure of type $(1,1)$).

Let $j:U\to X$ be the inclusion.  Consider the multiplicative $\Q$-mixed Hodge complex of sheaves on $X$ of \cite[Theorem 2.37]{mhsalexander}
\begin{equation}\label{eq:rationalMHC}
((\cK^\bullet_\infty,\wt W_{\lc}), (\Omega_X^\bullet(\log D), W_{\lc}, F^{\lc}),\varphi_\infty),
\end{equation}
which coincides with the one from \cite[Section 4.4]{peters2008mixed} except for a slight modification in the weight filtration of the rational part so as to make it biregular.

For the purposes of this section, we just need to recall this much of the definition of this mixed Hodge complex of sheaves:
\begin{align*}
	&\cK_\infty^p\coloneqq \varinjlim_{m\to \infty} \left(\Sym_\Q^{m-p}(\mathcal O_X)\otimes\left(\bigwedge_\Q^p\mathcal M_{X,D}^{\mathrm{gp}}\otimes_\Z\right)\right),
	\quad\wt W_m 	\cK_\infty^\bullet\coloneqq\left\{\begin{array}{lr} \cK_m^\bullet & \text{if }m\leq\dim_\C X,\\
		\cK_\infty^\bullet & \text{otherwise.}
	\end{array}\right. ,\\ &\f{\varphi_\infty}{\quad\quad\quad\quad\quad\quad\cK_\infty^p\quad\quad\quad\quad\quad\quad}{\Omega_X^p(\log D)}{g_1\cdot\ldots\cdot g_{m-p}\otimes y_1\wedge\ldots\wedge y_p}{\frac{1}{(2\pi i)^p}g_1\cdot\ldots\cdot g_{m-p}\frac{dy_1}{y_1}\wedge\ldots\wedge\frac{dy_p}{y_p}}
\end{align*}
where $\mathcal O_X$ is the sheaf of holomorphic functions on $X$ and $\mathcal M_{X,D}^{\mathrm{gp}}$ is the sheaf of abelian groups associated to $\mathcal M_{X,D}\coloneqq \mathcal O_X\cap j_*\mathcal O_U^*$, where $\mathcal O_U^*$ is the sheaf of non-vanishing holomorphic functions on $U$, as a sheaf of groups under multiplication.

\begin{remark}
	Note that $f_j\in \Gamma(X,j_*\cO_U^*)$ for all $j=1,\ldots, n$. These functions can be extended to $X$ as the quotient of two holomorphic functions, and thus $f_j\in\Gamma(X, \mathcal M_{X,D}^{\mathrm{gp}})$ for all $j=1,\ldots,n$.
\end{remark}

\begin{definition}
	We define the morphism $\Phi_\Q$ as
	$$
	\f{\Phi_\Q}{\left(H^1((\C^*)^n,\Q),W_{\lc}[1]\right)}{\Gamma\left(X,(\cK_{\infty}^{1,\cl},\wt W_{\lc})\right)}{\left[\frac{1}{2\pi i}\frac{dz_i}{z_i}\right]}{1\otimes f_i.}$$
\end{definition}

Clearly, $\Phi_\Q$ preserves the weights and, using \eqref{eq:explicitPhiPsi}, it is straightforward to see that $\varphi_\infty\circ\Phi_\Q=\ov f^*\circ\Phi_\C^{(\mathbb P^1)^n}$ in $H^1((\C^*)^n,\Q)$. In particular, we may apply Definition-Proposition~\ref{defprop:thickMHC} to get the following mixed Hodge complex of sheaves.

\begin{definition}[Thickened rational mixed Hodge complex of sheaves]\label{def:thickQMHC}
	Let $m\in\Z\setminus\{0\}$. The following is a $\Q$-mixed Hodge complex of sheaves in $X$:
	$$
	\left(\left((R_m\otimes_\Q \cK^\bullet_\infty,d+\Phi_\Q(\eps_\Q)), W_{\lc}\right), \left((R_m\otimes_\C \Omega_{X}^\bullet(\log D), d+\ov f^*\circ\Phi^{(\mathbb P^1)^n}_\C(\eps_\C)), W_{\lc}, F^{\lc}\right), \Id\otimes\varphi_\infty\right),
	$$
	where the filtrations are the tensor filtrations corresponding to $R_m$ and the mixed Hodge complex \eqref{eq:rationalMHC}, and $\Id\otimes\varphi_\infty:R_m\otimes_\Q \cK^\bullet_\infty\to R_m\otimes_\C \Omega_{X}^\bullet(\log D)$ is a quasi-isomorphism after tensoring the domain with $\C$ over $\Q$.
\end{definition}

\begin{remark}
	Let $m\in\Z\setminus\{0\}$, and let $n'\geq\max\{2,\dim_\R U\}$. Note that the mixed Hodge complex of sheaves from Definition~\ref{def:thickQMHC} can be given an extra term so that its complex part coincides with the complex part in the mixed Hodge complex of sheaves $\left(R_m\otimes\cN^\bullet_{X,D,n'}, d+\ov f^*\circ\Phi^{(\mathbb P^1)^n}(\eps)\right)$ of Definition~\ref{def:thickening}. Indeed,
 by Proposition~\ref{prop:qisoThickenings}, the composition of the bi-filtered quasi-isomorphisms from Theorem~\ref{thm:propertiesNA} and Definition-Proposition~\ref{defprop:modifiedNA}
		$(\Omega_X^\bullet(\log D),W_{\lc},F^{\lc})\hookrightarrow(\cA^\bullet_{X,\C}(\log D),W_{\lc}^{n'},F^{\lc})$ given by inclusion extends to a bi-filtered quasi-isomorphism between the complex parts of the thickened complexes of Definitions~\ref{def:thickQMHC}  and~\ref{def:thickening}
		$$
		\left((R_{m}\otimes_{\C}\Omega_X^\bullet(\log D),d+\ov f^*\circ \Phi_{\C}^Y(\eps_{\C})),W_{\lc},F^{\lc}\right)\hookrightarrow\left((R_{m}\otimes_{\C}\cA^\bullet_{X,\C}(\log D),d+\ov f^*\circ \Phi_{\C}^Y(\eps_{\C})),W_{\lc}^{n'},F^{\lc}\right)
		$$
\end{remark}

\begin{notation}\label{not:LQ}
	Let $\ov\cL_\Q$ (resp. $\ov\cL_\R$) be as $\ov\cL$ in Definition~\ref{def:L} but with $\Q$ (resp. $\R$) coefficients. For $m\in\Z\setminus\{0\}$, $R_m$ (resp. $R$)  in the expression $R_m\otimes_R\ov\cL_\Q$ will be as in Definition~\ref{def:Rminfty} for $k=\Q$  (resp. $\Q\left[\pi_1\left((\C^*)^n\right)\right]$), and similarly for $\R$-coefficients.
\end{notation}

We now introduce some notation. Let $k=\Q, \R$ or $\C$, depending on context. Let $s_j^\vee\coloneqq \left[\frac{1}{2\pi i}\frac{dz_j}{z_j}\right]$, let $\{s_j\mid j=1,\ldots,n\}\subset H_1((\C^*)^n,k)$ be the dual basis of $\{s_j^\vee\mid j=1,\ldots,n\}$. Let $\gamma_j$ be a loop around the origin in the $j$-th coordinate $\C^*$ of $(\C^*)^n$ for all $j=1,\ldots, n$. Identify $T(\C^*)^n$ with $\C^n$ (with coordinates $(w_1,\ldots, w_n)$) in a way that $\exp(w_1,\ldots,w_n)=(e^{w_1},\ldots, e^{w_n})$. With those identifications, $\log \gamma_j$ is seen in $T(\C^*)^n$ as $2\pi i e_j$, where $e_j$ is the $j$-th element of the canonical basis of $\C^n$. Denote by $x_j$ and $y_j$ the real and imaginary parts of $z_j$. Note that $\{e_1,\ldots,e_n\}\cup\{\log \gamma_i,\ldots,\log\gamma_n\}$ form an $\R$-basis of $T(\C^*)^n$. Note that $\Phi_\R^{(\C^*)^n}\left(\left[\frac{1}{2\pi i}\frac{dz_j}{z_j}\right]\right)=\Re\frac{1}{2\pi i} \frac{dz_j}{z_j}$, which, at the identity element of $(\C^*)^n$ takes the value $\frac{1}{2\pi}dy_j$. The form  $\frac{1}{2\pi}dy_j$  takes the value $1$ in $\log \gamma_j$ and $0$ in the rest of the elements of the fixed $\R$ basis of $T(\C^*)^n$. Hence, under these identifications, $(\Phi_\R^{(\C^*)^n})^\vee$ (as introduced in Construction~\ref{con:nu}) takes the following form:
$$
\begin{array}{cccc}
(\Phi_\R^{(\C^*)^n})^\vee: &TG&\longrightarrow&H_1(G,\R)\\
\ & \log \gamma_j &\longmapsto & \log \gamma_j=s_j\\
\ & e_j &\longmapsto & 0
\end{array}
$$
Let $V$ be a small open set in $U$ and let $\iota:V\to T(\C^*)^n$ be a holomorphic map such that $\exp\circ\iota=f$. Note that such $\iota$ form a $\Q$-basis of $\ov\cL_\Q$ in $V$. Under the identifications above, $\iota=(\iota_1,\ldots,\iota_n)$, where $\iota_j:V\to \C$ is holomorphic. Hence, it makes sense to talk about $\Re\iota_j$ and $\Im \iota_j$ for all $j=1,\ldots, n$. Notice that $\exp(\iota_j)=f_j$, so $\exp(\Re\iota_j)=|f_j|$.

We now do the analogue of Construction~\ref{con:nu} but for $\Q$-coefficients.
\begin{construction}\label{con:nuQ}
	Let $m\in\Z\setminus\{0\}$. Then, we can define an $R=\Q\left[\pi_1\left(\C^*\right)^n\right]$-linear morphism of sheaves $\nu_\Q:R_m\otimes_R\ov\cL_\Q\to j^{-1}(R_m\otimes_\Q \cK^0_\infty)$ locally by
	$$
	\f{\nu_\Q}{R_m\otimes_R\ov\cL_\Q}{j^{-1}(R_m\otimes_\Q \cK^0_\infty)}{\alpha\otimes\iota}{\alpha \exp\left(-\frac{1}{2\pi i}\sum\limits_{j=1}^n(s_j\otimes(\iota_j\otimes 1))\right).}
	$$
\end{construction}

The proof that $\nu_\Q$ is well-defined on the tensor product (over $R$) and that it is $R$-linear follows similar steps as its analogue for $\R$-coefficients (Proposition~\ref{prop:nuR-linear}), so we omit it. This time, it needs to use that for all $a_1,\ldots,a_n\in \Z$, $(\prod_{j=1}^n\gamma_j^{a_j})\cdot\iota=(\iota_1-2\pi ia_1,\ldots,\iota_n-2\pi ia_n)$. 

\begin{proposition}\label{prop:Qcomparison}
	Let $m\in\Z\setminus\{0\}$. Then, the restriction of the morphism $e^{f^*\circ\Psi^{(\C^*)^n}(\eps_\C)}\cdot e^{-(\Phi_\R^{(\C^*)^n})^\vee}:R_m\otimes_R \ov\cL_\R\to j^{-1}\left(R_m\otimes_\C \cA^0_{X,\C}(\log D)\right)$ to $R_m\otimes_R \ov\cL_\Q$ coincides with the composition $(\Id\otimes\varphi_\infty)\circ\nu_\Q$.
\end{proposition}

\begin{proof}
	The proof is a direct computation:
\begin{align*}
	e^{f^*\circ\Psi^{(\C^*)^n}(\eps_\C)}\cdot e^{-(\Phi_\R^{(\C^*)^n})^\vee}(\alpha\otimes \iota)&=\alpha\cdot e^{f^*\circ\Psi^{(\C^*)^n}(\eps_\C)}\cdot e^{-(\Phi_\R^{(\C^*)^n})^\vee\left(\sum_{j=1}^n e_j\otimes\Re\iota_j+\frac{1}{2\pi}\sum_{j=1}^n \log\gamma_j\otimes\Im\iota_j\right)}\\
	&=\alpha\cdot\exp\left(-\frac{1}{2\pi i}\sum_{j=1}^ns_j\otimes\log(|f_j|)\right)\cdot \exp\left(-\frac{1}{2\pi}\sum_{j=1}^n s_j\otimes\Im\iota_j\right)\\
	&=\alpha\cdot\exp\left(-\frac{1}{2\pi i}\sum_{j=1}^ns_j\otimes(\Re\iota_j+i\Im\iota_j)\right)\\
	&=\alpha\cdot\exp\left(-\frac{1}{2\pi i}\sum_{j=1}^ns_j\otimes\iota_j\right)=(\Id\otimes\varphi_\infty)\circ\nu_\Q(\alpha\otimes \iota)
\end{align*}
\end{proof}

\begin{remark}
	Since $\Id\otimes\varphi_\infty$ (resp. $e^{f^*\circ\Psi^{(\C^*)^n}(\eps_\C)}$) is a quasi-isomorphism when the domain is tensored by $\C$ over $\Q$ (resp. over $\R$) and $\C$ is faithfully flat over $\Q$ (resp. over $\R$), we have that $j^{-1}\left((R_m\otimes_\Q \cK^\bullet_\infty, d+\Phi_\Q(\eps_\Q))\right)$ resolves a free rank $1$ $R$-local system. Using Proposition~\ref{prop:Qcomparison} and Lemma~\ref{lem:nuResolves} we get that $\nu_\Q$ induces a quasi-isomorphism
	$$
	\nu_\Q:R_m\otimes_R\ov\cL_\Q\to j^{-1}\left(R_m\otimes_\Q \cK^\bullet_\infty, d+\Phi_\Q(\eps_\Q)\right).
	$$
	In particular, the mixed Hodge complex of sheaves from Definition~\ref{def:thickQMHC} endows $H^*(U,R_m\otimes_R\ov\cL_\Q)$ with a $\Q$-MHS following the same steps as in Definition~\ref{def:endowedMHS} (with the same shifts if $m>0$), using the adjunction $\Id\to Rj_*j^{-1}$.
\end{remark}

\begin{corollary}\label{cor:QMHS}
	Let $m\in\Z\setminus\{0\}$, and suppose that $G=(\C^*)^n$ for some $n\geq 1$. The MHS on $H^*(U,R_m\otimes_R\ov\cL_\R)$ from Definition~\ref{def:endowedMHS} is defined over $\Q$. In particular, all of the MHSs defined in this paper are defined over $\Q$ in this case, and the results in Sections~\ref{sec:MHS}, \ref{sec:functoriality}, \ref{sec:otherideals} and~\ref{sec:eigenspace} also hold for $\Q$-coefficients.
\end{corollary}
\begin{proof}
	 Proposition~\ref{prop:Qcomparison} implies that, after tensoring by $\R$ over $\Q$, the MHS on $H^*(U,R_m\otimes_R\ov\cL_\Q)$ induced by  the mixed Hodge complex of sheaves from Definition~\ref{def:thickQMHC} coincides with the MHS from Definition~\ref{def:endowedMHS}.
\end{proof}

\begin{remark}[The case $G=\C^*$]\label{rem:comparisonC^*}
	Suppose that $m>0$ and that $G=\C^*$. Let $s$ be a positive oriented loop around the origin in $H_1\left((\C^*)^n,\Q\right)$, and use it to identify $R_m$ with $\Q[s^{\pm 1}]/(s^m)$. Under the identifications and choice of coordinates explained in this section, the mixed Hodge complex of Definition~\ref{def:thickQMHC} and $\nu_\Q$ coincide with those of \cite{mhsalexander} (see Remark 5.12 and Theorem 5.24 therein). Therefore, the $\Q$-MHS on $H^*(U,R_m\otimes_R \ov\cL)$ from both papers is the same. Note that, in \cite{mhsalexander}, this MHS (with the same Tate twist as in Definition~\ref{def:endowedMHS}) was used to endow $\Tors_R H^*(U,\ov\cL)$ with a canonical MHS.
\end{remark}

\section{Relationship with the Milnor fiber of a central hyperplane arrangement complement}\label{sec:hyperplanes}

Let $f_i\in\C[x_1,\ldots,x_n]$ be homogeneous polynomials of degree $1$ for $i=1,\ldots, m$ such that if $i\neq j$, $f_i$ is not a product of $f_j$ by a constant. Suppose that $m>n$. Let $f=\prod_{i=1}^m f_i^{d_i}$ for some $d_i\geq 1$, let $d=\sum_{i=1}^m d_i$, let $H_i=V(f_i)\subset \C^n$, and let  $H=\cup_{i=1}^m H_i$. The $f_i$'s describe a central hyperplane arrangement in $\C^n$, but if we think of it as being determined by $f$, the arrangement is not necessarily reduced.

The Milnor fiber of $f$ is $f^{-1}(1)$, and it is equipped with the monodromy action
$$
\begin{array}{ccc}
	f^{-1}(1)&\longrightarrow& f^{-1}(1)\\
	(x_1,\ldots,x_n)&\longmapsto& (\xi x_1,\ldots,\xi x_n),
\end{array}
$$
where $\xi=e^{\frac{2\pi i}{d}}$. Note that this induces a semisimple action on the reduced homology groups $\wt H^j(f^{-1}(1),\C)$, and its possible eigenvalues are the $d$-th roots of unity.

\begin{definition}[Spectrum of $f$]\label{def:spectrum}
	The spectrum of $f$ is defined by  $\Sp(f)=\sum_{\alpha\in\Q}n_{f,\alpha} t^\alpha$, where
	\begin{itemize}
		\item $n_{f,\alpha}=\sum_{j}(-1)^{j-n+1}\dim_\C \Gr^p_F \wt H^j(f^{-1}(1),\C)_\lambda$,
		\item $\wt H^j(f^{-1}(1),\C)_\lambda$ is the eigenspace of eigenvalue $\lambda$ by the monodromy action on the reduced cohomology groups $\wt H^j(f^{-1}(1),\C)$,
		\item $\lambda=e^{-2\pi i\alpha}$ and
		\item $p=\lfloor n-\alpha\rfloor$.
	\end{itemize}
\end{definition}

The spectrum of a hypersurface singularity was first defined by Steenbrink \cite{steenbrinkSpectrum} as a local invariant of the Hodge filtration of the cohomology of the local Milnor fiber, but in the case of central hyperplane arrangements, the Milnor fibration corresponding to the singularity at the origin comes from a global fibration of the hyperplane complement over $\C^*$, and Definition~\ref{def:spectrum} coincides with Steenbrink's.

\begin{remark}\label{BudurSaito}
	Budur and Saito showed in \cite{budurSaito} that $\Sp(f)$ depends only on the combinatorial data of the (not necessarily reduced) arrangement defined by $f$. 
\end{remark}

Despite this positive result of Budur and Saito, one of the most important open problems of arrangement theory is the following.
\begin{question}\label{question:BettiGen}
	Are the Betti numbers of the Milnor fiber associated to a (reduced) central hyperplane  arrangement in $\C^n$ determined by the combinatorics of the arrangement?
\end{question}
This has been solved if $n=3$ and the projectivized arrangement in $\mathbb P^2$ only has double and triple points by Papadima and Suciu in \cite{PapaSuciu}. However, a general answer to this question is not known even in this particular case:
\begin{question}\label{question:Betti}
	Is the first Betti number of the Milnor fiber associated to a (reduced) central hyperplane  arrangement in $\C^n$ determined by the combinatorics of the arrangement?
\end{question}

In this section, we translate Question~\ref{question:Betti} to a question of whether the dimensions of the filtered pieces by the Hodge filtration of a MHS defined in this paper are combinatorially determined, motivating the future study of the objects introduced in this note.

Suppose that one wants to study Question~\ref{question:Betti}. It is enough to consider the case where the arrangement is essential and the number of hyperplanes is greater than the dimension of the ambient space. Indeed, if the arrangement is not essential, the complement has the homotopy type of a central essential line arrangement in an affine space of smaller dimension. If the arrangement is essential but the number of hyperplanes equals that of the ambient space, the arrangement complement is isomorphic to $(\C^*)^n$, so these arrangements all have the same combinatorics and all have isomorphic Milnor fibers.

Lemma~\ref{lem:reduction} below details the relation between Question~\ref{question:Betti} and the following stronger question. \begin{question}
	Let $f$ be the reduced defining polynomial of a (not necessarily central) essential line arrangement in $\C^2$ of $3$ or more lines, and let $U$ be the corresponding arrangement complement in $\C^2$. Let $\pi:U^f\to U$ be the pullback of $\exp:\C\to \C^*$ by $f:U\to\C^*$ as in \eqref{eq:Uf}. Is the first Betti number of the infinite cyclic cover $\pi:U^f\to U$  of an essential line arrangement complement $U$ in $\C^2$ with $3$ or more lines  determined by the combinatorics of the arrangement?
\end{question}

\begin{lemma}\label{lem:reduction}
	Let $\{H_1,\ldots,H_m\}$ be a (reduced) central arrangement of $m$ different hyperplanes in $\C^n$, where $m>n$, and let $\{L_1,\ldots,L_m\}$ a (reduced, not necessarily central) line arrangement in $\C^2$ which is obtained from $\{H_1,\ldots,H_m\}$ after intersection with $n-2$ generic hyperplanes. Let $f(x,y)\in\C[x,y]$ be a reduced defining polynomial of $\cup_{i=1}^m L_i$, and let $U=\C^2\setminus V(f)$. Let $U^f\to U$ be the pullback by $f$ of $\exp:\C\to\C^*$.
	
	If $\dim_{\C} H_1(U^f,\C)$ is determined by the combinatorics of $\{L_1,\ldots,L_m\}$, then the first Betti number of the Milnor fiber of $\{H_1,\ldots,H_m\}$ is determined by the combinatorics of $\{H_1,\ldots,H_m\}$.
	
	Moreover, the reverse implication holds if $n=3$.
\end{lemma}
\begin{proof}
	Note that the combinatorics of $\{H_1,\ldots,H_m\}$ determines the combinatorics of $\{L_1,\ldots,L_m\}$ and the reverse implication also holds if $n=3$. By the Lefschetz hyperplane section theorem, $\C^n\setminus\left(\cup_{j=1}H_j	\right)$ can be obtained (up to homotopy equivalence) by attaching cells of dimensions $3$ and higher to $U$. Consider the commutative diagram
	$$
	\begin{tikzcd}
		U\arrow[rr, hook]\arrow[rd, "f"]& & \C^n\setminus\left(\cup_{j=1}H_j	\right)\arrow[ld, "f_2"]\\
		& \C^*
	\end{tikzcd}
	$$
	(where $f_2$ is a reduced defining polynomial of the arrangement $\{H_1,\ldots,H_m\}$). Since $f_2$ is homogeneous of degree $m$,  $(\C^n\setminus\left(\cup_{j=1}H_j	\right))^{f_2}\cong (f_2)^{-1}(1)\times\C$. Hence, the inclusion in the commutative diagram above induces an isomorphism $H_1(U^f,\C)\to H_1((f_2)^{-1}(1),\C)$. If the dimension of $H_1(U^f,\C)$ is determined by the combinatorics of $\{L_1,\ldots,L_m\}$, then the dimension of the first Betti number of the Milnor fiber of the arrangement $\{H_1,\ldots,H_m\}$ is determined by the combinatorics of $\{H_1,\ldots,H_m\}$.
	%
\end{proof}

%

From now on in this section, this will be our setting: $m\geq 3$, $L_1,L_2,\ldots,L_m$ form an essential arrangement of  $m$ different lines in $\C^2$, and $f_i\in\C[x,y]$ is a polynomial of degree $1$ such that $L_i=V(f_i)$ for all $i=1,\ldots,m$. Let $f=\prod_{i=1}^m f_i$, and let $U\coloneqq \C^2\setminus\left(\cup_{i=1}^m L_i \right)$. The infinite cyclic cover $\pi:U^f\to U$ is constructed as the pullback of $\exp:\C\to \C^*$ by $f:U\to\C^*$. Moreover, we identify $R=\C[\pi_1(\C^*)]$ with $\C[t^{\pm 1}]$ by taking a positively oriented loop around the origin to $t$.

\begin{remark}\label{rem:Nhyperplanes}
By \cite[Proposition 2.24, Corollary 7.21]{mhsalexander}, there exists $N\in \N$ such that $t^N-1$ annihilates $H_1(U^f,\C)$. Moreover, by \cite[Theorem 5]{evaHyperplanes}, $N$ can be taken to be the least common multiple of all the numbers which are greater than 2 and appear as multiplicities of multiple points in the arrangement (so in particular, this choice of $N$ is combinatorially determined).
\end{remark}

\begin{remark}\label{rem:Nvsd}
	Let $N$ be as in Remark~\ref{rem:Nhyperplanes} (given by the least common multiple of the non-2  multiplicities of the multiple points in the arrangement), and suppose that $N<m=\deg f$. We may substitute $N$ by $\min\{Nk\mid k\in\N,\quad Nk> m\}$, which is again combinatorially determined.
\end{remark}

Let $N$ as in Remark~\ref{rem:Nvsd} (a combinatorially determined number such that $N\geq m$ and $t^N-1$ annihilates $H_1(U^f,\C)$), and let $\pi_N:U_N\to U$ be the covering space of $U$ obtained via the pullback diagram

\begin{equation}\label{eq:UN}
	\begin{tikzcd}
		U_N\subset U\times \C^*  \arrow[r,"f_N"] \arrow[d,"\pi_N"]\arrow[dr,phantom,very near start, "\lrcorner"]&
		\C^* \arrow[d,"w\mapsto w^N"] \\
		U\arrow[r,"f"] &
		\C^*,
	\end{tikzcd}
\end{equation}

Notice that if we see $U$ as the affine variety $V(z\cdot f(x,y)-1)\subset\C^3$, then
\begin{align*}
	U_N&=\{(x,y,z,w)\in \C^3\times\C^*\mid zf(x,y)=1,\quad f(x,y)=w^N\}\\
	&\cong\{(x,y,w)\in \C^2\times\C^*\mid f(x,y)=w^N\},
\end{align*}
Let $\wt f(x,y,z)$ be the homogenization of the polynomial $f(x,y)$. The following is an isomorphism of algebraic varieties
\begin{equation}\label{eq:isoMilnor}
	\begin{array}{ccc}
		V(z^{N-m}\wt f(x,y,z)=1)\subset\C^3&\longleftrightarrow &U_N=\{(x,y,w)\in \C^2\times\C^*\mid f(x,y)=w^N\}\subset\C^3\\
		(x,y,z)&\longmapsto &\left(\frac{x}{z},\frac{y}{z},\frac{1}{z}\right)\\
	\end{array}
\end{equation}
and under this identification, $\pi_N$ and $f_N$ in the pullback diagram~\eqref{eq:UN} become
$$
\f{\pi_N}{U_N=V(z^{N-m}\wt f(x,y,z)=1)\subset\C^3}{U=\C^2\setminus V( f)\subset\C^2}{(x,y,z)}{\left(\frac{x}{z},\frac{y}{z}\right)}
$$
and
$$
\f{f_N}{U_N=V(z^{N-m}\wt f(x,y,z)=1)\subset\C^3}{\C^*}{(x,y,z)}{\frac{1}{z}}
$$
respectively. Note that $U_N$ is a (possibly non-reduced) Milnor fiber of an essential central hyperplane arrangement in $\C^3$, so it makes sense to talk about $\Sp(z^{N-m}\wt f(x,y,z))$. Under the identification of $R$ with $\C[t^{\pm 1}]$, the $t$-action on $U_N$ given by Deck transformations of $\pi_N$ is
$$
\f{t}{U_N=V(z^{N-m}\wt f(x,y,z)=1)\subset\C^3}{U_N}{(x,y,z)}{e^{-\frac{2\pi i}{N}}(x,y,z)},
$$
which is the inverse of the monodromy of the Milnor fiber.

\begin{theorem}\label{thm:combinatorics}
	Let $\{L_1,\ldots,L_m\}$ be a (reduced) essential line arrangement in $\C^2$, with $m\geq 3$. Let $f(x,y)\in\C[x,y]$ be a reduced defining polynomial of $\cup_{i=1}^m L_i$, and let $U=\C^2\setminus V(f)$. Let $\cL=f^{-1}\exp_!\ul \C_\C$, and let $N$ as Remark~\ref{rem:Nvsd}, which is determined by the combinatorics of $\{L_1,\ldots,L_m\}$. Then, the following hold.
	\begin{enumerate}
		\item\label{combinatorics0} $\dim_{\C}\Gr^{-p}_F \frac{H_2(U^f,\R)}{(t^N-1)H_2(U^f,\R)}\neq 0$ $\Rightarrow$ $p=0,1,2$.
		\item\label{combinatorics1} $\dim_{\C}\Gr^{-1}_F \frac{H_2(U^f,\R)}{(t^N-1)H_2(U^f,\R)}$ is determined by the combinatorics of $\{L_1,\ldots,L_m\}$.
		\item\label{combinatorics2} $H_2(U^f,\R)$ is a free $\C[t^{\pm1}]$-module of rank $\chi(U)$, so $\dim_{\C} \frac{H_2(U^f,\R)}{(t^N-1)H_2(U^f,\R)}=N\chi(U)$, which is determined by the combinatorics of $\{L_1,\ldots,L_m\}$.
		\item\label{combinatorics3} The following are equivalent:
		\begin{itemize}
			\item $\dim_{\C} F^{-1} \frac{H_2(U^f,\C)}{(t^N-1)H_2(U^f,\C)}$ is determined by the combinatorics of $\{L_1,\ldots,L_m\}$.
			\item $\dim_{\C} F^{0} \frac{H_2(U^f,\C)}{(t^N-1)H_2(U^f,\C)}$ is determined by the combinatorics of $\{L_1,\ldots,L_m\}$.
			\item $\dim_\C  H_1(U^f,\C)=\dim_\C \Tors_R H_1(U^f,\C)$ is determined by the combinatorics of $\{L_1,\ldots,L_m\}$.
		\end{itemize}
	\end{enumerate}
\end{theorem}
\begin{proof}
	The statement in \eqref{combinatorics2} is true by \cite[Theorem 4]{evaHyperplanes}. The covering $\pi:U^f\to U$ factors through $U_N$ as $\pi_N\circ\pi'$, where $\pi':U^f\to U_N$ is a covering space. The short exact sequence at the level of singular chains
	$$
	0\to C_\bullet(U^f)\xrightarrow{t^N-1} C_\bullet(U^f)\xrightarrow{(\pi')_*}C_\bullet(U_N)\to 0
	$$
	yields the Milnor long exact sequence in homology, which we claim gives rise to the following exact sequences of MHS, where $(1)$ denotes a Tate twist.
\begin{align*}
	&0\to \frac{H_2(U^f,\C)}{(t^N-1)H_2(U^f,\C)}\to H_2(U_N,\C)\to \Tors_R H_1(U^f,\C)(1)\to 0,	\\
		&0\to \Tors_R H_1(U^f,\C)=H_1(U^f,\C)\to H_1(U_N,\C)\to H_0(U^f,\C)=\Tors_R H_0(U^f,\C)(1)\to 0,\\
		&0\to H_0(U^f,\C)=\Tors_R H_0(U^f,\C)\to H_0(U_N,\C)\to  0,
\end{align*}
Let us see that these are indeed short exact sequences of MHS, where $\Tors_R H_i(U^f,\C)$ is endowed with the MHS from \cite{mhsalexander}, $H_1(U_N,\C)$ is endowed with Deligne's MHS, and $\frac{H_2(U^f,\C)}{(t^N-1)H_2(U^f,\C)}$ is endowed with the MHS from Proposition~\ref{prop:otherideals}. Indeed, $H_i(U^f,\C)=\Tors_R H_i(U^f,\C)$ for $i=0,1$ by \cite[Theorem 4]{evaHyperplanes}. We have that  $\frac{H_2(U^f,\C)}{(t^N-1)H_2(U^f,\C)}\to H_2(U_N,\C)$ is a MHS morphism by Corollary~\ref{cor:compatibleDeligne} and Proposition~\ref{prop:otherideals}. The remaining maps are shown to be MHS morphisms in \cite[Corollary 5.9]{EvaMoises}.

The statement in \eqref{combinatorics0} holds by observing the first of these sequences, because since $U_N$ is smooth, the analogous statement holds for $H_2(U_N,\C)$ (see \cite[Corollaire 3.2.15]{DeligneII}).

The $t$-action is semisimple in all of the homology groups appearing in the three exact sequences above, in fact, $t^N$ acts as the identity. In the case of $U_N$, it acts by deck transformations realized by an algebraic isomorphism, so it is a MHS isomorphism in homology. By Theorem~\ref{thm:ss} (and its counterpart for the torsion in \cite[Theorem 1.3]{EvaMoises}) multiplication by $t$ is a MHS isomorphism for the rest of those homology groups. The exact sequences above induce exact sequences of MHSs in the corresponding eigenspaces, which in turn induce the following exact sequences:
	\begin{align*}
&	0\to \Gr_F^{-p} \left(\frac{H_2(U^f,\C)}{(t^N-1)H_2(U^f,\C)}\right)_\lambda\to \Gr_F^{-p} H_2(U_N,\C)_\lambda\to \Gr_F^{-p+1}\left(\Tors_R H_1(U^f,\C)\right)_\lambda\to 0,\\
&	0\to \Gr_F^{-p}\left(\Tors_R H_1(U^f,\C)\right)_\lambda\to \Gr_F^{-p} H_1(U_N,\C)_\lambda\to \Gr_F^{-p+1} \left(\Tors_R H_0(U^f,\C)\right)_\lambda\to 0,\\
	&\Gr_F^{-p} \left(\Tors_R H_0(U^f,\C)\right)_\lambda\cong 	\Gr_F^{-p} H_0(U_N,\C)_\lambda.
\end{align*}
By \cite[Theorem 10.5]{mhsalexander}, the following hold:
\begin{itemize}
	\item $\Tors_R H_1(U^f,\C)_{\neq 1}$ is a pure Hodge structure of weight $-1$, where the subindex $\neq 1$ denotes the direct sum of all of the eigenspaces of eigenvalue other than $1$. Hence, $$\dim_\C\Gr_F^0\left(\Tors_R H_1(U^f,\C)\right)_{\neq 1}=\dim_\C\Gr_F^{-1}\left(\Tors_R H_1(U^f,\C)\right)_{\neq 1}.$$
	\item $\left(\Tors_R H_1(U^f,\C)\right)_{1}$ is a pure Hodge structure of type $(-1,-1)$, so its only non-zero graded piece is $\Gr_F^{-1}\left(\Tors_R H_1(U^f,\C)\right)_{1}$. Moreover, this graded piece has dimension $m-1$, which is combinatorially determined.
	\item $\Tors_R H_0(U^f,\C)$ is a Hodge structure of weight $0$ and dimension $1$. The only nontrivial eigenspace is the eigenspace of eigenvalue $1$, which has dimension $1$, and its only nontrivial graded piece is $\Gr_F^{0}\left(\Tors_R H_0(U^f,\C)\right)_{1}$.
\end{itemize} 

 Recall from Remark~\ref{BudurSaito} that, by Budur and Saito's result, $n_{\wt f,\alpha}:=\sum_j (-1)^j\dim_\C \Gr_F^p \wt H^j(U_N,\C)_{\lambda}$ is a combinatorial invariant, where $\wt f(x,y,z)\coloneqq z^{N-m}\ov f(x,y,z)$, $\lambda=e^{2\pi i \alpha}$, $\wt H^j(U_N,\C)_{\lambda}$ is the \linebreak eigenspace of eigenvalue $\lambda$ for the $t$-action on the reduced cohomology groups $\wt H^j(U_N,\C)$ (which is the inverse of the monodromy action), and $p=\lfloor 3-\alpha\rfloor$. Let $\lambda_l\coloneqq e^{2\pi i\frac{l}{N}}$ for $l=1,\ldots, N$ be all the $N$-th roots of unity. We have that 
\begin{equation}
	n_{\wt f,\frac{l}{N}+(2-p)}=\dim_\C \Gr_F^{-p} H_2(U_N,\C)_{\lambda_l}-\dim_\C \Gr_F^{-p} H_1(U_N,\C)_{\lambda_l}
\end{equation}
is combinatorially determined for all $p=0,1,2$ (the only possibly non-zero graded pieces). Now, using the exact sequences above, we get that, for all  $l\neq N$,
	 \begin{align*}
	 	n_{\wt f,\frac{l}{N}+2}=&\dim_\C \Gr_F^{0} \left(\frac{H_2(U^f,\C)}{(t^N-1)H_2(U^f,\C)}\right)_{\lambda_l}-\dim_\C \Gr_F^{0} \left(\Tors_R H_1(U^f,\C)\right)_{\lambda_l},\\
	 n_{\wt f,\frac{l}{N}+1}=&\dim_\C \Gr_F^{-1} \left(\frac{H_2(U^f,\C)}{(t^N-1)H_2(U^f,\C)}\right)_{\lambda_l}+\dim_\C \Gr_F^{0} \left(\Tors_R H_1(U^f,\C)\right)_{\lambda_l}\\
	 &-\dim_\C \Gr_F^{-1} \left(\Tors_R H_1(U^f,\C)\right)_{\lambda_l},
\end{align*}
and for $l=N$,
	 \begin{align*}
	 		n_{\wt f,\frac{l}{N}+2}&=\dim_\C \Gr_F^{0} \left(\frac{H_2(U^f,\C)}{(t^N-1)H_2(U^f,\C)}\right)_{\lambda_l},\\
	n_{\wt f,\frac{l}{N}+1}&=\dim_\C \Gr_F^{-1} \left(\frac{H_2(U^f,\C)}{(t^N-1)H_2(U^f,\C)}\right)_{\lambda_l}-m,
\end{align*}
Hence, $\dim_\C \Gr_F^{-1} \frac{H_2(U^f,\C)}{(t^N-1)H_2(U^f,\C)}=\left(\sum_{l=1}^N n_{\wt f,\frac{l}{N}+1}\right)+m$ is combinatorially determined, which concludes the proof of the statement in \eqref{combinatorics1}. For the statement in \eqref{combinatorics3}, just note that
$$
\dim_{\C} \Gr_F^{-1} \frac{H_2(U^f,\R)}{(t^N-1)H_2(U^f,\R)}=\dim_{\C} F^{-1} \frac{H_2(U^f,\R)}{(t^N-1)H_2(U^f,\R)}-\dim_{\C} F^{0} \frac{H_2(U^f,\R)}{(t^N-1)H_2(U^f,\R)}
$$
and that
\begin{align*}
	\dim_{\C} F^{0} \frac{H_2(U^f,\C)}{(t^N-1)H_2(U^f,\C)}&=\dim_{\C} \Gr_F^{0} \frac{H_2(U^f,\C)}{(t^N-1)H_2(U^f,\C)}\\
	&=\left(\sum_{l=1}^N
	n_{\wt f,\frac{l}{N}+2}\right)-\frac{1}{2}\dim_\C \left(\Tors_R H_1(U^f,\C)\right)_{\neq 1}\\
	&=\left(\sum_{l=1}^N
	n_{\wt f,\frac{l}{N}+2}\right)-\frac{1}{2}\left(\dim_\C \Tors_R H_1(U^f,\C)-(m-1)\right).
\end{align*}
\end{proof}

\begin{remark}
	Note that, by  Theorem~\ref{thm:combinatorics}, $\frac{H_2(U^f,\C)}{(t^N-1)H_2(U^f,\C)}$ is a space whose dimension is combinatorially determined, and the dimension of one out of its three possible non-zero graded pieces by the Hodge filtration is also combinatorially determined.
\end{remark}

The following corollary summarizes the work done in this section.

\begin{corollary}\label{cor:summary}
		Let $\{L_1,\ldots,L_m\}$ be a (reduced) essential line arrangement in $\C^2$, with $m\geq 3$. Let $f(x,y)\in\C[x,y]$ be a reduced defining polynomial of $\cup_{i=1}^m L_i$, and let $U=\C^2\setminus V(f)$. Let  $N$ as Remark~\ref{rem:Nvsd}, which is determined by the combinatorics of $\{L_1,\ldots,L_m\}$. Consider the MHS on $\frac{H_2(U^f,\R)}{(t^N-1)H_2(U^f,\R)}$ from Definition~\ref{def:MHSalexander}. Then, if $\dim_{\C} F^{0} \frac{H_2(U^f,\C)}{(t^N-1)H_2(U^f,\C)}$ is always determined by the combinatorics of $\{L_1,\ldots,L_m\}$, Question~\ref{question:Betti} has a positive answer. 
\end{corollary}

\bibliographystyle{plain}
\bibliography{Bibliography}
\end{document}